\PassOptionsToPackage{top=2.5cm, bottom=2cm, left=2cm, right=2cm}{geometry}
\documentclass[final,3p,times, leqno]{elsarticle}

\makeatletter
\def\ps@pprintTitle{%
   \let\@oddhead\@empty
   \let\@evenhead\@empty
   \def\@oddfoot{\reset@font\hfil\thepage\hfil}
   \let\@evenfoot\@oddfoot
}
\makeatother

\journal{}

\usepackage[x11names]{xcolor}
\usepackage[
  colorlinks,
]{hyperref}

\AtBeginDocument{\hypersetup{citecolor=magenta}}

\newcommand\rurl[1]{%
  \href{http://#1}{\nolinkurl{#1}}%
}

 \usepackage{lmodern}
\usepackage[T5]{fontenc}
\usepackage{amscd}
\usepackage{graphicx}

\usepackage{amsthm, amssymb, amsfonts, mathrsfs, amsmath}

\usepackage{pb-diagram}

\usepackage{listings}

\theoremstyle{definition}
\newtheorem{dn}{Definition}[section]

\theoremstyle{definition}
\newtheorem{defn}{Definition}[subsection]
\newtheorem{rema}[defn]{Remark}

\newtheorem{note}[defn]{Note}

\theoremstyle{plain}
\newtheorem{thm}[defn]{Theorem}
\newtheorem{prop}[defn]{Proposition}
\newtheorem{conj}[defn]{Conjecture}
\newtheorem{corl}[defn]{Corollary}
\newtheorem{lema}[defn]{Lemma}

\theoremstyle{definition}

\newtheorem{rems}[defn]{Remark}

\def\DD{D\kern-.7em\raise0.4ex\hbox{\char '55}\kern.33em}

\usepackage{sectsty}
\subsectionfont{\fontsize{11.5}{11.5}\selectfont}

\begin{document}
\fontsize{11.5pt}{11.5}\selectfont

\begin{frontmatter}

\title{A note on the modular representation on the $\mathbb Z/2$-homology groups of the fourth power of real projective space and its application}

\author{\DD\d{\u a}ng V\~o Ph\'uc}
\address{{\fontsize{10pt}{10}\selectfont Faculty of Education Studies, University of Khanh Hoa,\\ 01 Nguyen Chanh, Nha Trang, Khanh Hoa, Viet Nam}}
\ead{dangphuc150488@gmail.com}

\begin{abstract}
We write $BV_h$ for the classifying space of the elementary Abelian 2-group $V_h$ of rank $h,$ which is homotopy equivalent to the cartesian product of $h$ copies of $\mathbb RP^{\infty}.$ Its cohomology with $\mathbb Z/2$-coefficients can be identified with the graded unstable algebra $P^{\otimes h} = \mathbb Z/2[t_1, \ldots, t_h]= \{P^{\otimes h}_n\}_{n\geq 0}$ over the Steenrod ring $\mathcal A$, where grading is by the degree of the homogeneous terms $P^{\otimes h}_n$ of degree $n$ in $h$ generators with the degree of each $t_i$ being one. Let $GL_h$ be the usual general linear group of rank $h$ over $\mathbb Z/2.$ The algebra $P^{\otimes h}$ admits a left action of $\mathcal A$ as well as a right action of $GL_h.$ A central problem of homotopy theory is to determine the structure of the space of $GL_h$-coinvariants,  $\mathbb Z/2\otimes_{GL_h}{\rm Ann}_{\overline{\mathcal A}}H_n(BV_h; \mathbb Z/2) ,$ where ${\rm Ann}_{\overline{\mathcal A}}H_n(BV_h; \mathbb Z/2) ={\rm Ann}_{\overline{\mathcal A}}[P^{\otimes h}_n]^{*}$ denotes the space of primitive homology classes, considered as a representation of $GL_h$ for all $n.$ Solving this problem is very difficult and still open for $h\geq 4.$

In this Note, our intent is of studying the dimension of $\mathbb Z/2\otimes_{GL_h}{\rm Ann}_{\overline{\mathcal A}}[P^{\otimes h}_n]^{*}$ for the case $h = 4$ and the "generic" degrees $n$ of the form $n_{k, r, s} = k(2^{s} - 1) + r.2^{s},$ where $k,\, r,\, s$ are positive integers. Applying the results, we investigate the behavior of the Singer cohomological "transfer" of rank $4.$ Singer's transfer is a homomorphism from a certain subquotient of the divided power algebra $\Gamma(a_1^{(1)}, \ldots, a_h^{(1)})$ to mod-2 cohomology groups ${\rm Ext}_{\mathcal A}^{h, h+*}(\mathbb Z/2, \mathbb Z/2)$ of the algebra $\mathcal A.$ Singer's algebraic transfer is one of the relatively efficient tools in determining the cohomology of the Steenrod algebra. Additionally, in higher ranks, by using the results on $\mathcal A$-generators for $P^{\otimes 5}$ and $P^{\otimes 6},$ we show in Appendix that the transfer of rank 5 is an isomorphism in som certain degrees of the form $n_{k, r, s}$, and that the transfer of rank 6 does not detect the non-zero elements $h_2^{2}g_1 = h_4Ph_2\in {\rm Ext}_{\mathcal A}^{6, 6+n_{6, 10, 1}}(\mathbb Z/2, \mathbb Z/2)$, and $D_2\in {\rm Ext}_{\mathcal A}^{6, 6+n_{6, 10, 2}}(\mathbb Z/2, \mathbb Z/2).$ Besides, we also probe the behavior of the Singer transfer of ranks 7 and 8 in internal degrees $\leq 15.$
\end{abstract}

\end{frontmatter}

\tableofcontents

\section{Introduction}

{\bf The hit problem and its dual.} Let $V_h$ be an elementary Abelian 2-group of rank $h$, which it will usually be more useful to consider as an $h$-dimensional vector space over the prime field $\mathbb Z/2$ of characteristic $2$. Denote by $BV_h$ the classifying space of $V_h,$ which is homotopy equivalent to the cartesian product of $h$ copies of the infinite dimensional projective space, $\mathbb RP^{\infty}.$ Writing $\mathcal A$ for the Steenrod algebra over $\mathbb Z/2,$ which is defined to be the algebra of stable cohomology operations for singular cohomology with coefficients in $\mathbb Z/2.$ The algebra $\mathcal A,$ which acts on the cohomology ring of a space, is widely studied by mathematicians whose interests range from Algebraic topology and homotopy theory, to manifold theory, combinatorics, and more. Let $P^{\otimes h} = \mathbb Z/2[t_1,\ldots, t_h] = \{P^{\otimes h}_n\}_{n\geq 0}$ be the graded unstable algebra over $\mathcal A,$ where each $t_i$ has grading one, and $P^{\otimes h}_n = \langle \{f\in P^{\otimes h}|\, \mbox{$f$ is a homogeneous polynomial of degree $n$}\}\rangle.$ It has been shown (see \cite{A.H.R}) that the ring of endomorphisms of the algebra $P^{\otimes h}$ over $\mathcal A$  is isomorphic to $\mathbb Z/2[M_h(\mathbb Z/2)],$ the ring of the semi-group ring of $h\times h$-matrices over $\mathbb Z/2$. So, we can be identified $P^{\otimes h}$ with the $\mathbb Z/2$-cohomology of $BV_h.$ As known, the algebra $\mathcal A$ is generated by the Steenrod squares $Sq^{i}$ ($i\geq 0$) and subject to Adem's relations: $Sq^iSq^j  =  \sum_{0\leq t\leq [i/2]}\binom{j-t-1}{i-2t}Sq^{i+j-t}Sq^t,\,\,0 < i < 2j,$ where the binomial coefficients are to be interpreted mod 2. The action of $\mathcal A$ over the polynomial algebra $P^{\otimes h} $ is determined by instability axioms. By Cartan's formula, it suffices to determine $Sq^{i}(t_j)$ and the instability axioms give $Sq^{1}(t_j)= t_j^{2}$ while $Sq^{i}(t_j) = 0$ if $i > 1.$  
One of the unsolved problems in Algebraic topology is to find minimal sets of $\mathcal A$-generators for $P^{\otimes h}.$ It is the same as the problem of investigating a basis for the space of the "non-hit" elements $\{QP_n^{\otimes h} := P_n^{\otimes h}/\overline{\mathcal A}P_n^{\otimes h} =  (\mathbb Z/2\otimes_{\mathcal A}P^{\otimes h})_n\}_{n\geq 0},$ where $\overline{\mathcal A}P_n^{\otimes h}:=P_n^{\otimes h}\cap \overline{\mathcal A}P^{\otimes h}$ and $\overline{\mathcal A}$ denotes the set of positive degree elements in $\mathcal A.$ This is called "hit problem" in literature \cite{Peterson} and is known for all integers $h$ less than or equal to $4$ (see \cite{Peterson}, \cite{Kameko}, \cite{Sum1, Sum2}). Though much work has been done (see, for instance \cite{Phuc-Sum, Phuc3, Phuc4, Phuc5, Phuc6, Phuc7, Phuc10, Phuc11}, \cite{MM}, \cite{MKR}, \cite{Sum4}, \cite{Tin, Tin3}, \cite{Walker-Wood}), solving the hit problem in the general case seems to be out of reach with the present techniques, even when $h$ equal to $5$ and in some certain generic degrees. Remarkably, by the works \cite{Kameko}, \cite{Sum4}, and \cite{Wood}, we need only to study this problem in degrees $n$ of the form  $ k(2^{s} - 1) + r.2^{s},$ whenever $k, s, r$ are positive integers satisfying $\mu(r) < k < h.$ Here $\mu(r)$ denotes the smallest integer $m$ such that $r$ can be represented as $r = \sum_{1\leq i\leq m}(2^{u_i}-1),$ where $u_i > 0.$ Several other aspects of the hit problem has been discovered in the recent works of Ault-Singer \cite{A.S}, Inoue \cite{Inoue},  Janfada-Wood \cite{J.W}, Pengelley-William \cite{P.W}. The dual of the hit problem is to determine a subring of elements of the Pontrjagin ring $H_*(BV_h; \mathbb Z/2) = [P^{\otimes h}]^{*},$ which is mapped to zero by all Steenrod squares of positive degrees, frequently denoted by ${\rm Ann}_{\overline{\mathcal A}}[P^{\otimes h}]^{*}.$ 

\medskip

{\bf The Singer cohomological transfer.} We knew that the $\mathbb Z/2$-cohomology of the Steenrod algebra is the $E_2$-term for the (2-local) Adams spectral sequence, whose abutment is the 2-component of the stable homotopy groups of spheres. More specifically, the $E_2$-page of this spectral sequence may be identified as ${\rm Ext}_{\mathcal A}^{*,*}(\mathbb Z/2, \mathbb Z/2).$ This is what is meant by the aphorism "the cohomology of the Steenrod algebra is an approximation to the stable homotopy groups of spheres".
So, studying the structure of the Adams $E_2$-terms becomes very important in homotopy theory. The study of the hit problem and its dual is closely related to describe these $E_2$-terms via a cohomological "transfer" of William Singer \cite{Singer}. This transfer is constructed as follows. Consider the polynomial ring of one variable, $P^{\otimes 1} = \mathbb Z/2[t_1].$ The canonical $\mathcal A$-action on $P^{\otimes 1}$ is extended to an $\mathcal A$-action on $\mathbb Z/2[t_1, t_1^{-1}],$ the ring of finite Laurent series with require that $Sq^{n}(t_1^{-1}) = t_1^{n-1}$ for all $n\geq 1.$  Then, $\mathcal P = \langle \{t_1^{a}|\ a\geq -1\}\rangle$ is $\mathcal A$-submodule of $\mathbb Z/2[t_1, t_1^{-1}].$ One has a short-exact sequence:  
\begin{equation}\label{dkn1}
 0\to P^{\otimes 1}\xrightarrow{q} \mathcal{P}\xrightarrow{\pi} \Sigma^{-1}\mathbb Z/2,
\end{equation}
where $q$ is the inclusion and $\pi$ is given by $\pi(t_1^{a}) = 0$ if $a\neq -1$ and $\pi(t_1^{-1}) = 1.$ Denote by $e_1$ the element in ${\rm Ext}_{\mathcal A}^1(\Sigma^{-1}\mathbb Z/2, P^{\otimes 1}),$ which is represented by the cocyle associated to \eqref{dkn1}. For each $i\geq 1,$ the short-exact sequence
\begin{equation}\label{dkn2}
 0\to P^{\otimes (i+1)}\cong P^{\otimes i}\otimes_{\mathbb Z/2} P^{\otimes 1}\xrightarrow{1\otimes_{\mathbb Z/2}q} P^{\otimes i}\otimes_{\mathbb Z/2}\mathcal{P}\xrightarrow{1\otimes_{\mathbb Z/2}\pi} \Sigma^{-1}P^{\otimes i},
\end{equation}
determines a class $(e_1\times P^{\otimes i})\in {\rm Ext}_{\mathcal A}^1(\Sigma^{-1}P^{\otimes i}, P^{\otimes (i+1)}).$ Then, using the cross product and Yoneda product, we have the element
$$ e_h = (e_1\times P^{\otimes (h-1)})\circ (e_1\times P^{\otimes (h-2)})\circ \cdots \circ (e_1\times P^{\otimes 1})\circ e_1 \in {\rm Ext}_{\mathcal A}^h(\Sigma^{-h}\mathbb Z/2, P^{\otimes h}).$$
Let $\Delta(e_1\times P^{\otimes i}): {\rm Tor}_{h-i}^{\mathcal A}(\mathbb Z/2, \Sigma^{-h}P^{\otimes i})\to {\rm Tor}_{h-i-1}^{\mathcal A}(\mathbb Z/2, P^{\otimes (i+1)})$ be the connecting homomorphism associated to \eqref{dkn2}. Then, we have a composition of the connecting homomorphisms $$ \overline{\varphi}^{\mathcal A}_h = \Delta(e_1\times  P^{\otimes (h-1)})\circ \Delta(e_1\times P^{\otimes (h-2)})\circ \cdots \circ \Delta(e_1\times P^{\otimes 1})\circ \Delta(e_1)$$from ${\rm Tor}_h^{\mathcal A}(\mathbb Z/2, \Sigma^{-h}\mathbb Z/2)$ to ${\rm Tor}_0^{\mathcal A}(\mathbb Z/2, P^{\otimes h}) = QP^{\otimes h} = \mathbb Z/2\otimes_{\mathcal A}P^{\otimes h}$, determined by \mbox{$\overline{\varphi}_h(z) = e_h\cap z$} for any $z\in {\rm Tor}_h^{\mathcal A}(\mathbb Z/2, \Sigma^{-h}\mathbb Z/2).$ Here $\cap$ denotes the \textit{cap product} in homology with $\mathbb Z/2$-coefficients. The image of $\overline{\varphi}_h$ is a submodule of the invariant space $[QP^{\otimes h}]^{GL_h}.$ Hence, $\overline{\varphi}^{\mathcal A}_h$ induces homomorphism $ \varphi_h^{\mathcal A}: {\rm Tor}_h^{\mathcal A}(\mathbb Z/2, \Sigma^{-h}\mathbb Z/2)\to [QP^{\otimes h}]^{GL_h}.$ Because the supension $\Sigma^{-h}$ induces an isomorphism ${\rm Tor}_{h, n}^{\mathcal A}(\mathbb Z/2, \Sigma^{-h}\mathbb Z/2)\cong {\rm Tor}_{h, h+n}^{\mathcal A}(\mathbb Z/2, \mathbb Z/2),$ we have the homomorphism $$ \varphi_h^{\mathcal A}: {\rm Tor}_{h, h+n}^{\mathcal A}(\mathbb Z/2, \mathbb Z/2)\to [QP^{\otimes h}_{n}]^{GL_h}.$$ Then its dual $Tr_h^{\mathcal A}: \mathbb Z/2\otimes_{GL_h}{\rm Ann}_{\overline{\mathcal A}}[P_{n}^{\otimes h}]^{*}\to {\rm Ext}_{\mathcal A}^{h, h+n}(\mathbb Z/2, \mathbb Z/2)$ is called the \textit{cohomological transfer} (or \textit{algebraic transfer}) of rank $h$. This transfer can be seen as an algebraic formulation of a stable transfer $B(V_h)_+^{S}\to S^{0}.$ It has been shown that $Tr_h^{\mathcal A}$ is an isomorphism for $h$ less than or equal to $3$ (see the work of Singer \cite{Singer} and of Boardman \cite{Boardman}). In addition, the "total" algebraic transfer $\{Tr_h^{\mathcal A}\}_{h\geq 0}$ is an algebraic homomorphism (see also \cite{Singer}). These events show that $Tr_h^{\mathcal A}$ is highly nontrivial and should be an useful tool to study the mysterious Ext groups. Specifically, Singer \cite{Singer} states the following yet left-open: \textit{$Tr_h^{\mathcal A}$ is a monomorphism, for all positive integers $h$}. One should note that by the works of Minami \cite{Minami, Minami2}, hit problems and this transfer also make a useful contribution for surveying the Kervaire invariant one problem. A review of known results on this problem is necessary for the interested readers. The Kervaire invariant was first introduced by Browder's work \cite{Browder}, where he shows that the classes $h_j^{2}\in {\rm Ext}_{\mathcal {A}}^{2,2^{j+1}}(\mathbb Z/2,\mathbb Z/2)$ are the permanent cycles in the classical Adams spectral sequence at the prime 2 \cite{Adams}, if and only if smooth framed manifolds of Kervaire invariant one exist only in dimensions $2^{j+1}-2.$ The hypothetical element in stable $2^{j+1}-2$-stem, $\pi^{S}_{2^{j+1}-2},$ represented by such a framed manifold is denoted by $\theta_j.$ These $\theta_j$ are known to be exist for $j\leq 5$ (see, for instance Lin, and Mahowald \cite{L.M}). Many open issues in Algebraic and Differential topology depend on knowing whether or not the Kervaire invariant one elements $\theta_j$ for $j\geq 6.$ In 2016, Hill, Hopkins, and Ravenel \cite{Hill} indicated that these elements do not exist for $j\geq 7$. The case $j = 6$ is still no answer. Thus, smooth framed manifolds of Kervaire invariant one therefore exist only in dimensions $2, 6, 14, 30, 62,$ and possibly $126.$

\medskip

{\bf The lambda algebra.} Beyond Singer's algebraic transfer, the mod 2 lambda algebra $\Lambda$ of six authors \cite{Bousfield} is also an important tool for computing the cohomology of $\mathcal A.$ This algebra considered as the $E_1$-term of the Adams spectral sequence converging to the 2-component of the stable homotopy groups of spheres. (Note that a spectral sequence arises from a filtration of the dual chain complex and it provides an alternative way to determine the cohomology of the dual chain complex.) Let us briefly review the related concepts of this algebra. Firstly, we know that $\Lambda$ is an associative differential bigraded algebra with generators $\lambda_{n}\in \Lambda^{1, n}$ ($n\geq 0$) and the Adem relations
\begin{equation}\label{qh}
\lambda_i\lambda_{2i+n+1} = \sum_{j\geq 0}\binom{n-j-1}{j}\lambda_{i+n-j}\lambda_{2i+1+j} \ \ \ (i\geq 0,\ n\geq 0)
\end{equation}
with differential 
\begin{equation}\label{vp}
\delta(\lambda_{n-1}) = \sum_{j\geq 1}\binom{n-j-1}{j}\lambda_{n-j-1}\lambda_{j-1}\  \ \ (n\geq 1),
\end{equation}
where the binomial coefficients $\binom{n-j-1}{j}$ modulo $2.$ Now, for non-negative integers $j_1, \ldots, j_h,$ the element $\lambda_{j_1}\ldots\lambda_{j_h}$ in $\Lambda$ is said to be \textit{a monomial of length $h$}. Then, by the relations \eqref{qh}, the $\mathbb Z/2$-vector subspace $\Lambda^{h, n} = \langle \{\lambda_{j_1}\ldots \lambda_{j_h}|\ j_m\geq 0,\ 1\leq m \leq h,\ \sum_{1\leq m\leq h}j_m = n\}\rangle$ of $\Lambda$ has an additive basis consisting of all admissible monomials of length $h$ (i.e., those of the form $\lambda_{j_1}\ldots \lambda_{j_h}$ where $j_i\leq 2j_{i+1}$ for all $0< i <h$.) Moreover, it is known that there is an endomorphism $Sq^{0}$ of  $\Lambda,$ determined by $Sq^{0}(\lambda_{j_1}\ldots \lambda_{j_h}) = \lambda_{2j_1+1}\ldots \lambda_{2j_h+1}.$ Further, this $Sq^{0}$ respects the relations in \eqref{qh} and commutes with the differential $\delta$ in \eqref{vp}. So, it induces the first Steenrod operation in Ext groups (see Palmieri \cite{Palmieri} for more details). It is known that in standard notation,  $H_*(V_h; \mathbb Z/2) = H_*(BV_h; \mathbb Z/2) = [P^{\otimes h}]^{*} = \Gamma(x_1^{(1)}, \ldots, x_h^{(1)}),$ the divided polynomial algebra on $V_h$ over $\mathbb Z/2,$ generated by $x_1^{(1)}, \ldots, x_h^{(1)},$ where $x_i^{(1)}$ is linear dual to $t_i.$ The (right) action of the Steenrod algebra on $[P^{\otimes h}]^{*}$ is determined by the usual Cartan formula and for $k > 0,$ 
$$ (x^{(n)}_i)Sq^{k} = \left\{\begin{array}{ll}
 \binom{n-k}{k}x^{(n-k)}&\mbox{if $2k < n$},\\
0&\mbox{otherwise}.
\end{array}\right.$$
This shows that $[P^{\otimes h}]^{*}$ has a natural right $\mathcal A$-module structure. An interesting $\mathbb Z/2$-linear map $\psi_h$ from $[P^{\otimes h}_{n}]^{*}$ to $\Lambda^{h, n},$ established by Ch\ohorn n and H\`a \cite{C.H2}, is determined by \mbox{$\psi_h(x_1^{(j_1)}\ldots x_h^{(j_h)}) =  \lambda_{j_h}$} if \mbox{$h = 1,$} and $\psi_h(x_1^{(j_1)}\ldots x_h^{(j_h)}) = \sum_{k\geq j_h}\psi_{h-1}(((x_{1}^{(j_{1})}\ldots x_{h-1}^{(j_{h-1})}))Sq^{k-j_h})\lambda_k$ otherwise. The authors indicated that this homomorphism can be considered as a representation in the algebra $\Lambda$ of the algebraic transfer, and that if $\zeta\in {\rm Ann}_{\overline{\mathcal A}}[P_{n}^{\otimes h}]^{*},$ then the image of $\zeta$ under $\psi_h$ is a cycle in $\Lambda^{h, n}$ and is a representative of the image of the class $[\zeta]$ under $Tr_h^{\mathcal A}.$ 

\medskip

The structure of the coinvariant space $\mathbb Z/2\otimes_{GL_h}{\rm Ann}_{\overline{\mathcal A}}[P_n^{\otimes h}]^{*}$ and the behavior of the Singer transfer have been proved surprisingly difficult by many researchers; for instance the interested reader can see the works by Boardman \cite{Boardman}, 
 Ch\ohorn n and H\`a \cite{C.H1, C.H2}, Crabb-Hubbuck \cite{C.H}, Crossley \cite{Crossley}, H\`a \cite{Ha}, H\uhorn ng \cite{Hung}, Minami \cite{Minami}, Nam \cite{Nam}, the present writer \cite{Phuc3, Phuc4, Phuc6, Phuc7, Phuc9, Phuc10, Phuc11}, Sum \cite{Sum3, Sum4}, etc. These problems are still unknown for $h\geq 4.$ In the rank 4 case, by the work of Sum \cite{Sum1}, it is sufficient to study them in the "generic degrees" $n$ of the following forms: 
\begin{enumerate}
\item[(i)] $2^{s+1} - i$ for $1\leq i\leq 3,$
\item[(ii)] $2^{s+m+1} +2^{s+1} - 3,$
\item [(iii)] $2^{s+m} + 2^{s} - 2,$
\item[(iv)] $2^{s+m+u} + 2^{s+m} + 2^{s} - 3,$
\end{enumerate}
for any positive integers $s,\, m,\, u$. It is easy to see that all these degrees $n$ can be written as $n = k(2^{s}-1) + r.2^{s},$ where $k,\, r,\, s$ are positive integers and $0<\mu(r) < k < 4.$  The information on the behavior of $Tr_4^{\mathcal A}$ in degrees of the form (i) is known by Sum \cite{Sum3}. The remaining cases have been partially studied by the present author \cite{Phuc9, Phuc11}. 

In this work, we will continue to probe the problems above for the case of the generic degree $n = 2^{s+m} + 2^{s} - 2$ with $m\in \{2, 4\}$ and $s$ an arbitrary positive integer. Thence, we show that the Singer transfer is an isomorphism in bidegree $(4, 4+n).$ Moreover, in Appendix, basing the works by Bruner \cite{Bruner}, Mothebe, and Uys \cite{Mothebe}, Moetele, and Mothebe \cite{MM}, Mothebe, Kaelo, and Ramatebele \cite{MKR}, Lin \cite{Lin}, Sum \cite{Sum1, Sum2, Sum, Sum4}, Walker, and Wood \cite{Walker-Wood}, Tin \cite{Tin, Tin3}, and our previous works \cite{Phuc4, Phuc5, Phuc6}, we would like to survey the behavior of the transfer of ranks $> 4$ in some internal degrees. Most notably, in the rank 6 case, we shall show that the transfer $Tr_6^{\mathcal A}$ does not detect the non-zero elements:
$$ \begin{array}{ll}
\medskip
& h_1Ph_1\in {\rm Ext}_{\mathcal A}^{6, 6+10}(\mathbb Z/2, \mathbb Z/2),\ \ \ \ \ \ \ \  \ \ \  h_0Ph_2\in {\rm Ext}_{\mathcal A}^{6, 6+11}(\mathbb Z/2, \mathbb Z/2),\\ &h_2^{2}g_1 = h_4Ph_2\in {\rm Ext}_{\mathcal A}^{6, 6+26}(\mathbb Z/2, \mathbb Z/2),\ D_2\in {\rm Ext}_{\mathcal A}^{6, 6+58}(\mathbb Z/2, \mathbb Z/2),
\end{array}$$
 and that in the rank 7 case, the transfer $Tr_7^{\mathcal A}$ does not detect the non-zero element $h_0^{2}Ph_2= h_1^{2}Ph_1\in {\rm Ext}_{\mathcal A}^{7, 7+11}(\mathbb Z/2, \mathbb Z/2).$ The results for the elements $h_1Ph_1,\, h_0Ph_2,$ and $h_0^{2}Ph_2= h_1^{2}Ph_1$ were also proved by Ch\ohorn n, and H\`a \cite{C.H1, C.H2} using other techniques. In addition, in ranks $h = 6,\, 7,\, 8,$ we refute the results by Moetele, and Mothebe \cite{MM} on the dimensions of $QP_{13}^{\otimes h}$, from which we also indicate that the Singer transfer is an isomorphism in bidegree $(h, h+13).$ Our method is based on techniques of the hit problem of four variables in \cite{Sum1, Sum2} and a representation in the lambda algebra of $Tr_h^{\mathcal A}$. The techniques can be applied to study Singer's transfer of a higher rank in some certain generic degrees. Our approach is different from the ones of  Singer \cite{Singer}, H\uhorn ng \cite{Hung}, Ch\ohorn n, and H\`a \cite{C.H1, C.H2},  Nam \cite{Nam} in studying the image of the cohomological transfer.

\section{Main results}

As mentioned above, and more precisely, we shall explicitly determine the structure of the coinvariants $\mathbb Z/2 \otimes_{GL_4} {\rm Ann}_{\overline{\mathcal A}}[P_{n}^{\otimes 4}]^{*}$ at the generic degrees $n = 5.2^{s}-2$ and $17.2^{s}-2$ for every positive integer $s.$  Then, we claim that the fourth cohomological transfer is an isomorphism in these internal degrees. The content of the present Note subsection-wise can be described as follows.

\subsection{The degree \boldmath{$n_s :=5.2^{s}-2$}}\label{s2.1}

Notice that the Kameko homomorphism 
$$ \begin{array}{ll}
(\widetilde {Sq^0_*})_{n_s}: QP^{\otimes 4}_{n_s}  &\longrightarrow QP^{\otimes 4}_{\frac{n_s-4}{2}}.\\
\ \ \ \ \ \ \ \ \ \ \ \ \ \ \ \ \mbox{[}\prod_{1\leq j\leq 4}t_j^{a_j}\mbox{]}&\longmapsto \left\{\begin{array}{ll}
\mbox{[}\prod_{1\leq j\leq 4}t_j^{\frac{a_j-1}{2}}\mbox{]}& \text{if $a_j$ odd, for all $j$},\\
0& \text{otherwise}
\end{array}\right.
\end{array}$$ 
is an epimorphism of $\mathbb Z/2[GL_4]$-modules, hence we have an isomorphism $$QP^{\otimes 4}_{n_s} \cong {\rm Ker}((\widetilde {Sq^0_*})_{n_s}) \bigoplus QP^{\otimes 4}_{\frac{n_s-4}{2}},$$
which implies that
\begin{equation}\label{bdt}
 \dim (\mathbb Z/2 \otimes_{GL_4} {\rm Ann}_{\overline{\mathcal A}}[P_{n_s}^{\otimes 4}]^{*})\leq \dim ([{\rm Ker}((\widetilde {Sq^0_*})_{n_s})]^{GL_4}) + \dim (\mathbb Z/2 \otimes_{GL_4} {\rm Ann}_{\overline{\mathcal A}}[P_{\frac{n_s-4}{2}}^{\otimes 4}]^{*}).
\end{equation}
According to Lin \cite{Lin}, we deduce that
\begin{equation}\label{kqL}
{\rm Ext}_{\mathcal A}^{4, 4+n_s}(\mathbb Z/2, \mathbb Z/2) = \left\{\begin{array}{ll}
\langle h_1^{2}h_2^{2} \rangle = 0 &\mbox{if $s = 1$},\\[1mm]
\langle h_1^{3}h_4, f_0 \rangle =\langle h_0^{2}h_2h_4, f_0 \rangle &\mbox{if $s = 2$},\\[1mm]
\langle h_0^{2}h_3h_5, e_1 \rangle  &\mbox{if $s = 3$},\\[1mm]
\langle h_0^{2}h_sh_{s+2} \rangle  &\mbox{if $s\geq 4$}.
\end{array}\right.
\end{equation}
By using these data and an admissible basis of ${\rm Ker}((\widetilde {Sq^0_*})_{n_s})$ in \cite{Sum1, Sum2}, we obtain the following results.

{\bf The case \boldmath{$s = 1$}}. We have

 \begin{thm}\label{dlc1}
The coinvariant space $\mathbb Z/2 \otimes_{GL_4} {\rm Ann}_{\overline{\mathcal A}}[P_{n_1}^{\otimes 4}]^{*}$ is trivial.
\end{thm}

In order to prove the theorem, we use a basis of ${\rm Ker}((\widetilde {Sq^0_*})_{n_1})$ and show that ${\rm Ker}((\widetilde {Sq^0_*})_{n_1})$ is trivial. At the same time, combining the inequality \eqref{bdt} and a theorem of Sum \cite{Sum3} that the coinvariant space $(\mathbb Z/2 \otimes_{GL_4} {\rm Ann}_{\overline{\mathcal A}}[P_{\frac{n_1-4}{2}}^{\otimes 4}]^{*})$ is trivial. Note that, this result has also been investigated in \cite{B.H.H}, in a different way. Now, by Theorem \ref{dlc1} and the equality \eqref{kqL}, it may be concluded that

\begin{corl}\label{hqc1}
The rank 4 algebraic transfer $$Tr_4^{\mathcal A}: \mathbb Z/2 \otimes_{GL_4} {\rm Ann}_{\overline{\mathcal A}}[P_{n_1}^{\otimes 4}]^{*}\to {\rm Ext}_{\mathcal A}^{4, 4+n_1}(\mathbb Z/2, \mathbb Z/2)$$
is a trivial isomorphism.
\end{corl}

Moreover, we have the following remarks.

\begin{rems}
We consider the generic degree $12.2^{s} - 4$ with $s$ an arbitrary positive integer. It is easy to see that $\mu(12.2^{s}-4) = 4$ for all $s > 0.$ (Recall that for each positive integer $n,$ by $\mu(n)$ one means the smallest
number $k$ for which it is possible to write $n = \sum_{1\leq i\leq k}(2^{d_i}-1),$ in which $d_i > 0.$) Then, by Kameko's Theorem \cite{Kameko}, it follows that the iterated Kameko map $$ (\widetilde {Sq^0_*})^{s}_{12.2^{s}-4}: QP^{\otimes 4}_{12.2^{s}-4}  \longrightarrow QP^{\otimes 4}_{8}$$ is an isomorphism of $\mathbb Z/2GL_4$-modules, for any $s\geq 0.$ Combining this fact and Theorem \ref{dlc1}, we conclude that the coinvariants $\mathbb Z/2 \otimes_{GL_4} {\rm Ann}_{\overline{\mathcal A}}[P_{12.2^{s}-4}^{\otimes 4}]^{*}$ is trivial for every non-negative integer $s.$ This also explained in \cite{B.H.H}, in another way. On the other hand, by calculations of Lin \cite{Lin}, clearly ${\rm Ext}_{\mathcal A}^{4, 12.2^{s}}(\mathbb Z/2, \mathbb Z/2) = \mathbb Z/2. g_s$ for all $s > 0.$ Therefore, the transfer homomorphism
$$Tr_4^{\mathcal A}: \mathbb Z/2 \otimes_{GL_4} {\rm Ann}_{\overline{\mathcal A}}[P_{12.2^{s}-4}^{\otimes 4}]^{*}\to {\rm Ext}_{\mathcal A}^{4, 12.2^{s}}(\mathbb Z/2, \mathbb Z/2)$$
is a monomorphism, but not an isomorphism for arbitrary $s > 0.$
\end{rems}

{\bf The case \boldmath{$s = 2$}}. It is known, from a theorem of Sum \cite{Sum3}, that $\mathbb Z/2 \otimes_{GL_4} {\rm Ann}_{\overline{\mathcal A}}[P_{\frac{n_2-4}{2}}^{\otimes 4}]^{*}$ is generated by $[x_1^{(0)}x_2^{(0)}x_3^{(0)}x_4^{(7)}].$ From this and a direct computation using a basis of ${\rm Ker}((\widetilde {Sq^0_*})_{n_2})$ and the inequality \eqref{bdt}, we obtain 

\begin{thm}\label{dlc2}
$\mathbb Z/2 \otimes_{GL_4} {\rm Ann}_{\overline{\mathcal A}}[P_{n_2}^{\otimes 4}]^{*}$ is the vector space of dimension $2$ with a basis consisting of 2 classes represented by $x_1^{(1)}x_2^{(1)}x_3^{(1)}x_4^{(15)}$ and 
$$ \begin{array}{ll}
\zeta&=  x_1^{(3)}x_2^{(5)}x_3^{(1)}x_4^{(9)}+
 x_1^{(3)}x_2^{(5)}x_3^{(2)}x_4^{(8)}+
 x_1^{(3)}x_2^{(6)}x_3^{(1)}x_4^{(8)}+
 x_1^{(3)}x_2^{(6)}x_3^{(2)}x_4^{(7)}+
\medskip
 x_1^{(3)}x_2^{(5)}x_3^{(4)}x_4^{(6)}\\
&\quad+
 x_1^{(3)}x_2^{(6)}x_3^{(3)}x_4^{(6)}+
 x_1^{(5)}x_2^{(6)}x_3^{(1)}x_4^{(6)}+
 x_1^{(3)}x_2^{(5)}x_3^{(5)}x_4^{(5)}+
 x_1^{(3)}x_2^{(6)}x_3^{(4)}x_4^{(5)}+
\medskip
 x_1^{(5)}x_2^{(6)}x_3^{(2)}x_4^{(5)}\\
&\quad+
 x_1^{(3)}x_2^{(9)}x_3^{(1)}x_4^{(5)}+
 x_1^{(5)}x_2^{(7)}x_3^{(1)}x_4^{(5)}+
 x_1^{(3)}x_2^{(9)}x_3^{(2)}x_4^{(4)}+
 x_1^{(5)}x_2^{(7)}x_3^{(2)}x_4^{(4)}+
\medskip
 x_1^{(3)}x_2^{(10)}x_3^{(1)}x_4^{(4)}\\
&\quad+
 x_1^{(6)}x_2^{(7)}x_3^{(1)}x_4^{(4)}+
 x_1^{(5)}x_2^{(6)}x_3^{(4)}x_4^{(3)}+
 x_1^{(6)}x_2^{(7)}x_3^{(2)}x_4^{(3)}+
 x_1^{(3)}x_2^{(10)}x_3^{(2)}x_4^{(3)}+
\medskip
 x_1^{(3)}x_2^{(11)}x_3^{(2)}x_4^{(2)}\\
&\quad+
 x_1^{(5)}x_2^{(9)}x_3^{(2)}x_4^{(2)}+
 x_1^{(6)}x_2^{(10)}x_3^{(1)}x_4^{(1)}+
 x_1^{(3)}x_2^{(5)}x_3^{(9)}x_4^{(1)}+
 x_1^{(3)}x_2^{(5)}x_3^{(8)}x_4^{(2)}+
\medskip
 x_1^{(3)}x_2^{(6)}x_3^{(8)}x_4^{(1)}\\
&\quad+
 x_1^{(3)}x_2^{(6)}x_3^{(7)}x_4^{(2)}+
 x_1^{(3)}x_2^{(5)}x_3^{(6)}x_4^{(4)}+
 x_1^{(3)}x_2^{(6)}x_3^{(6)}x_4^{(3)}+
 x_1^{(5)}x_2^{(6)}x_3^{(6)}x_4^{(1)}+
\medskip
 x_1^{(3)}x_2^{(5)}x_3^{(5)}x_4^{(5)}\\
&\quad+
 x_1^{(3)}x_2^{(6)}x_3^{(5)}x_4^{(4)}+
 x_1^{(5)}x_2^{(6)}x_3^{(5)}x_4^{(2)}+
 x_1^{(3)}x_2^{(9)}x_3^{(5)}x_4^{(1)}+
 x_1^{(5)}x_2^{(7)}x_3^{(5)}x_4^{(1)}+
\medskip
 x_1^{(3)}x_2^{(9)}x_3^{(4)}x_4^{(2)}\\
&\quad+
 x_1^{(5)}x_2^{(7)}x_3^{(4)}x_4^{(2)}+
 x_1^{(3)}x_2^{(10)}x_3^{(4)}x_4^{(1)}+
 x_1^{(6)}x_2^{(7)}x_3^{(4)}x_4^{(1)}+
 x_1^{(5)}x_2^{(6)}x_3^{(3)}x_4^{(4)}+
\medskip
 x_1^{(6)}x_2^{(7)}x_3^{(3)}x_4^{(2)}\\
&\quad+
 x_1^{(3)}x_2^{(10)}x_3^{(3)}x_4^{(2)}+
 x_1^{(3)}x_2^{(11)}x_3^{(2)}x_4^{(2)}+
 x_1^{(5)}x_2^{(9)}x_3^{(2)}x_4^{(2)}+
 x_1^{(6)}x_2^{(10)}x_3^{(1)}x_4^{(1)}+
\medskip
 x_1^{(3)}x_2^{(12)}x_3^{(1)}x_4^{(2)}\\
&\quad+
 x_1^{(7)}x_2^{(8)}x_3^{(1)}x_4^{(2)}+
 x_1^{(11)}x_2^{(4)}x_3^{(1)}x_4^{(2)}+
 x_1^{(13)}x_2^{(2)}x_3^{(1)}x_4^{(2)}+
 x_1^{(14)}x_2^{(1)}x_3^{(1)}x_4^{(2)}+
\medskip
 x_1^{(12)}x_2^{(3)}x_3^{(1)}x_4^{(2)}\\
&\quad+
 x_1^{(8)}x_2^{(7)}x_3^{(1)}x_4^{(2)}+
 x_1^{(4)}x_2^{(11)}x_3^{(1)}x_4^{(2)}+
 x_1^{(2)}x_2^{(13)}x_3^{(1)}x_4^{(2)}+
 x_1^{(1)}x_2^{(14)}x_3^{(1)}x_4^{(2)}+
\medskip
 x_1^{(6)}x_2^{(6)}x_3^{(3)}x_4^{(3)}\\
&\quad+
 x_1^{(5)}x_2^{(5)}x_3^{(5)}x_4^{(3)}+
 x_1^{(3)}x_2^{(3)}x_3^{(9)}x_4^{(3)}+
 x_1^{(5)}x_2^{(3)}x_3^{(7)}x_4^{(3)}+
 x_1^{(7)}x_2^{(7)}x_3^{(2)}x_4^{(2)}+
\medskip
 x_1^{(6)}x_2^{(9)}x_3^{(1)}x_4^{(2)}\\
&\quad+
 x_1^{(9)}x_2^{(6)}x_3^{(1)}x_4^{(2)}+
 x_1^{(10)}x_2^{(5)}x_3^{(1)}x_4^{(2)}+
 x_1^{(5)}x_2^{(10)}x_3^{(2)}x_4^{(1)}+
 x_1^{(13)}x_2^{(3)}x_3^{(1)}x_4^{(1)}+
\medskip
 x_1^{(5)}x_2^{(11)}x_3^{(1)}x_4^{(1)}\\
&\quad+
 x_1^{(9)}x_2^{(7)}x_3^{(1)}x_4^{(1)}.
\end{array}$$
\end{thm}

To check that $\zeta_2\in {\rm Ann}_{\overline{\mathcal A}}[P_{n_2}^{\otimes 4}]^{*},$ we need only to consider the effects of the Steenrod squares $Sq^{i}$ for $i = 1, 2, 4, 8$ because of the unstable condition. Now, noticing that $\lambda_1\in \Lambda^{1, 1},$ $\lambda_{15}\in \Lambda^{1, 15}$ and $\widetilde{f}_0 = \lambda_4\lambda_6\lambda_5\lambda_3 + \lambda_5\lambda_7\lambda^2_3 + \lambda_3^2\lambda_2\lambda_5\lambda_7 + \lambda_2\lambda_4\lambda_5\lambda_7  \in \Lambda^{4, n_2}$ are the cycles in the algebra $\Lambda.$ Moreover, they are representative of the non-zero elements \mbox{$h_i\in {\rm Ext}_{\mathcal A}^{1, 2^{i}}(\mathbb Z/2, \mathbb Z/2),$} for $i = 1, 4$ and $f_0\in {\rm Ext}_{\mathcal A}^{4, 4+n_2}(\mathbb Z/2, \mathbb Z/2)$ respectively. Then, using Theorems \ref{dlc2} and the representation of $Tr_4^{\mathcal A}$ over $\Lambda,$ we see that the cycles $ \lambda_1^{3}\lambda_{15} = \psi_4(x_1^{(1)}x_2^{(1)}x_3^{(1)}x_4^{(15)})$ and $\widetilde{f}_0 +  \delta(\lambda_3\lambda_5\lambda_{11}) = \psi_4(\zeta)$ in the lambda algebra are representative of the non-zero elements $h_0^{2}h_2h_4$ and $f_0$ respectively. This shows that $h_1^{3}h_4$ and $f_0$ are in the image of $Tr_4^{\mathcal A}.$ Combining this with Theorem \ref{dlc2} and the equality \eqref{kqL}, it follows that

\begin{corl}\label{hqc2}
The fourth cohomological transfer is a isomorphism in the internal degree $n_2.$
\end{corl}

{\bf The case \boldmath{$s = 3$}.} We consider the element $\widetilde{\zeta}\in [P_{n_3}^{\otimes 4}]^{*},$ which is the following sum:
$$ \begin{array}{ll}
&x_1^{(11)}x_2^{(11)}x_3^{(11)}x_4^{(5)}+
 x_1^{(11)}x_2^{(11)}x_3^{(13)}x_4^{(3)}+
 x_1^{(7)}x_2^{(11)}x_3^{(17)}x_4^{(3)}+
x_1^{(11)}x_2^{(7)}x_3^{(17)}x_4^{(3)} +
\medskip
 x_1^{(7)}x_2^{(13)}x_3^{(15)}x_4^{(3)}\\
&+ x_1^{(11)}x_2^{(15)}x_3^{(9)}x_4^{(3)}+
x_1^{(15)}x_2^{(11)}x_3^{(9)}x_4^{(3)}+
 x_1^{(7)}x_2^{(19)}x_3^{(9)}x_4^{(3)}+
 x_1^{(19)}x_2^{(7)}x_3^{(9)}x_4^{(3)}+
\medskip
x_1^{(7)}x_2^{(19)}x_3^{(7)}x_4^{(5)}\\
&+
x_1^{(19)}x_2^{(7)}x_3^{(7)}x_4^{(5)}+
x_1^{(11)}x_2^{(19)}x_3^{(5)}x_4^{(3)}+
x_1^{(19)}x_2^{(11)}x_3^{(5)}x_4^{(3)} +
x_1^{(11)}x_2^{(21)}x_3^{(3)}x_4^{(3)}+
\medskip
x_1^{(19)}x_2^{(13)}x_3^{(3)}x_4^{(3)}\\
&+
 x_1^{(7)}x_2^{(23)}x_3^{(5)}x_4^{(3)}+
 x_1^{(23)}x_2^{(7)}x_3^{(5)}x_4^{(3)} +
 x_1^{(11)}x_2^{(11)}x_3^{(7)}x_4^{(9)}+
 x_1^{(11)}x_2^{(7)}x_3^{(11)}x_4^{(9)}+
\medskip
 x_1^{(7)}x_2^{(11)}x_3^{(11)}x_4^{(9)}\\
&+
 x_1^{(7)}x_2^{(25)}x_3^{(3)}x_4^{(3)}+
  x_1^{(23)}x_2^{(9)}x_3^{(3)}x_4^{(3)}+
 x_1^{(15)}x_2^{(17)}x_3^{(3)}x_4^{(3)}+
 x_1^{(15)}x_2^{(15)}x_3^{(3)}x_4^{(5)}+
\medskip
x_1^{(27)}x_2^{(5)}x_3^{(3)}x_4^{(3)}\\
&+
 x_1^{(29)}x_2^{(3)}x_3^{(3)}x_4^{(3)}+
 x_1^{(13)}x_2^{(11)}x_3^{(7)}x_4^{(7)}+
 x_1^{(11)}x_2^{(7)}x_3^{(13)}x_4^{(7)}+
 x_1^{(7)}x_2^{(13)}x_3^{(11)}x_4^{(7)} +
\medskip
 x_1^{(13)}x_2^{(7)}x_3^{(7)}x_4^{(11)}\\
&+
 x_1^{(7)}x_2^{(7)}x_3^{(13)}x_4^{(11)}+
 x_1^{(7)}x_2^{(13)}x_3^{(7)}x_4^{(11)}+
  x_1^{(11)}x_2^{(7)}x_3^{(7)}x_4^{(13)}+
 x_1^{(7)}x_2^{(11)}x_3^{(7)}x_4^{(13)}+
\medskip
 x_1^{(7)}x_2^{(7)}x_3^{(11)}x_4^{(13)}\\
&+
 x_1^{(7)}x_2^{(7)}x_3^{(7)}x_4^{(17)}+
x_1^{(7)}x_2^{(7)}x_3^{(9)}x_4^{(15)} +
 x_1^{(7)}x_2^{(11)}x_3^{(5)}x_4^{(15)}+
 x_1^{(7)}x_2^{(13)}x_3^{(3)}x_4^{(15)}+
\medskip
 x_1^{(7)}x_2^{(7)}x_3^{(19)}x_4^{(5)}\\
&+
 x_1^{(7)}x_2^{(7)}x_3^{(21)}x_4^{(3)}+
x_1^{(11)}x_2^{(7)}x_3^{(15)}x_4^{(5)}+
 x_1^{(11)}x_2^{(15)}x_3^{(7)}x_4^{(5)}+
 x_1^{(15)}x_2^{(11)}x_3^{(7)}x_4^{(5)}.
\end{array}$$
In \cite{Phuc9}, we showed that $(\mathbb Z/2 \otimes_{GL_4} {\rm Ann}_{\overline{\mathcal A}}[P_{\frac{n_3-4}{2}}^{\otimes 4}]^{*})$ is 1-dimensional. Basing this coupled with the inequality \eqref{bdt} and an admissible basis of ${\rm Ker}((\widetilde {Sq^0_*})_{n_3}),$ it may be claimed that

\begin{thm}\label{dlc3}
The space of $GL_4$-coinvariants $\mathbb Z/2 \otimes_{GL_4} {\rm Ann}_{\overline{\mathcal A}}[P_{n_3}^{\otimes 4}]^{*}$ has dimension $2$ with the basis $\{[x_1^{(0)}x_2^{(0)}x_3^{(7)}x_4^{(31)}], [\widetilde{\zeta}]\}.$ 
\end{thm}

By the unstable condition,  to verify that $\widetilde{\zeta}\in {\rm Ext}_{\mathcal A}^{0, n_3}(\mathbb Z/2, P^{\otimes 4}),$ we need only to compute the actions of the Steenrod squares $Sq^{2^{i}}$ for $i = 0, 1, 2, 3.$ Now, clearly $$ \overline{e}_1 = \lambda_7^{3}\lambda_{17} + (\lambda_7\lambda_{11}^{2} + \lambda_7^{2}\lambda_{15})\lambda_9 + \lambda_{15}\lambda_{11}\lambda_7\lambda_5 + \lambda_7^{2}\lambda_{11}\lambda_{13}\in \Lambda^{4, n_3}$$ is a cycle in $\Lambda$ and $e_1 = [\overline{e}_1]\in {\rm Ext}_{\mathcal A}^{4, 4+n_3}(\mathbb Z/2, \mathbb Z/2)$. Then, based on Theorems \ref{dlc3} and the representation in the lambda algebra of the rank 4 transfer, we claim that the cycles
$$ \lambda_0^{2}\lambda_7\lambda_{31} = \psi_4(x_1^{(0)}x_2^{(0)}x_3^{(7)}x_4^{(31)}), \ \ \overline{e}_1 = \psi_4(\widetilde{\zeta})$$ in $\Lambda^{4, n_3}$ are representative of the non-zero elements $h_0^{2}h_3h_5$ and $e_1$ respectively. This shows that $h_0^{2}h_3h_5$ and $e_1$ are in the image of $Tr_4^{\mathcal A}$ and therefore by Theorem \ref{dlc2} and the equality \eqref{kqL}, it follows that

\begin{corl}\label{hqc3}
The Singer transfer $$Tr_4^{\mathcal A}: \mathbb Z/2 \otimes_{GL_4} {\rm Ann}_{\overline{\mathcal A}}[P_{n_3}^{\otimes 4}]^{*}\to {\rm Ext}_{\mathcal A}^{4, 4+n_3}(\mathbb Z/2, \mathbb Z/2)$$
is an isomorphism.
\end{corl}

{\bf The case \boldmath{$s \geq 4$}.} Direct calculating based on an admissible basis of ${\rm Ker}((\widetilde {Sq^0_*})_{n_s}),$ we get

\begin{thm}\label{dlc4}
For each $s\geq 4,$ the space $\mathbb Z/2 \otimes_{GL_4} {\rm Ann}_{\overline{\mathcal A}}[P_{n_s}^{\otimes 4}]^{*}$ is one-dimensional and generated by $[x_1^{(0)}x_2^{(0)}x_3^{(2^{s}-1)}x_4^{(2^{s+2}-1)}].$
\end{thm}

Using the representation of $Tr_4^{\mathcal A}$ over $\Lambda$, we conclude that the cycles $$ \lambda_0^{2}\lambda_{2^{s}-1}\lambda_{2^{s+2}-1} = \psi_4(x_1^{(0)}x_2^{(0)}x_3^{(2^{s}-1)}x_4^{(2^{s+2}-1)})$$ in $\Lambda^{4, n_s}$ are representative of the non-zero elements $h_0^{2}h_sh_{s+2}\in {\rm Ext}_{\mathcal A}^{4, 4+n_s}(\mathbb Z/2, \mathbb Z/2)$ for all $s\geq 4.$ This leads to $h_0^{2}h_sh_{s+2}$ being in the image of $Tr_4^{\mathcal A}.$  So, combining these data with the equality \eqref{kqL} and Theorem \ref{dlc4}, the following is immediate.

\begin{corl}\label{hqc4}
The cohomological transfers $Tr_4^{\mathcal A}: \mathbb Z/2 \otimes_{GL_4} {\rm Ann}_{\overline{\mathcal A}}[P_{n_s}^{\otimes 4}]^{*}\to {\rm Ext}_{\mathcal A}^{4, 4+n_s}(\mathbb Z/2, \mathbb Z/2)$ are isomorphisms, for all $s\geq 4.$
\end{corl}

\subsection{The degree \boldmath{$n'_s :=17.2^{s}-2$}}\label{s2.2}

Because Kameko's squaring operation $(\widetilde {Sq^0_*})_{n'_s}: QP^{\otimes 4}_{n'_s}  \longrightarrow QP^{\otimes 4}_{\frac{n'_s-4}{2}}$ is an epimorphism of $\mathbb Z/2[GL_4]$-modules, we have
\begin{equation}\label{bdt2}
 \dim (\mathbb Z/2 \otimes_{GL_4} {\rm Ann}_{\overline{\mathcal A}}[P_{n'_s}^{\otimes 4}]^{*})\leq \dim ([{\rm Ker}((\widetilde {Sq^0_*})_{n'_s})]^{GL_4}) + \dim (\mathbb Z/2 \otimes_{GL_4} {\rm Ann}_{\overline{\mathcal A}}[P_{\frac{n'_s-4}{2}}^{\otimes 4}]^{*}).
\end{equation}
By direct calculations using a monomial basis of ${\rm Ker}((\widetilde {Sq^0_*})_{n'_s})$ in Sum \cite{Sum1, Sum2}, we get the following results.

{\bf The case {\boldmath{$s = 1$}}}. Consider the following element in $[P^{\otimes 4}_{n'_1}]^{*}$:
$$ \begin{array}{ll}
\overline{\zeta}&=  x_1^{(3)}x_2^{(13)}x_3^{(7)}x_4^{(9)}+
 x_1^{(3)}x_2^{(13)}x_3^{(11)}x_4^{(5)}+
 x_1^{(3)}x_2^{(13)}x_3^{(13)}x_4^{(3)}+
 x_1^{(5)}x_2^{(11)}x_3^{(7)}x_4^{(9)}+
\medskip
 x_1^{(5)}x_2^{(11)}x_3^{(11)}x_4^{(5)}\\
&+
 x_1^{(5)}x_2^{(11)}x_3^{(13)}x_4^{(3)}+
 x_1^{(5)}x_2^{(13)}x_3^{(7)}x_4^{(7)}+
 x_1^{(7)}x_2^{(3)}x_3^{(11)}x_4^{(11)}+
 x_1^{(7)}x_2^{(3)}x_3^{(13)}x_4^{(9)}+
\medskip
 x_1^{(7)}x_2^{(5)}x_3^{(11)}x_4^{(9)}\\
&+
 x_1^{(7)}x_2^{(5)}x_3^{(13)}x_4^{(7)}+
 x_1^{(7)}x_2^{(7)}x_3^{(7)}x_4^{(11)}+
 x_1^{(7)}x_2^{(7)}x_3^{(9)}x_4^{(9)}+
 x_1^{(7)}x_2^{(7)}x_3^{(13)}x_4^{(5)}+
\medskip
 x_1^{(7)}x_2^{(9)}x_3^{(7)}x_4^{(9)}\\
&+
 x_1^{(7)}x_2^{(11)}x_3^{(5)}x_4^{(9)}+
 x_1^{(7)}x_2^{(13)}x_3^{(3)}x_4^{(9)}+
 x_1^{(7)}x_2^{(13)}x_3^{(5)}x_4^{(7)}+
 x_1^{(9)}x_2^{(7)}x_3^{(7)}x_4^{(9)}+
\medskip
 x_1^{(9)}x_2^{(7)}x_3^{(11)}x_4^{(5)}\\
&+
 x_1^{(9)}x_2^{(7)}x_3^{(13)}x_4^{(3)}+
 x_1^{(11)}x_2^{(3)}x_3^{(7)}x_4^{(11)}+
 x_1^{(11)}x_2^{(3)}x_3^{(13)}x_4^{(5)}+
 x_1^{(11)}x_2^{(5)}x_3^{(11)}x_4^{(5)}+
\medskip
 x_1^{(11)}x_2^{(7)}x_3^{(3)}x_4^{(11)}\\
&+
 x_1^{(11)}x_2^{(7)}x_3^{(9)}x_4^{(5)}+
 x_1^{(11)}x_2^{(9)}x_3^{(7)}x_4^{(5)}+
 x_1^{(11)}x_2^{(11)}x_3^{(3)}x_4^{(7)}+
 x_1^{(11)}x_2^{(11)}x_3^{(5)}x_4^{(5)}+
\medskip
 x_1^{(11)}x_2^{(13)}x_3^{(3)}x_4^{(5)}\\
&+
 x_1^{(13)}x_2^{(3)}x_3^{(13)}x_4^{(3)}+
 x_1^{(13)}x_2^{(5)}x_3^{(11)}x_4^{(3)}+
 x_1^{(13)}x_2^{(7)}x_3^{(7)}x_4^{(5)}+
 x_1^{(13)}x_2^{(7)}x_3^{(7)}x_4^{(5)}+
\medskip
 x_1^{(13)}x_2^{(7)}x_3^{(9)}x_4^{(3)}\\
&+
 x_1^{(13)}x_2^{(9)}x_3^{(7)}x_4^{(3)}+
 x_1^{(13)}x_2^{(11)}x_3^{(5)}x_4^{(3)}+
 x_1^{(13)}x_2^{(13)}x_3^{(3)}x_4^{(3)}.
\end{array}$$

\begin{thm}\label{dlc5}
The space $\mathbb Z/2 \otimes_{GL_4} {\rm Ann}_{\overline{\mathcal A}}[P_{n'_1}^{\otimes 4}]^{*}$ is 1-dimensional and generated by $[\overline{\zeta}].$
\end{thm}

An effective approach, based on the inequality \eqref{bdt2} coupled with a result of Sum \cite{Sum3} and the calculations of the invariant space $[{\rm Ker}_{n'_1}]^{GL_4}$, is presented in the proof of the theorem.  Now, it is easy to see that $ \overline{d}_1 = \lambda_7^2\lambda_5\lambda_{13} + \lambda_7^2\lambda_9^2 + \lambda_7\lambda_{11}\lambda_9\lambda_5 + \lambda_{15}\lambda_3\lambda_{11}\lambda_3 \in \Lambda^{4, n'_1}$ is a cycle in $\Lambda$ and is a representative of the non-zero element $d_1\in {\rm Ext}_{\mathcal A}^{4, 4+n'_1}(\mathbb Z/2, \mathbb Z/2).$ Then, using Theorem \ref{dlc5} and a representation in $\Lambda$ of $Tr_4^{\mathcal A},$ we deduce that $d_1 = Tr_4([\overline{\zeta}]) = [\psi_4(\overline{\zeta})].$ This implies that $d_1$ is in the image of $Tr_4^{\mathcal A}.$ Combining this with Theorem \ref{dlc5} and the fact that ${\rm Ext}_{\mathcal A}^{4, 4+n'_1}(\mathbb Z/2, \mathbb Z/2)$ has dimension one, we immediately obtain

\begin{corl}\label{hqc5}
The transfer $Tr_4^{\mathcal A}$ is an isomorphism when acting on $\mathbb Z/2 \otimes_{GL_4} {\rm Ann}_{\overline{\mathcal A}}[P_{n'_1}^{\otimes 4}]^{*}.$
\end{corl}

{\bf The case \boldmath{$s = 2$}}. Using the inequality \eqref{bdt2} and a result in Sum \cite{Sum3}, we have

\begin{thm}\label{dlc6}
The coinvariant space $\mathbb Z/2 \otimes_{GL_4} {\rm Ann}_{\overline{\mathcal A}}[P_{n'_2}^{\otimes 4}]^{*}$ is the $\mathbb Z/2$-vector space of 1 dimension with the basis $\{[x_1^{(1)}x_2^{(1)}x_3^{(1)}x_4^{(63)}]\}.$
\end{thm}

We observe that $\lambda_1\in \Lambda^{1, 1}$ and $\lambda_{63}\in \Lambda^{1, 63}$ are the cycles in $\Lambda$ and $[\lambda_1] = h_1\in {\rm Ext}_{\mathcal A}^{1, 2}(\mathbb Z/2, \mathbb Z/2)$ and $[\lambda_{63}] = h_6\in {\rm Ext}_{\mathcal A}^{1, 64}(\mathbb Z/2, \mathbb Z/2).$ So, from Theorem \ref{dlc6} and the representation of the fourth transfer homomorphism over the algebra $\Lambda,$ we claim that the cycle $\lambda_1^{3}\lambda_{63} = \psi_4(x_1^{(1)}x_2^{(1)}x_3^{(1)}x_4^{(63)})$ in $\Lambda$ is a representative of the non-zero element $h_1^{3}h_6 = h_0^{2}h_2h_6\in {\rm Ext}_{\mathcal A}^{4, 4+n'_2}(\mathbb Z/2, \mathbb Z/2).$ This shows that $h_1^{3}h_6\in {\rm Im}(Tr_4^{\mathcal A})$ and so, the following corollary is immediate from the fact that $ {\rm Ext}_{\mathcal A}^{4, 4+n'_2}(\mathbb Z/2, \mathbb Z/2)$ is 1-dimensional.

\begin{corl}\label{hqc6}
The algebraic transfer $Tr_4^{\mathcal A}: \mathbb Z/2 \otimes_{GL_4} {\rm Ann}_{\overline{\mathcal A}}[P_{n'_2}^{\otimes 4}]^{*}\to {\rm Ext}_{\mathcal A}^{4, 4+n'_2}(\mathbb Z/2, \mathbb Z/2)$ is an isomorphism.
\end{corl}

{\bf The case \boldmath{$s \geq 3$}}. The following theorem is proved by using the equality \eqref{bdt2} and our result in \cite{Phuc9}.

\begin{thm}\label{dlc7}
Let $s$ be a positive integer such that $s\geq 3.$ Then, we have
$$ \dim \mathbb Z/2 \otimes_{GL_4} {\rm Ann}_{\overline{\mathcal A}}[P_{n'_s}^{\otimes 4}]^{*} =\left\{\begin{array}{ll}
1 &\mbox{if $s = 4$},\\
2&\mbox{if $s \neq 4$}.
\end{array}\right.$$
Furthermore, 
$$ \begin{array}{ll}
\medskip
& \mathbb Z/2 \otimes_{GL_4} {\rm Ann}_{\overline{\mathcal A}}[P_{n'_s}^{\otimes 4}]^{*} \\
&=\left\{\begin{array}{ll}
\langle [x_1^{(1)}x_2^{(7)}x_3^{(63)}x_4^{(63)}], [x_1^{(0)}x_2^{(0)}x_3^{(7)}x_4^{(127)}] \rangle &\mbox{if $s = 3$},\\[1mm]
\langle [x_1^{(1)}x_2^{(15)}x_3^{(127)}x_4^{(127)}] \rangle &\mbox{if $s = 4$},\\[1mm]
\langle [x_1^{(1)}x_2^{(2^{s-1}-1)}x_3^{(2^{s-1}-1)}x_4^{(2^{s+4}-1)}] , [x_1^{(1)}x_2^{(2^{s}-1)}x_3^{(2^{s+3}-1)}x_4^{(2^{s+3}-1)}]  \rangle &\mbox{if $s \geq 5$}.
\end{array}\right.
\end{array}$$
\end{thm}

From this theorem and the representation in $\Lambda$ of the rank 4 transfer, it may be concluded that 
$$ \begin{array}{ll} 
&[\psi_4(x_1^{(1)}x_2^{(7)}x_3^{(63)}x_4^{(63)})] = [\lambda_1\lambda_7\lambda_{63}^{2}]\\
\medskip
&= Tr_4^{\mathcal A}([x_1^{(1)}x_2^{(7)}x_3^{(63)}x_4^{(63)}])  = h_1h_3h_6^{2}\in {\rm Ext}_{\mathcal A}^{4, 4+n'_3}(\mathbb Z/2, \mathbb Z/2),\\
&[\psi_4(x_1^{(0)}x_2^{(0)}x_3^{(7)}x_4^{(127)})] = [\lambda_0\lambda_7\lambda_{127}^{2}] \\
\medskip
&= Tr_4^{\mathcal A}([x_1^{(0)}x_2^{(0)}x_3^{(7)}x_4^{(127)}]) = h_0^{2}h_3h_7\in {\rm Ext}_{\mathcal A}^{4, 4+n'_3}(\mathbb Z/2, \mathbb Z/2),\\
 &[\psi_4(x_1^{(1)}x_2^{(15)}x_3^{(127)}x_4^{(127)})] = [\lambda_1\lambda_{15}\lambda_{127}^{2}] \\
 \medskip
&= Tr_4^{\mathcal A}([x_1^{(1)}x_2^{(15)}x_3^{(127)}x_4^{(127)}]) = h_1h_4h_7^{3}\in {\rm Ext}_{\mathcal A}^{4, 4+n'_4}(\mathbb Z/2, \mathbb Z/2),\\
&[\psi_4(x_1^{(1)}x_2^{(2^{s-1}-1)}x_3^{(2^{s-1}-1)}x_4^{(2^{s+4}-1)})] = [\lambda_1\lambda^{2}_{2^{s-1}-1}\lambda_{2^{s+4}-1}] \\
\medskip
&= Tr_4^{\mathcal A}([x_1^{(1)}x_2^{(2^{s-1}-1)}x_3^{(2^{s-1}-1)}x_4^{(2^{s+4}-1)}]) = h_1h_{s-1}^{2}h_{s+4}\in {\rm Ext}_{\mathcal A}^{4, 4+n'_s}(\mathbb Z/2, \mathbb Z/2)\ \mbox{for $s\geq 5$}\\
&[\psi_4(x_1^{(1)}x_2^{(2^{s}-1)}x_3^{(2^{s+3}-1)}x_4^{(2^{s+3}-1)})] = [\lambda_1\lambda_{2^{s}-1}\lambda_{2^{s+3}-1}^{2}]\\
\medskip
 &=  Tr_4^{\mathcal A}([x_1^{(1)}x_2^{(2^{s}-1)}x_3^{(2^{s+3}-1)}x_4^{(2^{s+3}-1)}]) = h_1h_{s}h_{s+3}^{2}\in {\rm Ext}_{\mathcal A}^{4, 4+n'_s}(\mathbb Z/2, \mathbb Z/2)\ \mbox{for $s\geq 5$}.
\end{array}$$ 
Combining these with the fact that
$$ \dim {\rm Ext}_{\mathcal A}^{4, 4+n'_s}(\mathbb Z/2, \mathbb Z/2) = \left\{\begin{array}{ll}
1&\mbox{if $s = 4$},\\
2&\mbox{if $s  = 3$ and $s\geq 5$},
\end{array}\right.$$
we have immediately

\begin{corl}\label{hqc7}
For each $s > 2,$ the cohomological transfer $$Tr_4^{\mathcal A}: \mathbb Z/2 \otimes_{GL_4} {\rm Ann}_{\overline{\mathcal A}}[P_{n'_s}^{\otimes 4}]^{*}\to {\rm Ext}_{\mathcal A}^{4, 4+n'_s}(\mathbb Z/2, \mathbb Z/2)$$ is also an isomorphism.
\end{corl}

Thus, Corollaries \ref{hqc1} - \ref{hqc7} confirmed Singer's conjecture that $Tr_4^{\mathcal A}$ is a monomorphism in all the internal degrees $q.2^s-2$ for $q\in \{5, 17\}$ and all positive integers $s.$ Moreover, motivated by the data in this paper, we propose the following.

\begin{conj}
Let us considere degrees $2^{s+m} + 2^{s} - 2$, where $s,\, m$ are positive integers. Then, Singer's conjecture for the rank 4 transfer holds in these internal degrees.  
\end{conj}

So, by our previous results \cite{Phuc9, Phuc11} and the work of Sum \cite{Sum3}, we see that if this conjecture is true then the fourth algebraic transfer is a monomorphism for all degrees $n.$ This means that Singer's conjecture is true for cohomological degrees less than or equal to 4.


\section{Proofs of main results}

We first provide some definitions, which will be used in proofs of our main results.

\begin{dn}\label{dnmd}
Let us consider the polynomials $u$ and $v$ in $P^{\otimes h}_{n}.$ We say that $u\equiv v$ if and only if $(u+v)$ is $\mathcal A$-decomposable (or "hit"). 
\end{dn}
The readers can see that the binary relation "$\equiv$" on $P^{\otimes h}_{n}$ is an equivalence relation.

\begin{dn}
For each $1\leq j\leq h,$ we define the $\mathbb Z/2$-linear map $$\theta_j: V_h\cong \langle t_1, \ldots, t_h \rangle\to V_h\cong \langle t_1, \ldots, t_h \rangle$$ by subsituting $\theta_j(t_j) = t_{j+1},$ $\theta_j(t_{j+1}) = t_{j},$ $\theta_j(t_i) = t_{i},$ for $i\neq j,\, j+1,$ $1\leq j < 4,$ and $\theta_h(t_1) = t_1+ t_2,$ $\theta_4(t_i) = t_{i},$ for $1 < i\leq h.$
\end{dn}

From this definition, the group $GL_h\cong GL(V_h)$ is generated by $\theta_j,$ for $1\leq j\leq h,$ and the symmetric group $S_h$ is generated by $\theta_j$ for $1\leq j < h.$ Further, $\theta_j$ induces a homomorphism of $\mathcal A$-algebras which is also denoted by $\theta_j: P^{\otimes h}\to P^{\otimes h}.$ This together with Definition \ref{dnmd} say that a class $[u]\in QP_{n}^{\otimes h}$ is an $S_h$-invariant (resp. $GL_h$-invariant) if and only if $\theta_j(u)\equiv u$ for $1\leq j < h$ (resp. $1\leq j\leq h$).

Now, in what follows: for any admissible monomials ${\rm adm}_1, \ldots, {\rm adm}_k$ in $P^{\otimes 4}_{n}$ and for a subgroup $\mathscr G$ of $GL_4,$ let us denote $\mathscr G({\rm adm}_1, \ldots, {\rm adm}_k)$ the $\mathbb Z/2\mathscr G$-submodule of ${\rm Ker}_{n}:= {\rm Ker}((\widetilde {Sq^0_*})_{n})$ generated by the set $\{[{\rm adm}_j]:\, 1\leq j\leq k\}.$ 

We have an isomorphism $${\rm Ker}_{n}\cong \underline{{\rm Ker}_{n}}\bigoplus \widehat{{\rm Ker}_{n}},$$ where 
$$ \begin{array}{ll}
\medskip
  \underline{{\rm Ker}_{n}}&:= \langle \{[t_1^{a_1}t_2^{a_2}t_3^{a_3}t_4^{a_4}]\in {\rm Ker}_{n}:\,a_1a_2a_3a_4 = 0\}\rangle,\\
  \widehat{{\rm Ker}_{n}}&:= \langle \{[t_1^{a_1}t_2^{a_2}t_3^{a_3}t_4^{a_4}]\in {\rm Ker}_{n}:\, a_1a_2a_3a_4\neq 0\}\rangle
\end{array}$$
are $\mathbb Z/2$-subspaces of ${\rm Ker}_{n}.$ 

\subsection{The degree \boldmath{$n_s = 5.2^{s}-2$}}

In order to prove Theorems \ref{dlc1}, \ref{dlc2}, \ref{dlc3} and \ref{dlc4}, we need to some related results.

{\bf The case \boldmath{$s = 1$}}. According to Sum \cite{Sum1, Sum2}, $\underline{{\rm Ker}_{n_1}}$ has a basis consisting of all the classes represented by the admissible monomials ${\rm adm}_{1,\, j}$ for all $1\leq j\leq 42,$ where  

\begin{center}
\begin{tabular}{lcrr}
${\rm adm}_{1,\,1}= t_3t_4^{7}$, & ${\rm adm}_{1,\,2}= t_3^{7}t_4$, & \multicolumn{1}{l}{${\rm adm}_{1,\,3}= t_2t_4^{7}$,} & \multicolumn{1}{l}{${\rm adm}_{1,\,4}= t_2t_3^{7}$,} \\
${\rm adm}_{1,\,5}= t_2^{7}t_4$, & ${\rm adm}_{1,\,6}= t_2^{7}t_3$, & \multicolumn{1}{l}{${\rm adm}_{1,\,7}= t_1t_4^{7}$,} & \multicolumn{1}{l}{${\rm adm}_{1,\,8}= t_1t_3^{7}$,} \\
${\rm adm}_{1,\,9}= t_1t_2^{7}$, & ${\rm adm}_{1,\,10}= t_1^{7}t_4$, & \multicolumn{1}{l}{${\rm adm}_{1,\,11}= t_1^{7}t_3$,} & \multicolumn{1}{l}{${\rm adm}_{1,\,12}= t_1^{7}t_2$,} \\
${\rm adm}_{1,\,13}= t_3^{3}t_4^{5}$, & ${\rm adm}_{1,\,14}= t_2^{3}t_4^{5}$, & \multicolumn{1}{l}{${\rm adm}_{1,\,15}= t_2^{3}t_3^{5}$,} & \multicolumn{1}{l}{${\rm adm}_{1,\,16}= t_1^{3}t_4^{5}$,} \\
${\rm adm}_{1,\,17}= t_1^{3}t_3^{5}$, & ${\rm adm}_{1,\,18}= t_1^{3}t_2^{5}$, & \multicolumn{1}{l}{${\rm adm}_{1,\,19}= t_2t_3t_4^{6}$,} & \multicolumn{1}{l}{${\rm adm}_{1,\,20}= t_2t_3^{6}t_4$,} \\
\end{tabular}%
\end{center}

\newpage
\begin{center}
\begin{tabular}{lcrr}
${\rm adm}_{1,\,21}= t_1t_3t_4^{6}$, & ${\rm adm}_{1,\,22}= t_1t_3^{6}t_4$, & \multicolumn{1}{l}{${\rm adm}_{1,\,23}= t_1t_2t_4^{6}$,} & \multicolumn{1}{l}{${\rm adm}_{1,\,24}= t_1t_2t_3^{6}$,} \\
${\rm adm}_{1,\,25}= t_1t_2^{6}t_4$, & ${\rm adm}_{1,\,26}= t_1t_2^{6}t_3$, & \multicolumn{1}{l}{${\rm adm}_{1,\,27}= t_2t_3^{2}t_4^{5}$,} & \multicolumn{1}{l}{${\rm adm}_{1,\,28}= t_1t_3^{2}t_4^{5}$,} \\
${\rm adm}_{1,\,29}= t_1t_2^{2}t_4^{5}$, & ${\rm adm}_{1,\,30}= t_1t_2^{2}t_3^{5}$, & \multicolumn{1}{l}{${\rm adm}_{1,\,31}= t_2t_3^{3}t_4^{4}$,} & \multicolumn{1}{l}{${\rm adm}_{1,\,32}= t_2^{3}t_3t_4^{4}$,} \\
${\rm adm}_{1,\,33}= t_2^{3}t_3^{4}t_4$, & ${\rm adm}_{1,\,34}= t_1t_3^{3}t_4^{4}$, & \multicolumn{1}{l}{${\rm adm}_{1,\,35}= t_1t_2^{3}t_4^{4}$,} & \multicolumn{1}{l}{${\rm adm}_{1,\,36}= t_1t_2^{3}t_3^{4}$,} \\
${\rm adm}_{1,\,37}= t_1^{3}t_3t_4^{4}$, & ${\rm adm}_{1,\,38}= t_1^{3}t_2t_4^{4}$, & \multicolumn{1}{l}{${\rm adm}_{1,\,39}= t_1^{3}t_2t_3^{4}$,} & \multicolumn{1}{l}{${\rm adm}_{1,\,40}= t_1^{3}t_3^{4}t_4$,} \\
${\rm adm}_{1,\,41}= t_1^{3}t_2^{4}t_4$, & ${\rm adm}_{1,\,42}= t_1^{3}t_2^{4}t_3$. &       &  
\end{tabular}%
\end{center}

From the admissible basis above,  by a simple computation, we see that 
$$ \begin{array}{ll}
\medskip
&S_4({\rm adm}_{1,\,1}) = \langle \{[{\rm adm}_{1,\,j}]:\, 1\leq j\leq 12\}  \rangle,\\
\medskip
&S_4({\rm adm}_{1,\,13}) = \langle \{[{\rm adm}_{1,\,j}]:\, 13\leq j\leq 18]\}  \rangle,\\
\medskip
&S_4({\rm adm}_{1,\,19},\ {\rm adm}_{1,\,27}) = \langle \{[{\rm adm}_{1,\,j}]:\, 19\leq j\leq 42\}  \rangle\\
\end{array}$$
are $S_4$-submodules of $\underline{{\rm Ker}_{n_1}}.$ Therefore, we have an isomorphism
$$ \underline{{\rm Ker}_{n_1}} \cong S_4({\rm adm}_{1,\,1})\bigoplus S_4({\rm adm}_{1,\,13})\bigoplus S_4({\rm adm}_{1,\,19},\ {\rm adm}_{1,\,27}).$$ 

\begin{lema}\label{bdc1-1}
The following statements are true:
\begin{itemize}

\item[i)] $[S_4({\rm adm}_{1,\,1})]^{S_4} = \langle [p_{1,\, 1}] \rangle$ with $p_{1,\, 1}:= \sum_{1\leq j\leq 12}{\rm adm}_{1,\,j}.$

\item[ii)] $[S_4({\rm adm}_{1,\,13})]^{S_4} = \langle [p_{1,\, 2}] \rangle$ with $p_{1,\, 2}:= \sum_{13\leq j\leq 18}{\rm adm}_{1,\,j}.$

\item[iii)] $[S_4({\rm adm}_{1,\,19},\ {\rm adm}_{1,\,27})]^{S_4} = \langle [p_{1,\, 3}] \rangle$ with $p_{1,\, 3}:= \sum_{\mathbb J\setminus \{27, \ldots, 31, 34, 35, 36\}}{\rm adm}_{1,\,j},$\\[1mm] where the set \mbox{$\mathbb J:=\{19, 20, \ldots, 42\}.$}
\end{itemize}
\end{lema}

\begin{proof}
We prove Part iii) in detail. The others can be proved by the similar computations. 

It is easy to see that the set $\{[{\rm adm}_{1,\,j}]:\, 19\leq j\leq 42\}$ is a basis of $S_4({\rm adm}_{1,\,19},\ {\rm adm}_{1,\,27}).$ Suppose that $[f]\in [S_4({\rm adm}_{1,\,19},\ {\rm adm}_{1,\,27})]^{S_4},$ then $f\equiv \sum_{19\leq j\leq 42}\gamma_j{\rm adm}_{1,\,j}$ in which $\gamma_j\in \mathbb Z/2.$ Acting the homomorphisms $\theta_i$ for $1\leq i\leq 3$ on both sides of this equality, we get 
$$ \begin{array}{ll}
\medskip
\theta_1(f)&\equiv\gamma_{21}{\rm adm}_{1,\,19} + \gamma_{22}{\rm adm}_{1,\,20} + \gamma_{19}{\rm adm}_{1,\,21} + \gamma_{20}{\rm adm}_{1,\,22} +\gamma_{23}{\rm adm}_{1,\,23}\\
\medskip
&+ \gamma_{24}{\rm adm}_{1,\,24} + \gamma_{28}{\rm adm}_{1,\,27} + \gamma_{27}{\rm adm}_{1,\,28} + \gamma_{34}{\rm adm}_{1,\,34} + \gamma_{37}{\rm adm}_{1,\,32}\\
\medskip
&+\gamma_{40}{\rm adm}_{1,\,33} + \gamma_{31}{\rm adm}_{1,\,34} + \gamma_{38}{\rm adm}_{1,\,35} + \gamma_{39}{\rm adm}_{1,\,36} + \gamma_{32}{\rm adm}_{1,\,37} \\
\medskip
&+ \gamma_{35}{\rm adm}_{1,\,38}+\gamma_{36}{\rm adm}_{1,\,39} + \gamma_{33}{\rm adm}_{1,\,40}+ \gamma_{25}t_1^{6}t_2t_4 + \gamma_{26}t_1^{6}t_2t_3\\
\medskip
& + \gamma_{29}t_1^{2}t_2t_4^{5} + \gamma_{30}t_1^{2}t_2t_3^{5} + \gamma_{41}t_1^{4}t_2^{3}t_4 + \gamma_{42}t_1^{4}t_2^{3}t_3,\\
\medskip
\theta_2(f)&\equiv   \gamma_{19}{\rm adm}_{1,\,19} + \gamma_{23}{\rm adm}_{1,\,21} + \gamma_{25}{\rm adm}_{1,\,22} + \gamma_{21}{\rm adm}_{1,\,23} +\gamma_{26}{\rm adm}_{1,\,24}\\
\medskip
&+ \gamma_{22}{\rm adm}_{1,\,25} + \gamma_{24}{\rm adm}_{1,\,26} + \gamma_{29}{\rm adm}_{1,\,28} + \gamma_{28}{\rm adm}_{1,\,29} + \gamma_{32}{\rm adm}_{1,\,31}\\
\medskip
&+\gamma_{31}{\rm adm}_{1,\,32} + \gamma_{35}{\rm adm}_{1,\,34} + \gamma_{34}{\rm adm}_{1,\,35} + \gamma_{38}{\rm adm}_{1,\,37} + \gamma_{37}{\rm adm}_{1,\,38} \\
\medskip
&+ \gamma_{42}{\rm adm}_{1,\,39}+\gamma_{41}{\rm adm}_{1,\,40} + \gamma_{40}{\rm adm}_{1,\,41}+ \gamma_{39}{\rm adm}_{1,\,42} + \gamma_{20}t_2^{6}t_3t_4\\
\medskip
& + \gamma_{27}t_2^{2}t_3t_4^{5} + \gamma_{30}t_1t_2^{5}t_3^{2} + \gamma_{33}t_2^{4}t_3^{3}t_4 + \gamma_{36}t_1t_2^{4}t_3^{3},\\
\medskip
\theta_3(f)&\equiv   \gamma_{20}{\rm adm}_{1,\,19} + \gamma_{19}{\rm adm}_{1,\,20} + \gamma_{22}{\rm adm}_{1,\,21} + \gamma_{21}{\rm adm}_{1,\,22} +\gamma_{24}{\rm adm}_{1,\,23}\\
\medskip
&+ \gamma_{23}{\rm adm}_{1,\,24} + \gamma_{26}{\rm adm}_{1,\,25} + \gamma_{25}{\rm adm}_{1,\,26} + \gamma_{30}{\rm adm}_{1,\,29} + \gamma_{29}{\rm adm}_{1,\,30}\\
\medskip
&+\gamma_{33}{\rm adm}_{1,\,32} + \gamma_{32}{\rm adm}_{1,\,33} + \gamma_{36}{\rm adm}_{1,\,35} + \gamma_{35}{\rm adm}_{1,\,36} + \gamma_{40}{\rm adm}_{1,\,37} \\
\medskip
&+ \gamma_{39}{\rm adm}_{1,\,38}+\gamma_{38}{\rm adm}_{1,\,39} + \gamma_{37}{\rm adm}_{1,\,40}+ \gamma_{42}{\rm adm}_{1,\,41} +  \gamma_{41}{\rm adm}_{1,\,42}\\
\medskip
& + \gamma_{27}t_2t_3^{5}t_4^{2} + \gamma_{28}t_1t_3^{5}t_4^{2} + \gamma_{31}t_2t_3^{4}t_4^{3} + \gamma_{34}t_1t_3^{4}t_4^{3}.
\end{array}$$ 
As well known, the (left) action of Steenrod algebra on $P^{\otimes 4}_{n_s}$ is determined by 
$$Sq^n(t_j) = \left\{\begin{array}{lll}
t_j &\mbox{if}& n = 0,\\
t_j^2 &\mbox{if}& n = 1,\\
0 &\mbox{if}& n > 1,
\end{array}\right.$$
and Cartan's formula $Sq^{n}(fg) = \sum_{k+\ell= n}Sq^{k}(f)Sq^{\ell}(g)$ for all $f,\, g\in P^{\otimes 4}_{n_s}.$ Using this action, we have
$$ \begin{array}{ll}
\medskip
t_2t_3^{4}t_4^{3}& = Sq^{2}(t_2t_3^{2}t_4^{3}) + {\rm adm}_{1,\,27} \mod (\overline{\mathcal A}P^{\otimes 4}_{n_1}),\\
\medskip
 t_2t_3^{5}t_4^{2} &= Sq^{2}(t_2t_3^{3}t_4^{2}) + {\rm adm}_{1,\,31}  \mod (\overline{\mathcal A}P^{\otimes 4}_{n_1}),\\
\medskip
t_2^{2}t_3t_4^{5}& = Sq^{1}(t_2t_3t_4^{5}) + {\rm adm}_{1,\,19}+{\rm adm}_{1,\,27} \mod (\overline{\mathcal A}P^{\otimes 4}_{n_1}),\\
\medskip
 t_2^{4}t_3^{3}t_4 &= Sq^{1}(t_2t_3^{5}t_4) +  Sq^{2}( t_2^{2}t_3^{3}t_4 +  t_2t_3^{3}t_4^{2}) + {\rm adm}_{1,\,20} +{\rm adm}_{1,\,31}  \mod (\overline{\mathcal A}P^{\otimes 4}_{n_1}),\\
\medskip
t_2^{6}t_3t_4& = Sq^{1}(t_2^{5}t_3t_4) + Sq^{2}(t_2^{3}t_3^{2}t_4 +  t_2^{3}t_3t_4^{2})  + {\rm adm}_{1,\,32}+{\rm adm}_{1,\,33} \mod (\overline{\mathcal A}P^{\otimes 4}_{n_1}),\\
\medskip
t_1t_3^{4}t_4^{3}& = Sq^{2}(t_1t_3^{2}t_4^{3}) + {\rm adm}_{1,\,28} \mod (\overline{\mathcal A}P^{\otimes 4}_{n_1}),\\
\medskip
t_1t_3^{5}t_4^{2}& = Sq^{2}(t_1t_3^{3}t_4^{2}) + {\rm adm}_{1,\,34} \mod (\overline{\mathcal A}P^{\otimes 4}_{n_1}),\\
\medskip
t_1t_2^{4}t_3^{3}& = Sq^{2}(t_1t_2^{2}t_3^{3}) + {\rm adm}_{1,\,30} \mod (\overline{\mathcal A}P^{\otimes 4}_{n_1}),\\
\medskip
t_1t_2^{5}t_3^{2}& = Sq^{2}(t_1t_2^{3}t_3^{2}) + {\rm adm}_{1,\,36} \mod (\overline{\mathcal A}P^{\otimes 4}_{n_1}),\\
\medskip
t_1^{2}t_2t_4^{5}& = Sq^{1}(t_1t_2t_4^{5}) + {\rm adm}_{1,\,23} + {\rm adm}_{1,\,29} \mod (\overline{\mathcal A}P^{\otimes 4}_{n_1}),\\
\medskip
t_1^{2}t_2t_3^{5}& = Sq^{1}(t_1t_2t_3^{5}) + {\rm adm}_{1,\,24} + {\rm adm}_{1,\,30} \mod (\overline{\mathcal A}P^{\otimes 4}_{n_1}),\\
\medskip
t_1^{4}t_2^{3}t_4& = Sq^{1}(t_1t_2^{5}t_4) + Sq^{2}(t_1^{2}t_2^{3}t_4 + t_1t_2^{3}t_4^{2})  + {\rm adm}_{1,\,25} + {\rm adm}_{1,\,35} \mod (\overline{\mathcal A}P^{\otimes 4}_{n_1}),\\
\medskip
t_1^{4}t_2^{3}t_3& = Sq^{1}(t_1t_2^{5}t_3) + Sq^{2}(t_1^{2}t_2^{3}t_3 + t_1t_2^{3}t_3^{2})  + {\rm adm}_{1,\,26} + {\rm adm}_{1,\,36} \mod (\overline{\mathcal A}P^{\otimes 4}_{n_1}),\\
\medskip
t_1^{6}t_2t_4& = Sq^{1}(t_1^{5}t_2t_4) + Sq^{2}(t_1^{3}t_2^{2}t_4 + t_1^{3}t_2t_4^{2})  + {\rm adm}_{1,\,38} + {\rm adm}_{1,\,41} \mod (\overline{\mathcal A}P^{\otimes 4}_{n_1}),\\
\medskip
t_1^{6}t_2t_3& = Sq^{1}(t_1^{5}t_2t_3) + Sq^{2}(t_1^{3}t_2^{2}t_3 + t_1^{3}t_2t_3^{2})  + {\rm adm}_{1,\,39} + {\rm adm}_{1,\,42} \mod (\overline{\mathcal A}P^{\otimes 4}_{n_1}).
\end{array}$$
Combining the above computations and the fact that $\theta_i(f) + f \equiv 0$ for $1\leq i\leq 3,$ we obtain
$ \gamma_j = 0$ for $j\in \mathbb J_1 = \{27, \ldots, 31, 34, 35, 36\}$ and $\gamma_j = \gamma_{19}$ for $j\not\in \mathbb J_1.$ This finishes the proof of  the lemma.
\end{proof}

\begin{lema}\label{bdc1-2}
There is a direct summand decomposition of the $S_4$-submodules:
$$ \widehat{{\rm Ker}_{n_1}} = S_4({\rm adm}_{1,\,43}) \bigoplus S_4( {\rm adm}_{1,\,46}),$$
where $$ \begin{array}{ll}
\medskip
S_4({\rm adm}_{1,\,43}) &= \langle [{\rm adm}_{1,\,43}], [{\rm adm}_{1,\,44}], [{\rm adm}_{1,\,45}]\rangle,\ \mbox{with }\\
\medskip
&\quad {\rm adm}_{1,\,43} =t_1t_2t_3^{2}t_4^{4}, \, {\rm adm}_{1,\,44}= t_1t_2^{2}t_3t_4^{4},\, {\rm adm}_{1,\,45}= t_1t_2^{2}t_3^{4}t_4,\\
\medskip
S_4({\rm adm}_{1,\,46}) &= \langle [{\rm adm}_{1,\,46}], [{\rm adm}_{1,\,47}], [{\rm adm}_{1,\,48}],  [{\rm adm}_{1,\,49}] \rangle \ \mbox{with}\\
&\quad {\rm adm}_{1,\,46} =t_1t_2^{2}t_3^{2}t_4^{3}, \, {\rm adm}_{1,\,47}= t_1t_2^{2}t_3^{3}t_4^{2},\, {\rm adm}_{1,\,48}= t_1t_2^{3}t_3^{2}t_4^{2},\, {\rm adm}_{1,\,48}= t_1^{3}t_2t_3^{2}t_4^{2}. 
\end{array}$$
Moreover, $[S_4({\rm adm}_{1,\,43})]^{S_4} = \langle [p_{1,\, 4}] \rangle$ with $p_{1,\, 4}:= \sum_{43\leq j\leq 45}{\rm adm}_{1,\,j}$ and $[S_4({\rm adm}_{1,\,46})]^{S_4} = 0.$
\end{lema}

\begin{proof}
It is known, from a result of Sum \cite{Sum1, Sum2}, that the set $\{[{\rm adm}_{1,\,j}]:\, 43\leq j\leq 49\}$ is an admissible monomial basis of $ \widehat{{\rm Ker}_{n_1}}.$ So, it is easy to see that $ \widehat{{\rm Ker}_{n_1}} \cong S_4({\rm adm}_{1,\,43}) \bigoplus S_4( {\rm adm}_{1,\,46}).$ 

Now if $[g]\in [S_4({\rm adm}_{1,\,43})]^{S_4}$ and $[h]\in [S_4({\rm adm}_{1,\,46})]^{S_4},$ then 
$$ g\equiv \sum_{43\leq j\leq 45}\gamma_j{\rm adm}_{1,\,j}, \ \ h\equiv \sum_{46\leq j\leq 49}\gamma_j{\rm adm}_{1,\,j},$$
where $\gamma_j\in \mathbb Z/2.$ For each $1\leq i\leq 3,$ direct calculating the equalities $\theta_i(g)$ and $\theta_i(h)$ in the terms ${\rm adm}_{1,\,j}$ for $43\leq j\leq 49$ mod ($\overline{\mathcal A}P^{\otimes 4}_{n_1}$), and using the relations $\theta_i(g) + g\equiv 0$ and $\theta_i(h) + h\equiv 0,$ we get
$$ \begin{array}{ll}
\medskip
\theta_1(g) + g &\equiv \gamma_{43}({\rm adm}_{1,\,44} + {\rm adm}_{1,\,45}) \equiv 0,\\
\medskip
\theta_2(g) + g &\equiv  (\gamma_{43}+\gamma_{44})({\rm adm}_{1,\,43} + {\rm adm}_{1,\,44}) \equiv 0,\\
\medskip
\theta_3(g) + g &\equiv (\gamma_{44}+\gamma_{45})({\rm adm}_{1,\,43} + {\rm adm}_{1,\,44})\equiv 0,\\
\medskip
\theta_1(h) + h &\equiv (\gamma_{46} + \gamma_{47}){\rm adm}_{1,\,43} + (\gamma_{48}+\gamma_{49})({\rm adm}_{1,\,48} + {\rm adm}_{1,\,49}) \equiv 0,\\
\medskip
\theta_2(h) + h &\equiv  \gamma_{49}({\rm adm}_{1,\,43} + {\rm adm}_{1,\,44}) + (\gamma_{47}+\gamma_{48})({\rm adm}_{1,\,47} + {\rm adm}_{1,\,48}) \equiv 0,\\
\medskip
\theta_3(h) + h &\equiv (\gamma_{46}+\gamma_{47})({\rm adm}_{1,\,46} + {\rm adm}_{1,\,47})\equiv 0.
 \end{array}$$ 
From the equalities above, we get $\gamma_{43} = \gamma_{44} = \gamma_{45}$ and $\gamma_j = 0$ for $46\leq j\leq 49.$ The lemma follows.
\end{proof}

{\bf The case \boldmath{$s \geq 2$}}. By Sum \cite{Sum1, Sum2}, $\underline{{\rm Ker}_{n_s}}$ has a basis consisting of all the classes represented by the following admissible monomials ${\rm adm}_{s,\, j}$:

For $s\geq 2,$

\begin{center}
\begin{tabular}{llr}
${\rm adm}_{s,\,1}= t_2t_3^{2^{s}-2}t_4^{2^{s+2}-1}$, & ${\rm adm}_{s,\,2}= t_2t_3^{2^{s+2}-1}t_4^{2^{s}-2}$, & \multicolumn{1}{l}{${\rm adm}_{s,\,3}= t_2^{2^{s+2}-1}t_3t_4^{2^{s}-2}$,} \\
${\rm adm}_{s,\,4}= t_1t_3^{2^{s}-2}t_4^{2^{s+2}-1}$, & ${\rm adm}_{s,\,5}= t_1t_3^{2^{s+2}-1}t_4^{2^{s}-2}$, & \multicolumn{1}{l}{${\rm adm}_{s,\,6}= t_1^{2^{s+2}-1}t_3t_4^{2^{s}-2}$,} \\
${\rm adm}_{s,\,7}= t_1t_2^{2^{s}-2}t_4^{2^{s+2}-1}$, & ${\rm adm}_{s,\,8}= t_1t_2^{2^{s+2}-1}t_4^{2^{s}-2}$, & \multicolumn{1}{l}{${\rm adm}_{s,\,9}= t_1^{2^{s+2}-1}t_2t_4^{2^{s}-2}$,} \\
${\rm adm}_{s,\,10}= t_1t_2^{2^{s}-2}t_3^{2^{s+2}-1}$, & ${\rm adm}_{s,\,11}= t_1t_2^{2^{s+2}-1}t_3^{2^{s}-2}$, & \multicolumn{1}{l}{${\rm adm}_{s,\,12}= t_1^{2^{s+2}-1}t_2t_3^{2^{s}-2}$,} \\
${\rm adm}_{s,\,13}= t_2t_3^{2^{s}-1}t_4^{2^{s+2}-2}$, & ${\rm adm}_{s,\,14}= t_2t_3^{2^{s+2}-2}t_4^{2^{s}-1}$, & \multicolumn{1}{l}{${\rm adm}_{s,\,15}= t_2^{2^{s}-1}t_3t_4^{2^{s+2}-2}$,} \\
${\rm adm}_{s,\,16}= t_1t_3^{2^{s}-1}t_4^{2^{s+2}-2}$, & ${\rm adm}_{s,\,17}= t_1t_3^{2^{s+2}-2}t_4^{2^{s}-1}$, & \multicolumn{1}{l}{${\rm adm}_{s,\,18}= t_1^{2^{s}-1}t_3t_4^{2^{s+2}-2}$,} \\
${\rm adm}_{s,\,19}= t_1t_2^{2^{s}-1}t_4^{2^{s+2}-2}$, & ${\rm adm}_{s,\,20}= t_1t_2^{2^{s+2}-2}t_4^{2^{s}-1}$, & \multicolumn{1}{l}{${\rm adm}_{s,\,21}= t_1^{2^{s}-1}t_2t_4^{2^{s+2}-2}$,} \\
${\rm adm}_{s,\,22}= t_1t_2^{2^{s}-1}t_3^{2^{s+2}-2}$, & ${\rm adm}_{s,\,23}= t_1t_2^{2^{s+2}-2}t_3^{2^{s}-1}$, & \multicolumn{1}{l}{${\rm adm}_{s,\,24}= t_1^{2^{s}-1}t_2t_3^{2^{s+2}-2}$,} \\
${\rm adm}_{s,\,25}= t_3^{2^{s}-1}t_4^{2^{s+2}-1}$, & ${\rm adm}_{s,\,26}= t_3^{2^{s+2}-1}t_4^{2^{s}-1}$, & \multicolumn{1}{l}{${\rm adm}_{s,\,27}= t_2^{2^{s}-1}t_4^{2^{s+2}-1}$,} \\
${\rm adm}_{s,\,28}= t_2^{2^{s}-1}t_3^{2^{s+2}-1}$, & ${\rm adm}_{s,\,29}= t_2^{2^{s+2}-1}t_4^{2^{s}-1}$, & \multicolumn{1}{l}{${\rm adm}_{s,\,30}= t_2^{2^{s+2}-1}t_3^{2^{s}-1}$,} \\
${\rm adm}_{s,\,31}= t_1^{2^{s}-1}t_4^{2^{s+2}-1}$, & ${\rm adm}_{s,\,32}= t_1^{2^{s}-1}t_3^{2^{s+2}-1}$, & \multicolumn{1}{l}{${\rm adm}_{s,\,33}= t_1^{2^{s+2}-1}t_4^{2^{s}-1}$,} \\
${\rm adm}_{s,\,34}= t_1^{2^{s+2}-1}t_3^{2^{s}-1}$, & ${\rm adm}_{s,\,35}= t_1^{2^{s}-1}t_2^{2^{s+2}-1}$, & \multicolumn{1}{l}{${\rm adm}_{s,\,36}= t_1^{2^{s+2}-1}t_2^{2^{s}-1}$,} \\
${\rm adm}_{s,\,37}= t_2t_3^{2^{s+1}-2}t_4^{2^{s+2}-2^{s}-1}$, & ${\rm adm}_{s,\,38}= t_1t_3^{2^{s+1}-2}t_4^{2^{s+2}-2^{s}-1}$, & \multicolumn{1}{l}{${\rm adm}_{s,\,39}= t_1t_2^{2^{s+1}-2}t_4^{2^{s+2}-2^{s}-1}$,} \\
${\rm adm}_{s,\,40}= t_1t_2^{2^{s+1}-2}t_3^{2^{s+2}-2^{s}-1}$, & ${\rm adm}_{s,\,41}= t_2t_3^{2^{s+1}-1}t_4^{2^{s+2}-2^{s}-2}$, & \multicolumn{1}{l}{${\rm adm}_{s,\,42}= t_2^{2^{s+1}-1}t_3t_4^{2^{s+2}-2^{s}-2}$,} \\
${\rm adm}_{s,\,43}= t_1t_3^{2^{s+1}-1}t_4^{2^{s+2}-2^{s}-2}$, & ${\rm adm}_{s,\,44}= t_1^{2^{s+1}-1}t_3t_4^{2^{s+2}-2^{s}-2}$, & \multicolumn{1}{l}{${\rm adm}_{s,\,45}= t_1t_2^{2^{s+1}-1}t_4^{2^{s+2}-2^{s}-2}$,} \\
${\rm adm}_{s,\,46}= t_1^{2^{s+1}-1}t_2t_4^{2^{s+2}-2^{s}-2}$, & ${\rm adm}_{s,\,47}= t_1t_2^{2^{s+1}-1}t_3^{2^{s+2}-2^{s}-2}$, & \multicolumn{1}{l}{${\rm adm}_{s,\,48}= t_1^{2^{s+1}-1}t_2t_3^{2^{s+2}-2^{s}-2}$,} \\
${\rm adm}_{s,\,49}= t_3^{2^{s+1}-1}t_4^{2^{s+2}-2^{s}-1}$, & ${\rm adm}_{s,\,50}= t_2^{2^{s+1}-1}t_4^{2^{s+2}-2^{s}-1}$, & \multicolumn{1}{l}{${\rm adm}_{s,\,51}= t_2^{2^{s+1}-1}t_3^{2^{s+2}-2^{s}-1}$,} \\
${\rm adm}_{s,\,52}= t_1^{2^{s+1}-1}t_4^{2^{s+2}-2^{s}-1}$, & ${\rm adm}_{s,\,53}= t_1^{2^{s+1}-1}t_3^{2^{s+2}-2^{s}-1}$, & \multicolumn{1}{l}{${\rm adm}_{s,\,54}= t_1^{2^{s+1}-1}t_2^{2^{s+2}-2^{s}-1}$,} \\
${\rm adm}_{s,\,55}= t_2^{3}t_3^{2^{s+2}-3}t_4^{2^{s}-2}$, & ${\rm adm}_{s,\,56}= t_1^{3}t_3^{2^{s+2}-3}t_4^{2^{s}-2}$, & \multicolumn{1}{l}{${\rm adm}_{s,\,57}= t_1^{3}t_2^{2^{s+2}-3}t_4^{2^{s}-2}$,} \\
${\rm adm}_{s,\,58}= t_1^{3}t_2^{2^{s+2}-3}t_3^{2^{s}-2}$, & ${\rm adm}_{s,\,59}= t_2^{3}t_3^{2^{s+1}-3}t_4^{2^{s+2}-2^{s}-2}$, & \multicolumn{1}{l}{${\rm adm}_{s,\,60}= t_1^{3}t_3^{2^{s+1}-3}t_4^{2^{s+2}-2^{s}-2}$,} \\
${\rm adm}_{s,\,61}= t_1^{3}t_2^{2^{s+1}-3}t_4^{2^{s+2}-2^{s}-2}$, & ${\rm adm}_{s,\,62}= t_1^{3}t_2^{2^{s+1}-3}t_3^{2^{s+2}-2^{s}-2}$. &  
\end{tabular}%
\end{center}

For $s = 2,$

\begin{center}
\begin{tabular}{llll}
${\rm adm}_{2,\,63}= t_2^{3}t_3^{3}t_4^{12}$, & ${\rm adm}_{2,\,64}= t_1^{3}t_3^{3}t_4^{12}$, & ${\rm adm}_{2,\,65}= t_1^{3}t_2^{3}t_4^{12}$, & ${\rm adm}_{2,\,66}= t_1^{3}t_2^{3}t_3^{12}$.
\end{tabular}%
\end{center}

For $s\geq 3,$

\begin{center}
\begin{tabular}{lrr}
${\rm adm}_{s,\,63}= t_2^{3}t_3^{2^{s}-3}t_4^{2^{s+2}-2}$, & \multicolumn{1}{l}{${\rm adm}_{s,\,64}= t_1^{3}t_3^{2^{s}-3}t_4^{2^{s+2}-2}$,} & \multicolumn{1}{l}{${\rm adm}_{s,\,65}= t_1^{3}t_2^{2^{s}-3}t_4^{2^{s+2}-2}$,} \\
${\rm adm}_{s,\,66}= t_1^{3}t_2^{2^{s}-3}t_3^{2^{s+2}-2}$. &       &  
\end{tabular}%
\end{center}

\newpage
\begin{lema}\label{bdc2-1}
For each $s\geq 2,$  we have $ \dim [\underline{{\rm Ker}_{n_s}}]^{S_4}\leq 6.$
\end{lema}

\begin{proof}
By direct computations using the admissible basis of $\underline{{\rm Ker}_{n_s}}$ above, we notice that
$$ \begin{array}{ll}
\medskip
&S_4({\rm adm}_{s,\,1})  = \langle \{[{\rm adm}_{s,\,j}]:\, 1\leq j\leq 12\} \rangle,\\
&S_4({\rm adm}_{s,\,13},\ {\rm adm}_{s,\,37},\ {\rm adm}_{s,\,59},\ {\rm adm}_{s,\,63})  = \langle \{[{\rm adm}_{s,\,j}]:\, j\in \mathbb J\} \rangle,\ \mbox{where}\\
\medskip
&\quad \mathbb J = \{13, \ldots, 24, 37, \ldots, 48, 55, \ldots, 66\},\\
\medskip
&S_4({\rm adm}_{s,\,25})  = \langle \{[{\rm adm}_{s,\,j}]:\, 25\leq j\leq 36\} \rangle,\ \ S_4({\rm adm}_{s,\,49})  = \langle \{[{\rm adm}_{s,\,j}]:\, 49\leq j\leq 54\} \rangle,\\
\end{array}$$
are $S_4$-submodules of $\underline{{\rm Ker}_{n_s}}.$ Hence, we have an isomorphism
$$ \underline{{\rm Ker}_{n_s}} \cong S_4({\rm adm}_{s,\,1}) \bigoplus S_4({\rm adm}_{s,\,13}, {\rm adm}_{s,\,37}, {\rm adm}_{s,\,59}, {\rm adm}_{s,\,63})\bigoplus S_4({\rm adm}_{s,\,25})\bigoplus  S_4({\rm adm}_{s,\,49}).$$
Then, by the similar techniques as in the proof of Lemma \ref{bdc1-1}, it may be concluded that
\begin{itemize}
\item[i)] $[S_4({\rm adm}_{s,\,1})]^{S_4} = \langle [p_{s,\, 1}] \rangle$ with $p_{s,\, 1}:= \sum_{1\leq j\leq 12}{\rm adm}_{s,\,j}$ 

\item[ii)]  $\begin{array}{ll}
&[S_4({\rm adm}_{s,\,13}, {\rm adm}_{s,\,37}, {\rm adm}_{s,\,59}, {\rm adm}_{s,\,63})\bigoplus S_4({\rm adm}_{s,\,25})]^{S_4}\\
& = \left\{\begin{array}{ll}
\medskip
\langle [p_{2,\, 2} + p_{2,\, 3} + p_{2,\, 4} + p_{2,\, 5}], [p_{2,\, 3} + p_{2,\, 4} + p_{2,\, 6}], [p_{2,\, 2} + p_{2,\, 7}] \rangle&\mbox{if $s = 2$},\\
\langle [p_{s,\, 2} + p_{s,\, 3}],  [p_{s,\, 2} + p_{s,\, 4}], [p_{s,\, 2} + p_{s,\, 5}]  \rangle&\mbox{if $s\geq 3$},
\end{array}\right.
\end{array}$\\[1mm]
where 
$$ \begin{array}{ll}
\medskip
 p_{2,\, 2} &= {\rm adm}_{2,\,13} + {\rm adm}_{2,\,16}+{\rm adm}_{2,\,18}+{\rm adm}_{2,\,19},\\
\medskip
 p_{2,\, 3} &= {\rm adm}_{2,\,14} + {\rm adm}_{2,\,17}+{\rm adm}_{2,\,20}+{\rm adm}_{2,\,21},\\
\medskip
 p_{2,\, 4} &= {\rm adm}_{2,\,15} + {\rm adm}_{2,\,22}+{\rm adm}_{2,\,23}+{\rm adm}_{2,\,24},\\
\medskip
 p_{2,\, 5} &=  \sum_{37\leq j\leq 48}{\rm adm}_{2,\,j},\ \  p_{2,\, 6} =  \sum_{55\leq j\leq 58}{\rm adm}_{2,\,j} +  \sum_{63\leq j\leq 66}{\rm adm}_{2,\,j},\\
\medskip
 p_{2,\, 7} &=  \sum_{59\leq j\leq 62}{\rm adm}_{2,\,j},\ \ p_{s,\, 2} =  \sum_{13\leq j\leq 24}{\rm adm}_{s,\,j},\ \ p_{s,\, 3} =  \sum_{37\leq j\leq 48}{\rm adm}_{s,\,j},\\
\medskip
 p_{s,\, 4} &=  \sum_{55\leq j\leq 58}{\rm adm}_{s,\,j},\ \  p_{s,\, 5} =  \sum_{59\leq j\leq 66}{\rm adm}_{s,\,j}.
\end{array}$$

\item[iii)] $[S_4({\rm adm}_{s,\,25})]^{S_4} = \langle [p_{s,\, 6}] \rangle$ with $p_{s,\, 6}:= \sum_{25\leq j\leq 36}{\rm adm}_{s,\,j}.$ 

\item[iv)] $[S_4({\rm adm}_{s,\,49})]^{S_4} = \langle [p_{s,\, 7}] \rangle$ with $p_{s,\, 7}:= \sum_{49\leq j\leq 54}{\rm adm}_{s,\,j}.$ 

\end{itemize}
These imply that 
$$ \begin{array}{ll} 
\medskip
&\dim [S_4({\rm adm}_{s,\,j})]^{S_4} = 1\ \mbox{for}\ j = 1,\, 25,\, 49,\\
&\dim [S_4({\rm adm}_{s,\,13}, {\rm adm}_{s,\,37}, {\rm adm}_{s,\,59}, {\rm adm}_{s,\,63})]^{S_4} = 3.
\end{array}$$
Now, the lemma follows from the fact that 
$$ \begin{array}{ll}
\medskip
 \dim [\underline{{\rm Ker}_{n_s}}]^{S_4}&\leq \dim [S_4({\rm adm}_{s,\,1})]^{S_4} + \dim[S_4({\rm adm}_{s,\,13}, {\rm adm}_{s,\,37}, {\rm adm}_{s,\,59}, {\rm adm}_{s,\,63})]^{S_4} \\
&\quad + \dim [S_4({\rm adm}_{s,\,25})]^{S_4} + \dim \dim [S_4({\rm adm}_{s,\,49})]^{S_4}.
\end{array}$$
\end{proof}

From a result in Sum \cite{Sum1, Sum2}, $\widehat{{\rm Ker}_{n_s}}$ has a basis consisting of all the classes represented by the following admissible monomials ${\rm adm}_{s,\, j}$:

\newpage
For $s\geq 2,$

\begin{center}
\begin{tabular}{llr}
${\rm adm}_{s,\,67}= t_1t_2t_3^{2^{s}-2}t_4^{2^{s+2}-2}$, & ${\rm adm}_{s,\,68}= t_1t_2t_3^{2^{s+2}-2}t_4^{2^{s}-2}$, & \multicolumn{1}{l}{${\rm adm}_{s,\,69}= t_1t_2^{2^{s}-2}t_3t_4^{2^{s+2}-2}$,} \\
${\rm adm}_{s,\,70}= t_1t_2^{2^{s+2}-2}t_3t_4^{2^{s}-2}$, & ${\rm adm}_{s,\,71}= t_1t_2^{2}t_3^{2^{s+2}-3}t_4^{2^{s}-2}$, & \multicolumn{1}{l}{${\rm adm}_{s,\,72}= t_1t_2t_3^{2^{s+1}-2}t_4^{2^{s+2}-2^{s}-2}$,} \\
${\rm adm}_{s,\,73}= t_1t_2^{2^{s+1}-2}t_3t_4^{2^{s+2}-2^{s}-2}$, & ${\rm adm}_{s,\,74}= t_1t_2^{2}t_3^{2^{s}-1}t_4^{2^{s+2}-4}$, & \multicolumn{1}{l}{${\rm adm}_{s,\,75}= t_1t_2^{2}t_3^{2^{s+2}-4}t_4^{2^{s}-1}$,} \\
${\rm adm}_{s,\,76}= t_1t_2^{3}t_3^{2^{s}-2}t_4^{2^{s+2}-4}$, & ${\rm adm}_{s,\,77}= t_1t_2^{3}t_3^{2^{s+2}-4}t_4^{2^{s}-2}$, & \multicolumn{1}{l}{${\rm adm}_{s,\,78}= t_1^{3}t_2t_3^{2^{s}-2}t_4^{2^{s+2}-4}$,} \\
${\rm adm}_{s,\,79}= t_1^{3}t_2t_3^{2^{s+2}-4}t_4^{2^{s}-2}$, & ${\rm adm}_{s,\,80}= t_1t_2^{2}t_3^{2^{s+1}-4}t_4^{2^{s+2}-2^{s}-1}$, & \multicolumn{1}{l}{${\rm adm}_{s,\,81}= t_1t_2^{2}t_3^{2^{s+1}-3}t_4^{2^{s+2}-2^{s}-2}$,} \\
${\rm adm}_{s,\,82}= t_1t_2^{2}t_3^{2^{s+1}-1}t_4^{2^{s+2}-2^{s}-2}$, & ${\rm adm}_{s,\,83}= t_1t_2^{2^{s+1}-1}t_3^{2}t_4^{2^{s+2}-2^{s}-2}$, & \multicolumn{1}{l}{${\rm adm}_{s,\,84}= t_1^{2^{s+1}-1}t_2t_3^{2}t_4^{2^{s+2}-2^{s}-2}$,} \\
${\rm adm}_{s,\,85}= t_1t_2^{3}t_3^{2^{s+1}-4}t_4^{2^{s+2}-2^{s}-2}$, & ${\rm adm}_{s,\,86}= t_1^{3}t_2t_3^{2^{s+1}-4}t_4^{2^{s+2}-2^{s}-2}$, & \multicolumn{1}{l}{${\rm adm}_{s,\,87}= t_1t_2^{3}t_3^{2^{s+1}-2}t_4^{2^{s+2}-2^{s}-2}$,} \\
${\rm adm}_{s,\,88}= t_1^{3}t_2t_3^{2^{s+1}-2}t_4^{2^{s+2}-2^{s}-2}$, & ${\rm adm}_{s,\,89}= t_1^{3}t_2^{2^{s+1}-3}t_3^{2}t_4^{2^{s+2}-2^{s}-2}$. &  
\end{tabular}%
\end{center}

For $s = 2,$\ \ ${\rm adm}_{2,\,90}= t_1^{3}t_2^{3}t_3^{4}t_4^{8}$,\ \  ${\rm adm}_{2,\,91}= t_1^{3}t_2^{5}t_3^{8}t_4^{2}$.

For $s\geq 3,$

\begin{center}
\begin{tabular}{lll}
${\rm adm}_{s,\,90}= t_1t_2^{2}t_3^{2^{s}-4}t_4^{2^{s+2}-1}$, & ${\rm adm}_{s,\,91}= t_1t_2^{2}t_3^{2^{s+2}-1}t_4^{2^{s}-4}$, & ${\rm adm}_{s,\,92}= t_1t_2^{2^{s+2}-1}t_3^{2}t_4^{2^{s}-4}$, \\
${\rm adm}_{s,\,93}= t_1^{2^{s+2}-1}t_2t_3^{2}t_4^{2^{s}-4}$, & ${\rm adm}_{s,\,94}= t_1t_2^{2}t_3^{2^{s}-3}t_4^{2^{s+2}-2}$, & ${\rm adm}_{s,\,95}= t_1t_2^{3}t_3^{2^{s}-4}t_4^{2^{s+2}-2}$, \\
${\rm adm}_{s,\,96}= t_1t_2^{3}t_3^{2^{s+2}-2}t_4^{2^{s}-4}$, & ${\rm adm}_{s,\,97}= t_1^{3}t_2t_3^{2^{s}-4}t_4^{2^{s+2}-2}$, & ${\rm adm}_{s,\,98}= t_1^{3}t_2t_3^{2^{s+2}-2}t_4^{2^{s}-4}$, \\
${\rm adm}_{s,\,99}= t_1t_2^{2^{s}-1}t_3^{2}t_4^{2^{s+2}-4}$, & ${\rm adm}_{s,\,100}= t_1^{2^{s}-1}t_2t_3^{2}t_4^{2^{s+2}-4}$, & ${\rm adm}_{s,\,101}= t_1^{3}t_2^{2^{s+2}-3}t_3^{2}t_4^{2^{s}-4}$, \\
${\rm adm}_{s,\,102}= t_1^{3}t_2^{2^{s}-3}t_3^{2}t_4^{2^{s+2}-4}$, & ${\rm adm}_{s,\,103}= t_1^{3}t_2^{5}t_3^{2^{s+2}-6}t_4^{2^{s}-4}$, & ${\rm adm}_{s,\,104}= t_1^{3}t_2^{5}t_3^{2^{s+1}-6}t_4^{2^{s+2}-2^{s}-2}$.
\end{tabular}%
\end{center}

For $s = 3,$ \ \  ${\rm adm}_{3,\,105}= t_1^{3}t_2^{5}t_3^{6}t_4^{24}.$

For $s\geq 4,$ \ \ ${\rm adm}_{s,\,105}= t_1^{3}t_2^{5}t_3^{2^{s}-6}t_4^{2^{s+2}-4}.$

\begin{prop}\label{mdc1}
Let $s$ be a positive integer. Then, 
$$ [{\rm Ker}_{n_s}]^{GL_4}= \left\{\begin{array}{ll}
0&\mbox{if $s = 1$},\\
\langle [\widehat{p_{2,\, 1}}] \rangle &\mbox{if $s = 2$},\\
\langle [\sum_{1\leq j\leq 105}{\rm adm}_{s,\,j}] \rangle &\mbox{if $s \geq 3$},\\
\end{array}\right.$$
where 
$$ \begin{array}{ll}
\medskip
 \widehat{p_{2,\, 1}} &= {\rm adm}_{2,\,67} + {\rm adm}_{2,\,68} + {\rm adm}_{2,\,72} + {\rm adm}_{2,\,76} + {\rm adm}_{2,\,77}\\
&\quad + {\rm adm}_{2,\,78} + {\rm adm}_{2,\,79} + {\rm adm}_{2,\,89} + {\rm adm}_{2,\,91}.
\end{array}$$
\end{prop}

\begin{proof}
Firstly, we consider the case $s = 1.$ Suppose that $[\rho]\in [{\rm Ker}_{n_1}]^{GL_4}.$ Then , since $S_4\subset GL_4,$ by Lemmas \ref{bdc1-1} and \ref{bdc1-2}, we have $\rho \equiv  \beta_1p_{1,\, 1} + \beta_2p_{1,\, 2} + \beta_3p_{1,\, 3} +\beta_4p_{1,\, 4},$
where $\beta_i\in \mathbb Z/2.$ Direct calculating $\theta_4(\rho)$ in the admissible terms ${\rm adm}_{1,\,j}$ for $1\leq j\leq 49$ and using the relation $\theta_4(\rho) + \rho \equiv 0,$ we obtain $ \theta_4(\rho) + \rho \equiv [\beta_1{\rm adm}_{1,\,3} + (\beta_1 + \beta_2){\rm adm}_{1,\,9} + \beta_3{\rm adm}_{1,\,19} + (\beta_3 + \beta_4){\rm adm}_{1,\,32}+\ \mbox{other terms}] \equiv 0.$ This equality implies $\beta_1 = \beta_2 = \beta_3 = \beta_4 = 0.$ Thus, $[{\rm Ker}_{n_1}]^{GL_4}$ is trivial.

Next, for $s = 2,$ we see that the set $\{[{\rm adm}_{2,\,j}]:\, 67\leq j\leq 91\}$ is a basis of $\widehat{{\rm Ker}_{n_2}}.$ Assume that $[f]\in [\widehat{{\rm Ker}_{n_2}}]^{S_4}.$ Then, $f\equiv \sum_{67\leq j\leq 91}\gamma_j{\rm adm}_{2,\,j}$ with $\gamma_j\in \mathbb Z/2.$ Using the relations $\theta_i(f) + f\equiv 0$ for $1\leq i\leq 3,$ it follows that
$$ \begin{array}{ll}
\medskip
\theta_1(f) + f &\equiv [(\gamma_{69} + \gamma_{73}){\rm adm}_{2,\,67} + \gamma_{71}{\rm adm}_{2,\,68}+\gamma_{81}{\rm adm}_{2,\,72} + (\gamma_{73} + \gamma_{76} + \gamma_{78}){\rm adm}_{2,\,76}\\
\medskip
&\quad +(\gamma_{70} + \gamma_{77} + \gamma_{79}){\rm adm}_{2,\,77} + (\gamma_{83} + \gamma_{84}){\rm adm}_{2,\,83} + (\gamma_{85} + \gamma_{86}){\rm adm}_{2,\,85} \\
\medskip
&\quad + (\gamma_{87} + \gamma_{88}){\rm adm}_{2,\,87} +  \gamma_{70}{\rm adm}_{2,\, 89} + (\gamma_{70} + \gamma_{73} + \gamma_{89} + \gamma_{91}){\rm adm}_{2,\,90} \\
\medskip
&\quad+\mbox{other terms}]\ \equiv 0,\\ 
\end{array}$$

\newpage
$$ \begin{array}{ll}
\medskip
\theta_2(f) + f &\equiv [(\gamma_{67} +\gamma_{69} + \gamma_{75} + \gamma_{80} + \gamma_{86} + \gamma_{87} + \gamma_{90} + \gamma_{91}){\rm adm}_{2,\,67} \\
\medskip
&\quad + (\gamma_{67} +\gamma_{69} + \gamma_{75} + \gamma_{80} + + \gamma_{87} + \gamma_{90} + \gamma_{91}){\rm adm}_{2,\,69} \\
\medskip
&\quad +  (\gamma_{68} +\gamma_{70} + \gamma_{91}){\rm adm}_{2,\,68} + (\gamma_{68} +\gamma_{70} + \gamma_{79}){\rm adm}_{2,\,70} \\
\medskip
&\quad + (\gamma_{71} +\gamma_{77} + \gamma_{91}){\rm adm}_{2,\,71} +  (\gamma_{72} +\gamma_{73} + \gamma_{87} + \gamma_{90} + \gamma_{91}){\rm adm}_{2,\,72}\\
\medskip
&\quad +  (\gamma_{72} +\gamma_{73} + \gamma_{86} + \gamma_{87} + \gamma_{90} + \gamma_{91}){\rm adm}_{2,\,73}\\
\medskip
&\quad +  (\gamma_{74} +\gamma_{76} + \gamma_{77} + \gamma_{85} + \gamma_{87} + \gamma_{90}){\rm adm}_{2,\,74}\\
\medskip
&\quad +  (\gamma_{71} +\gamma_{74} + \gamma_{76} + \gamma_{81} + \gamma_{86} + \gamma_{87}+ \gamma_{90} + \gamma_{91}){\rm adm}_{2,\,76}\\
\medskip
&\quad +  (\gamma_{71} +\gamma_{77} + \gamma_{79}){\rm adm}_{2,\,77} + \gamma_{86}{\rm adm}_{2,\, 78} + (\gamma_{79} + \gamma_{91}){\rm adm}_{2,\, 79}\\
\medskip
&\quad + (\gamma_{81} + \gamma_{85}){\rm adm}_{2,\, 81} + (\gamma_{82} + \gamma_{83}){\rm adm}_{2,\, 82} + (\gamma_{88} + \gamma_{89} + \gamma_{91}){\rm adm}_{2,\, 88}\\
\medskip
&\quad + (\gamma_{79} + \gamma_{88} + \gamma_{89}){\rm adm}_{2,\, 89} + (\gamma_{79} + \gamma_{91}){\rm adm}_{2,\, 91}\\
\medskip
&\quad + (\gamma_{79} + \gamma_{86} + \gamma_{91}){\rm adm}_{2,\, 90} + \ \mbox{other terms}]\ \equiv 0,\\
\medskip
\theta_3(f) + f &\equiv [(\gamma_{67} +\gamma_{68} + \gamma_{71} + \gamma_{73}){\rm adm}_{2,\,67} +(\gamma_{67} +\gamma_{68} + \gamma_{69}){\rm adm}_{2,\,68}   \\
\medskip
&\quad (\gamma_{69} +\gamma_{71} + \gamma_{73}){\rm adm}_{2,\,69} + (\gamma_{73} +\gamma_{81}){\rm adm}_{2,\,72} + (\gamma_{74} +\gamma_{75}){\rm adm}_{2,\,74}\\
\medskip
&\quad + (\gamma_{70} +\gamma_{73} +\gamma_{76} + \gamma_{77} + \gamma_{83}){\rm adm}_{2,\,76} + (\gamma_{78} + \gamma_{79} + \gamma_{84}){\rm adm}_{2,\, 78}\\
\medskip
&\quad  + (\gamma_{80} +\gamma_{82}){\rm adm}_{2,\,80} + (\gamma_{85} +\gamma_{87}){\rm adm}_{2,\,85} + (\gamma_{86} +\gamma_{88}){\rm adm}_{2,\,86}\\
\medskip
&\quad + (\gamma_{89} +\gamma_{91}){\rm adm}_{2,\,89}\ +\mbox{other terms}]\ \equiv 0.
\end{array}$$
From these computations, it may be concluded that $ [\widehat{{\rm Ker}_{n_2}}]^{S_4} = \langle [\widehat{p_{2,\, 1}}], [\widehat{p_{2,\, 2}}], [\widehat{p_{2,\, 3}}] \rangle,$ where $$ \begin{array}{ll}
\widehat{p_{2,\, 1}}&:= \sum_{\gamma_j\in \mathbb J_1}{\rm adm}_{2,\, j},\ \mbox{with\ $\mathbb J_1 = \{67, 68, 72, 76, 77, 78, 79, 89, 91\},$}\\[1mm]
\widehat{p_{2,\, 2}}&:= \sum_{\gamma_j\in \mathbb J_2}{\rm adm}_{2,\, j},\ \mbox{with\ $\mathbb J_2 = \{74, 75, 76, 78, 80, 82, 83, 84\}$},\\[1mm]
\widehat{p_{2,\, 3}}&:= \sum_{\gamma_j\in \mathbb J_3}{\rm adm}_{2,\, j},\ \mbox{with\ $\mathbb J_3 = \{72, 74, 75, 90\}$}.
\end{array}$$

Now, suppose that $[\rho]\in [{\rm Ker}_{n_2}]^{GL_4},$ then combining the proof of Lemma \ref{bdc2-1} and the fact that $S_4\subset GL_4,$ we have
$$ \begin{array}{ll}
\medskip
 \rho&\equiv \alpha_1p_{2,\, 1} + \alpha_2\sum_{2\leq i\leq 5}p_{2,\, i} + \alpha_3(p_{2,\, 3} + p_{2,\, 4} + p_{2,\, 6}) \\
&\quad + \alpha_4(p_{2,\, 2} + p_{2,\, 7})   +  \alpha_5p_{2,\, 6} + \alpha_6p_{2,\, 7} + \alpha_7\widehat{p_{2,\, 1}} + \alpha_8\widehat{p_{2,\, 2}} +\alpha_9\widehat{p_{2,\, 3}}  .
\end{array}$$
Direct calculating $\theta_4(\rho)$ in the admissible terms ${\rm adm}_{2,\,j}$ for $1\leq j\leq 91$ and using the relation $\theta_4(\rho) + \rho \equiv 0,$ we find that 
$$ \begin{array}{ll}
\medskip
\theta_4(\rho) + \rho &\equiv [\alpha_1{\rm adm}_{2,\, 1} + (\gamma_{1} + \gamma_{6}){\rm adm}_{2,\, 8} + (\gamma_{1} + \gamma_{2}){\rm adm}_{2,\, 27} + (\gamma_{2} + \gamma_{4}+ \gamma_{6}){\rm adm}_{2,\, 29}\\
\medskip
&\quad + (\gamma_{2} + \gamma_{3}){\rm adm}_{2,\, 33} + (\gamma_{4} + \gamma_{8}){\rm adm}_{2,\, 8} + \gamma_{5}{\rm adm}_{2,\, 59}\\
\medskip
&\quad + (\gamma_{4} + \gamma_{6} + \gamma_8 + \gamma_9){\rm adm}_{2,\, 14} + \ \mbox{other terms}]\ \equiv 0.
\end{array}$$
This equality shows that $\alpha_i = 0$ for $i\neq 7.$ Therefore, $[{\rm Ker}_{n_2}]^{GL_4} = \langle [\widehat{p_{2,\, 1}}]\rangle.$

Finally, by using Lemma \ref{bdc2-1} and the similar techniques as above, for each $s\geq 3,$ we get $[{\rm Ker}_{n_s}]^{GL_4} = \langle [\sum_{1\leq j\leq 105}{\rm adm}_{s,\,j}] \rangle.$ The proposition is proved.
\end{proof}

\begin{proof}[{\bf Proof of Theorem \ref{dlc1}}]
The theorem follows from the inequality \eqref{bdt}, Proposition \ref{mdc1} and the fact that $\mathbb Z/2 \otimes_{GL_4} {\rm Ann}_{\overline{\mathcal A}}[P_{\frac{n_1-4}{2}}^{\otimes 4}]^{*}$ is trivial (see Sum \cite{Sum3}).
\end{proof}

\begin{proof}[{\bf Proof of Theorem \ref{dlc2}}]
It is known, from a theorem of Sum \cite{Sum3} that $\mathbb Z/2 \otimes_{GL_4} {\rm Ann}_{\overline{\mathcal A}}[P_{\frac{n_2-4}{2}}^{\otimes 4}]^{*}$ is 1-dimensional and generated by $[x_1^{(0)}x_2^{(0)}x_3^{(0)}x_4^{(7)}].$ Combining this with the ineequality \eqref{bdt} and Proposition \ref{mdc1}, we deduce that $ \dim \mathbb Z/2 \otimes_{GL_4} {\rm Ann}_{\overline{\mathcal A}}[P_{n_2}^{\otimes 4}]^{*}\leq 2.$ On the other hand, consider the element
$$ \begin{array}{ll}
\zeta&=  (x_1^{(3)}x_2^{(5)}x_3^{(1)}x_4^{(9)}+
 x_1^{(3)}x_2^{(5)}x_3^{(2)}x_4^{(8)}+
 x_1^{(3)}x_2^{(6)}x_3^{(1)}x_4^{(8)}+
 x_1^{(3)}x_2^{(6)}x_3^{(2)}x_4^{(7)}+
\medskip
 x_1^{(3)}x_2^{(5)}x_3^{(4)}x_4^{(6)}\\
&\quad+
 x_1^{(3)}x_2^{(6)}x_3^{(3)}x_4^{(6)}+
 x_1^{(5)}x_2^{(6)}x_3^{(1)}x_4^{(6)}+
 x_1^{(3)}x_2^{(5)}x_3^{(5)}x_4^{(5)}+
 x_1^{(3)}x_2^{(6)}x_3^{(4)}x_4^{(5)}+
\medskip
 x_1^{(5)}x_2^{(6)}x_3^{(2)}x_4^{(5)}\\
&\quad+
 x_1^{(3)}x_2^{(9)}x_3^{(1)}x_4^{(5)}+
 x_1^{(5)}x_2^{(7)}x_3^{(1)}x_4^{(5)}+
 x_1^{(3)}x_2^{(9)}x_3^{(2)}x_4^{(4)}+
 x_1^{(5)}x_2^{(7)}x_3^{(2)}x_4^{(4)}+
\medskip
 x_1^{(3)}x_2^{(10)}x_3^{(1)}x_4^{(4)}\\
&\quad+
 x_1^{(6)}x_2^{(7)}x_3^{(1)}x_4^{(4)}+
 x_1^{(5)}x_2^{(6)}x_3^{(4)}x_4^{(3)}+
 x_1^{(6)}x_2^{(7)}x_3^{(2)}x_4^{(3)}+
 x_1^{(3)}x_2^{(10)}x_3^{(2)}x_4^{(3)}+
\medskip
 x_1^{(3)}x_2^{(11)}x_3^{(2)}x_4^{(2)}\\
&\quad+
 x_1^{(5)}x_2^{(9)}x_3^{(2)}x_4^{(2)}+
 x_1^{(6)}x_2^{(10)}x_3^{(1)}x_4^{(1)}+
 x_1^{(3)}x_2^{(5)}x_3^{(9)}x_4^{(1)}+
 x_1^{(3)}x_2^{(5)}x_3^{(8)}x_4^{(2)}+
\medskip
 x_1^{(3)}x_2^{(6)}x_3^{(8)}x_4^{(1)}\\
&\quad+
 x_1^{(3)}x_2^{(6)}x_3^{(7)}x_4^{(2)}+
 x_1^{(3)}x_2^{(5)}x_3^{(6)}x_4^{(4)}+
 x_1^{(3)}x_2^{(6)}x_3^{(6)}x_4^{(3)}+
 x_1^{(5)}x_2^{(6)}x_3^{(6)}x_4^{(1)}+
\medskip
 x_1^{(3)}x_2^{(5)}x_3^{(5)}x_4^{(5)}\\
&\quad+
 x_1^{(3)}x_2^{(6)}x_3^{(5)}x_4^{(4)}+
 x_1^{(5)}x_2^{(6)}x_3^{(5)}x_4^{(2)}+
 x_1^{(3)}x_2^{(9)}x_3^{(5)}x_4^{(1)}+
 x_1^{(5)}x_2^{(7)}x_3^{(5)}x_4^{(1)}+
\medskip
 x_1^{(3)}x_2^{(9)}x_3^{(4)}x_4^{(2)}\\
&\quad+
 x_1^{(5)}x_2^{(7)}x_3^{(4)}x_4^{(2)}+
 x_1^{(3)}x_2^{(10)}x_3^{(4)}x_4^{(1)}+
 x_1^{(6)}x_2^{(7)}x_3^{(4)}x_4^{(1)}+
 x_1^{(5)}x_2^{(6)}x_3^{(3)}x_4^{(4)}+
\medskip
 x_1^{(6)}x_2^{(7)}x_3^{(3)}x_4^{(2)}\\
&\quad+
 x_1^{(3)}x_2^{(10)}x_3^{(3)}x_4^{(2)}+
 x_1^{(3)}x_2^{(11)}x_3^{(2)}x_4^{(2)}+
 x_1^{(5)}x_2^{(9)}x_3^{(2)}x_4^{(2)}+
 x_1^{(6)}x_2^{(10)}x_3^{(1)}x_4^{(1)}+
\medskip
 x_1^{(3)}x_2^{(12)}x_3^{(1)}x_4^{(2)}\\
&\quad+
 x_1^{(7)}x_2^{(8)}x_3^{(1)}x_4^{(2)}+
 x_1^{(11)}x_2^{(4)}x_3^{(1)}x_4^{(2)}+
 x_1^{(13)}x_2^{(2)}x_3^{(1)}x_4^{(2)}+
 x_1^{(14)}x_2^{(1)}x_3^{(1)}x_4^{(2)}+
\medskip
 x_1^{(12)}x_2^{(3)}x_3^{(1)}x_4^{(2)}\\
&\quad+
 x_1^{(8)}x_2^{(7)}x_3^{(1)}x_4^{(2)}+
 x_1^{(4)}x_2^{(11)}x_3^{(1)}x_4^{(2)}+
 x_1^{(2)}x_2^{(13)}x_3^{(1)}x_4^{(2)}+
 x_1^{(1)}x_2^{(14)}x_3^{(1)}x_4^{(2)}+
\medskip
 x_1^{(6)}x_2^{(6)}x_3^{(3)}x_4^{(3)}\\
&\quad+
 x_1^{(5)}x_2^{(5)}x_3^{(5)}x_4^{(3)}+
 x_1^{(3)}x_2^{(3)}x_3^{(9)}x_4^{(3)}+
 x_1^{(5)}x_2^{(3)}x_3^{(7)}x_4^{(3)}+
 x_1^{(7)}x_2^{(7)}x_3^{(2)}x_4^{(2)}+
\medskip
 x_1^{(6)}x_2^{(9)}x_3^{(1)}x_4^{(2)}\\
&\quad+
 x_1^{(9)}x_2^{(6)}x_3^{(1)}x_4^{(2)}+
 x_1^{(10)}x_2^{(5)}x_3^{(1)}x_4^{(2)}+
 x_1^{(5)}x_2^{(10)}x_3^{(2)}x_4^{(1)}+
 x_1^{(13)}x_2^{(3)}x_3^{(1)}x_4^{(1)}+
\medskip
 x_1^{(5)}x_2^{(11)}x_3^{(1)}x_4^{(1)}\\
&\quad+
 x_1^{(9)}x_2^{(7)}x_3^{(1)}x_4^{(1)})\in [P_{n_2}^{\otimes 4}]^{*}.
\end{array}$$
Then, it is $\overline{\mathcal A}$-annahilated. Indeed, by the unstable condition, we need only to effects of the Steenrod squares $Sq^{2^{i}}$ for $i = 0, 1, 2, 3.$ It is easy to check that $(\zeta)Sq^{2^{i}} = 0$ for $i = 0, 1, 2.$ A direct computation shows that
$$ \begin{array}{ll}
(\zeta)Sq^{8} &=  x_1^{(3)}x_2^{(3)}x_3^{(1)}x_4^{(3)}+
 x_1^{(3)}x_2^{(3)}x_3^{(1)}x_4^{(3)}+
 x_1^{(3)}x_2^{(5)}x_3^{(1)}x_4^{(1)}+
 x_1^{(3)}x_2^{(5)}x_3^{(1)}x_4^{(1)}+
\medskip
 x_1^{(3)}x_2^{(3)}x_3^{(3)}x_4^{(1)}\\
&+
 x_1^{(3)}x_2^{(3)}x_3^{(3)}x_4^{(1)}+
 x_1^{(3)}x_2^{(5)}x_3^{(1)}x_4^{(1)}+
 x_1^{(3)}x_2^{(5)}x_3^{(1)}x_4^{(1)}+
 x_1^{(7)}x_2^{(1)}x_3^{(1)}x_4^{(1)}+
\medskip
 x_1^{(7)}x_2^{(1)}x_3^{(1)}x_4^{(1)}\\
&+
 x_1^{(1)}x_2^{(7)}x_3^{(1)}x_4^{(1)}+
 x_1^{(1)}x_2^{(7)}x_3^{(1)}x_4^{(1)}+
 x_1^{(3)}x_2^{(5)}x_3^{(1)}x_4^{(1)}+
 x_1^{(3)}x_2^{(5)}x_3^{(1)}x_4^{(1)}+
\medskip
 x_1^{(5)}x_2^{(3)}x_3^{(1)}x_4^{(1)}\\
&+
 x_1^{(5)}x_2^{(3)}x_3^{(1)}x_4^{(1)} = 0.
\end{array}$$
Moreover, it is easy to see that the element $x_1^{(1)}x_2^{(1)}x_3^{(1)}x_4^{(15)}\in [P_{n_2}^{\otimes 4}]^{*}$ is also $\overline{\mathcal A}$-annahilated. Then, using the representation in the algebra $\Lambda$ of the fourth cohomological transfer, we deduce that 
$$ \begin{array}{ll}
\medskip
& \psi_4(x_1^{(1)}x_2^{(1)}x_3^{(1)}x_4^{(15)}) = \lambda_1^{3}\lambda_{15},\\
 &\psi_4(\zeta) = \lambda_4\lambda_6\lambda_5\lambda_3 + \lambda_5\lambda_7\lambda^2_3 + \lambda_3^2\lambda_2\lambda_5\lambda_7 + \lambda_2\lambda_4\lambda_5\lambda_7 + \delta(\lambda_3\lambda_5\lambda_{11})
\end{array}$$ are the cycles in $\Lambda^{4, n_2}$ and 
$$ \begin{array}{ll}
\medskip
&Tr_4^{\mathcal A}([x_1^{(1)}x_2^{(1)}x_3^{(1)}x_4^{(15)}]) = [\psi_4(x_1^{(1)}x_2^{(1)}x_3^{(1)}x_4^{(15)})] =  h_1^{3}h_4\in  {\rm Ext}_{\mathcal A}^{4, 4+n_2}(\mathbb Z/2, \mathbb Z/2), \\
&Tr_4^{\mathcal A}([\zeta]) = [\psi_4(\zeta)] =  f_0\in  {\rm Ext}_{\mathcal A}^{4, 4+n_2}(\mathbb Z/2, \mathbb Z/2).
\end{array}$$
Combining the arguments above with the equality \eqref{kqL}, we claim that $\mathbb Z/2 \otimes_{GL_4} {\rm Ann}_{\overline{\mathcal A}}[P_{n_2}^{\otimes 4}]^{*}$ is 2-dimensional.

\medskip

Now, let $\rho\in P^{\otimes 4}_{n_2}$ such that $[\rho]\in [QP^{\otimes 4}_{n_2}]^{GL_4}.$ Since Kameko's homomorphism $$(\widetilde {Sq^0_*})_{n_2}: QP^{\otimes 4}_{n_2}  \longrightarrow QP^{\otimes 4}_{\frac{n_2-4}{2}}$$ is an epimorphism of $\mathbb Z/2GL_4$-modules,  $(\widetilde {Sq^0_*})_{n_2}([\rho])$ belongs to the invariant space $[QP^{\otimes 4}_{\frac{n_2-4}{2}}]^{GL_4}.$ By Sum \cite[Proposition 4.3.2]{Sum3}, $[QP^{\otimes 4}_{\frac{n_2-4}{2}}]^{GL_4} = \langle [\overline{p}_{4,\, 2}] \rangle,$ where $$\overline{p}_{4,\, 2} = \sum_{1\leq \ell\leq 3}\sum_{1\leq i_1\leq \ldots\leq i_{\ell}\leq 4}t_{i_1}t_{i_2}^{2}\ldots t^{2^{\ell-2}}_{i_{\ell-1}}t^{8-2^{\ell-1}}_{i_{\ell}} + t_1t_2^{2}t_3^{2}t_4^{2}.$$
So, $(\widetilde {Sq^0_*})_{n_2}([\rho]) = \gamma[\varphi(\overline{p}_{4,\, 2})]$ where $\gamma\in \mathbb Z/2$ and the $\mathbb Z/2$-linear map $\varphi: P^{\otimes 4}_{\frac{n_2-4}{2}}\to P^{\otimes 4}_{n_2}$ is determined by $\varphi(u) = t_1t_2t_3t_4u^{2}$ for any $u\in P^{\otimes 4}_{\frac{n_2-4}{2}}.$ Because $[\rho]\in [QP^{\otimes 4}_{n_2}]^{GL_4},$ we have
$ \rho \equiv \gamma\varphi(\overline{p}_{4,\, 2}) + \rho^{*},$ where $\rho^{*}\in P^{\otimes 4}_{n_2}$ such that $[\rho^{*}]\in {\rm Ker}_{n_2}.$ By a direct computation using Proposition \ref{mdc1} and the relations $\theta_i(\rho) + \rho \equiv 0$ for $1\leq i\leq 4$, we get $[QP^{\otimes 4}_{n_2}]^{GL_4} = \langle[\varphi(\overline{p}_{4,\, 2})], [\widehat{p_{2,\, 1}}]\rangle$. Furthermore, we see that 
$$ \begin{array}{ll}
\medskip
&\langle [\varphi(\overline{p}_{4,\, 2})], [\zeta]\rangle = 0, \ \ \langle [\widehat{p_{2,\, 1}})], [x_1^{(1)}x_2^{(1)}x_3^{(1)}x_4^{(15)}]\rangle = 0,\\
&\langle [\widehat{p_{2,\, 1}}], [\zeta]\rangle = 1,\ \ \ \ \ \ \langle [\varphi(\overline{p}_{4,\, 2})], [x_1^{(1)}x_2^{(1)}x_3^{(1)}x_4^{(15)}] \rangle = 1.
\end{array}$$ 
Therefore, $\mathbb Z/2 \otimes_{GL_4} {\rm Ann}_{\overline{\mathcal A}}[P_{n_2}^{\otimes 4}]^{*} = \langle [x_1^{(1)}x_2^{1}x_3^{(1)}x_4^{(15)}], [\zeta] \rangle.$ The theorem is proved.
\end{proof}

\begin{proof}[{\bf Proof of Theorem \ref{dlc3}}]
We consider the following elements in $[P_{n_3}^{\otimes 4}]^{*}$:\\[1mm]
$ \begin{array}{ll}
\medskip
&\overline{\zeta} = x_1^{(0)}x_2^{(0)}x_3^{(7)}x_4^{(31)},\\
&\zeta=x_1^{(11)}x_2^{(11)}x_3^{(11)}x_4^{(5)}+
 x_1^{(11)}x_2^{(11)}x_3^{(13)}x_4^{(3)}+
 x_1^{(7)}x_2^{(11)}x_3^{(17)}x_4^{(3)}+
\medskip
x_1^{(11)}x_2^{(7)}x_3^{(17)}x_4^{(3)}\\
&+ x_1^{(11)}x_2^{(15)}x_3^{(9)}x_4^{(3)}+
x_1^{(15)}x_2^{(11)}x_3^{(9)}x_4^{(3)}+
 x_1^{(7)}x_2^{(19)}x_3^{(9)}x_4^{(3)}+
 x_1^{(19)}x_2^{(7)}x_3^{(9)}x_4^{(3)}+
\medskip
x_1^{(7)}x_2^{(19)}x_3^{(7)}x_4^{(5)}\\
&+
x_1^{(19)}x_2^{(7)}x_3^{(7)}x_4^{(5)}+
x_1^{(11)}x_2^{(19)}x_3^{(5)}x_4^{(3)}+
x_1^{(19)}x_2^{(11)}x_3^{(5)}x_4^{(3)} +
x_1^{(11)}x_2^{(21)}x_3^{(3)}x_4^{(3)}+
\medskip
x_1^{(19)}x_2^{(13)}x_3^{(3)}x_4^{(3)}\\
&+
 x_1^{(7)}x_2^{(23)}x_3^{(5)}x_4^{(3)}+
 x_1^{(23)}x_2^{(7)}x_3^{(5)}x_4^{(3)} +
 x_1^{(11)}x_2^{(11)}x_3^{(7)}x_4^{(9)}+
 x_1^{(11)}x_2^{(7)}x_3^{(11)}x_4^{(9)}+
\medskip
 x_1^{(7)}x_2^{(11)}x_3^{(11)}x_4^{(9)}\\
&+
 x_1^{(7)}x_2^{(25)}x_3^{(3)}x_4^{(3)}+
  x_1^{(23)}x_2^{(9)}x_3^{(3)}x_4^{(3)}+
 x_1^{(15)}x_2^{(17)}x_3^{(3)}x_4^{(3)}+
 x_1^{(15)}x_2^{(15)}x_3^{(3)}x_4^{(5)}+
\medskip
x_1^{(27)}x_2^{(5)}x_3^{(3)}x_4^{(3)}\\
&+
 x_1^{(29)}x_2^{(3)}x_3^{(3)}x_4^{(3)}+
 x_1^{(13)}x_2^{(11)}x_3^{(7)}x_4^{(7)}+
 x_1^{(11)}x_2^{(7)}x_3^{(13)}x_4^{(7)}+
 x_1^{(7)}x_2^{(13)}x_3^{(11)}x_4^{(7)} +
\medskip
 x_1^{(13)}x_2^{(7)}x_3^{(7)}x_4^{(11)}\\
&+
 x_1^{(7)}x_2^{(7)}x_3^{(13)}x_4^{(11)}+
 x_1^{(7)}x_2^{(13)}x_3^{(7)}x_4^{(11)}+
  x_1^{(11)}x_2^{(7)}x_3^{(7)}x_4^{(13)}+
 x_1^{(7)}x_2^{(11)}x_3^{(7)}x_4^{(13)}+
\medskip
 x_1^{(7)}x_2^{(7)}x_3^{(11)}x_4^{(13)}\\
&+
 x_1^{(7)}x_2^{(7)}x_3^{(7)}x_4^{(17)}+
x_1^{(7)}x_2^{(7)}x_3^{(9)}x_4^{(15)} +
 x_1^{(7)}x_2^{(11)}x_3^{(5)}x_4^{(15)}+
 x_1^{(7)}x_2^{(13)}x_3^{(3)}x_4^{(15)}+
\medskip
 x_1^{(7)}x_2^{(7)}x_3^{(19)}x_4^{(5)}\\
&+
 x_1^{(7)}x_2^{(7)}x_3^{(21)}x_4^{(3)}+
x_1^{(11)}x_2^{(7)}x_3^{(15)}x_4^{(5)}+
 x_1^{(11)}x_2^{(15)}x_3^{(7)}x_4^{(5)}+
 x_1^{(15)}x_2^{(11)}x_3^{(7)}x_4^{(5)} + x_1^{(7)}x_2^{(13)}x_3^{(15)}x_4^{(3)}.
\end{array}$\\[2mm]
Then, we can immediately see that $\overline{\zeta}\in {\rm Ext}_{\mathcal A}^{0, n_3}(\mathbb Z/2, P^{\otimes 4}).$ Further, by direct calculations, we get $(\zeta)Sq^{2^{i}} = 0$ for $0\leq i\leq 2,$ and\\[2mm]
$\begin{array}{ll}
(\zeta)Sq^{8} &=  x_1^{(7)}x_2^{(7)}x_3^{(11)}x_4^{(5)}+
 x_1^{(7)}x_2^{(11)}x_3^{(7)}x_4^{(5)}+
 x_1^{(11)}x_2^{(7)}x_3^{(7)}x_4^{(5)}+
\medskip
 x_1^{(7)}x_2^{(7)}x_3^{(13)}x_4^{(3)}\\
&+
 x_1^{(7)}x_2^{(7)}x_3^{(13)}x_4^{(3)}+
 x_1^{(11)}x_2^{(7)}x_3^{(9)}x_4^{(3)}+
 x_1^{(7)}x_2^{(7)}x_3^{(13)}x_4^{(3)}+
\medskip
 x_1^{(7)}x_2^{(15)}x_3^{(5)}x_4^{(3)}\\
&+
 x_1^{(15)}x_2^{(7)}x_3^{(5)}x_4^{(3)}+
 x_1^{(7)}x_2^{(11)}x_3^{(9)}x_4^{(3)}+
 x_1^{(7)}x_2^{(15)}x_3^{(5)}x_4^{(3)}+
\medskip
 x_1^{(11)}x_2^{(7)}x_3^{(9)}x_4^{(3)}\\
&+
 x_1^{(15)}x_2^{(7)}x_3^{(5)}x_4^{(3)}+
 x_1^{(7)}x_2^{(11)}x_3^{(7)}x_4^{(5)}+
 x_1^{(11)}x_2^{(7)}x_3^{(7)}x_4^{(5)}+
\medskip
 x_1^{(11)}x_2^{(11)}x_3^{(5)}x_4^{(3)}\\
\end{array}$

\newpage
$\begin{array}{ll}
&+
 x_1^{(7)}x_2^{(15)}x_3^{(5)}x_4^{(3)}+
 x_1^{(11)}x_2^{(11)}x_3^{(5)}x_4^{(3)}+
 x_1^{(15)}x_2^{(7)}x_3^{(5)}x_4^{(3)}+
\medskip
 x_1^{(11)}x_2^{(13)}x_3^{(3)}x_4^{(3)}\\
&+
 x_1^{(11)}x_2^{(13)}x_3^{(3)}x_4^{(3)}+
 x_1^{(7)}x_2^{(15)}x_3^{(5)}x_4^{(3)}+
 x_1^{(15)}x_2^{(7)}x_3^{(5)}x_4^{(3)}+
\medskip
 x_1^{(7)}x_2^{(7)}x_3^{(7)}x_4^{(9)}\\
&+
 x_1^{(7)}x_2^{(11)}x_3^{(7)}x_4^{(5)}+
 x_1^{(11)}x_2^{(7)}x_3^{(7)}x_4^{(5)}+
 x_1^{(7)}x_2^{(7)}x_3^{(7)}x_4^{(9)}+
\medskip
 x_1^{(7)}x_2^{(7)}x_3^{(11)}x_4^{(5)}\\
&+
 x_1^{(11)}x_2^{(7)}x_3^{(7)}x_4^{(5)}+
 x_1^{(7)}x_2^{(7)}x_3^{(7)}x_4^{(9)}+
 x_1^{(7)}x_2^{(7)}x_3^{(11)}x_4^{(5)}+
\medskip
 x_1^{(7)}x_2^{(11)}x_3^{(7)}x_4^{(5)}\\
&+
 x_1^{(15)}x_2^{(9)}x_3^{(3)}x_4^{(3)}+
 x_1^{(15)}x_2^{(9)}x_3^{(3)}x_4^{(3)}+
\medskip
 x_1^{(7)}x_2^{(7)}x_3^{(7)}x_4^{(9)}\\
&+
 x_1^{(7)}x_2^{(7)}x_3^{(11)}x_4^{(5)}+
 x_1^{(7)}x_2^{(7)}x_3^{(13)}x_4^{(3)}+
 x_1^{(7)}x_2^{(11)}x_3^{(9)}x_4^{(3)} = 0.
\end{array}$\\[2mm]
Thus, by the unstable condition, we claim that $\zeta\in {\rm Ext}_{\mathcal A}^{0, n_3}(\mathbb Z/2, P^{\otimes 4}).$  Using the representation of the rank 4 algebraic transfer over the algebra $\Lambda$, we find that the following cycles 
$$ \begin{array}{ll}
\medskip
 \psi_4(\overline{\zeta}) &= \lambda_0^{2}\lambda_{7}\lambda_{31}\in \Lambda^{4, n_3}\\
\medskip
\psi_4(\zeta) &= (\lambda_7^{3}\lambda_{17} + (\lambda_7\lambda_{11}^{2} + \lambda_7^{2}\lambda_{15})\lambda_9 + \lambda_{15}\lambda_{11}\lambda_7\lambda_5 + \lambda_7^{2}\lambda_{11}\lambda_{13}\\
&\quad +\delta(\lambda_7\lambda_{11}\lambda_{21} + \lambda_7\lambda_{25}\lambda_7 + \lambda_9\lambda_{15}^{2} + \lambda_1\lambda_{23}\lambda_{15}))\in \Lambda^{4, n_3}
\end{array}$$
in $\Lambda$ are representative of the non-zero elements $h_0^{2}h_3h_{5}$ and $e_1\in {\rm Ext}_{\mathcal A}^{4, 4+n_3}(\mathbb Z/2, \mathbb Z/2)$ respectively. So, by the equality \eqref{kqL}, we get $\dim \mathbb Z/2 \otimes_{GL_4} {\rm Ann}_{\overline{\mathcal A}}[P_{n_3}^{\otimes 4}]^{*}\geq 2.$ On the other hand, we have shown in \cite{Phuc9} that $\dim \mathbb Z/2 \otimes_{GL_4} {\rm Ann}_{\overline{\mathcal A}}[P_{\frac{n_3-4}{2}}^{\otimes 4}]^{*} = 1.$ Combining these data with the inequality \eqref{bdt} and Proposition \ref{mdc1}, we obtain $\dim \mathbb Z/2 \otimes_{GL_4} {\rm Ann}_{\overline{\mathcal A}}[P_{n_3}^{\otimes 4}]^{*} = 2.$

\medskip

Now, if $[\widetilde{\rho}]\in [QP^{\otimes 4}_{n_3}]^{GL_4},$ then the invariant space $[QP^{\otimes 4}_{\frac{n_3-4}{2}}]^{GL_4}$ contains the image of $[\widetilde{\rho}]$ under Kameko's homomorphism $(\widetilde {Sq^0_*})_{n_3}.$ On the other hand, in \cite{Phuc9}, we showed that
$[QP^{\otimes 4}_{\frac{n_3-4}{2}}]^{GL_4} = \langle [\widetilde{\zeta}] \rangle$ with 
$$ \begin{array}{ll}
\widetilde{\zeta}&=  t_1t_2^{2}t_3^{7}t_4^{7}+
 t_1t_2^{7}t_3^{2}t_4^{7}+
 t_1t_2^{3}t_3t_4^{12}+
\medskip
 t_1t_2^{3}t_3^{12}t_4\\
&\quad+
 t_1^{3}t_2t_3t_4^{12}+
 t_1^{3}t_2t_3^{12}t_4+
 t_1^{3}t_2^{5}t_3t_4^{8}+
 t_1^{3}t_2^{5}t_3^{8}t_4.
\end{array}$$  
Hence, $(\widetilde {Sq^0_*})_{n_3}([\widetilde{\rho}]) = \beta[\overline{\varphi}(\widetilde{\zeta})]$ where $\beta\in \mathbb Z/2$ and the $\mathbb Z/2$-linear map $\overline{\varphi}: P^{\otimes 4}_{\frac{n_3-4}{2}}\to P^{\otimes 4}_{n_3}$ is a linear map determined by $\overline{\varphi}(v) = \prod_{1\leq i\leq 4}t_iv^{2}$ for all $v = \prod_{1\leq i\leq 4}t_i^{a_i}\in P^{\otimes 4}_{\frac{n_3-4}{2}}.$ Thus, we have $ \widetilde{\rho} \equiv  \beta\overline{\varphi}(\widetilde{\zeta})+ \overline{\rho},$ where $\overline{\rho}\in P^{\otimes 4}_{n_3}$ such that $[\overline{\rho}]\in {\rm Ker}_{n_3}.$ Using Proposition \ref{mdc1} and the relations $\theta_i(\widetilde{\rho})\equiv \widetilde{\rho},\, i = 1, 2, 3,4,$ we obtain $\widetilde{\rho}\equiv \beta\overline{\varphi}(\widetilde{\zeta})+ \beta'\sum_{1\leq j\leq 105}{\rm adm}_{3,\, j},\ \ (\beta'\in \mathbb Z/2).$ Combining this and the fact that
$$ \begin{array}{ll}
\medskip
&\langle [\overline{\varphi}(\widetilde{\zeta})], [\overline{\zeta}]\rangle = 0, \ \  \langle [\sum_{1\leq j\leq 105}{\rm adm}_{3,\, j}], [\zeta]\rangle = 0,\\
&\langle [\sum_{1\leq j\leq 105}{\rm adm}_{3,\, j}], [\overline{\zeta}]\rangle = 1,\ \  \langle [\overline{\varphi}(\widetilde{\zeta})], [\zeta] \rangle = 1,
\end{array}$$ 
we get $\mathbb Z/2 \otimes_{GL_4} {\rm Ann}_{\overline{\mathcal A}}[P_{n_3}^{\otimes 4}]^{*} = \langle [\overline{\zeta}], [\zeta] \rangle.$ The proof of the theorem is completed.
\end{proof}

\begin{proof}[{\bf Proof of Theorem \ref{dlc4}}]
From the inequality \eqref{bdt}, Proposition \ref{mdc1} and our result in \cite{Phuc9} that the coinvarian space $\mathbb Z/2 \otimes_{GL_4} {\rm Ann}_{\overline{\mathcal A}}[P_{\frac{n_s-4}{2}}^{\otimes 4}]^{*}$ is trivial for all $s\geq 4$, one gets an estimate that $\dim \mathbb Z/2 \otimes_{GL_4} {\rm Ann}_{\overline{\mathcal A}}[P_{n_3}^{\otimes 4}]^{*}\leq 1.$ On the other hand, it is easy to see that the element $\zeta_s = x_1^{(0)}x_2^{(0)}x_3^{(2^{s}-1)}x_4^{(2^{s+2}-1)}$ belongs to ${\rm Ext}_{\mathcal A}^{0, n_s}(\mathbb Z/2, P^{\otimes 4}).$ So, based on the equality \eqref{kqL} and a representation of the rank 4 transfer   over the lambda algebra, we claim that $h_0^{2}h_sh_{s+2} = Tr_4^{\mathcal A}([\zeta]) = [\psi_4(\zeta)] = [\lambda_0^{2}\lambda_{2^{s}-1}\lambda_{2^{s+2}-1}]\in {\rm Ext}_{\mathcal A}^{4, 4+n_s}(\mathbb Z/2, \mathbb Z/2)$ and $\dim \mathbb Z/2 \otimes_{GL_4} {\rm Ann}_{\overline{\mathcal A}}[P_{n_s}^{\otimes 4}]^{*}\geq 1.$ It should be noted that $\lambda_0^{2}\lambda_{2^{s}-1}\lambda_{2^{s+2}-1}\in \Lambda^{4, n_s}$ is a cycle in $\Lambda.$ Thus, $ \mathbb Z/2 \otimes_{GL_4} {\rm Ann}_{\overline{\mathcal A}}[P_{n_3}^{\otimes 4}]^{*}$ is 1-dimensional. Furthermore, because $\langle [\sum_{1\leq j\leq 105}{\rm adm}_{s,\, j}], [\zeta_s] \rangle = 1,$ for any $s > 3,$ by Proposition \ref{mdc1}, we conclude that the set $\{[\zeta_s]\}$ is a basis of $\mathbb Z/2 \otimes_{GL_4} {\rm Ann}_{\overline{\mathcal A}}[P_{n_s}^{\otimes 4}]^{*}.$ This completes the proof of the theorem.
\end{proof}

\subsection{The degree \boldmath{$n'_s = 17.2^{s}-2$}}

In this section, we prove Theorems \ref{dlc5}, \ref{dlc6} and \ref{dlc7}. Firstly, we need the following results. Let us denote ${\rm Ker}_{n'_s}$ the kernel of Kameko's homomorphism $ (\widetilde {Sq^0_*})_{n'_s}: QP^{\otimes 4}_{n'_s}  \longrightarrow QP^{\otimes 4}_{\frac{n'_s-4}{2}}.$ According to Sum \cite[Propositions 7.2.1, 7.4.1, Theorems 7.2.2, 7.4.2]{Sum1}), ${\rm Ker}_{n'_s}$ has a basis consisting all the classes represented by the admissible monomials $b_{4,\, 1,\, j},$ $c_{4,\, 1,\, j}$ ($s = 1$), and $b_{4,\, s,\, j}$ ($s\geq 2).$ Based on this result, we obtain the following.

\begin{prop}\label{mdc2}
Let $s$ be a positive integer. Then,
$$ [{\rm Ker}_{n'_s}]^{GL_4}  = \left\{\begin{array}{ll}
\langle [\sum_{1\leq j\leq 105}b_{4,\, s,\, j}] \rangle&\mbox{if $s = 3$},\\
0&\mbox{if $s\neq 3$}.
\end{array}\right.$$
\end{prop}

\begin{proof}
The proof of the proposition for the cases $s \geq 2$ is similar to Proposition \ref{mdc1}. Now, we prove the case $s = 1$ in detail. As well known (see Sum \cite{Sum1}),  an admissible monomial basis for ${\rm Ker}_{n'_1}$ is the set $\{[{\rm adm}_j]:\, 1\leq j\leq 45\} = \{[b_{4,\, 1,\, j}]:\, 1\leq j\leq 38\}\bigcup \{[c_{4,\,1,\, j}]:\, 1\leq j\leq 7\},$ where

\begin{center}
\begin{tabular}{lrrr}
${\rm adm}_{1}= t_3t_4^{31}$, & \multicolumn{1}{l}{${\rm adm}_{2}= t_3^{31}t_4$,} & \multicolumn{1}{l}{${\rm adm}_{3}= t_2t_4^{31}$,} & \multicolumn{1}{l}{${\rm adm}_{4}= t_2t_3^{31}$,} \\
${\rm adm}_{5}= t_2^{31}t_4$, & \multicolumn{1}{l}{${\rm adm}_{6}= t_2^{31}t_3$,} & \multicolumn{1}{l}{${\rm adm}_{7}= t_1t_4^{31}$,} & \multicolumn{1}{l}{${\rm adm}_{8}= t_1t_3^{31}$,} \\
${\rm adm}_{9}= t_1t_2^{31}$, & \multicolumn{1}{l}{${\rm adm}_{10}= t_1^{31}t_4$,} & \multicolumn{1}{l}{${\rm adm}_{11}= t_1^{31}t_3$,} & \multicolumn{1}{l}{${\rm adm}_{12}= t_1^{31}t_2$,} \\
${\rm adm}_{13}= t_3^{3}t_4^{29}$, & \multicolumn{1}{l}{${\rm adm}_{14}= t_2^{3}t_4^{29}$,} & \multicolumn{1}{l}{${\rm adm}_{15}= t_2^{3}t_3^{29}$,} & \multicolumn{1}{l}{${\rm adm}_{16}= t_1^{3}t_4^{29}$,} \\
${\rm adm}_{17}= t_1^{3}t_3^{29}$, & \multicolumn{1}{l}{${\rm adm}_{18}= t_1^{3}t_2^{29}$,} & \multicolumn{1}{l}{${\rm adm}_{19}= t_2t_3t_4^{30}$,} & \multicolumn{1}{l}{${\rm adm}_{20}= t_2t_3^{30}t_4$,} \\
${\rm adm}_{21}= t_1t_3t_4^{30}$, & \multicolumn{1}{l}{${\rm adm}_{22}= t_1t_3^{30}t_4$,} & \multicolumn{1}{l}{${\rm adm}_{23}= t_1t_2t_4^{30}$,} & \multicolumn{1}{l}{${\rm adm}_{24}= t_1t_2t_3^{30}$,} \\
${\rm adm}_{25}= t_1t_2^{30}t_4$, & \multicolumn{1}{l}{${\rm adm}_{26}= t_1t_2^{30}t_3$,} & \multicolumn{1}{l}{${\rm adm}_{27}= t_2t_3^{2}t_4^{29}$,} & \multicolumn{1}{l}{${\rm adm}_{28}= t_1t_3^{2}t_4^{29}$,} \\
${\rm adm}_{29}= t_1t_2^{2}t_4^{29}$, & \multicolumn{1}{l}{${\rm adm}_{30}= t_1t_2^{2}t_3^{29}$,} & \multicolumn{1}{l}{${\rm adm}_{31}= t_2t_3^{3}t_4^{28}$,} & \multicolumn{1}{l}{${\rm adm}_{32}= t_2^{3}t_3t_4^{28}$,} \\
${\rm adm}_{33}= t_1t_3^{3}t_4^{28}$, & \multicolumn{1}{l}{${\rm adm}_{34}= t_1t_2^{3}t_4^{28}$,} & \multicolumn{1}{l}{${\rm adm}_{35}= t_1t_2^{3}t_3^{28}$,} & \multicolumn{1}{l}{${\rm adm}_{36}= t_1^{3}t_3t_4^{28}$,} \\
${\rm adm}_{37}= t_1^{3}t_2t_4^{28}$, & \multicolumn{1}{l}{${\rm adm}_{38}= t_1^{3}t_2t_3^{28}$,} & \multicolumn{1}{l}{${\rm adm}_{39}= t_1t_2t_3^{2}t_4^{28}$,} & \multicolumn{1}{l}{${\rm adm}_{40}= t_1t_2^{2}t_3t_4^{28}$,} \\
${\rm adm}_{41}= t_1t_2^{2}t_3^{28}t_4$, & \multicolumn{1}{l}{${\rm adm}_{42}= t_1t_2^{2}t_3^{4}t_4^{25}$,} & \multicolumn{1}{l}{${\rm adm}_{43}= t_1t_2^{2}t_3^{5}t_4^{24}$,} & \multicolumn{1}{l}{${\rm adm}_{44}= t_1t_2^{3}t_3^{4}t_4^{24}$,} \\
${\rm adm}_{45}= t_1^{3}t_2t_3^{4}t_4^{24}$. &       &       &  
\end{tabular}%
\end{center}

Then we have a direct summand decomposition of $S_4$-submodules: $ {\rm Ker}_{n'_1}  = \bigoplus_{1\leq i\leq 4}\mathbb M_i,$ 
where $$ \begin{array}{ll}
\medskip
&\mathbb M_1:=S_4({\rm adm}_{1}) = \langle \{[{\rm adm}_{j}]:\, 1\leq j\leq 12\} \rangle,\\
\medskip
&\mathbb M_2:=S_4({\rm adm}_{13}) = \langle \{[{\rm adm}_{j}]:\, 13\leq j\leq 18\} \rangle,\\
\medskip
&\mathbb M_3:=S_4({\rm adm}_{19}, {\rm adm}_{27}, {\rm adm}_{31}) =\langle \{[{\rm adm}_{j}]:\, 19\leq j\leq 38\} \rangle,\\
\medskip
&\mathbb M_4 := \langle \{[{\rm adm}_{j}]:\, 39\leq j\leq 45\} \rangle.
\end{array}$$ 
Denote by $\mathscr B_i$ the bases of the spaces $\mathbb M_i.$ Let $f_i\in P^{\otimes 4}_{n'_1}$ such that $[f_i]\in \mathbb M_i^{S_4},$ then $f_i\equiv \sum_{u\in \mathscr B_i}\gamma_uu$ with $\gamma_u\in \mathbb Z/2.$ For each $1\leq k\leq 3,$ we compute explicitly the equalities $\theta_k(f_i)$ in the terms $u\in \mathscr B_i$ modulo ($\overline{\mathcal A}P^{\otimes 4}_{n'_1}$). Then, from the relations $\theta_k(f_i) + f_i\equiv 0$, we find that
$$ \begin{array}{ll}
\medskip
&\mathbb M_1^{S_4}= \langle [q_1] \rangle \ \mbox{with\ $q_1 = \sum_{1\leq j\leq 12}{\rm adm}_{j}$},\ \ \mathbb M_2^{S_4}= \langle [q_2] \rangle \ \mbox{with\ $q_2 = \sum_{13\leq j\leq 18}{\rm adm}_{j}$},\\
\medskip
&\mathbb M_3^{S_4}= \langle [q_3] \rangle \ \mbox{with\ $q_3 = {\rm adm}_{20} + {\rm adm}_{22} +   \sum_{25\leq j\leq 38}{\rm adm}_{j}$},\\
\medskip
&\mathbb M_4^{S_4}= \langle [q_4 + q_5], [q_5 + q_6] \rangle \ \mbox{with\ $q_4 = \sum_{42\leq j\leq 45}{\rm adm}_{j}$, $q_5 = {\rm adm}_{39}$, $q_6 = {\rm adm}_{40} + {\rm adm}_{41}$}.
\end{array}$$
Now, suppose that $[h]\in  [{\rm Ker}_{n'_s}]^{GL_4},$ then because $S_4\subset GL_4,$ we have $$ h\equiv \xi_1q_1 + \xi_2q_2 + \xi_3q_3 + \xi_4(q_4 + q_5) + \xi_5(q_5 + q_6),$$
in which $\xi_i\in \mathbb Z/2.$ A direct computation shows that
$$ \begin{array}{ll}
\medskip
 \theta_4(h) + h &\equiv (\xi_1{\rm adm}_{3} + (\xi_1+\xi_3){\rm adm}_{5} + (\xi_1+\xi_2){\rm adm}_{9} + (\xi_3+\xi_4){\rm adm}_{27} \\
&\qquad + (\xi_3+\xi_5){\rm adm}_{20}+ \ \mbox{other terms})\ \equiv 0.
\end{array}$$
This equality implies that $\xi_i = 0$ for $1\leq i\leq 5.$ The proposition follows.
\end{proof}

\begin{rema}
From a result of Lin \cite{Lin}, we claim that
\begin{equation}\label{kqL2}
{\rm Ext}_{\mathcal A}^{4, 4+n'_s}(\mathbb Z/2, \mathbb Z/2) = \left\{\begin{array}{ll}
\langle d_1 \rangle  &\mbox{if $s = 1$},\\[1mm]
\langle h_1^{3}h_6 \rangle = \langle h_0^{2}h_2h_6 \rangle &\mbox{if $s = 2$},\\[1mm]
\langle h_0^{2}h_3h_7, h_1h_3h_6^{2} \rangle  &\mbox{if $s = 3$},\\[1mm]
\langle h_1h^{2}_{s-1}h_{s+4}, h_1h_sh_{s+3}^{2} \rangle  &\mbox{if $s\geq 4$},
\end{array}\right.
\end{equation}
where $h_1h^{2}_{s-1}h_{s+4} = 0$ for $s = 4$ and $h_1h^{2}_{s-1}h_{s+4}\neq h_1h_sh_{s+3}^{2}$ for $s\geq 5.$
\end{rema}

\begin{proof}[{\bf Proof of Theorem \ref{dlc5}}]
By Sum \cite{Sum3}, the coinvariant space $\mathbb Z/2 \otimes_{GL_4} {\rm Ann}_{\overline{\mathcal A}}[P_{\frac{n'_1-4}{2}}^{\otimes 4}]^{*}$ has dimension 1. Combining this with the inequality \eqref{bdt2} and Proposition \ref{mdc2}, we deduce an estimate that $\dim \mathbb Z/2 \otimes_{GL_4} {\rm Ann}_{\overline{\mathcal A}}[P_{n'_1}^{\otimes 4}]^{*} \leq 1.$ On the other hand, it is easy to verify that the element 
$$ \begin{array}{ll}
\overline{\zeta}&=  x_1^{(3)}x_2^{(13)}x_3^{(7)}x_4^{(9)}+
 x_1^{(3)}x_2^{(13)}x_3^{(11)}x_4^{(5)}+
 x_1^{(3)}x_2^{(13)}x_3^{(13)}x_4^{(3)}+
 x_1^{(5)}x_2^{(11)}x_3^{(7)}x_4^{(9)}+
\medskip
 x_1^{(5)}x_2^{(11)}x_3^{(11)}x_4^{(5)}\\
&+
 x_1^{(5)}x_2^{(11)}x_3^{(13)}x_4^{(3)}+
 x_1^{(5)}x_2^{(13)}x_3^{(7)}x_4^{(7)}+
 x_1^{(7)}x_2^{(3)}x_3^{(11)}x_4^{(11)}+
 x_1^{(7)}x_2^{(3)}x_3^{(13)}x_4^{(9)}+
\medskip
 x_1^{(7)}x_2^{(5)}x_3^{(11)}x_4^{(9)}\\
&+
 x_1^{(7)}x_2^{(5)}x_3^{(13)}x_4^{(7)}+
 x_1^{(7)}x_2^{(7)}x_3^{(7)}x_4^{(11)}+
 x_1^{(7)}x_2^{(7)}x_3^{(9)}x_4^{(9)}+
 x_1^{(7)}x_2^{(7)}x_3^{(13)}x_4^{(5)}+
\medskip
 x_1^{(7)}x_2^{(9)}x_3^{(7)}x_4^{(9)}\\
&+
 x_1^{(7)}x_2^{(11)}x_3^{(5)}x_4^{(9)}+
 x_1^{(7)}x_2^{(13)}x_3^{(3)}x_4^{(9)}+
 x_1^{(7)}x_2^{(13)}x_3^{(5)}x_4^{(7)}+
 x_1^{(9)}x_2^{(7)}x_3^{(7)}x_4^{(9)}+
\medskip
 x_1^{(9)}x_2^{(7)}x_3^{(11)}x_4^{(5)}\\
&+
 x_1^{(9)}x_2^{(7)}x_3^{(13)}x_4^{(3)}+
 x_1^{(11)}x_2^{(3)}x_3^{(7)}x_4^{(11)}+
 x_1^{(11)}x_2^{(3)}x_3^{(13)}x_4^{(5)}+
 x_1^{(11)}x_2^{(5)}x_3^{(11)}x_4^{(5)}+
\medskip
 x_1^{(11)}x_2^{(7)}x_3^{(3)}x_4^{(11)}\\
&+
 x_1^{(11)}x_2^{(7)}x_3^{(9)}x_4^{(5)}+
 x_1^{(11)}x_2^{(9)}x_3^{(7)}x_4^{(5)}+
 x_1^{(11)}x_2^{(11)}x_3^{(3)}x_4^{(7)}+
 x_1^{(11)}x_2^{(11)}x_3^{(5)}x_4^{(5)}+
\medskip
 x_1^{(11)}x_2^{(13)}x_3^{(3)}x_4^{(5)}\\
&+
 x_1^{(13)}x_2^{(3)}x_3^{(13)}x_4^{(3)}+
 x_1^{(13)}x_2^{(5)}x_3^{(11)}x_4^{(3)}+
 x_1^{(13)}x_2^{(7)}x_3^{(7)}x_4^{(5)}+
 x_1^{(13)}x_2^{(7)}x_3^{(7)}x_4^{(5)}+
\medskip
 x_1^{(13)}x_2^{(7)}x_3^{(9)}x_4^{(3)}\\
&+
 x_1^{(13)}x_2^{(9)}x_3^{(7)}x_4^{(3)}+
 x_1^{(13)}x_2^{(11)}x_3^{(5)}x_4^{(3)}+
 x_1^{(13)}x_2^{(13)}x_3^{(3)}x_4^{(3)}
\end{array}$$
belongs to ${\rm Ext}_{\mathcal A}^{0, n'_1}(\mathbb Z/2, P^{\otimes 4}).$ It should be note that we need only to consider effects of the Steenrod squares $Sq^{i}$ for $i = 1,2, 4, 8$ because of the unstable condition. So, using the representation of $Tr_4^{\mathcal A}$ over $\Lambda$, we find that $\psi_4(\overline{\zeta}) =  \lambda_7^2\lambda_5\lambda_{13} + \lambda_7^2\lambda_9^2 + \lambda_7\lambda_{11}\lambda_9\lambda_5 + \lambda_{15}\lambda_3\lambda_{11}\lambda_3 + \delta(\lambda_7^2\lambda_{19} + \lambda_7\lambda_{19}\lambda_7)$ is a cycle in $\Lambda$ and is a representative of the non-zero element $d_1\in {\rm Ext}_{\mathcal A}^{4, 4+n'_1}(\mathbb Z/2, \mathbb Z/2).$ Combining this with the equality \eqref{kqL2} and the fact that $\dim \mathbb Z/2 \otimes_{GL_4} {\rm Ann}_{\overline{\mathcal A}}[P_{n'_1}^{\otimes 4}]^{*} \leq 1$, we claim that the coinvariant space $ \dim \mathbb Z/2 \otimes_{GL_4} {\rm Ann}_{\overline{\mathcal A}}[P_{n'_1}^{\otimes 4}]^{*}$ is 1-dimensional.  Now, let $\sigma\in P^{\otimes 4}_{n'_1}$ such that $[\sigma]\in [QP^{\otimes 4}_{n'_1}]^{GL_4}.$ Consider the linear transformation
$$ \begin{array}{ll}
\widetilde{\varphi}: P^{\otimes 4}_{\frac{n'_1-4}{2}}&\longrightarrow P^{\otimes 4}_{n'_1},\\
\prod_{1\leq i\leq 4}t_i^{a_i}&\longmapsto \prod_{1\leq i\leq 4}t_i^{2a_i+1}.
\end{array}$$
 According to Sum \cite[Proposition 4.2.2]{Sum3}, the invariant space $[QP^{\otimes 4}_{\frac{n'_1-4}{2}}]^{GL_4}$ is generated by the class $[t_1t_2t_3^{6}t_4^{6} + t_1^{3}t_2^{3}t_3^{4}t_4^{4}].$ Then, since the Kameko squaring operation $(\widetilde {Sq^0_*})_{n'_1}$ is an epimorphism of $\mathbb Z/2GL_4$-modules, $(\widetilde {Sq^0_*})_{n'_1}([\sigma]) = \gamma[\widetilde{\varphi}(t_1t_2t_3^{6}t_4^{6} + t_1^{3}t_2^{3}t_3^{4}t_4^{4})]$ with $\gamma\in \mathbb Z/2.$ We therefore deduce that $\sigma \equiv \gamma\widetilde{\varphi}(t_1t_2t_3^{6}t_4^{6} + t_1^{3}t_2^{3}t_3^{4}t_4^{4}) + \sigma^{*},$ where $\sigma^{*}\in P^{\otimes 4}_{n'_1}$ such that $[\sigma^{*}]\in {\rm Ker}_{n'_1}.$ By a direct calculation using Proposition \ref{mdc2} and the relations $\theta_i(\sigma)  + \sigma \equiv 0$ for $1\leq i\leq 4,$ we get $$[QP^{\otimes 4}_{n'_1}]^{GL_4} = \langle [\widetilde{\varphi}(t_1t_2t_3^{6}t_4^{6} + t_1^{3}t_2^{3}t_3^{4}t_4^{4})]\rangle.$$ Now the theorem follows from the fact that $\langle [\widetilde{\varphi}(t_1t_2t_3^{6}t_4^{6} + t_1^{3}t_2^{3}t_3^{4}t_4^{4})], [\overline{\zeta}] \rangle = 1.$  
\end{proof}

\begin{proof}[{\bf Proof of Theorem \ref{dlc6}}]
The proof of this theorem is similar to the one in Theorem \ref{dlc2}. Note that the invariant space $[QP^{\otimes 4}_{n'_2}]^{GL_4}$ generated over $\mathbb Z/2$ by the class $$ \big[\varphi\big(\sum_{1\leq \ell\leq 3}\sum_{1\leq i_1\leq \ldots\leq i_{\ell}\leq 4}t_{i_1}t_{i_2}^{2}\ldots t^{2^{\ell-2}}_{i_{\ell-1}}t^{32-2^{\ell-1}}_{i_{\ell}} + t_1t_2^{2}t_3^{4}t_4^{24}\big)\big].$$
 \end{proof}

\begin{proof}[{\bf Proof of Theorem \ref{dlc7}}]
Firstly, we recall a result in our work \cite{Phuc9} which is used in the proof of the theorem.

\begin{thm}\label{dlct}
Let $s$ and $t$ be positive integers such that $t\geq 4.$ Then, 
$$ 
\dim \mathbb Z/2 \otimes_{GL_4} {\rm Ann}_{\overline{\mathcal A}}[P_{2^{s+t+1} + 2^{s+1}-3}^{\otimes 4}]^{*} = \left\{\begin{array}{ll}
1&\mbox{if $s = 1, 2$},\\
2&\mbox{if $s \geq 3$}.
\end{array}\right.$$
Moreover, $ \begin{array}{ll}
\medskip
&\mathbb Z/2 \otimes_{GL_4} {\rm Ann}_{\overline{\mathcal A}}[P_{2^{s+t+1} + 2^{s+1}-3}^{\otimes 4}]^{*}\\
&  = \left\{\begin{array}{ll}
\langle [x_1^{(0)}x_2^{(2^{s+1}-1)}x_3^{(2^{s+t}-1)}x_4^{(2^{s+t}-1)}] \rangle&\mbox{if $s = 1, 2$},\\[1mm]
\langle [x_1^{(0)}x_2^{(2^{s+1}-1)}x_3^{(2^{s+t}-1)}x_4^{(2^{s+t}-1)}] , [x_1^{(0)}x_2^{(2^{s}-1)}x_3^{(2^{s}-1)}x_4^{(2^{s+t+1}-1)}] \rangle&\mbox{if $s \geq 3$}.
\end{array}\right.
\end{array}$
\end{thm}

Theorem \ref{dlct} implies that
$$ \dim \mathbb Z/2 \otimes_{GL_4} {\rm Ann}_{\overline{\mathcal A}}[P_{\frac{n'_s-4}{2}}^{\otimes 4}]^{*} =  \left\{\begin{array}{ll}
1&\mbox{if $s = 3, 4$}\\
2&\mbox{if $s\geq 5$}.
\end{array}\right.$$
Then, by the inequality \eqref{bdt2} and Proposition \ref{mdc2}, we get
\begin{equation}\label{bdt3}
 \dim \mathbb Z/2 \otimes_{GL_4} {\rm Ann}_{\overline{\mathcal A}}[P_{n'_s}^{\otimes 4}]^{*} \leq 
 \left\{\begin{array}{ll}
1&\mbox{if $s = 4$}\\
2&\mbox{if $s\neq 4$}.
\end{array}\right.
\end{equation}
On the other hand, it is not too difficult to check that the elements
$$ \begin{array}{ll}
\medskip
 \zeta_1 &= x_1^{(1)}x_2^{(7)}x_3^{(63)}x_4^{(63)}\in [P_{n'_3}^{\otimes 4}]^{*}, \ \ \zeta_2 =  x_1^{(0)}x_2^{(0)}x_3^{(7)}x_4^{(127)}\in [P_{n'_3}^{\otimes 4}]^{*},\\
\medskip
 \zeta_3 &= x_1^{(1)}x_2^{(15)}x_3^{(127)}x_4^{(127)}\in [P_{n'_4}^{\otimes 4}]^{*}\\
\medskip
\zeta_s &= x^{(1)}x_2^{(2^{s-1}-1)}x_3^{(2^{s-1}-1)}x_4^{(2^{s+4}-1)}\in [P_{n'_s}^{\otimes 4}]^{*}, \ \mbox{for $s\geq 5$}. \\
\medskip
 \zeta'_s &= x_1^{(1)}x_2^{(2^{s}-1)}x_3^{(2^{s+3}-1)}x_4^{(2^{s+3}-1}\in [P_{n'_s}^{\otimes 4}]^{*}\ \mbox{for $s\geq 5$}.
\end{array}$$
are $\overline{\mathcal A}$-annahilated. Then, using the representation in $\Lambda$ of the rank 4 algebraic transfer, we get
$$ \begin{array}{ll} 
\medskip
&[\psi_4(\zeta_1)] = [\lambda_1\lambda_7\lambda_{63}^{2}]= Tr_4^{\mathcal A}([\zeta_1])  = h_1h_3h_6^{2}\in {\rm Ext}_{\mathcal A}^{4, 4+n'_3}(\mathbb Z/2, \mathbb Z/2),\\
\medskip
&[\psi_4(\zeta_2)] = [\lambda_0\lambda_7\lambda_{127}^{2}] = Tr_4^{\mathcal A}([\zeta_2]) = h_0^{2}h_3h_7\in {\rm Ext}_{\mathcal A}^{4, 4+n'_3}(\mathbb Z/2, \mathbb Z/2),\\
\medskip
 &[\psi_4(\zeta_3)] = [\lambda_1\lambda_{15}\lambda_{127}^{2}] = Tr_4^{\mathcal A}([\zeta_3]) = h_1h_4h_7^{3}\in {\rm Ext}_{\mathcal A}^{4, 4+n'_4}(\mathbb Z/2, \mathbb Z/2),\\
\medskip
&[\psi_4(\zeta_s)] = [\lambda_1\lambda^{2}_{2^{s-1}-1}\lambda_{2^{s+4}-1}] = Tr_4^{\mathcal A}([\zeta_s]) = h_1h_{s-1}^{2}h_{s+4}\in {\rm Ext}_{\mathcal A}^{4, 4+n'_s}(\mathbb Z/2, \mathbb Z/2)\ \mbox{for $s\geq 5$}\\
\medskip
&[\psi_4(\zeta'_s)] = [\lambda_1\lambda_{2^{s}-1}\lambda_{2^{s+3}-1}^{2}]=  Tr_4^{\mathcal A}([\zeta'_s]) = h_1h_{s}h_{s+3}^{2}\in {\rm Ext}_{\mathcal A}^{4, 4+n'_s}(\mathbb Z/2, \mathbb Z/2)\ \mbox{for $s\geq 5$}.
\end{array}$$ 
Combining these computations with the equality \eqref{kqL2} and the inequality \eqref{bdt3} gives 
$$  \dim \mathbb Z/2 \otimes_{GL_4} {\rm Ann}_{\overline{\mathcal A}}[P_{n'_s}^{\otimes 4}]^{*} =
 \left\{\begin{array}{ll}
1&\mbox{if $s = 4$}\\
2&\mbox{if $s\neq 4$}.
\end{array}\right.$$
Furthermore, by using Proposition \ref{mdc2}, Theorem \ref{dlct} and the similar arguments as in the proof of Theorem \ref{dlc5}, we obtain
$$ \begin{array}{ll}
\medskip
& \mathbb Z/2 \otimes_{GL_4} {\rm Ann}_{\overline{\mathcal A}}[P_{n'_s}^{\otimes 4}]^{*} \\
&=\left\{\begin{array}{ll}
\langle [x_1^{(1)}x_2^{(7)}x_3^{(63)}x_4^{(63)}], [x_1^{(0)}x_2^{(0)}x_3^{(7)}x_4^{(127)}] \rangle &\mbox{if $s = 3$},\\[1mm]
\langle [x_1^{(1)}x_2^{(15)}x_3^{(127)}x_4^{(127)}] \rangle &\mbox{if $s = 4$},\\[1mm]
\langle [x_1^{(1)}x_2^{(2^{s-1}-1)}x_3^{(2^{s-1}-1)}x_4^{(2^{s+4}-1)}] , [x_1^{(1)}x_2^{(2^{s}-1)}x_3^{(2^{s+3}-1)}x_4^{(2^{s+3}-1)}]  \rangle &\mbox{if $s \geq 5$}.
\end{array}\right.
\end{array}$$
The theorem is proved.
\end{proof}

\section{Appendix}

Let us recall that the structure of the cohomology $ {\rm Ext}_{\mathcal A}^{*, *}(\mathbb Z/2, \mathbb Z/2)$ of the Steenrod ring is still a great mystery for all homological degrees $\geq 6.$ So far, Bruner's computer calculations \cite{Bruner} tell us only a  little about the groups ${\rm Ext}_{\mathcal A}^{h, *}(\mathbb Z/2, \mathbb Z/2)$ for $h\geq 6.$ The Singer algebraic transfer $Tr_h^{\mathcal A}$ serves as a reliable tool for depicting those Ext groups. This is concretized by the works of Singer \cite{Singer}, and Boardman \cite{Boardman} that $Tr_h^{\mathcal A}$ is an isomorphism for all homological degrees $h\leq 3$ and any internal degree $>0$. Singer \cite{Singer} predicted that the algebraic transfer is a monomorphism in every $h > 0,$ but it remains open for all $h \geq 5.$ Very little information is known this conjecture for the rank 5 case (see also the present author \cite{Phuc3, Phuc4, Phuc6, Phuc10}, Sum \cite{Sum4}). In particular, it is no information yet on the cases $h = 6, 7, 8.$ So, in this appendix, we would like to investigate the behavior of the Singer transfer of ranks $\geq 5$ in some internal degrees. More precisely, we first study $Tr_5^{\mathcal A}$ in degrees of the form $n_s:=k(2^{s} - 1) + r.2^{s}$ with $k = 4,\, 5.$ Next, based on the known results by Mothebe, Kaelo, and Ramatebele \cite{MKR} and direct calculations, we investigate the behavior of the sixth transfer, $Tr_6^{\mathcal A}$ in internal degrees $\leq 25.$ At the same time, using the results by Bruner \cite{Bruner}, we will show that $Tr_6^{\mathcal A}$ is not an epimorphism in degrees $n_s$ with $k = 6,\, r = 10,$ and $s\leq 2.$ More explicitly, $Tr_6^{\mathcal A}$ does not detect the non-zero elements $h_2^{2}g_1 = h_4Ph_2\in {\rm Ext}_{\mathcal A}^{6, 6+n_1}(\mathbb Z/2, \mathbb Z/2)$, and $D_2\in {\rm Ext}_{\mathcal A}^{6, 6+n_2}(\mathbb Z/2, \mathbb Z/2).$ Finally, we probe the behavior of $Tr_7^{\mathcal A},$ and $Tr_8^{\mathcal A}$ in internal degrees $\leq 15.$ As a consequence, Singer's conjecture for the ranks $h = 6, 7, 8$ is true in respective degrees.  Remarkably, based on the techniques in this paper, we showed that the non-zero elements $h_1Ph_1\in {\rm Ext}_{\mathcal A}^{6, 6+10}(\mathbb Z/2, \mathbb Z/2)$, $h_0Ph_2\in {\rm Ext}_{\mathcal A}^{6, 6+11}(\mathbb Z/2, \mathbb Z/2)$ are not in the image of $Tr_6^{\mathcal A},$ and that $h_0^{2}Ph_2\in {\rm Ext}_{\mathcal A}^{6, 6+11}(\mathbb Z/2, \mathbb Z/2)$ is not in the image of $Tr_7^{\mathcal A}.$ These events have also been proved by Ch\ohorn n, and H\`a \cite{C.H1, C.H2} using another method, but their approaches were not directly applicable in confirming or refuting the Singer conjecture in those cases. Therefore, our calculations are important and of non-trivial value. Additionally, in ranks $h = 6,\, 7,\, 8,$ we prove that the results of Moetele, and Mothebe \cite{MM} for the dimensions of the cohit $\mathbb Z/2$-module $QP_{13}^{\otimes h}$ are not true. Therefrom, as described above, the behavior of the Singer transfer in bidegree $(h, h+13)$ is obtained. Our results will play a crucial role in the study of the Singer  algebraic transfer in the certain generic internal degrees. The calculation techniques are similar to the proofs of the main results of this article and to our previous works \cite{Phuc4, Phuc6, Phuc7, Phuc10}. We emphasize that along this section, because the dimensions of the cohit spaces $QP^{\otimes h}$ in degrees $n$  given are too large, we do not list the bases of these spaces in detail; more precisely, they are the equivalence classes of admissible monomials of degrees $n$ in $\mathcal A$-modules $P^{\otimes h}$ for $5\leq h\leq 8.$  Moreover, the MAGMA computer algebra system \cite{Magma} has also been used for verifying the results. 

\medskip

For the reader's convenience, some related knowledge are mentioned here, since they will be needed later. We also refer the readers to the literatures \cite{Kameko}, \cite{Phuc4} for more details. First of all, let us recall that the set of all the admissible monomials of degree $n$ in $P^{\otimes h}$ is \textit{a minimal set of $\mathcal {A}$-generators for $P^{\otimes h}$ in degree $n$.} Therefore, the space of indecomposables $QP^{\otimes h}_{n}$ has a basis consisting of all the classes represent by the admissible monomials of degree $n$ in $P^{\otimes h}.$ We are now going to review the following well-known. For a natural number $k,$ writing $\alpha_j(k)$ and $\alpha(k)$ for the $j$-th coefficients and the number of $1$'s in dyadic expansion of $k$, respectively. Thence, $\alpha(k) = \sum_{j\geq 0}\alpha_j(k),$ and $k$ can be written in the form $\sum_{j\geq 0}\alpha_j(k)2^j,$ where $\alpha_j(k)$ belongs to $\{0, 1\},$ for all $j\geq 0.$ For a monomial $t = t_1^{a_1}t_2^{a_2}\ldots t_h^{a_h}$ in $P_{n}^{\otimes h},$ we consider a sequence associated with $t$ by $\omega(t) :=(\omega_1(t), \omega_2(t), \ldots, \omega_i(t), \ldots)$ where $\omega_i(t)=\sum_{1\leq j\leq h}\alpha_{i-1}(a_j)\leq h,$ for all $i\geq 1.$ This sequence is called the {\it weight vector} of $t.$ One defines $\deg(\omega(t)) = \sum_{j\geq 1}2^{j-1}\omega_j(t).$ We say that $a(t) = (a_1, \ldots, a_h)$ is the exponent vector of $t.$ The sets of all weight vectors and exponent vectors are given the left lexicographical order. Assume that $t = t_1^{a_1}t_2^{a_2}\ldots t_h^{a_h}$ and $t' = t_1^{a'_1}t_2^{a'_2}\ldots t_h^{a'_h}$ are the monomials in $P_{n}^{\otimes h}.$ We say that $t  < t'$ if and only if either $\omega(t) < \omega(t')$ or $\omega(t) = \omega(t'),$ $a(t) < a'(t').$ For a weight vector $\omega$ of degree $n,$ we denote two subspaces associated with $\omega$ by
$$ \begin{array}{ll}
\medskip
 P_n^{\otimes h}(\omega) &= \langle\{ t\in P^{\otimes h}_n|\, \deg(t) = \deg(\omega) = n,\  \omega(t)\leq \omega\}\rangle,\\
P_n^{\otimes h}(< \omega) &= \langle \{t\in P^{\otimes h}_n|\, \deg(t) = \deg(\omega) = n,\  \omega(t) < \omega\}\rangle.
\end{array}$$ 
Let $u$ and $v$ be two homogeneous polynomials in $P_n^{\otimes h}(\omega).$ We define the equivalence relation "$\equiv_{\omega}$" on $P_n^{\otimes h}$ as $u \equiv_{\omega} v$ if and only if $(u +v)$ belongs to $((\overline{\mathcal {A}}P_n^{\otimes h} \cap P_n^{\otimes h}(\omega) + P_n^{\otimes h}(< \omega)).$ By this, one has a quotient space $QP_n^{\otimes h}(\omega) := P_n^{\otimes h}(\omega)/ ((\overline{\mathcal {A}}P_n^{\otimes h} \cap P_n^{\otimes h}(\omega))+ P_n^{\otimes h}(< \omega)),$ and an isomorphism $QP_n^{\otimes h}\cong \bigoplus_{\deg(\omega) = n}QP_n^{\otimes h}(\omega).$ Due to Definition \ref{dnmd} (see section two), we say that a monomial $t\in P_{n}^{\otimes h}$ is {\it inadmissible} if there exist monomials $z_1, z_2,\ldots, z_k\in P_{n}^{\otimes h}$ such that $z_j < t$ for $1\leq j\leq k$ and $t \equiv  \sum_{1\leq j\leq k}z_j,$ from which $t$ is said to be {\it admissible} if it is not inadmissible. Our motivating tool is the following.\\[1mm]
$\bullet$ {\it Kameko's criterion {\rm \cite{Kameko}} on inadmissible monomials}: Let $t$ and $z$ be monomials in $P_{n}^{\otimes h}.$ For an integer $r >0,$ assume that there exists an index $i>r$ such that $\omega_i(t) = 0$, from which if $z$ is inadmissible, then $tz^{2^r}$ is, too. 

A monomial $z = \prod_{1\leq j\leq h}t_j^{\alpha_j}$ in $P^{\otimes h}_{n}$ is called a {\it spike} if every exponent $\alpha_j$ is of the form $2^{\beta_j} - 1.$ In particular, if the exponents $\beta_j$ can be arranged to satisfy $\beta_1 > \beta_2 > \ldots > \beta_{r-1}\geq \beta_r \geq 1,$ where only the last two smallest exponents can be equal, and $\beta_j = 0$ for $ r < j  \leq t,$ then $z$ is called a {\it minimal spike}. Another important tool to recall here is that of relations on hit polynomials.\\[1mm]
$\bullet$ {\it Singer's criterion {\rm \cite{Singer2}} on hit monomials}: Suppose that $t\in P^{\otimes h}_{n}$ and $\mu(n)\leq h.$ Then,  if $z$ is a minimal spike in $P^{\otimes h}_{n}$ such that $\omega(t) < \omega(z),$ then $t\equiv 0$ (or in other words, $t$ is hit).

\medskip

It will be useful to adopt the following notations:
$$ \begin{array}{ll}
\medskip
(QP_{n}^{\otimes h})^{0}:= \langle \{[\prod_{1\leq j\leq h}t_j^{\alpha_j}] \in QP_{n}^{\otimes h}|\, \deg(\prod_{1\leq j\leq h}t_j^{\alpha_j}) = n,\, \prod_{1\leq j\leq h}\alpha_j = 0\}\rangle,\\
\medskip
(QP_{n}^{\otimes h})^{>0}:= \langle \{[\prod_{1\leq j\leq h}t_j^{\alpha_j}]\in QP_{n}^{\otimes h}|\, \deg(\prod_{1\leq j\leq h}t_j^{\alpha_j}) = n,\, \prod_{1\leq j\leq h}\alpha_j \neq 0\}\rangle,\\
\medskip
(QP_{n}^{\otimes h})^{0}(\omega):= \langle \{[\prod_{1\leq j\leq h}t_j^{\alpha_j}]\in QP_{n}^{\otimes h}(\omega)|\, \deg(\prod_{1\leq j\leq h}t_j^{\alpha_j}) = n,\, \prod_{1\leq j\leq h}\alpha_j = 0,\, \omega(\prod_{1\leq j\leq h}t_j^{\alpha_j}) = \omega\}\rangle,\\
(QP_{n}^{\otimes h})^{>0}(\omega):= \langle \{[\prod_{1\leq j\leq h}t_j^{\alpha_j}]\in QP_{n}^{\otimes h}(\omega)|\, \deg(\prod_{1\leq j\leq h}t_j^{\alpha_j}) = n,\, \prod_{1\leq j\leq h}\alpha_j \neq 0,\, \omega(\prod_{1\leq j\leq h}t_j^{\alpha_j})=\omega\}\rangle,
\end{array}$$ 
where $\omega$ is a weight vector of degree $n.$ Then, we have the isomorphisms:
$$QP_n^{\otimes h} \cong (QP_n^{\otimes h})^0\,\bigoplus\, (QP_n^{\otimes h})^{>0},\ \mbox{ and }\ QP_n^{\otimes h}(\omega) \cong (QP_n^{\otimes h})^0(\omega)\,\bigoplus\, (QP_n^{\otimes h})^{>0}(\omega).$$

\medskip

Denote the set of all the admissible monomials of degree $n$ in $\mathcal A$-module $P^{\otimes h}$ by $\mathscr C_{n}^{\otimes h}.$ For a monomial $t\in P^{\otimes h}_n,$ we denote by $[t]$ the equivalence class of $t$ in $QP_n^{\otimes h}.$ If $\omega$ is a weight vector of degree $n$ and $\omega(t) = \omega,$ then denote by $[t]_\omega$ the equivalence class of $t$ in $QP_n^{\otimes h}(\omega).$ As usual, we write $|\mathscr C_{n}^{\otimes h}|$ for the cardinality of $\mathscr C_{n}^{\otimes h},$ and put $[\mathscr C_{n}^{\otimes h}] = \{[t]\, :\, t\in \mathscr C_{n}^{\otimes h}\}.$ For any admissible monomials ${\rm adm}_1, \ldots, {\rm adm}_k$ in $P^{\otimes h}_{n}(\omega)$ and for a subgroup $\mathscr G$ of $GL_h.$ Let us denote by $\mathscr G({\rm adm}_1, \ldots, {\rm adm}_k)$ the $\mathbb Z/2\mathscr G$-submodule of $QP_n^{\otimes h}(\omega)$ generated by $\{[{\rm adm}_j]_{\omega}:\, 1\leq j\leq k\}.$ It should be noted that if $\omega$ is a weight vector of a minimal spike, then $[{\rm adm}_j]_{\omega} = [{\rm adm}_j]$ for all $j.$

\subsection{The behavior of $Tr_5^{\mathcal A}$ in some internal degrees}

The goal of this section is to study the fifth Singer transfer in some internal degrees of the form $m_{r, s}:=5(2^{s} - 1) + r.2^{s}$ with $s = 0,$ and $r\in \{30,\, 33,\, 34,\, 36,\, 38,\, 39,\, 40,\, 42,\, 45, 46,\, 47,\, 48\}.$ A direct computation using previous results by the present author \cite{Phuc4, Phuc5, Phuc6}, Mothebe, and Uys \cite{Mothebe}, Sum \cite{Sum1, Sum2, Sum, Sum4}, Ly, and Tin \cite{LT}, Tin \cite{Tin, Tin3}, we find that
\begin{thm}\label{dlsc}
The dimension of $QP_{m_{r, 0}}^{\otimes 5}$ is determined by the following table:

\centerline{\begin{tabular}{c|ccccccccccccccccc}
$r$ & $30$ &$33$ & $34$ & $36$ & $38$ & $39$ & $40$ & $42$ & $45$ & $46$ & $47$ & $48$  \cr
\hline
\ $\dim QP_{m_{r, 0}}^{\otimes 5}$ & $840$ & $1322$ & $1554$ & $1189$ &$2015$ &$2130$ & $2047$ & $2520$ & $1731$ & $2349$ & $1894$ & $2374$  \cr
\end{tabular}}
\end{thm}
The theorem has also been proved by us \cite{Phuc5, Phuc6} for the cases $r = 39,$ and $47.$ Note that since the Kameko map $(\widetilde {Sq^0_*})_{m_{r, 0}}: QP^{\otimes 5}_{m_{r, 0}}  \longrightarrow QP^{\otimes 5}_{\frac{m_{r, 0}-5}{2}}$ is an epimorphism of $\mathbb Z/2[GL_5]$-modules, for $r = 33,\, 45,$ we have an isomorphism $QP^{\otimes 5}_{m_{r, 0}}\cong {\rm Ker}((\widetilde{Sq^0_*})_{m_{r, 0}})\bigoplus QP^{\otimes 5}_{\frac{m_{r, 0}-5}{2}}.$ On the other hand, since $\dim QP^{\otimes 5}_{\frac{m_{33, 0}-5}{2}} = 320,$ (see Ly, and Tin \cite{LT}), and $\dim QP^{\otimes 5}_{\frac{m_{45, 0}-5}{2}} = 641$ (see Tin \cite{Tin3}), we only need to determine ${\rm Ker}((\widetilde{Sq^0_*})_{m_{r, 0}}).$ Because the calculation methods are similar to our previous works \cite{Phuc4, Phuc5, Phuc6}, we will outline the proof of the theorem for the cases $r = 30$ and $33.$ Noting that the case $r = 30$ has also been informed by Tin \cite{Tin4} using a computer program of V.H. Viet.

We first consider the case $r = 30.$ We remark that if $t$ is an admissible monomial of degree $m_{30, 0}$ in $P^{\otimes 5}$, then $\omega(t)$ is one of the following sequences:
$$ \widetilde{\omega}_{(1)}:= (2,2,2,2),\ \widetilde{\omega}_{(2)}:=(2,2,4,1),\ \widetilde{\omega}_{(3)}:=(2,4,3,1),\ \widetilde{\omega}_{(4)}:=(4,3,3,1).$$
Indeed, observe that $t_1^{15}t_2^{15}\in P_{30}^{\otimes 5}$ is the minimal spike, and $\omega(t_1^{15}t_2^{15}) = (2,2,2,2).$ Since $[t]\neq 0,$ and  $\deg(t)$ is even, by Singer's criterion on hit monomials, either $\omega_1(t) = 2$ or $\omega_1(t) = 4.$ If $\omega_1(t) = 2,$ then $t = t_it_j\underline{t}^2$ with $\underline{t}$ a monomial of degree $14$ in $P^{\otimes 5},$ and $1\leq i < j \leq 5.$ Since $t$ is admissible, according to Kameko's criterion on inadmissible monomials, $\underline{t}$ is admissible. Following Ly, and Tin \cite{LT}, the weight vector $\omega(\underline{t})$ belongs to $\{(2,2,2),\, (2,4,1),\, (4,3,1)\},$ from which $\omega(t)$ belongs to $\{ \widetilde{\omega}_{(i)}:\, 1\leq i\leq 3\}.$ Now if  $\omega_1(t) = 4,$ then $t = t_it_jt_kt_l\underline{t}^2$ in which $\underline{t}$ is an admissible monomial of degree $13$ in $P^{\otimes 5},$ and $1\leq i < j < k <l \leq 5.$ In \cite{Phuc4}, we have shown that $\omega(\underline{t}) = (3,3,1),$ which implies that $\omega(t) = \widetilde{\omega}_{(4)}.$ Thus, $\omega(t)$ belongs to the set $\{ \widetilde{\omega}_{(i)}:\, 1\leq i\leq 4\},$ and one has an isomorphism 
$$ QP_{m_{30, 0}}^{\otimes 5} \cong \bigoplus_{1\leq i\leq 4}QP_{m_{30, 0}}^{\otimes 5}(\widetilde{\omega}_{(i)}).$$
By direct calculations using Kameko's criterion on inadmissible monomials and the results of Peterson \cite{Peterson}, Kameko \cite{Kameko}, Mothebe, and Uys \cite{Mothebe}, Sum \cite{Sum1, Sum2}, we find that 
$$ 
\dim QP_{m_{30, 0}}^{\otimes 5}(\widetilde{\omega}_{(i)})=\left\{\begin{array}{ll} 
154 &\mbox{if $i = 1$},\\
0&\mbox{if $i = 2$},\\
1&\mbox{if $i = 3$},\\
685&\mbox{if $i = 4$},
\end{array}\right.$$
where since 
$$ \begin{array}{ll}
\medskip
QP_{m_{30, 0}}^{\otimes 5}(\widetilde{\omega}_{(i)})&\cong (QP_{m_{30, 0}}^{\otimes 5})^{0}(\widetilde{\omega}_{(i)})\bigoplus (QP_{m_{30, 0}}^{\otimes 5})^{>0}(\widetilde{\omega}_{(i)}),\\
QP_{m_{30, 0}}^{\otimes 5}(\widetilde{\omega}_{(3)}) &= (QP_{m_{30, 0}}^{\otimes 5})^{>0}(\widetilde{\omega}_{(3)}),
\end{array}$$
we have that 
$$ 
\dim (QP_{m_{30, 0}}^{\otimes 5})^{0}(\widetilde{\omega}_{(i)})=\left\{\begin{array}{ll} 
115 &\mbox{if $i = 1$},\\
0&\mbox{if $i = 2$},\\
0 &\mbox{if $i = 3$},\\
175&\mbox{if $i = 4$},
\end{array}\right.$$
and 
$$ 
\dim (QP_{m_{30, 0}}^{\otimes 5})^{> 0}(\widetilde{\omega}_{(i)})=\left\{\begin{array}{ll} 
39 &\mbox{if $i = 1$},\\
0&\mbox{if $i = 2$},\\
1 &\mbox{if $i = 3$},\\
510&\mbox{if $i = 4$}.
\end{array}\right.$$
Therefore, we get $$\dim QP_{m_{30, 0}}^{\otimes 5} =\sum_{1\leq i\leq 4}\dim QP_{m_{30, 0}}^{\otimes 5}(\widetilde{\omega}_{(i)}) = 154 + 1 + 685 = 840.$$

Next, we will give the sketch of proof of the case $r  = 33.$ As mentioned above, we need to compute the kernel of $(\widetilde {Sq^0_*})_{m_{33, 0}}.$ We notice that if $t$ is an admissible monomial of degree $m_{33, 0}$ in $P^{\otimes 5}$ such that $[t]$ belongs to ${\rm Ker}((\widetilde {Sq^0_*})_{m_{33, 0}}),$ then $\omega_1(t) = 3.$ Indeed, we observe that $t_1^{31}t_2t_3\in P^{\otimes 5}$ is minimal spike of degree $m_{33, 0}$, and $\omega(t_1^{31}t_2t_3) = (3,1,1,1,1).$ Then, $t_1^{31}t_2t_3$ is admissible. Because $t$ is admissible, and $\deg(t) = m_{33, 0}$ is odd, by Singer's criterion on hit polynomials, either $\omega_1(t) = 1$ or $\omega_1(t) = 5.$ If $\omega_1(t) = 5,$ then $t$ is of the form $\prod_{1\leq j\leq 5}t_j\underline{t}^{2},$ in which $\underline{t}$ is a monomial of degree $14$ in $P^{\otimes 5}.$ Since $\prod_{1\leq j\leq 5}t_j\underline{t}^{2}$ is admissible, following Kameko's criterion on inadmissible monomials, one has that $\underline{t}$ is admissible, and so, $(\widetilde {Sq^0_*})_{m_{33, 0}}([\prod_{1\leq j\leq 5}t_j\underline{t}^{2}]) = [\underline{t}]\neq [0].$ This contradicts with the fact that $[t] = [\prod_{1\leq j\leq 5}t_j\underline{t}^{2}]$ belongs to the kernel of $(\widetilde {Sq^0_*})_{m_{33, 0}}.$ Thus, we must have  $\omega_1(t) = 3$ and therefore by the above arguments, $t$ can be represented in the form $t_it_jt_k\underline{t}^{2}$, whenever $1\leq i<j<k\leq 5,$ and $\underline{t}$ is an admissible monomial of degree $15$ in $P^{\otimes 5}.$ Therefrom, due to Sum \cite{Sum}, we conclude that the weight vector $\omega(t)$ of $t$ is one of the following sequences: 
$$ \omega_{(1)}:= (3,1,1,1,1),\ \omega_{(2)}:= (3,1,1,3),\ \omega_{(3)}:= (3,3,2,2),\ \omega_{(4)}:= (3,3,4,1).$$
This leads to ${\rm Ker}((\widetilde {Sq^0_*})_{m_{33, 0}})\cong \bigoplus_{1\leq i\leq 4}QP_{m_{33, 0}}^{\otimes 5}(\omega_{(i)}),$ Then, from a result in Sum \cite{Sum1, Sum2}, it is straightforward to check that
$$ 
\dim (QP_{m_{33, 0}}^{\otimes 5})^{0}(\omega_{(i)})=\left\{\begin{array}{ll} 
0&\mbox{if $i = 2,\, 4$},\\
155&\mbox{if $i = 1$},\\
395&\mbox{if $i = 3$}.
\end{array}\right.$$
By direct calculations, we obtain that 
$$ 
\dim (QP_{m_{33, 0}}^{\otimes 5})^{>0}(\omega_{(i)}) =\left\{\begin{array}{ll} 
0&\mbox{if $i = 2,\, 4$},\\
31&\mbox{if $i = 1$},\\
421&\mbox{if $i = 3$}.
\end{array}\right.$$
Indeed, it should be noted that we only need to consider the monomials of the form $t:= t_it_jt_k\underline{t}^{2} = \prod_{1\leq l\leq 5}t_l^{\alpha_l},$ with $\underline{t}\in P^{\otimes 5}_{15},\, 1\leq i< j<k\leq 5,$ $\alpha_1\in \{1, 3, 7, 15\},$ and $\prod_{2\leq l\leq 5}\alpha_l> 0.$\\
For $i = 2,$ one has that $ \omega(t) = \omega_{(2)},$ and that $t$ is one of the following monomials: $t_1^{3}t_2^{4}t_3^{8}t_4^{9}t_5^{9},$ $t_1^{3}t_2^{4}t_3^{9}t_4^{8}t_5^{9},$ $t_1^{3}t_2^{4}t_3^{9}t_4^{9}t_5^{8},$ $t_1^{3}t_2^{5}t_3^{8}t_4^{8}t_5^{9},$ $t_1^{3}t_2^{5}t_3^{8}t_4^{9}t_5^{8},$ $t_1^{3}t_2^{5}t_3^{9}t_4^{8}t_5^{8}.$ Obviously, these monomials are inadmissible.\\
 For $i = 4,$ $\omega(t) = \omega_{(4)},$ we see that $t$ is a permutation of one of the following monomials: $t_1^{3}t_2^{4}t_3^{4}t_4^{7}t_5^{15}$, $t_1^{3}t_2^{4}t_3^{5}t_4^{6}t_5^{15}$, $t_1^{3}t_2^{4}t_3^{5}t_4^{7}t_5^{14}$, $t_1^{3}t_2^{4}t_3^{6}t_4^{7}t_5^{13}$, $t_1^{3}t_2^{4}t_3^{7}t_4^{7}t_5^{12}$,  $t_1^{3}t_2^{5}t_3^{5}t_4^{6}t_5^{14}$, $t_1^{3}t_2^{5}t_3^{6}t_4^{6}t_5^{13}$, $t_1^{3}t_2^{5}t_3^{6}t_4^{7}t_5^{12}$, $t_1^{7}t_2^{2}t_3^{4}t_4^{5}t_5^{15}$,  $t_1^{7}t_2^{2}t_3^{4}t_4^{7}t_5^{13}$, $t_1^{7}t_2^{2}t_3^{5}t_4^{5}t_5^{14}$, $t_1^{7}t_2^{2}t_3^{5}t_4^{6}t_5^{13}$,
$t_1^{7}t_2^{2}t_3^{5}t_4^{7}t_5^{12}$,  $t_1^{7}t_2^{4}t_3^{4}t_4^{7}t_5^{11}$, $t_1^{7}t_2^{4}t_3^{5}t_4^{6}t_5^{11}$, $t_1^{7}t_2^{4}t_3^{5}t_4^{7}t_5^{10}$, $t_1^{7}t_2^{5}t_3^{5}t_4^{6}t_5^{10}$, $t_1^{15}t_2^{2}t_3^{5}t_4^{5}t_5^{6}$. 
So, it is not difficult to show that these monomials are inadmissible. For instance, using the Cartan formula, we find that $t_1^{3}t_2^{4}t_3^{4}t_4^{7}t_5^{15} = Sq^{1}(t_1^{3}t_2t_3^{2}t_4^{11}t_5^{15} + t_1^{3}t_2t_3^{2}t_4^{7}t_5^{19}) + Sq^{2}(t_1^{5}t_2^{2}t_3^{2}t_4^{7}t_5^{15}) + Sq^{4}(t_1^{3}t_2^{2}t_3^{2}t_4^{7}t_5^{15}) + \mbox{smaller monomials}.$\\
 For $i  =1,$ $\omega(t) = \omega_{(1)},$ a simple computation indicates that if $t\neq {\rm adm}_i,\, 1\leq i\leq 31,$ where 

\begin{center}
\begin{tabular}{lllr}
${\rm adm}_{1}=t_1t_2t_3t_4^{2}t_5^{28}$, & ${\rm adm}_{2}=t_1t_2t_3^{2}t_4t_5^{28}$, & ${\rm adm}_{3}=t_1t_2t_3^{2}t_4^{28}t_5$, & \multicolumn{1}{l}{${\rm adm}_{4}=t_1t_2^{2}t_3t_4t_5^{28}$,} \\
${\rm adm}_{5}=t_1t_2^{2}t_3t_4^{28}t_5$, & ${\rm adm}_{6}=t_1t_2^{2}t_3^{28}t_4t_5$, & ${\rm adm}_{7}=t_1t_2t_3^{2}t_4^{4}t_5^{25}$, & \multicolumn{1}{l}{${\rm adm}_{8}=t_1t_2^{2}t_3t_4^{4}t_5^{25}$,} \\
${\rm adm}_{9}=t_1t_2^{2}t_3^{4}t_4t_5^{25}$, & ${\rm adm}_{10}=t_1t_2^{2}t_3^{4}t_4^{25}t_5$, & ${\rm adm}_{11}=t_1t_2t_3^{2}t_4^{5}t_5^{24}$, & \multicolumn{1}{l}{${\rm adm}_{12}=t_1t_2^{2}t_3t_4^{5}t_5^{24}$,} \\
${\rm adm}_{13}=t_1t_2^{2}t_3^{5}t_4t_5^{24}$, & ${\rm adm}_{14}=t_1t_2^{2}t_3^{5}t_4^{24}t_5$, & ${\rm adm}_{15}=t_1t_2t_3^{3}t_4^{4}t_5^{24}$, & \multicolumn{1}{l}{${\rm adm}_{16}=t_1t_2^{3}t_3t_4^{4}t_5^{24}$,} \\
${\rm adm}_{17}=t_1^{3}t_2t_3t_4^{4}t_5^{24}$, & ${\rm adm}_{18}=t_1t_2^{3}t_3^{4}t_4t_5^{24}$, & ${\rm adm}_{19}=t_1^{3}t_2t_3^{4}t_4t_5^{24}$, & \multicolumn{1}{l}{${\rm adm}_{20}=t_1t_2^{3}t_3^{4}t_4^{24}t_5$,} \\
${\rm adm}_{21}=t_1^{3}t_2t_3^{4}t_4^{24}t_5$, & ${\rm adm}_{22}=t_1t_2^{2}t_3^{4}t_4^{9}t_5^{17}$, & ${\rm adm}_{23}=t_1t_2^{2}t_3^{5}t_4^{8}t_5^{17}$, & \multicolumn{1}{l}{${\rm adm}_{24}=t_1t_2^{2}t_3^{5}t_4^{9}t_5^{16}$,} \\
${\rm adm}_{25}=t_1t_2^{3}t_3^{4}t_4^{8}t_5^{17}$, & ${\rm adm}_{26}=t_1^{3}t_2t_3^{4}t_4^{8}t_5^{17}$, & ${\rm adm}_{27}=t_1t_2^{3}t_3^{4}t_4^{9}t_5^{16}$, & \multicolumn{1}{l}{${\rm adm}_{28}=t_1^{3}t_2t_3^{4}t_4^{9}t_5^{16}$,} \\
${\rm adm}_{29}=t_1t_2^{3}t_3^{5}t_4^{8}t_5^{16}$, & ${\rm adm}_{30}=t_1^{3}t_2t_3^{5}t_4^{8}t_5^{16}$, & ${\rm adm}_{31}=t_1^{3}t_2^{5}t_3t_4^{8}t_5^{16},$ &  
\end{tabular}%
\end{center}

then either $t\in \{t_1^{3}t_2^{5}t_3^{8}t_4t_5^{16}, t_1^{3}t_2^{5}t_3^{8}t_4^{16}t_5\},$ or $t$ is of the form $t'z^{16},$ where $t'$ is one of the following inadmissible monomials: $t_1^{3}t_2^{4}t_3t_4^{9},\, t_1^{3}t_2^{4}t_3^{9}t_4,\, t_1^{3}t_2^{4}t_3^{9}t_5,\, t_1^{3}t_2^{4}t_3t_4t_5^{8},\, t_1^{3}t_2^{4}t_3t_4^{8}t_5,\, t_1^{3}t_2^{4}t_3^{8}t_4t_5.$ Applying Cartan's formula, we see that
$$ \begin{array}{ll}
 t_1^{3}t_2^{5}t_3^{8}t_4t_5^{16} &= Sq^{4}(t_1^{3}t_2^{3}t_3^{8}t_4t_5^{14} + t_1^{3}t_2^{3}t_3^{4}t_4t_5^{18} + t_1^{3}t_2^{3}t_3^{2}t_4t_5^{20}+ t_1^{3}t_2^{3}t_3t_4^{2}t_5^{20}+ t_1^{3}t_2^{3}t_3t_4^{4}t_5^{18} + t_1^{3}t_2^{3}t_3t_4^{8}t_5^{14})\\
&\quad +Sq^{2}(t_1^{5}t_2^{3}t_3^{8}t_4t_5^{14}+ t_1^{5}t_2^{3}t_3^{4}t_4t_5^{18}+ t_1^{5}t_2^{3}t_3^{2}t_4t_5^{20}+ t_1^{5}t_2^{3}t_3t_4^{8}t_5^{14}+ t_1^{5}t_2^{3}t_3t_4^{4}t_5^{18}+ t_1^{5}t_2^{3}t_3t_4^{2}t_5^{20} \\
&\quad + t_1^{2}t_2^{3}t_3t_4t_5^{24}) + Sq^{1}(t_1^{3}t_2^{3}t_3t_4t_5^{24}) + t_1^{3}t_2^{5}t_3t_4^{8}t_5^{16} +  t_1^{3}t_2^{4}t_3t_4t_5^{24} + t_1^{2}t_2^{5}t_3t_4t_5^{24} + \sum X,
\end{array}$$
where $\omega(X) < \omega(t_1^{3}t_2^{5}t_3^{8}t_4t_5^{16}) = \omega_{(1)}.$ Since the monomials $t_1^{3}t_2^{5}t_3t_4^{8}t_5^{16},$ $t_1^{3}t_2^{4}t_3t_4t_5^{24},$ $t_1^{2}t_2^{5}t_3t_4t_5^{24},$ $X$ are less than $t_1^{3}t_2^{5}t_3^{8}t_4t_5^{16},$ the monomial $t_1^{3}t_2^{5}t_3^{8}t_4t_5^{16}$ is inadmissible, and so is $t_1^{3}t_2^{5}t_3^{8}t_4^{16}t_5$. Thus, it follows that the indecomposables $(QP_{m_{33, 0}}^{\otimes 5})^{>0}(\omega_{(1)}) = \langle \{[{\rm adm}_i]_{\omega_{(1)}} = [{\rm adm}_i]|\, 1\leq i\leq 31\}\rangle.$ We shall show that the set $\{[{\rm adm}_i]|\, 1\leq i\leq 31\}$ is linearly dependent in $(QP_{m_{33, 0}}^{\otimes 5})^{>0}(\omega_{(1)}).$ For any $1\leq u < v\leq 4,$ we consider the $\mathcal A$-homomorphism $\varphi_{(u, v)}: P^{\otimes 5}\to P^{\otimes 4},$ which is given by 
$$ \varphi_{(u, v)}(t_j) = \left\{ \begin{array}{ll}
{t_j}&\text{if }\;1\leq j \leq u-1, \\
t_{v-1}& \text{if}\; j = u,\\
t_{j-1}&\text{if}\; u+1 \leq j \leq 5.
\end{array} \right.$$
Using a result in \cite{Phuc4}, we have $\varphi_{(u, v)}({\rm adm}_i)\in P_{m_{33, 0}}^{\otimes 4}(\omega_{(1)}),$ for each $1\leq i\leq 31.$  Applying this, we assume that there is a linear relation $S = \sum_{1\leq i\leq 31}\gamma_i{\rm adm}_i\equiv  0,$ in which $\gamma_i\in \mathbb Z/2$ for every $i.$ Based upon the calculations by Sum \cite{Sum1, Sum2}, we determine explicitly $\varphi_{(u, v)}(S)$ in terms of admissible monomials in $P_{m_{33, 0}}^{\otimes 4}$ modulo ($\overline{\mathcal A}P^{\otimes 4}_{m_{33, 0}}$). Then, a direct computation using the relations $\varphi_{(u, v)}(S)\equiv 0,$ for $1\leq u\leq 3,\, 2\leq v\leq 4,$ one gets $\gamma_i = 0,$ for all $i.$\\
Finally, by similar calculations, we also obtain the result for the case $i = 3$ in the above statement. Thus, we could claim that
$$ 
\dim (QP_{m_{33, 0}}^{\otimes 5})(\omega_{(i)}) = \dim (QP_{m_{33, 0}}^{\otimes 5})^{0}(\omega_{(i)}) + \dim (QP_{m_{33, 0}}^{\otimes 5})^{>0}(\omega_{(i)}) =\left\{\begin{array}{ll} 
0&\mbox{if $i = 2,\, 4$},\\
186&\mbox{if $i = 1$},\\
816&\mbox{if $i = 3$},
\end{array}\right.$$
and so $$\dim {\rm Ker}((\widetilde {Sq^0_*})_{m_{33, 0}})=\sum_{1\leq i\leq 4}\dim QP_{m_{33, 0}}^{\otimes 5}(\omega_{(i)})=
186 + 816 = 1002.$$
Because $QP_{m_{33, 0}}^{\otimes 5}\cong {\rm Ker}((\widetilde {Sq^0_*})_{m_{33, 0}})\bigoplus QP_{14}^{\otimes 5},$ and $\dim QP_{14}^{\otimes 5} = 320,$ it may be concluded that $QP_{m_{33, 0}}^{\otimes 5}$ is $1322$-dimensional. 

\medskip

Now, to determine the behavior of $Tr_5^{\mathcal A}$ in degrees $m_{r, 0},$ we wish to compute the coinvariants $\mathbb Z/2 \otimes_{GL_5} {\rm Ann}_{\overline{\mathcal A}}[P_{m_{r, 0}}^{\otimes 5}]^{*}.$ It may be outlined for the degrees $m_{30, 0},$ and $m_{33, 0}$ as follows: From the above results, we have $$QP_{m_{30, 0}}^{\otimes 5} \cong [QP_{m_{30, 0}}^{\otimes 5}(\widetilde{\omega}_{(1)})\bigoplus QP_{m_{30, 0}}^{\otimes 5}(\widetilde{\omega}_{(3)})]\bigoplus QP_{m_{30, 0}}^{\otimes 5}(\widetilde{\omega}_{(4)}),$$ where $\dim (QP_{m_{30, 0}}^{\otimes 5}(\widetilde{\omega}_{(1)})\bigoplus QP_{m_{30, 0}}^{\otimes 5}(\widetilde{\omega}_{(3)})) = 155,$ and $\dim QP_{m_{30, 0}}^{\otimes 5}(\widetilde{\omega}_{(4)}) = 685.$ So direct computations show that $$\dim [QP_{m_{30, 0}}^{\otimes 5}(\widetilde{\omega}_{(1)})\bigoplus QP_{m_{30, 0}}^{\otimes 5}(\widetilde{\omega}_{(3)})]^{GL_5} = 1,\ \mbox{ and }\ \dim [QP_{m_{30, 0}}^{\otimes 5}(\widetilde{\omega}_{(4)})]^{GL_5} = 0.$$
So $\dim \mathbb Z/2 \otimes_{GL_5} {\rm Ann}_{\overline{\mathcal A}}[P_{m_{30, 0}}^{\otimes 5}]^{*} = \dim [QP^{\otimes 5}_{m_{30,0}}]^{GL_5} = 1.$ Moreover, based on this dimension, we find that $\mathbb Z/2 \otimes_{GL_5} {\rm Ann}_{\overline{\mathcal A}}[P_{m_{30, 0}}^{\otimes 5}]^{*}  = \langle [\zeta_{30}] \rangle.$ Next, let us note again that 
$$ \begin{array}{ll}
\medskip
 QP^{\otimes 5}_{m_{33, 0}}&\cong {\rm Ker}((\widetilde{Sq^0_*})_{m_{33, 0}})\bigoplus QP^{\otimes 5}_{14},\\
{\rm Ker}((\widetilde{Sq^0_*})_{m_{33, 0}})&\cong QP^{\otimes 5}_{m_{33, 0}}(\omega_{(1)})\bigoplus QP^{\otimes 5}_{m_{33, 0}}(\omega_{(3)}).
\end{array}$$
In \cite{Phuc10}, we have shown that $[QP^{\otimes 5}_{14}]^{GL_5} = \langle [\zeta(t_1, \ldots, t_5)] \rangle$, where
$$ \begin{array}{ll}
\medskip
\zeta(t_1, \ldots, t_5) &= W_{51} + W_{53} + W_{55} + W_{56} + W_{57}+W_{111} + W_{112} + W_{113} + W_{114} + W_{115}\\
\medskip
&\quad+ W_{116} + W_{117} + W_{118} + W_{119} + W_{115} + W_{121} + W_{122} + W_{124} + W_{125}\\
&\quad + W_{127} + W_{128} + W_{129} + W_{130}.
\end{array}$$
The admissible monomials $W_i$ are given in the same paper \cite[Subsection 3.1]{Phuc10}. Noting that the variables $x_j,\,1\leq j\leq 5$ in the monomials $W_i$ are replaced by $t_j.$ Further, by a simple computation, we see that $$[QP^{\otimes 5}_{m_{33, 0}}(\omega_{(1)})]^{GL_5} = 0,\ \mbox{ and }\ [QP^{\otimes 5}_{m_{33, 0}}(\omega_{(3)})]^{GL_5} = 0,$$ which implies that $[{\rm Ker}((\widetilde{Sq^0_*})_{m_{33, 0}})]^{GL_5} = 0,$ and therefore $$[QP^{\otimes 5}_{m_{33, 0}}]^{GL_5} = \langle [\prod_{1\leq j\leq 5}t_j\zeta(t_1,\ldots, t_5)^{2}]\rangle.$$ Thus, since $\mathbb Z/2 \otimes_{GL_5} {\rm Ann}_{\overline{\mathcal A}}[P_{m_{33, 0}}^{\otimes 5}]^{*}\cong [QP^{\otimes 5}_{m_{33, 0}}]^{GL_5},$ the coinvariant $\mathbb Z/2 \otimes_{GL_5} {\rm Ann}_{\overline{\mathcal A}}[P_{m_{r, 0}}^{\otimes 5}]^{*}$, which has dimension $1$, is generated by the element $([\prod_{1\leq j\leq 5}t_j\zeta(t_1, \ldots, t_5)^{2}])^{*}=[\zeta_{33}].$ Here the $\overline{\mathcal A}$-annahilated elements $\zeta_{30},$ and $\zeta_{33}$ are determined explicitly as in the technical claim below.

\begin{thm}\label{dlbb5}
The dimension of  the coinvariant $ \mathbb Z/2 \otimes_{GL_5} {\rm Ann}_{\overline{\mathcal A}}[P_{m_{r, 0}}^{\otimes 5}]^{*}$ is given by

\centerline{
\scalebox{0.9}{
\begin{tabular}{c|ccccccccccccccc}
$r$ & $30$ &$33$ & $34$ & $36$ & $38$ & $39$ & $40$ & $42$ & $45$ & $46$ & $47$ & $48$\cr
\hline
\ $\dim \mathbb Z/2 \otimes_{GL_5} {\rm Ann}_{\overline{\mathcal A}}[P_{m_{r, 0}}^{\otimes 5}]^{*} = \dim [QP_{m_{r, 0}}^{\otimes 5}]^{GL_5}$ & $1$ & $1$ & $0$ & $0$ &$1$ &$1$ & $2$ & $0$ & $2$ & $0$ & $1$ & $1$ \cr
\end{tabular}}
}

\medskip

Furthermore, the generators of $\mathbb Z/2 \otimes_{GL_5} {\rm Ann}_{\overline{\mathcal A}}[P_{m_{r, 0}}^{\otimes 5}]^{*}$ for $r\in \{30,\, 33,\, 38,\, 39,\, 40,\, 45,\, 47,\, 48\}$ are determined as follows:

\centerline{
\scalebox{0.9}{
\begin{tabular}{c|cccccccc}
$r$ & $30$ &$33$ & $38$ & $39$ & $40$ & $45$ & $47$ & $48$\cr
\hline
\ $\mathbb Z/2 \otimes_{GL_5} {\rm Ann}_{\overline{\mathcal A}}[P_{m_{r, 0}}^{\otimes 5}]^{*}=$ & $\langle [\zeta_{30, 0}] \rangle$ & $\langle [\zeta_{33, 0}] \rangle$ & $\langle [\zeta_{38, 0}] \rangle$ & $\langle [\zeta_{39, 0}] \rangle$ & $\langle [\zeta_{40, 0}],\, [\overline{\zeta}_{40, 0}] \rangle$ & $\langle [\zeta_{45, 0}],\, [\overline{\zeta}_{45, 0}] \rangle$ & $\langle [\zeta_{47, 0}] \rangle$ & $\langle [\zeta_{48, 0}] \rangle$  \cr
\end{tabular}}
}
\medskip

where 
$$ \begin{array}{ll}
\medskip
\zeta_{30, 0} &=  x_1^{(0)}x_2^{(0)}x_3^{(0)}x_4^{(15)}x_5^{(15)},\\
\zeta_{33, 0} &=  x_1^{(1)}x_2^{(3)}x_3^{(13)}x_4^{(7)}x_5^{(9)}+
 x_1^{(1)}x_2^{(3)}x_3^{(13)}x_4^{(11)}x_5^{(5)}+
 x_1^{(1)}x_2^{(3)}x_3^{(13)}x_4^{(13)}x_5^{(3)}+
\medskip
 x_1^{(1)}x_2^{(5)}x_3^{(11)}x_4^{(7)}x_5^{(9)}\\
&+ x_1^{(1)}x_2^{(5)}x_3^{(11)}x_4^{(11)}x_5^{(5)}+
 x_1^{(1)}x_2^{(5)}x_3^{(11)}x_4^{(13)}x_5^{(3)}+
 x_1^{(1)}x_2^{(5)}x_3^{(13)}x_4^{(7)}x_5^{(7)}+
\medskip
 x_1^{(1)}x_2^{(7)}x_3^{(3)}x_4^{(11)}x_5^{(11)}\\
&+ x_1^{(1)}x_2^{(7)}x_3^{(3)}x_4^{(13)}x_5^{(9)}+
 x_1^{(1)}x_2^{(7)}x_3^{(5)}x_4^{(11)}x_5^{(9)}+
 x_1^{(1)}x_2^{(7)}x_3^{(5)}x_4^{(13)}x_5^{(7)}+
\medskip
 x_1^{(1)}x_2^{(7)}x_3^{(7)}x_4^{(7)}x_5^{(11)}\\
& x_1^{(1)}x_2^{(7)}x_3^{(7)}x_4^{(9)}x_5^{(9)}+
 x_1^{(1)}x_2^{(7)}x_3^{(7)}x_4^{(13)}x_5^{(5)}+
 x_1^{(1)}x_2^{(7)}x_3^{(9)}x_4^{(7)}x_5^{(9)}+
\medskip
 x_1^{(1)}x_2^{(7)}x_3^{(11)}x_4^{(5)}x_5^{(9)}\\
&+ x_1^{(1)}x_2^{(7)}x_3^{(13)}x_4^{(3)}x_5^{(9)}+
 x_1^{(1)}x_2^{(7)}x_3^{(13)}x_4^{(5)}x_5^{(7)}+
 x_1^{(1)}x_2^{(9)}x_3^{(7)}x_4^{(7)}x_5^{(9)}+
\medskip
 x_1^{(1)}x_2^{(9)}x_3^{(7)}x_4^{(11)}x_5^{(5)}\\
&+
 x_1^{(1)}x_2^{(9)}x_3^{(7)}x_4^{(13)}x_5^{(3)}+
 x_1^{(1)}x_2^{(11)}x_3^{(3)}x_4^{(7)}x_5^{(11)}+
 x_1^{(1)}x_2^{(11)}x_3^{(3)}x_4^{(13)}x_5^{(5)}+
\medskip
 x_1^{(1)}x_2^{(11)}x_3^{(5)}x_4^{(11)}x_5^{(5)}\\
&+x_1^{(1)}x_2^{(11)}x_3^{(7)}x_4^{(3)}x_5^{(11)}+
 x_1^{(1)}x_2^{(11)}x_3^{(7)}x_4^{(9)}x_5^{(5)}+
 x_1^{(1)}x_2^{(11)}x_3^{(9)}x_4^{(7)}x_5^{(5)}+
\medskip
 x_1^{(1)}x_2^{(11)}x_3^{(11)}x_4^{(3)}x_5^{(7)}\\
 &+x_1^{(1)}x_2^{(11)}x_3^{(11)}x_4^{(5)}x_5^{(5)}+
 x_1^{(1)}x_2^{(11)}x_3^{(13)}x_4^{(3)}x_5^{(5)}+
 x_1^{(1)}x_2^{(13)}x_3^{(3)}x_4^{(13)}x_5^{(3)}+
\medskip
 x_1^{(1)}x_2^{(13)}x_3^{(5)}x_4^{(11)}x_5^{(3)}\\
&+ x_1^{(1)}x_2^{(13)}x_3^{(7)}x_4^{(7)}x_5^{(5)}+
 x_1^{(1)}x_2^{(13)}x_3^{(7)}x_4^{(7)}x_5^{(5)}+
 x_1^{(1)}x_2^{(13)}x_3^{(7)}x_4^{(9)}x_5^{(3)}+
\medskip
 x_1^{(1)}x_2^{(13)}x_3^{(9)}x_4^{(7)}x_5^{(3)}\\
 &+x_1^{(1)}x_2^{(13)}x_3^{(11)}x_4^{(5)}x_5^{(3)}+
\medskip
 x_1^{(1)}x_2^{(13)}x_3^{(13)}x_4^{(3)}x_5^{(3)},\\
\medskip
\zeta_{38, 0} &= x_1^{(0)}x_2^{(0)}x_3^{(0)}x_4^{(7)}x_5^{(31)},\\
\zeta_{39, 0} &=  x_1^{(1)}x_2^{(11)}x_3^{(11)}x_4^{(11)}x_5^{(5)}+
  x_1^{(1)}x_2^{(11)}x_3^{(11)}x_4^{(13)}x_5^{(3)}+
 x_1^{(1)} x_2^{(7)}x_3^{(11)}x_4^{(17)}x_5^{(3)}+
\medskip
 x_1^{(1)}x_2^{(11)}x_3^{(7)}x_4^{(17)}x_5^{(3)} \\
&  + x_1^{(1)}x_2^{(7)}x_3^{(13)}x_4^{(15)}x_5^{(3)}+
x_1^{(1)}x_2^{(11)}x_3^{(15)}x_4^{(9)}x_5^{(3)}+
 x_1^{(1)}x_2^{(15)}x_3^{(11)}x_4^{(9)}x_5^{(3)}+
\medskip
  x_1^{(1)}x_2^{(7)}x_3^{(19)}x_4^{(9)}x_5^{(3)}\\
 &+ x_1^{(1)}x_2^{(19)}x_3^{(7)}x_4^{(9)}x_5^{(3)}+
 x_1^{(1)}x_2^{(7)}x_3^{(19)}x_4^{(7)}x_5^{(5)}+
 x_1^{(1)}x_2^{(19)}x_3^{(7)}x_4^{(7)}x_5^{(5)}+
\medskip
x_1^{(1)}x_2^{(11)}x_3^{(19)}x_4^{(5)}x_5^{(3)}\\
&+ x_1^{(1)}x_2^{(19)}x_3^{(11)}x_4^{(5)}x_5^{(3)} +
 x_1^{(1)}x_2^{(11)}x_3^{(21)}x_4^{(3)}x_5^{(3)}+
 x_1^{(1)}x_2^{(19)}x_3^{(13)}x_4^{(3)}x_5^{(3)}+
\medskip
x_1^{(1)}x_2^{(7)}x_3^{(23)}x_4^{(5)}x_5^{(3)}\\
&+  x_1^{(1)}x_2^{(23)}x_3^{(7)}x_4^{(5)}x_5^{(3)} +
  x_1^{(1)}x_2^{(11)}x_3^{(11)}x_4^{(7)}x_5^{(9)}+
  x_1^{(1)}x_2^{(11)}x_3^{(7)}x_4^{(11)}x_5^{(9)}+
\medskip
x_1^{(1)}x_2^{(7)}x_3^{(11)}x_4^{(11)}x_5^{(9)}\\
&+  x_1^{(1)}x_2^{(7)}x_3^{(25)}x_4^{(3)}x_5^{(3)}+
   x_1^{(1)}x_2^{(23)}x_3^{(9)}x_4^{(3)}x_5^{(3)}+
  x_1^{(1)}x_2^{(15)}x_3^{(17)}x_4^{(3)}x_5^{(3)}+
\medskip
x_1^{(1)} x_2^{(15)}x_3^{(15)}x_4^{(3)}x_5^{(5)}\\
 &+x_1^{(1)}x_2^{(27)}x_3^{(5)}x_4^{(3)}x_5^{(3)}+
  x_1^{(1)}x_2^{(29)}x_3^{(3)}x_4^{(3)}x_5^{(3)}+
  x_1^{(1)}x_2^{(13)}x_3^{(11)}x_4^{(7)}x_5^{(7)}+
\medskip
 x_1^{(1)}x_2^{(11)}x_3^{(7)}x_4^{(13)}x_5^{(7)}\\
  &+x_1^{(1)}x_2^{(7)}x_3^{(13)}x_4^{(11)}x_5^{(7)} +
  x_1^{(1)}x_2^{(13)}x_3^{(7)}x_4^{(7)}x_5^{(11)}+
x_1^{(1)}x_2^{(7)}x_3^{(7)}x_4^{(13)}x_5^{(11)}+
\medskip
  x_1^{(1)}x_2^{(7)}x_3^{(13)}x_4^{(7)}x_5^{(11)}\\
\end{array}$$

\newpage
$$ \begin{array}{ll}
 &+  x_1^{(1)}x_2^{(11)}x_3^{(7)}x_4^{(7)}x_5^{(13)}+
  x_1^{(1)}x_2^{(7)}x_3^{(11)}x_4^{(7)}x_5^{(13)}+
x_1^{(1)}x_2^{(7)}x_3^{(7)}x_4^{(11)}x_5^{(13)}+
\medskip
  x_1^{(1)}x_2^{(7)}x_3^{(7)}x_4^{(7)}x_5^{(17)}\\
&+ x_1^{(1)}x_2^{(7)}x_3^{(7)}x_4^{(9)}x_5^{(15)} +
  x_1^{(1)}x_2^{(7)}x_3^{(11)}x_4^{(5)}x_5^{(15)}+
x_1^{(1)}x_2^{(7)}x_3^{(13)}x_4^{(3)}x_5^{(15)}+
\medskip
  x_1^{(1)}x_2^{(7)}x_3^{(7)}x_4^{(19)}x_5^{(5)}\\
 &+ x_1^{(1)}x_2^{(7)}x_3^{(7)}x_4^{(21)}x_5^{(3)}+
 x_1^{(1)}x_2^{(11)}x_3^{(7)}x_4^{(15)}x_5^{(5)}+
 x_1^{(1)}x_2^{(11)}x_3^{(15)}x_4^{(7)}x_5^{(5)}+
\medskip
  x_1^{(1)}x_2^{(15)}x_3^{(11)}x_4^{(7)}x_5^{(5)},\\
\zeta_{40, 0} &=   x_1^{(0)}x_2^{(7)}x_3^{(11)}x_4^{(3)}x_5^{(19)}+
 x_1^{(0)}x_2^{(7)}x_3^{(11)}x_4^{(5)}x_5^{(17)}+
 x_1^{(0)}x_2^{(7)}x_3^{(13)}x_4^{(3)}x_5^{(17)}+
\medskip
 x_1^{(0)}x_2^{(7)}x_3^{(13)}x_4^{(5)}x_5^{(15)}\\
&\quad +
 x_1^{(0)}x_2^{(7)}x_3^{(11)}x_4^{(9)}x_5^{(13)}+
 x_1^{(0)}x_2^{(7)}x_3^{(13)}x_4^{(7)}x_5^{(13)}+
 x_1^{(0)}x_2^{(11)}x_3^{(13)}x_4^{(3)}x_5^{(13)}+
\medskip
 x_1^{(0)}x_2^{(7)}x_3^{(11)}x_4^{(11)}x_5^{(11)}\\
&\quad +
 x_1^{(0)}x_2^{(7)}x_3^{(13)}x_4^{(9)}x_5^{(11)}+
x_1^{(0)}x_2^{(11)}x_3^{(13)}x_4^{(5)}x_5^{(11)}+
 x_1^{(0)}x_2^{(7)}x_3^{(19)}x_4^{(3)}x_5^{(11)}+
\medskip
 x_1^{(0)}x_2^{(11)}x_3^{(15)}x_4^{(3)}x_5^{(11)}\\
&\quad +
 x_1^{(0)}x_2^{(7)}x_3^{(19)}x_4^{(5)}x_5^{(9)}+
 x_1^{(0)}x_2^{(11)}x_3^{(15)}x_4^{(5)}x_5^{(9)}+
 x_1^{(0)}x_2^{(7)}x_3^{(21)}x_4^{(3)}x_5^{(9)}+
\medskip
 x_1^{(0)}x_2^{(13)}x_3^{(15)}x_4^{(3)}x_5^{(9)}\\
&\quad +
 x_1^{(0)}x_2^{(11)}x_3^{(13)}x_4^{(9)}x_5^{(7)}+
 x_1^{(0)}x_2^{(13)}x_3^{(15)}x_4^{(5)}x_5^{(7)}+
 x_1^{(0)}x_2^{(7)}x_3^{(21)}x_4^{(5)}x_5^{(7)}+
\medskip
 x_1^{(0)}x_2^{(7)}x_3^{(23)}x_4^{(5)}x_5^{(5)}\\
&\quad +
 x_1^{(0)}x_2^{(11)}x_3^{(19)}x_4^{(5)}x_5^{(5)}+
 x_1^{(0)}x_2^{(13)}x_3^{(21)}x_4^{(3)}x_5^{(3)}+
 x_1^{(0)}x_2^{(7)}x_3^{(11)}x_4^{(19)}x_5^{(3)}+
\medskip
 x_1^{(0)}x_2^{(7)}x_3^{(11)}x_4^{(17)}x_5^{(5)}\\
&\quad +
 x_1^{(0)}x_2^{(7)}x_3^{(13)}x_4^{(17)}x_5^{(3)}+
 x_1^{(0)}x_2^{(7)}x_3^{(13)}x_4^{(15)}x_5^{(5)}+
 x_1^{(0)}x_2^{(7)}x_3^{(11)}x_4^{(13)}x_5^{(9)}+
\medskip
 x_1^{(0)}x_2^{(7)}x_3^{(13)}x_4^{(13)}x_5^{(7)}\\
&\quad +
 x_1^{(0)}x_2^{(11)}x_3^{(13)}x_4^{(13)}x_5^{(3)}+
 x_1^{(0)}x_2^{(7)}x_3^{(11)}x_4^{(11)}x_5^{(11)}+
 x_1^{(0)}x_2^{(7)}x_3^{(13)}x_4^{(11)}x_5^{(9)}+
\medskip
 x_1^{(0)}x_2^{(11)}x_3^{(13)}x_4^{(11)}x_5^{(5)}\\
&\quad +
 x_1^{(0)}x_2^{(7)}x_3^{(19)}x_4^{(11)}x_5^{(3)}+
 x_1^{(0)}x_2^{(11)}x_3^{(15)}x_4^{(11)}x_5^{(3)}+
 x_1^{(0)}x_2^{(7)}x_3^{(19)}x_4^{(9)}x_5^{(5)}+
\medskip
 x_1^{(0)}x_2^{(11)}x_3^{(15)}x_4^{(9)}x_5^{(5)}\\
&\quad +
 x_1^{(0)}x_2^{(7)}x_3^{(21)}x_4^{(9)}x_5^{(3)}+
 x_1^{(0)}x_2^{(13)}x_3^{(15)}x_4^{(9)}x_5^{(3)}+
 x_1^{(0)}x_2^{(11)}x_3^{(13)}x_4^{(7)}x_5^{(9)}+
\medskip
 x_1^{(0)}x_2^{(13)}x_3^{(15)}x_4^{(7)}x_5^{(5)}\\
&\quad +
 x_1^{(0)}x_2^{(7)}x_3^{(21)}x_4^{(7)}x_5^{(5)}+
 x_1^{(0)}x_2^{(7)}x_3^{(23)}x_4^{(5)}x_5^{(5)}+
 x_1^{(0)}x_2^{(11)}x_3^{(19)}x_4^{(5)}x_5^{(5)}+
\medskip
 x_1^{(0)}x_2^{(13)}x_3^{(21)}x_4^{(3)}x_5^{(3)}\\
&\quad+
 x_1^{(0)}x_2^{(7)}x_3^{(25)}x_4^{(3)}x_5^{(5)}+
 x_1^{(0)}x_2^{(15)}x_3^{(17)}x_4^{(3)}x_5^{(5)}+
 x_1^{(0)}x_2^{(23)}x_3^{(9)}x_4^{(3)}x_5^{(5)}+
\medskip
 x_1^{(0)}x_2^{(27)}x_3^{(5)}x_4^{(3)}x_5^{(5)}\\
&\quad +
 x_1^{(0)}x_2^{(29)}x_3^{(3)}x_4^{(3)}x_5^{(5)}+
 x_1^{(0)}x_2^{(25)}x_3^{(7)}x_4^{(3)}x_5^{(5)}+
 x_1^{(0)}x_2^{(17)}x_3^{(15)}x_4^{(3)}x_5^{(5)}+
\medskip
 x_1^{(0)}x_2^{(9)}x_3^{(23)}x_4^{(3)}x_5^{(5)}\\
&\quad +
 x_1^{(0)}x_2^{(5)}x_3^{(27)}x_4^{(3)}x_5^{(5)}+
 x_1^{(0)}x_2^{(3)}x_3^{(29)}x_4^{(3)}x_5^{(5)}+
 x_1^{(0)}x_2^{(13)}x_3^{(13)}x_4^{(7)}x_5^{(7)}+
\medskip
 x_1^{(0)}x_2^{(11)}x_3^{(11)}x_4^{(11)}x_5^{(7)}\\
&\quad +
 x_1^{(0)}x_2^{(7)}x_3^{(7)}x_4^{(19)}x_5^{(7)}+
 x_1^{(0)}x_2^{(11)}x_3^{(7)}x_4^{(15)}x_5^{(7)}+
 x_1^{(0)}x_2^{(15)}x_3^{(15)}x_4^{(5)}x_5^{(5)}+
\medskip
 x_1^{(0)}x_2^{(13)}x_3^{(19)}x_4^{(3)}x_5^{(5)}\\
&\quad +
 x_1^{(0)}x_2^{(19)}x_3^{(13)}x_4^{(3)}x_5^{(5)}+
 x_1^{(0)}x_2^{(21)}x_3^{(11)}x_4^{(3)}x_5^{(5)}+
 x_1^{(0)}x_2^{(11)}x_3^{(21)}x_4^{(5)}x_5^{(3)}+
\medskip
 x_1^{(0)}x_2^{(27)}x_3^{(7)}x_4^{(3)}x_5^{(3)}\\
&\quad +
 x_1^{(0)}x_2^{(11)}x_3^{(23)}x_4^{(3)}x_5^{(3)}+
\medskip
 x_1^{(0)}x_2^{(19)}x_3^{(15)}x_4^{(3)}x_5^{(3)},\\
\medskip
\overline{\zeta}_{40, 0} &=  x_1^{(1)}x_2^{(31)}x_3^{(3)}x_4^{(3)}x_5^{(2)} + x_1^{(1)}x_2^{(31)}x_3^{(3)}x_4^{(4)}x_5^{(1)} + x_1^{(1)}x_2^{(31)}x_3^{(5)}x_4^{(2)}x_5^{(1)} + x_1^{(1)}x_2^{(31)}x_3^{(6)}x_4^{(1)}x_5^{(1)},\\
\zeta_{45, 0} &= x_1^{(31)}x_2^{(6)}x_3^{(5)}x_4^{(2)}x_5^{(1)}+
x_1^{(31)}x_2^{(6)}x_3^{(4)}x_4^{(3)}x_5^{(1)}+
x_1^{(31)}x_2^{(6)}x_3^{(3)}x_4^{(4)}x_5^{(1)}+
\medskip
x_1^{(31)}x_2^{(6)}x_3^{(2)}x_4^{(5)}x_5^{(1)}\\
&\quad+
x_1^{(31)}x_2^{(6)}x_3^{(1)}x_4^{(6)}x_5^{(1)}+
x_1^{(31)}x_2^{(5)}x_3^{(6)}x_4^{(1)}x_5^{(2)}+
x_1^{(31)}x_2^{(5)}x_3^{(5)}x_4^{(2)}x_5^{(2)}+
\medskip
x_1^{(31)}x_2^{(5)}x_3^{(4)}x_4^{(3)}x_5^{(2)}\\
&\quad+
x_1^{(31)}x_2^{(5)}x_3^{(3)}x_4^{(4)}x_5^{(2)}+
x_1^{(31)}x_2^{(5)}x_3^{(2)}x_4^{(5)}x_5^{(2)}+
x_1^{(31)}x_2^{(5)}x_3^{(1)}x_4^{(6)}x_5^{(2)}+
\medskip
x_1^{(31)}x_2^{(5)}x_3^{(5)}x_4^{(1)}x_5^{(3)}\\
&\quad+
x_1^{(31)}x_2^{(3)}x_3^{(6)}x_4^{(2)}x_5^{(3)}+
x_1^{(31)}x_2^{(6)}x_3^{(2)}x_4^{(3)}x_5^{(3)}+
x_1^{(31)}x_2^{(6)}x_3^{(1)}x_4^{(4)}x_5^{(3)}+
\medskip
x_1^{(31)}x_2^{(5)}x_3^{(2)}x_4^{(4)}x_5^{(3)}\\
&\quad+
x_1^{(31)}x_2^{(3)}x_3^{(4)}x_4^{(4)}x_5^{(3)}+
x_1^{(31)}x_2^{(3)}x_3^{(2)}x_4^{(6)}x_5^{(3)}+
x_1^{(31)}x_2^{(3)}x_3^{(6)}x_4^{(1)}x_5^{(4)}+
\medskip
x_1^{(31)}x_2^{(3)}x_3^{(5)}x_4^{(2)}x_5^{(4)}\\
&\quad+
x_1^{(31)}x_2^{(3)}x_3^{(4)}x_4^{(3)}x_5^{(4)}+
x_1^{(31)}x_2^{(3)}x_3^{(3)}x_4^{(4)}x_5^{(4)}+
x_1^{(31)}x_2^{(3)}x_3^{(2)}x_4^{(5)}x_5^{(4)}+
\medskip
x_1^{(31)}x_2^{(3)}x_3^{(1)}x_4^{(6)}x_5^{(4)}\\
&\quad+
x_1^{(31)}x_2^{(5)}x_3^{(3)}x_4^{(1)}x_5^{(5)}+
x_1^{(31)}x_2^{(6)}x_3^{(1)}x_4^{(2)}x_5^{(5)}+
x_1^{(31)}x_2^{(5)}x_3^{(2)}x_4^{(2)}x_5^{(5)}+
\medskip
x_1^{(31)}x_2^{(3)}x_3^{(4)}x_4^{(2)}x_5^{(5)}\\
&\quad+
x_1^{(31)}x_2^{(5)}x_3^{(1)}x_4^{(3)}x_5^{(5)}+
x_1^{(31)}x_2^{(3)}x_3^{(3)}x_4^{(3)}x_5^{(5)}+
x_1^{(31)}x_2^{(3)}x_3^{(1)}x_4^{(5)}x_5^{(5)}+
\medskip
x_1^{(31)}x_2^{(6)}x_3^{(1)}x_4^{(1)}x_5^{(6)}\\
&\quad+
x_1^{(31)}x_2^{(5)}x_3^{(2)}x_4^{(1)}x_5^{(6)}+
x_1^{(31)}x_2^{(3)}x_3^{(4)}x_4^{(1)}x_5^{(6)}+
x_1^{(31)}x_2^{(3)}x_3^{(3)}x_4^{(2)}x_5^{(6)}+
\medskip
x_1^{(31)}x_2^{(6)}x_3^{(6)}x_4^{(1)}x_5^{(1)},\\
\end{array}$$

\newpage
$$ \begin{array}{ll}
\overline{\zeta}_{45, 0} &= x_1^{(7)}x_2^{(11)}x_3^{(11)}x_4^{(11)}x_5^{(5)}+
  x_1^{(7)}x_2^{(11)}x_3^{(11)}x_4^{(13)}x_5^{(3)}+
 x_1^{(7)} x_2^{(7)}x_3^{(11)}x_4^{(17)}x_5^{(3)}+
\medskip
 x_1^{(7)}x_2^{(11)}x_3^{(7)}x_4^{(17)}x_5^{(3)} \\
&  + x_1^{(7)}x_2^{(7)}x_3^{(13)}x_4^{(15)}x_5^{(3)}+
x_1^{(7)}x_2^{(11)}x_3^{(15)}x_4^{(9)}x_5^{(3)}+
 x_1^{(7)}x_2^{(15)}x_3^{(11)}x_4^{(9)}x_5^{(3)}+
\medskip
  x_1^{(7)}x_2^{(7)}x_3^{(19)}x_4^{(9)}x_5^{(3)}\\
 &+ x_1^{(7)}x_2^{(19)}x_3^{(7)}x_4^{(9)}x_5^{(3)}+
 x_1^{(7)}x_2^{(7)}x_3^{(19)}x_4^{(7)}x_5^{(5)}+
 x_1^{(7)}x_2^{(19)}x_3^{(7)}x_4^{(7)}x_5^{(5)}+
\medskip
x_1^{(7)}x_2^{(11)}x_3^{(19)}x_4^{(5)}x_5^{(3)}\\
&+ x_1^{(7)}x_2^{(19)}x_3^{(11)}x_4^{(5)}x_5^{(3)} +
 x_1^{(7)}x_2^{(11)}x_3^{(21)}x_4^{(3)}x_5^{(3)}+
 x_1^{(7)}x_2^{(19)}x_3^{(13)}x_4^{(3)}x_5^{(3)}+
\medskip
x_1^{(7)}x_2^{(7)}x_3^{(23)}x_4^{(5)}x_5^{(3)}\\
&+  x_1^{(7)}x_2^{(23)}x_3^{(7)}x_4^{(5)}x_5^{(3)} +
  x_1^{(7)}x_2^{(11)}x_3^{(11)}x_4^{(7)}x_5^{(9)}+
  x_1^{(7)}x_2^{(11)}x_3^{(7)}x_4^{(11)}x_5^{(9)}+
\medskip
x_1^{(7)}x_2^{(7)}x_3^{(11)}x_4^{(11)}x_5^{(9)}\\
&+  x_1^{(7)}x_2^{(7)}x_3^{(25)}x_4^{(3)}x_5^{(3)}+
   x_1^{(7)}x_2^{(23)}x_3^{(9)}x_4^{(3)}x_5^{(3)}+
  x_1^{(7)}x_2^{(15)}x_3^{(17)}x_4^{(3)}x_5^{(3)}+
\medskip
x_1^{(7)} x_2^{(15)}x_3^{(15)}x_4^{(3)}x_5^{(5)}\\
 &+x_1^{(7)}x_2^{(27)}x_3^{(5)}x_4^{(3)}x_5^{(3)}+
  x_1^{(7)}x_2^{(29)}x_3^{(3)}x_4^{(3)}x_5^{(3)}+
  x_1^{(7)}x_2^{(13)}x_3^{(11)}x_4^{(7)}x_5^{(7)}+
\medskip
 x_1^{(7)}x_2^{(11)}x_3^{(7)}x_4^{(13)}x_5^{(7)}\\
  &+x_1^{(7)}x_2^{(7)}x_3^{(13)}x_4^{(11)}x_5^{(7)} +
  x_1^{(7)}x_2^{(13)}x_3^{(7)}x_4^{(7)}x_5^{(11)}+
x_1^{(7)}x_2^{(7)}x_3^{(7)}x_4^{(13)}x_5^{(11)}+
\medskip
  x_1^{(7)}x_2^{(7)}x_3^{(13)}x_4^{(7)}x_5^{(11)}\\
 &+  x_1^{(7)}x_2^{(11)}x_3^{(7)}x_4^{(7)}x_5^{(13)}+
  x_1^{(7)}x_2^{(7)}x_3^{(11)}x_4^{(7)}x_5^{(13)}+
x_1^{(7)}x_2^{(7)}x_3^{(7)}x_4^{(11)}x_5^{(13)}+
\medskip
  x_1^{(7)}x_2^{(7)}x_3^{(7)}x_4^{(7)}x_5^{(17)}\\
&+ x_1^{(7)}x_2^{(7)}x_3^{(7)}x_4^{(9)}x_5^{(15)} +
  x_1^{(7)}x_2^{(7)}x_3^{(11)}x_4^{(5)}x_5^{(15)}+
x_1^{(7)}x_2^{(7)}x_3^{(13)}x_4^{(3)}x_5^{(15)}+
\medskip
  x_1^{(7)}x_2^{(7)}x_3^{(7)}x_4^{(19)}x_5^{(5)}\\
 &+ x_1^{(7)}x_2^{(7)}x_3^{(7)}x_4^{(21)}x_5^{(3)}+
 x_1^{(7)}x_2^{(11)}x_3^{(7)}x_4^{(15)}x_5^{(5)}+
 x_1^{(7)}x_2^{(11)}x_3^{(15)}x_4^{(7)}x_5^{(5)}+
\medskip
  x_1^{(7)}x_2^{(15)}x_3^{(11)}x_4^{(7)}x_5^{(5)},\\
\zeta_{47, 0} &=x_1^{(7)}x_2^{(7)}x_3^{(11)}x_4^{(3)}x_5^{(19)}+
 x_1^{(7)}x_2^{(7)}x_3^{(11)}x_4^{(5)}x_5^{(17)}+
 x_1^{(7)}x_2^{(7)}x_3^{(13)}x_4^{(3)}x_5^{(17)}+
\medskip
 x_1^{(7)}x_2^{(7)}x_3^{(13)}x_4^{(5)}x_5^{(15)}\\
&\quad +
 x_1^{(7)}x_2^{(7)}x_3^{(11)}x_4^{(9)}x_5^{(13)}+
 x_1^{(7)}x_2^{(7)}x_3^{(13)}x_4^{(7)}x_5^{(13)}+
 x_1^{(7)}x_2^{(11)}x_3^{(13)}x_4^{(3)}x_5^{(13)}+
\medskip
 x_1^{(7)}x_2^{(7)}x_3^{(11)}x_4^{(11)}x_5^{(11)}\\
&\quad +
 x_1^{(7)}x_2^{(7)}x_3^{(13)}x_4^{(9)}x_5^{(11)}+
x_1^{(7)}x_2^{(11)}x_3^{(13)}x_4^{(5)}x_5^{(11)}+
 x_1^{(7)}x_2^{(7)}x_3^{(19)}x_4^{(3)}x_5^{(11)}+
\medskip
 x_1^{(7)}x_2^{(11)}x_3^{(15)}x_4^{(3)}x_5^{(11)}\\
&\quad +
 x_1^{(7)}x_2^{(7)}x_3^{(19)}x_4^{(5)}x_5^{(9)}+
 x_1^{(7)}x_2^{(11)}x_3^{(15)}x_4^{(5)}x_5^{(9)}+
 x_1^{(7)}x_2^{(7)}x_3^{(21)}x_4^{(3)}x_5^{(9)}+
\medskip
 x_1^{(7)}x_2^{(13)}x_3^{(15)}x_4^{(3)}x_5^{(9)}\\
&\quad +
 x_1^{(7)}x_2^{(11)}x_3^{(13)}x_4^{(9)}x_5^{(7)}+
 x_1^{(7)}x_2^{(13)}x_3^{(15)}x_4^{(5)}x_5^{(7)}+
 x_1^{(7)}x_2^{(7)}x_3^{(21)}x_4^{(5)}x_5^{(7)}+
\medskip
 x_1^{(7)}x_2^{(7)}x_3^{(23)}x_4^{(5)}x_5^{(5)}\\
&\quad +
 x_1^{(7)}x_2^{(11)}x_3^{(19)}x_4^{(5)}x_5^{(5)}+
 x_1^{(7)}x_2^{(13)}x_3^{(21)}x_4^{(3)}x_5^{(3)}+
 x_1^{(7)}x_2^{(7)}x_3^{(11)}x_4^{(19)}x_5^{(3)}+
\medskip
 x_1^{(7)}x_2^{(7)}x_3^{(11)}x_4^{(17)}x_5^{(5)}\\
&\quad +
 x_1^{(7)}x_2^{(7)}x_3^{(13)}x_4^{(17)}x_5^{(3)}+
 x_1^{(7)}x_2^{(7)}x_3^{(13)}x_4^{(15)}x_5^{(5)}+
 x_1^{(7)}x_2^{(7)}x_3^{(11)}x_4^{(13)}x_5^{(9)}+
\medskip
 x_1^{(7)}x_2^{(7)}x_3^{(13)}x_4^{(13)}x_5^{(7)}\\
&\quad +
 x_1^{(7)}x_2^{(11)}x_3^{(13)}x_4^{(13)}x_5^{(3)}+
 x_1^{(7)}x_2^{(7)}x_3^{(11)}x_4^{(11)}x_5^{(11)}+
 x_1^{(7)}x_2^{(7)}x_3^{(13)}x_4^{(11)}x_5^{(9)}+
\medskip
 x_1^{(7)}x_2^{(11)}x_3^{(13)}x_4^{(11)}x_5^{(5)}\\
&\quad +
 x_1^{(7)}x_2^{(7)}x_3^{(19)}x_4^{(11)}x_5^{(3)}+
 x_1^{(7)}x_2^{(11)}x_3^{(15)}x_4^{(11)}x_5^{(3)}+
 x_1^{(7)}x_2^{(7)}x_3^{(19)}x_4^{(9)}x_5^{(5)}+
\medskip
 x_1^{(7)}x_2^{(11)}x_3^{(15)}x_4^{(9)}x_5^{(5)}\\
&\quad +
 x_1^{(7)}x_2^{(7)}x_3^{(21)}x_4^{(9)}x_5^{(3)}+
 x_1^{(7)}x_2^{(13)}x_3^{(15)}x_4^{(9)}x_5^{(3)}+
 x_1^{(7)}x_2^{(11)}x_3^{(13)}x_4^{(7)}x_5^{(9)}+
\medskip
 x_1^{(7)}x_2^{(13)}x_3^{(15)}x_4^{(7)}x_5^{(5)}\\
&\quad +
 x_1^{(7)}x_2^{(7)}x_3^{(21)}x_4^{(7)}x_5^{(5)}+
 x_1^{(7)}x_2^{(7)}x_3^{(23)}x_4^{(5)}x_5^{(5)}+
 x_1^{(7)}x_2^{(11)}x_3^{(19)}x_4^{(5)}x_5^{(5)}+
\medskip
 x_1^{(7)}x_2^{(13)}x_3^{(21)}x_4^{(3)}x_5^{(3)}\\
&\quad+
 x_1^{(7)}x_2^{(7)}x_3^{(25)}x_4^{(3)}x_5^{(5)}+
 x_1^{(7)}x_2^{(15)}x_3^{(17)}x_4^{(3)}x_5^{(5)}+
 x_1^{(7)}x_2^{(23)}x_3^{(9)}x_4^{(3)}x_5^{(5)}+
\medskip
 x_1^{(7)}x_2^{(27)}x_3^{(5)}x_4^{(3)}x_5^{(5)}\\
&\quad +
 x_1^{(7)}x_2^{(29)}x_3^{(3)}x_4^{(3)}x_5^{(5)}+
 x_1^{(7)}x_2^{(25)}x_3^{(7)}x_4^{(3)}x_5^{(5)}+
 x_1^{(7)}x_2^{(17)}x_3^{(15)}x_4^{(3)}x_5^{(5)}+
\medskip
 x_1^{(7)}x_2^{(9)}x_3^{(23)}x_4^{(3)}x_5^{(5)}\\
&\quad +
 x_1^{(7)}x_2^{(5)}x_3^{(27)}x_4^{(3)}x_5^{(5)}+
 x_1^{(7)}x_2^{(3)}x_3^{(29)}x_4^{(3)}x_5^{(5)}+
 x_1^{(7)}x_2^{(13)}x_3^{(13)}x_4^{(7)}x_5^{(7)}+
\medskip
 x_1^{(7)}x_2^{(11)}x_3^{(11)}x_4^{(11)}x_5^{(7)}\\
&\quad +
 x_1^{(7)}x_2^{(7)}x_3^{(7)}x_4^{(19)}x_5^{(7)}+
 x_1^{(7)}x_2^{(11)}x_3^{(7)}x_4^{(15)}x_5^{(7)}+
 x_1^{(7)}x_2^{(15)}x_3^{(15)}x_4^{(5)}x_5^{(5)}+
\medskip
 x_1^{(7)}x_2^{(13)}x_3^{(19)}x_4^{(3)}x_5^{(5)}\\
&\quad +
 x_1^{(7)}x_2^{(19)}x_3^{(13)}x_4^{(3)}x_5^{(5)}+
 x_1^{(7)}x_2^{(21)}x_3^{(11)}x_4^{(3)}x_5^{(5)}+
 x_1^{(7)}x_2^{(11)}x_3^{(21)}x_4^{(5)}x_5^{(3)}+
\medskip
 x_1^{(7)}x_2^{(27)}x_3^{(7)}x_4^{(3)}x_5^{(3)}\\
&\quad +
 x_1^{(7)}x_2^{(11)}x_3^{(23)}x_4^{(3)}x_5^{(3)}+
\medskip
 x_1^{(7)}x_2^{(19)}x_3^{(15)}x_4^{(3)}x_5^{(3)},\\
\end{array}$$

\newpage
$$ \begin{array}{ll}
\zeta_{48, 0} &= x_1^{(31)}x_2^{(5)}x_3^{(5)}x_4^{(5)}x_5^{(2)}+
 x_1^{(31)}x_2^{(5)}x_3^{(5)}x_4^{(6)}x_5^{(1)}+
 x_1^{(31)}x_2^{(3)}x_3^{(5)}x_4^{(8)}x_5^{(1)}+
\medskip
 x_1^{(31)}x_2^{(5)}x_3^{(3)}x_4^{(8)}x_5^{(1)}\\
&\quad +
 x_1^{(31)}x_2^{(3)}x_3^{(6)}x_4^{(7)}x_5^{(1)}+
 x_1^{(31)}x_2^{(5)}x_3^{(7)}x_4^{(4)}x_5^{(1)}+
x_1^{(31)}x_2^{(7)}x_3^{(5)}x_4^{(4)}x_5^{(1)} + 
\medskip
 x_1^{(31)}x_2^{(3)}x_3^{(9)}x_4^{(4)}x_5^{(1)}\\
&\quad +
 x_1^{(31)}x_2^{(9)}x_3^{(3)}x_4^{(4)}x_5^{(1)}+
 x_1^{(31)}x_2^{(3)}x_3^{(9)}x_4^{(3)}x_5^{(2)}+
 x_1^{(31)}x_2^{(9)}x_3^{(3)}x_4^{(3)}x_5^{(2)}+
\medskip
 x_1^{(31)}x_2^{(5)}x_3^{(9)}x_4^{(2)}x_5^{(1)}\\
&\quad +
 x_1^{(31)}x_2^{(9)}x_3^{(5)}x_4^{(2)}x_5^{(1)} +
 x_1^{(31)}x_2^{(5)}x_3^{(10)}x_4^{(1)}x_5^{(1)}+
 x_1^{(31)}x_2^{(9)}x_3^{(6)}x_4^{(1)}x_5^{(1)}+
\medskip
 x_1^{(31)}x_2^{(3)}x_3^{(11)}x_4^{(2)}x_5^{(1)}\\
&\quad +
 x_1^{(31)}x_2^{(11)}x_3^{(3)}x_4^{(2)}x_5^{(1)} +
 x_1^{(31)}x_2^{(5)}x_3^{(5)}x_4^{(3)}x_5^{(4)}+
 x_1^{(31)}x_2^{(5)}x_3^{(3)}x_4^{(5)}x_5^{(4)}+
\medskip
 x_1^{(31)}x_2^{(3)}x_3^{(5)}x_4^{(5)}x_5^{(4)}\\
&\quad +
 x_1^{(31)}x_2^{(3)}x_3^{(12)}x_4^{(1)}x_5^{(1)}+
 x_1^{(31)}x_2^{(11)}x_3^{(4)}x_4^{(1)}x_5^{(1)}+
 x_1^{(31)}x_2^{(7)}x_3^{(8)}x_4^{(1)}x_5^{(1)}+
\medskip
 x_1^{(31)}x_2^{(7)}x_3^{(7)}x_4^{(1)}x_5^{(2)}\\
&\quad +  
x_1^{(31)}x_2^{(13)}x_3^{(2)}x_4^{(1)}x_5^{(1)} +
 x_1^{(31)}x_2^{(14)}x_3^{(1)}x_4^{(1)}x_5^{(1)}+
 x_1^{(31)}x_2^{(6)}x_3^{(5)}x_4^{(3)}x_5^{(3)}+
\medskip
 x_1^{(31)}x_2^{(5)}x_3^{(3)}x_4^{(6)}x_5^{(3)}\\
&\quad +
 x_1^{(31)}x_2^{(3)}x_3^{(6)}x_4^{(5)}x_5^{(3)} + 
x_1^{(31)}x_2^{(6)}x_3^{(3)}x_4^{(3)}x_5^{(5)}+
 x_1^{(31)}x_2^{(3)}x_3^{(3)}x_4^{(6)}x_5^{(5)}+
\medskip
 x_1^{(31)}x_2^{(3)}x_3^{(6)}x_4^{(3)}x_5^{(5)}\\
&\quad +
x_1^{(31)}x_2^{(5)}x_3^{(3)}x_4^{(3)}x_5^{(6)}+
 x_1^{(31)}x_2^{(3)}x_3^{(5)}x_4^{(3)}x_5^{(6)}+
 x_1^{(31)}x_2^{(3)}x_3^{(3)}x_4^{(5)}x_5^{(6)}+
\medskip
 x_1^{(31)}x_2^{(3)}x_3^{(3)}x_4^{(3)}x_5^{(8)}\\
&\quad + 
x_1^{(31)}x_2^{(3)}x_3^{(3)}x_4^{(4)}x_5^{(7)} +
 x_1^{(31)}x_2^{(3)}x_3^{(5)}x_4^{(2)}x_5^{(7)}+
 x_1^{(31)}x_2^{(3)}x_3^{(6)}x_4^{(1)}x_5^{(7)}+
\medskip
 x_1^{(31)}x_2^{(3)}x_3^{(3)}x_4^{(9)}x_5^{(2)}\\
&\quad +
 x_1^{(31)}x_2^{(3)}x_3^{(3)}x_4^{(10)}x_5^{(1)} + 
x_1^{(31)}x_2^{(5)}x_3^{(3)}x_4^{(7)}x_5^{(2)}+
 x_1^{(31)}x_2^{(5)}x_3^{(7)}x_4^{(3)}x_5^{(2)}+
 x_1^{(31)}x_2^{(7)}x_3^{(5)}x_4^{(3)}x_5^{(2)}.
\end{array}$$
\end{thm}

The proof of the theorem is similar to the calculations of the cases $r = 30$ and $33$. Basing the representation of the fifth transfer over $\Lambda,$ it may be concluded that
$$ \begin{array}{ll}
\medskip
Tr_5^{\mathcal A}([\zeta_{30, 0}]) &=[\psi_5(\zeta_{30, 0})] = [\lambda_0^{3}\lambda_{15}^{2}] = h_0^{3}h_4^{2}\in {\rm Ext}_{\mathcal A}^{5, 5+m_{30,0}}(\mathbb Z/2, \mathbb Z/2),\\
Tr_5^{\mathcal A}([\zeta_{33, 0}]) &=[\psi_5(\zeta_{33, 0})] = [\lambda_1\lambda_7^2\lambda_5\lambda_{13} + \lambda_1\lambda_7^2\lambda_9^2 + \lambda_1\lambda_7\lambda_{11}\lambda_9\lambda_5 + \lambda_1\lambda_{15}\lambda_3\lambda_{11}\lambda_3\\
&  + \delta(\lambda_1\lambda_7^2\lambda_{19} + \lambda_1\lambda_7\lambda_{19}\lambda_7)]\\
\medskip
&=h_1d_1\in {\rm Ext}_{\mathcal A}^{5, 5+m_{33,0}}(\mathbb Z/2, \mathbb Z/2),\\
\medskip
Tr_5^{\mathcal A}([\zeta_{38, 0}]) &=[\psi_5(\zeta_{38, 0})] = [\lambda_0^{3}\lambda_{7}\lambda_{31}] = h_0^{3}h_3h_5\in {\rm Ext}_{\mathcal A}^{5, 5+m_{38,0}}(\mathbb Z/2, \mathbb Z/2),\\
Tr_5^{\mathcal A}([\zeta_{39, 0}]) &=[\psi_5(\zeta_{39, 0})] = [\lambda_1\lambda_7^{3}\lambda_{17} + \lambda_1(\lambda_7\lambda_{11}^{2} + \lambda_7^{2}\lambda_{15})\lambda_9 + \lambda_1\lambda_{15}\lambda_{11}\lambda_7\lambda_5 + \lambda_1\lambda_7^{2}\lambda_{11}\lambda_{13}\\
&+\delta(\lambda_1\lambda_7\lambda_{11}\lambda_{21} + \lambda_1\lambda_7\lambda_{25}\lambda_7 + \lambda_1\lambda_9\lambda_{15}^{2} + \lambda^{2}_1\lambda_{23}\lambda_{15}))]\\
\medskip
&=h_1e_1\in {\rm Ext}_{\mathcal A}^{5, 5+m_{39,0}}(\mathbb Z/2, \mathbb Z/2),\\
Tr_5^{\mathcal A}([\zeta_{40, 0}]) &=[\psi_5(\zeta_{40, 0})] = [\lambda_0\lambda_9\lambda_{13}\lambda_{11}\lambda_7 + \lambda_0\lambda_{11}\lambda_{15}\lambda^2_7 + \lambda_0\lambda_7^2\lambda_5\lambda_{11}\lambda_{15} \\
&+ \lambda_0\lambda_5\lambda_9\lambda_{11}\lambda_{15} +\delta(\lambda_0\lambda_7\lambda_{11}\lambda_{23})]\\
\medskip
&= h_0f_1\in {\rm Ext}_{\mathcal A}^{5, 5+m_{40,0}}(\mathbb Z/2, \mathbb Z/2),\\
\medskip
Tr_5^{\mathcal A}([\overline{\zeta}_{40, 0}]) &=[\psi_5(\overline{\zeta}_{40, 0})] = [\lambda_1\lambda_{31}\lambda_3^{2}\lambda_2] = h_1h_5c_0\in {\rm Ext}_{\mathcal A}^{5, 5+m_{40,0}}(\mathbb Z/2, \mathbb Z/2),\\
Tr_5^{\mathcal A}([\zeta_{45, 0}]) &=[\psi_5(\zeta_{45, 0})] = [\lambda_{31}\lambda_3^2\lambda_2\lambda_6 + \lambda_{31}\lambda_3^2\lambda_4^2 + \lambda_{31}\lambda_3\lambda_5\lambda_4\lambda_2 + \lambda_{31}\lambda_7\lambda_1\lambda_5\lambda_1\\
& + \delta(\lambda_{31}\lambda_3^2\lambda_9 + \lambda_{31}\lambda_3\lambda_9\lambda_3)]\\
\medskip
&= h_5d_0\in {\rm Ext}_{\mathcal A}^{5, 5+m_{45,0}}(\mathbb Z/2, \mathbb Z/2),\\
Tr_5^{\mathcal A}([\overline{\zeta}_{45, 0}]) &=[\psi_5(\overline{\zeta}_{45, 0})] = [\lambda_7\lambda_7^{3}\lambda_{17} + \lambda_7(\lambda_7\lambda_{11}^{2} + \lambda_7^{2}\lambda_{15})\lambda_9 + \lambda_7\lambda_{15}\lambda_{11}\lambda_7\lambda_5 + \lambda_7^{3}\lambda_{11}\lambda_{13}\\
&+\delta(\lambda_7^{2}\lambda_{11}\lambda_{21} + \lambda_7^{2}\lambda_{25}\lambda_7 + \lambda_7\lambda_9\lambda_{15}^{2} + \lambda_7\lambda_1\lambda_{23}\lambda_{15}))]\\
\medskip
&=h_3e_1\in {\rm Ext}_{\mathcal A}^{5, 5+m_{45,0}}(\mathbb Z/2, \mathbb Z/2),\\
Tr_5^{\mathcal A}([\zeta_{47, 0}]) &=[\psi_5(\zeta_{47, 0})] = [\lambda_7\lambda_9\lambda_{13}\lambda_{11}\lambda_7 + \lambda_7\lambda_{11}\lambda_{15}\lambda^2_7 + \lambda_7^3\lambda_5\lambda_{11}\lambda_{15} \\
&+ \lambda_7\lambda_5\lambda_9\lambda_{11}\lambda_{15} +\delta(\lambda_7^{2}\lambda_{11}\lambda_{23})]\\
\medskip
&= h_3f_1\in {\rm Ext}_{\mathcal A}^{5, 5+m_{47,0}}(\mathbb Z/2, \mathbb Z/2),\\
Tr_5^{\mathcal A}([\zeta_{48, 0}]) &=[\psi_5(\zeta_{48, 0})] = [\lambda_{31}\lambda_3^{3}\lambda_8 + \lambda_{31}(\lambda_3\lambda_5^{2} + \lambda_{31}\lambda_3^{2}\lambda_7)\lambda_4 + \lambda_{31}\lambda_7\lambda_5\lambda_3\lambda_2 + \lambda_{31}\lambda_3^{2}\lambda_5\lambda_6\\
 &+ \delta(\lambda_{31}\lambda_3\lambda_5\lambda_{10} + \lambda_{31}\lambda_3\lambda_{12}\lambda_3 + \lambda_{31}\lambda_4\lambda_7^{2} + \lambda_{31}\lambda_0\lambda_{11}\lambda_7)]\\
&=h_5e_0\in {\rm Ext}_{\mathcal A}^{5, 5+m_{48,0}}(\mathbb Z/2, \mathbb Z/2),\\
\end{array}$$

\newpage
According to Lin \cite{Lin}, one has that

\centerline{
\scalebox{0.8}{
\begin{tabular}{c|ccccccccccccccccc}
$r$ & $30$  &$33$ & $34$ & $36$ & $38$ & $39$ & $40$ & $42$ & $45$ & $46$ & $47$ & $48$  \cr
\hline
\ ${\rm Ext}_{\mathcal A}^{5, 5+m_{r,0}}(\mathbb Z/2, \mathbb Z/2)=$ & $\langle h_0^{3}h_4^{2} \rangle$ & $\langle h_0p_0 \rangle$ & $0$ & $0$ &$\langle h_0^{3}h_3h_5 \rangle$ &$\langle h_3d_1 \rangle $ & $\langle h_3p_0, h_1h_5c_0 \rangle$ & $0$ & $\langle h_5d_0, h_1g_2 \rangle$ & $0$ & $\langle h_2g_2 \rangle $ & $\langle  h_5e_0\rangle$  \cr
\end{tabular}}
}
\medskip

where $h_0p_0 = h_1d_1,$ $h_3d_1 = h_1e_1,$ $h_3p_0 = h_0f_1,$ $h_1g_2 = h_3e_1,$ and $h_2g_2 = h_3f_1.$ This data together with the above calculations show that the non-zero elements $h_0^{3}h_4^{2},$ $h_0p_0,$ $h_0^{3}h_3h_5,$ $h_3d_1,$ $h_3p_0,$ $h_1h_5c_0,$ $h_5d_0,$ $h_1g_2,$ $h_2g_2,$ and $h_5e_0$ are detected by $Tr_5^{\mathcal A}.$ Therefore, the reader can see that
\begin{corl}\label{hqkqpl}
Singer's fifth transfer $$ Tr_5^{\mathcal A}: \mathbb Z/2 \otimes_{GL_5} {\rm Ann}_{\overline{\mathcal A}}[P_{m_{r, 0}}^{\otimes 5}]^{*}\to {\rm Ext}_{\mathcal A}^{5, 5+m_{r, 0}}(\mathbb Z/2, \mathbb Z/2)$$
is an isomorphism for $r\in \{30,\, 33,\, 34,\, 36,\, 38,\, 39,\, 40,\, 42,\, 45, 46,\, 47,\, 48\}.$
\end{corl}

Next, we probe the behavior of the fifth cohomological transfer in degrees $n_{r, s}:= 4(2^{s} - 1) + r.2^{s}.$  Following Walker, and Wood \cite{Walker-Wood}, Tin \cite{Tin}, one gets
$$ \begin{array}{ll}
\dim QP_{n_{r, s}}^{\otimes 5}
&=\left\{\begin{array}{ll}
641 &\mbox{if $r = 8$, $s = 1$},\\[1mm]
961 &\mbox{if $r = 10$, $s = 1$},\\[1mm]
2171 &\mbox{if $r = 10$,  $s = 3$},\\[1mm]
2170 &\mbox{if $r = 10$, $s \geq 4$},\\[1mm]
1024 &\mbox{if $r = 11$, $s =1$},\\[1mm]
1984 &\mbox{if $r = 11$, $s \geq 2$}.
\end{array}\right.
\end{array}$$
The case $r = 11$, $s =1$ was proved by Walker, and Wood \cite{Walker-Wood}. The remaining cases were indicated by Tin \cite{Tin}, but the details have not appeared at the time of the writing. The applications of these results give the following theorem.
\begin{thm}\label{dlkqpl}
The coinvariants $\mathbb Z/2 \otimes_{GL_5} {\rm Ann}_{\overline{\mathcal A}}[P_{n_{r, s}}^{\otimes 5}]^{*} $ are given by
$$ \begin{array}{ll}
\mathbb Z/2 \otimes_{GL_5} {\rm Ann}_{\overline{\mathcal A}}[P_{n_{r, s}}^{\otimes 5}]^{*}
&=\left\{\begin{array}{ll}
\langle [\zeta_{8, 1}] \rangle &\mbox{if $r = 8$, $s = 1$},\\[1mm]
\langle [\zeta_{10, 1}] \rangle &\mbox{if $r = 10$, $s = 1$},\\[1mm]
0 &\mbox{if $r = 10$, $s \geq 3$},\\[1mm]
0 &\mbox{if $r = 11$, $s \geq 1$}.
\end{array}\right.
\end{array}$$
where the generators $ \zeta_{8, 1},$ and $ \zeta_{10, 1}$ are determined as follows:
$$  \begin{array}{ll}
 \zeta_{8, 1} &= x_1^{(3)}x_2^{(5)}x_3^{(5)}x_4^{(5)}x_5^{(2)}+
 x_1^{(3)}x_2^{(5)}x_3^{(5)}x_4^{(6)}x_5^{(1)}+
 x_1^{(3)}x_2^{(3)}x_3^{(5)}x_4^{(8)}x_5^{(1)}+
\medskip
 x_1^{(3)}x_2^{(5)}x_3^{(3)}x_4^{(8)}x_5^{(1)}\\
&\quad +
 x_1^{(3)}x_2^{(3)}x_3^{(6)}x_4^{(7)}x_5^{(1)}+
 x_1^{(3)}x_2^{(5)}x_3^{(7)}x_4^{(4)}x_5^{(1)}+
x_1^{(3)}x_2^{(7)}x_3^{(5)}x_4^{(4)}x_5^{(1)} + 
\medskip
 x_1^{(3)}x_2^{(3)}x_3^{(9)}x_4^{(4)}x_5^{(1)}\\
&\quad +
 x_1^{(3)}x_2^{(9)}x_3^{(3)}x_4^{(4)}x_5^{(1)}+
 x_1^{(3)}x_2^{(3)}x_3^{(9)}x_4^{(3)}x_5^{(2)}+
 x_1^{(3)}x_2^{(9)}x_3^{(3)}x_4^{(3)}x_5^{(2)}+
\medskip
 x_1^{(3)}x_2^{(5)}x_3^{(9)}x_4^{(2)}x_5^{(1)}\\
&\quad +
 x_1^{(3)}x_2^{(9)}x_3^{(5)}x_4^{(2)}x_5^{(1)} +
 x_1^{(3)}x_2^{(5)}x_3^{(10)}x_4^{(1)}x_5^{(1)}+
 x_1^{(3)}x_2^{(9)}x_3^{(6)}x_4^{(1)}x_5^{(1)}+
\medskip
 x_1^{(3)}x_2^{(3)}x_3^{(11)}x_4^{(2)}x_5^{(1)}\\
&\quad +
 x_1^{(3)}x_2^{(11)}x_3^{(3)}x_4^{(2)}x_5^{(1)} +
 x_1^{(3)}x_2^{(5)}x_3^{(5)}x_4^{(3)}x_5^{(4)}+
 x_1^{(3)}x_2^{(5)}x_3^{(3)}x_4^{(5)}x_5^{(4)}+
\medskip
 x_1^{(3)}x_2^{(3)}x_3^{(5)}x_4^{(5)}x_5^{(4)}\\
&\quad +
 x_1^{(3)}x_2^{(3)}x_3^{(12)}x_4^{(1)}x_5^{(1)}+
 x_1^{(3)}x_2^{(11)}x_3^{(4)}x_4^{(1)}x_5^{(1)}+
 x_1^{(3)}x_2^{(7)}x_3^{(8)}x_4^{(1)}x_5^{(1)}+
\medskip
 x_1^{(3)}x_2^{(7)}x_3^{(7)}x_4^{(1)}x_5^{(2)}\\
&\quad +  
x_1^{(3)}x_2^{(13)}x_3^{(2)}x_4^{(1)}x_5^{(1)} +
 x_1^{(3)}x_2^{(14)}x_3^{(1)}x_4^{(1)}x_5^{(1)}+
 x_1^{(3)}x_2^{(6)}x_3^{(5)}x_4^{(3)}x_5^{(3)}+
\medskip
 x_1^{(3)}x_2^{(5)}x_3^{(3)}x_4^{(6)}x_5^{(3)}\\
&\quad +
 x_1^{(3)}x_2^{(3)}x_3^{(6)}x_4^{(5)}x_5^{(3)} + 
x_1^{(3)}x_2^{(6)}x_3^{(3)}x_4^{(3)}x_5^{(5)}+
 x_1^{(3)}x_2^{(3)}x_3^{(3)}x_4^{(6)}x_5^{(5)}+
\medskip
 x_1^{(3)}x_2^{(3)}x_3^{(6)}x_4^{(3)}x_5^{(5)}\\
&\quad +
x_1^{(3)}x_2^{(5)}x_3^{(3)}x_4^{(3)}x_5^{(6)}+
 x_1^{(3)}x_2^{(3)}x_3^{(5)}x_4^{(3)}x_5^{(6)}+
 x_1^{(3)}x_2^{(3)}x_3^{(3)}x_4^{(5)}x_5^{(6)}+
\medskip
 x_1^{(3)}x_2^{(3)}x_3^{(3)}x_4^{(3)}x_5^{(8)}\\
&\quad + 
x_1^{(3)}x_2^{(3)}x_3^{(3)}x_4^{(4)}x_5^{(7)} +
 x_1^{(3)}x_2^{(3)}x_3^{(5)}x_4^{(2)}x_5^{(7)}+
 x_1^{(3)}x_2^{(3)}x_3^{(6)}x_4^{(1)}x_5^{(7)}+
\medskip
 x_1^{(3)}x_2^{(3)}x_3^{(3)}x_4^{(9)}x_5^{(2)}\\
&\quad +
 x_1^{(3)}x_2^{(3)}x_3^{(3)}x_4^{(10)}x_5^{(1)} + 
x_1^{(3)}x_2^{(5)}x_3^{(3)}x_4^{(7)}x_5^{(2)}+
 x_1^{(3)}x_2^{(5)}x_3^{(7)}x_4^{(3)}x_5^{(2)}+
\medskip
 x_1^{(3)}x_2^{(7)}x_3^{(5)}x_4^{(3)}x_5^{(2)},\\
\zeta_{10, 1} &= x_1^{(1)}x_2^{(15)}x_3^{(3)}x_4^{(3)}x_5^{(2)}+
 x_1^{(1)}x_2^{(15)}x_3^{(3)}x_4^{(4)}x_5^{(1)}+
 x_1^{(1)}x_2^{(15)}x_3^{(5)}x_4^{(2)}x_5^{(1)}+
 x_1^{(1)}x_2^{(15)}x_3^{(6)}x_4^{(1)}x_5^{(1)}.
\end{array}$$
\end{thm}

The sketch of proof of the case $r = 8,\, s  =1$ will be provided. The proofs of other cases use similar idea. We first claim that $$ QP_{n_{8, 1}}^{\otimes 5}\cong QP_{n_{8, 1}}^{\otimes 5}(4,2,1,1)\bigoplus QP_{n_{8, 1}}^{\otimes 5}(4,2,3)\bigoplus QP_{n_{8, 1}}^{\otimes 5}(4,4,2).$$
Indeed, suppose that $t$ is an admissible monomial of degree $n_{8, 1}$ in $\mathcal A$-module $P^{\otimes 5}.$ Noting that the monomial $t_1^{15}t_2^3t_3t_4\in P_{n_{8, 1}}^{\otimes 5}$ is the minimal spike whose its weight vector is $(4,2,1,1).$ By this and $\deg(t)$ is even, one gets $\omega_1(t) = 4.$ This means that $t = t_it_jt_kt_l\underline{t}^2,$ in which $1\leq i<j<k<l\leq 5,$ and $\underline{t}$ belongs to $\mathcal A$-module $P^{\otimes 5}$ in degree $8.$ Obviously, basing Kameko's criterion on inadmissible monomials, $\underline{t}$ is an admissible monomial. So, using a result in Tin \cite{Tin}, we notice that the weight vector of $t$ belongs to the set $\{(4,2,1,1),\, (4,2,3),\, (4,4,2)\}.$ This follows the above assert. Next, by direct calculations using a result in Sum \cite{Sum1, Sum2} and Kameko's criterion on inadmissible monomials, we may state the following:
$$  \begin{array}{ll}
\medskip
& \dim (QP_{n_{8, 1}}^{\otimes 5})^{0}(4,2,1,1) = 225,\ \ \dim (QP_{n_{8, 1}}^{\otimes 5})^{0}(4,2,3) = 20,\ \ \dim (QP_{n_{8, 1}}^{\otimes 5})^{0}(4,4,2) = 30,\\
& \dim (QP_{n_{8, 1}}^{\otimes 5})^{> 0}(4,2,1,1) = 225,\ \ \dim (QP_{n_{8, 1}}^{\otimes 5})^{> 0}(4,2,3) = 50,\ \ \dim (QP_{n_{8, 1}}^{\otimes 5})^{> 0}(4,4,2) = 91.
\end{array}$$
If $\omega$ is a weight vector of degree $n_{8, 1},$ then $QP_{n_{8, 1}}^{\otimes 5}(\omega) \cong (QP_{n_{8, 1}}^{\otimes 5})^{0}(\omega) \bigoplus (QP_{n_{8, 1}}^{\otimes 5})^{> 0}(\omega).$ So, from these data, we get:
$$ \dim QP_{n_{8, 1}}^{\otimes 5}(4,2,1,1)  = 450,\ \ \dim QP_{n_{8, 1}}^{\otimes 5}(4,2,3) = 70,\ \ \dim QP_{n_{8, 1}}^{\otimes 5}(4,4,2) = 121.$$  
Then, it is not difficult to verify that $$\dim [(QP_{n_{8, 1}}^{\otimes 5}(4,2,1,1)\bigoplus QP_{n_{8, 1}}^{\otimes 5}(4,4,2)]^{GL_5} = 1,\ \mbox{ and }\ \dim [QP_{n_{8, 1}}^{\otimes 5}(4,2,3)]^{GL_5} = 0,$$
from which $\mathbb Z/2 \otimes_{GL_5} {\rm Ann}_{\overline{\mathcal A}}[P_{n_{8, 1}}^{\otimes 5}]^{*}\cong [QP_{n_{8, 1}}^{\otimes 5}]^{GL_5}$ has dimension 1.  

Now, using the $E_1$-level of $Tr_5^{\mathcal A},$ one concludes that the elements
$$ \begin{array}{ll}
\psi_5(\zeta_{8, 1} ) &= \lambda_3^{4}\lambda_8 + (\lambda_3^{2}\lambda_5^{2} + \lambda_3^{3}\lambda_7)\lambda_4 + \lambda_3\lambda_7\lambda_5\lambda_3\lambda_2 + \lambda_3^{3}\lambda_5\lambda_6) \\
&\quad+ \delta(\lambda_3^{2}\lambda_5\lambda_{10} + \lambda_3^{2}\lambda_{12}\lambda_3 + \lambda_3\lambda_4\lambda_7^{2} + \lambda_3\lambda_0\lambda_{11}\lambda_7),\\
\psi_5(\zeta_{10, 1}) &= \lambda_1\lambda_{15}\lambda_3^{2}\lambda_2
\end{array}$$
are cycles in $\Lambda,$ and are representative of the elements $h_2e_0 = h_0g_1\in {\rm Ext}_{\mathcal {A}}^{5, 5+n_{8, 1}}(\mathbb Z/2, \mathbb Z/2)$ and $h_1h_4c_0 = h_3e_0\in {\rm Ext}_{\mathcal {A}}^{5, 5+n_{10, 1}}(\mathbb Z/2, \mathbb Z/2),$ respectively. On the other side, following Lin \cite{Lin}, the elements $h_2e_0,$ and $h_1h_4c_0$ are non-zero in the fifth cohomology groups of Steenrod algebra. At the same time, ${\rm Ext}_{\mathcal A}^{5, 5+n_{10, s}}(\mathbb Z/2, \mathbb Z/2) =\langle h_0h_s^{2}h_{s+2}h_{s+3} \rangle = 0,$ for all $s\geq 3,$ and ${\rm Ext}_{\mathcal A}^{5, 5+n_{11, s}}(\mathbb Z/2, \mathbb Z/2) = \langle h_0h_sh_{s+1}h_{s+2}h_{s+3} \rangle = 0,$ for all $s\geq 1,$
So, the following is immediate.

\begin{corl}\label{hqkqpl}
The transfer homomorphism $$ Tr_5^{\mathcal A}: \mathbb Z/2 \otimes_{GL_5} {\rm Ann}_{\overline{\mathcal A}}[P_{n_{r, s}}^{\otimes 5}]^{*}\to {\rm Ext}_{\mathcal A}^{5, 5+n_{r, s}}(\mathbb Z/2, \mathbb Z/2)$$
is an isomorphism, where $r,$ and $s$ given as in Theorem \ref{dlkqpl}.
\end{corl}


\subsection{The behavior of $Tr_6^{\mathcal A}$ in some internal degrees}\label{pl2}

Firstly, we survey the behavior of the sixth algebraic transfer in internal degrees $n\leq 25.$ For $n\leq 5,$ it is easy to see that $QP_{n}^{\otimes 6} = (QP_{n}^{\otimes 6})^{0}.$ From a formula in \cite{MKR}, we have
$$ \dim QP_{n}^{\otimes 6}=\dim (QP_{n}^{\otimes 6})^0 = \sum_{1\leq  k\leq 5}\binom{6}{k}\dim (QP_{n}^{\otimes\, k})^{>0}.$$
On the other hand, by Peterson \cite{Peterson}, Kameko \cite{Kameko}, Sum \cite{Sum2, Sum}, for each $1\leq k\leq 5,$ we have the following:\\
For $n = 1,$ then $\dim (QP_{1}^{\otimes\, k})^{>0} = 1$ if $k = 1,$ and $\dim (QP_{1}^{\otimes\, k})^{>0} = 0$ otherwise.\\
For $n = 2,$ then $\dim (QP_{2}^{\otimes\, k})^{>0} = 1$ if $k = 2,$ and $\dim (QP_{2}^{\otimes\, k})^{>0} = 0$ otherwise.\\
For $n = 3,$ then $\dim (QP_{3}^{\otimes\, k})^{>0} = 1$ if $k \leq 3$ and $\dim (QP_{3}^{\otimes\, k})^{>0} = 0$ otherwise.\\
For $n = 4,$ then 
$$ \dim (QP_{4}^{\otimes\, k})^{>0} =  \left\{\begin{array}{ll}
0&\mbox{if $k = 1, 5$},\\
2&\mbox{if $k = 2, 3$},\\
1&\mbox{if $k = 4$}.
\end{array}\right.$$
For $n = 5,$ then 
$$ \dim (QP_{5}^{\otimes\, k})^{>0} =  \left\{\begin{array}{ll}
0&\mbox{if $k = 1, 2$},\\
3&\mbox{if $k = 3, 4$},\\
1&\mbox{if $k = 5$}.
\end{array}\right.$$
Therefore, we obtain

\centerline{\begin{tabular}{c|ccccc}
$n$  &$1$ & $2$ & $3$ & $4$ & $5$ \cr
\hline
\ $\dim QP_{n}^{\otimes 6}=\dim (QP_{n}^{\otimes 6})^0 = \sum_{1\leq  k\leq 5}\binom{6}{k}\dim (QP_{n}^{\otimes\, k})^{>0}$ & $6$ & $15$ & $41$ &$85$ &$111$ \cr
\end{tabular}}

\medskip

and the sets
$$ \begin{array}{ll}
\medskip
{\rm [}\mathscr C_{1}^{\otimes 6}{\rm ]}&= \big\{[t_j]:\, 1\leq j\leq 6\big\},\ \ {\rm [}\mathscr C_{2}^{\otimes 6}{\rm ]}= \big\{[t_it_j]:\, 1\leq i, j\leq 6,\, i\neq j\big\},\\
\medskip
{\rm [}\mathscr C_{3}^{\otimes 6}{\rm ]}&= \big\{[t_i^{3}],\, [t_jt_kt_l],\, [t_rt^{2}_s]:\, 1\leq i, j, k, l, r, s\leq 6,\, i\neq k, j\neq l, k\neq l,  r< s\big\},\\
{\rm [}\mathscr C_{4}^{\otimes 6}{\rm ]}&= \big\{[t_it_j^{3}],\, [t_kt_lt_mt_p],\, [t_qt_rt^{2}_s]:\, 1\leq i, j, k, l, m, p, q, r, s\leq 6,\, i\neq j\\
\medskip
&\hspace{1cm} k\neq l, k\neq m, k\neq p, l\neq m, l\neq p, m\neq p, q < s,\, q\neq r, r\neq s\big\},\\
{\rm [}\mathscr C_{5}^{\otimes 6}{\rm ]}&= \big\{[t_it_jt_k^{3}],\, [t_lt_mt_rt_s^{2}]:\, 1\leq i, j, k, l, m, r, s\leq 6,\, i\neq j, i\neq k, j\neq k,\, l< s,\\
&\hspace{4cm} l\neq m, l\neq r,  m\neq r, m\neq s, r\neq s\big\}
\end{array}$$
are the monomial bases of $QP_{1}^{\otimes 6},$ $QP_{2}^{\otimes 6},$ $QP_{3}^{\otimes 6},$ $QP_{4}^{\otimes 6},$ and $QP_{5}^{\otimes 6}$  respectively. Basing these results, a direct computation shows that $[QP_{n}^{\otimes 6}]^{GL_6} = 0.$ Indeed, for simplicity, we would like to describe the invariant $[QP_{2}^{\otimes 6}]^{GL_6}$ in detail. The others can be proved by the similar computations. Notice that  $QP_{2}^{\otimes 6} = \langle [\mathscr C_{2}^{\otimes 6}]\rangle = \langle \{[{\rm adm}_j]|\, 1\leq j\leq 15\} \rangle,$ where the admissible monomials ${\rm adm}_j$ are determined as follows:

\begin{center}
\begin{tabular}{lllll}
${\rm adm}_{1}=t_5t_6$, & ${\rm adm}_{2}=t_4t_6$, & ${\rm adm}_{3}=t_4t_5$, & ${\rm adm}_{4}=t_3t_6$, & ${\rm adm}_{5}=t_3t_5$, \\
${\rm adm}_{6}=t_3t_4$, & ${\rm adm}_{7}=t_2t_6$, & ${\rm adm}_{8}=t_2t_5$, & ${\rm adm}_{9}=t_2t_4$, & ${\rm adm}_{10}=t_2t_3$, \\
${\rm adm}_{11}=t_1t_6$, & ${\rm adm}_{12}=t_1t_5$, & ${\rm adm}_{13}=t_1t_4$, & ${\rm adm}_{14}=t_1t_3$, & ${\rm adm}_{15}=t_1t_2$.
\end{tabular}
\end{center}
Then, $[QP_{2}^{\otimes 6}]^{S_6} = \langle [\sum_{1\leq j\leq 15}{\rm adm}_j]\rangle.$ This can be determined as follows: Suppose that $[f]\in [QP_{2}^{\otimes 6}]^{S_6},$ then $f\equiv \sum_{1\leq j\leq 15}\gamma_j{\rm adm}_j,$ where $\gamma_j\in \mathbb Z/2$ for every $j.$ Using the homomorphisms $\theta_i: P^{\otimes 6}\longrightarrow P^{\otimes 6}$ and the relations $\theta_i(f) + f\equiv 0$ for $1\leq i\leq 5,$ one gets:
$$ \begin{array}{ll}
\theta_1(f) + f &\equiv (\gamma_7 + \gamma_{11})({\rm adm}_{7}+{\rm adm}_{11}) + (\gamma_8 + \gamma_{12})({\rm adm}_{8}+{\rm adm}_{12})\\
\medskip
&\quad +(\gamma_9 + \gamma_{13})({\rm adm}_{9}+{\rm adm}_{13}) + (\gamma_{10} + \gamma_{14})({\rm adm}_{10}+{\rm adm}_{14}) \equiv 0,\\
\theta_2(f) + f &\equiv (\gamma_4 + \gamma_{7})({\rm adm}_{4}+{\rm adm}_{7}) + (\gamma_5 + \gamma_{8})({\rm adm}_{5}+{\rm adm}_{8})\\
\medskip
&\quad +(\gamma_6 + \gamma_{9})({\rm adm}_{6}+{\rm adm}_{9}) + (\gamma_{14} + \gamma_{15})({\rm adm}_{14}+{\rm adm}_{15}) \equiv 0,\\
\theta_3(f) + f &\equiv (\gamma_2 + \gamma_{4})({\rm adm}_{2}+{\rm adm}_{4}) + (\gamma_3 + \gamma_{5})({\rm adm}_{3}+{\rm adm}_{5})\\
\medskip
&\quad +(\gamma_9 + \gamma_{10})({\rm adm}_{9}+{\rm adm}_{10}) + (\gamma_{13} + \gamma_{14})({\rm adm}_{13}+{\rm adm}_{14}) \equiv 0,\\
\theta_4(f) + f &\equiv (\gamma_1 + \gamma_{2})({\rm adm}_{1}+{\rm adm}_{2}) + (\gamma_5 + \gamma_{6})({\rm adm}_{5}+{\rm adm}_{6})\\
\medskip
&\quad +(\gamma_8 + \gamma_{9})({\rm adm}_{8}+{\rm adm}_{9}) + (\gamma_{12} + \gamma_{13})({\rm adm}_{12}+{\rm adm}_{13}) \equiv 0,\\
\theta_5(f) + f &\equiv (\gamma_2 + \gamma_{3})({\rm adm}_{2}+{\rm adm}_{3}) + (\gamma_4 + \gamma_{5})({\rm adm}_{4}+{\rm adm}_{5})\\
&\quad +(\gamma_7 + \gamma_{8})({\rm adm}_{7}+{\rm adm}_{8}) + (\gamma_{11} + \gamma_{12})({\rm adm}_{11}+{\rm adm}_{12}) \equiv 0.
\end{array}$$
These equalities show that $\gamma_1 = \gamma_2 = \cdots = \gamma_{15}.$ Now, since $S_6\subset GL_6,$ if $[g]\in [QP_{2}^{\otimes 6}]^{GL_6},$ then $g\equiv \beta\sum_{1\leq j\leq 15}{\rm adm}_j.$ By the relation $\theta_6(g)\equiv g,$ one gets $\beta({\rm adm}_7 + {\rm adm}_8 + {\rm adm}_9 + {\rm adm}_{10}) \equiv 0,$ which implies that $\beta = 0.$ Therefore $[QP_{2}^{\otimes 6}]^{GL_6}$ is zero.

\medskip

For $6\leq n\leq 12,$ Mothebe, Kaelo, and Ramatebele showed in \cite{MKR} that

\centerline{\begin{tabular}{c|ccccccc}
$n$  &$6$ & $7$ & $8$ & $9$ & $10$ & $11$ & $12$\cr
\hline
\ $\dim QP_{n}^{\otimes 6}$ & $190$ & $301$ & $489$ &$630$ &$945$ & $1205$ & $1001$ \cr
\end{tabular}}

\medskip

The monomial bases of $QP_{n}^{\otimes 6}$ are also given in the same paper \cite{MKR}. Using these results, we find that the spaces of  $GL_6$-invarians $[QP_{n}^{\otimes 6}]^{GL_6}$ are also trivial, for $6\leq n\leq 12.$ We shall prove the cases $n = 8,10, 11, 12$ in detail, and leave the rest to the reader. We should note that since the Kameko homomorphism $(\widetilde {Sq^0_*})_{n}: QP^{\otimes 6}_{n} \longrightarrow QP^{\otimes 6}_{\frac{n-6}{2}}$ is an epimorphism for $n\in \{8, 10, 12\},$ we have an estimate $$\dim [QP_{n}^{\otimes 6}]^{GL_6}\leq \dim [{\rm Ker}((\widetilde {Sq^0_*})_{n})]^{GL_6} + \dim [QP^{\otimes 6}_{\frac{n-6}{2}}]^{GL_6}.$$
As shown above, $\dim [QP^{\otimes 6}_{\frac{n-6}{2}}]^{GL_6} = 0$ for $n = 8,\, 10,\, 12.$ So, we need only to determine $[{\rm Ker}((\widetilde {Sq^0_*})_{n})]^{GL_6}.$ To this goal, we use the following homomorphisms: For any $1\leq l\leq 6,\, l\in \mathbb Z,$ we consider the maps $\mathsf{q}_{l}: P^{\otimes 5}\longrightarrow P^{\otimes 6}$ of algebras, which are determined by 
$$ \mathsf{q}_{l}(t_j) = \left\{ \begin{array}{ll}
{t_j}&\text{if }\;1\leq j \leq l-1, \\
t_{j+1}& \text{if}\; l\leq j \leq 5.
\end{array} \right.$$
The \textit{up Kameko map} $\varphi: P_{\frac{n-6}{2}}^{\otimes 6}\longrightarrow P_{n}^{\otimes 6}$ is an injective linear map defined on monomials by $\varphi(t) = t_1t_2t_3t_4t_5t_6t^2.$ Then, from a result of Mothebe, and Uys \cite{Mothebe}, we deduce that for each $0\leq d\leq 3,\, d\in\mathbb Z,$ if $t\in \mathscr {C}^{\otimes 5}_{n+1-2^{d}},$ then $t_l^{2^{d}-1}\mathsf{q}_{l}(t)\in \mathscr {C}^{\otimes 6}_{n}$ for any $1\leq l\leq 6.$ We put
$$ \mathscr C(d, n):= \big\{t_l^{2^{d}-1}\mathsf{q}_{l}(t)|\, t\in \mathscr {C}^{\otimes 5}_{n+1-2^{d}},\, 1\leq l\leq 6\big\},\ 0\leq d\leq 3.$$

Now, for $n = 8,$ we see that ${\rm Ker}((\widetilde {Sq^0_*})_{8})$ has dimension $483$, since $\dim QP^{\otimes 6}_1 = 6,$ and $\dim QP^{\otimes 6}_8 = 489.$ On the other hand, we have $|\mathscr {C}^{\otimes 5}_{9-2^{0}}| = 174$ (see \cite{Tin}), $|\mathscr {C}^{\otimes 5}_{9-2^{1}}| = 110$ (see \cite{Phuc3})$|\mathscr {C}^{\otimes 5}_{9-2^{2}}| = 46$ (see \cite{Sum}), and $|\mathscr {C}^{\otimes 5}_{9-2^{3}}| = 5$ (see \cite{Peterson}). These results indicate that the set
$$ B(8):=\big(\bigcup_{0\leq d\leq 3}\mathscr C(d, 8)\big)\setminus \varphi(\mathscr C_{1}^{\otimes 6}),$$
where $|\mathscr C_{1}^{\otimes 6}| = 6,$ has all 930 admissible monomials. Furthermore, for any $z\in B(8),$ we have $[z]\neq [0],$ and $(\widetilde {Sq^0_*})_{8}([z]) = [0].$ Therefore, a basis of ${\rm Ker}((\widetilde {Sq^0_*})_{8})$ is a set consisting of all the equivalence classes of the admissible monomials in $B(8).$ Furthermore, $${\rm Ker}((\widetilde {Sq^0_*})_{8})\cong QP_8^{\otimes 6}(2, 1, 1) \bigoplus QP_8^{\otimes 6}(2, 3)\bigoplus QP_8^{\otimes 6}(4, 2).$$ Indeed, assume that $[t]\in {\rm Ker}((\widetilde {Sq^0_*})_{8}),$ and $t\in P_8^{\otimes 6}$ is admissible. We observe that $\deg(t)$ is even, and that $t_1^{7}t_2$ is the minimal spike monomial in $P^{\otimes 6}_8,$ and $\omega(t_1^{7}t_2) = (2,1,1);$ so by Singer's criterion on hit polynomials, either $\omega_1(t) = 2$ or $\omega_1(t) = 4$ or $\omega_1(t) = 6.$ Since $[t]\neq [0],$ and $t$ is admissible, by Kameko's criterion on inadmissible monomials, $\omega_1(t)\neq 6.$ This means that either $t = t_it_j\underline{t}^{2},\, 1\leq i < j\leq 6,\, \underline{t}\in P^{\otimes 6}_3,$ or $t = t_it_jt_kt_l\underline{t}^{2},\, 1\leq i < j<k<l\leq 6,\, \underline{t}\in P^{\otimes 6}_2.$ Since $t$ is admissible, $\underline{t}$ is, too. As computed above, $\dim QP_{2}^{\otimes 6} = 15,$ and $\dim QP_{3}^{\otimes 6} = 41.$ Moreover, it is easy to see that $\omega(\underline{t}) = (2, 0)$ with $\underline{t}\in P^{\otimes 6}_2,$ and that $\omega(\underline{t}) \in \{(1, 1),\, (3, 0)\}$ with $\underline{t}\in P^{\otimes 6}_3.$ Therefore, the weight vector $\omega(t)\in \{(2,1,1),\, (2, 3),\, (4, 2)\},$ which leads to the above statement. Since $\dim {\rm Ker}((\widetilde {Sq^0_*})_{8}) = 483,$ we obtain that $$ \dim QP_8^{\otimes 6}(2, 1, 1) = 210,\ \dim QP_8^{\otimes 6}(2, 3) = 84, \ \dim QP_8^{\otimes 6}(4, 2) = 189.$$
Note that $QP_8^{\otimes 6}(2, 1, 1) = (QP_8^{\otimes 6})^{0}(2, 1, 1),$ and $QP_8^{\otimes 6}(2, 3) = (QP_8^{\otimes 6})^{0}(2, 3 ).$ By this, it is not difficult to verify that the invariants $[QP_8^{\otimes 6}(2, 1, 1)]^{GL_6},$ $[QP_8^{\otimes 6}(2, 3)]^{GL_6},$ and $[QP_8^{\otimes 6}(4, 2)]^{GL_6}$ are zero. Indeed, the calculations of the space $[QP_8^{\otimes 6}(2, 3)]^{GL_6}$ will be presented. The others can be found by using similar ideas. We have that $QP_8^{\otimes 6}(\omega) = \langle \{[{\rm adm}_j]_{\omega}:\, 1\leq j\leq 84\} \rangle,$ where $\omega:= (2,3),$ and the admissible monomials ${\rm adm}_j$ are given as follows:

\begin{center}
\begin{tabular}{llll}
${\rm adm}_{1}=t_3t_4^{2}t_5^{2}t_6^{3}$, & ${\rm adm}_{2}=t_3t_4^{2}t_5^{3}t_6^{2}$, & ${\rm adm}_{3}=t_3t_4^{3}t_5^{2}t_6^{2}$, & ${\rm adm}_{4}=t_3^{3}t_4t_5^{2}t_6^{2}$, \\
${\rm adm}_{5}=t_2t_4^{2}t_5^{2}t_6^{3}$, & ${\rm adm}_{6}=t_2t_4^{2}t_5^{3}t_6^{2}$, & ${\rm adm}_{7}=t_2t_4^{3}t_5^{2}t_6^{2}$, & ${\rm adm}_{8}=t_2t_3^{2}t_5^{2}t_6^{3}$, \\
${\rm adm}_{9}=t_2t_3^{2}t_5^{3}t_6^{2}$, & ${\rm adm}_{10}=t_2t_3^{2}t_4^{2}t_6^{3}$, & ${\rm adm}_{11}=t_2t_3^{2}t_4^{2}t_5^{3}$, & ${\rm adm}_{12}=t_2t_3^{2}t_4^{3}t_6^{2}$, \\
${\rm adm}_{13}=t_2t_3^{2}t_4^{3}t_5^{2}$, & ${\rm adm}_{14}=t_2t_3^{3}t_5^{2}t_6^{2}$, & ${\rm adm}_{15}=t_2t_3^{3}t_4^{2}t_6^{2}$, & ${\rm adm}_{16}=t_2t_3^{3}t_4^{2}t_5^{2}$, \\
${\rm adm}_{17}=t_2^{3}t_4t_5^{2}t_6^{2}$, & ${\rm adm}_{18}=t_2^{3}t_3t_5^{2}t_6^{2}$, & ${\rm adm}_{19}=t_2^{3}t_3t_4^{2}t_6^{2}$, & ${\rm adm}_{20}=t_2^{3}t_3t_4^{2}t_5^{2}$, \\
${\rm adm}_{21}=t_1t_4^{2}t_5^{2}t_6^{3}$, & ${\rm adm}_{22}=t_1t_4^{2}t_5^{3}t_6^{2}$, & ${\rm adm}_{23}=t_1t_4^{3}t_5^{2}t_6^{2}$, & ${\rm adm}_{24}=t_1t_3^{2}t_5^{2}t_6^{3}$, \\
${\rm adm}_{25}=t_1t_3^{2}t_5^{3}t_6^{2}$, & ${\rm adm}_{26}=t_1t_3^{2}t_4^{2}t_6^{3}$, & ${\rm adm}_{27}=t_1t_3^{2}t_4^{2}t_5^{3}$, & ${\rm adm}_{28}=t_1t_3^{2}t_4^{3}t_6^{2}$, \\
${\rm adm}_{29}=t_1t_3^{2}t_4^{3}t_5^{2}$, & ${\rm adm}_{30}=t_1t_3^{3}t_5^{2}t_6^{2}$, & ${\rm adm}_{31}=t_1t_3^{3}t_4^{2}t_6^{2}$, & ${\rm adm}_{32}=t_1t_3^{3}t_4^{2}t_5^{2}$, \\
${\rm adm}_{33}=t_1t_2^{2}t_5^{2}t_6^{3}$, & ${\rm adm}_{34}=t_1t_2^{2}t_5^{3}t_6^{2}$, & ${\rm adm}_{35}=t_1t_2^{2}t_4^{2}t_6^{3}$, & ${\rm adm}_{36}=t_1t_2^{2}t_4^{2}t_5^{3}$, \\
${\rm adm}_{45}=t_1t_2^{3}t_5^{2}t_6^{2}$, & ${\rm adm}_{46}=t_1t_2^{3}t_4^{2}t_6^{2}$, & ${\rm adm}_{47}=t_1t_2^{3}t_4^{2}t_5^{2}$, & ${\rm adm}_{48}=t_1t_2^{3}t_3^{2}t_6^{2}$, \\
${\rm adm}_{49}=t_1t_2^{3}t_3^{2}t_5^{2}$, & ${\rm adm}_{50}=t_1t_2^{3}t_3^{2}t_4^{2}$, & ${\rm adm}_{51}=t_1^{3}t_4t_5^{2}t_6^{2}$, & ${\rm adm}_{52}=t_1^{3}t_3t_5^{2}t_6^{2}$, \\
${\rm adm}_{53}=t_1^{3}t_3t_4^{2}t_6^{2}$, & ${\rm adm}_{54}=t_1^{3}t_3t_4^{2}t_5^{2}$, & ${\rm adm}_{55}=t_1^{3}t_2t_5^{2}t_6^{2}$, & ${\rm adm}_{56}=t_1^{3}t_2t_4^{2}t_6^{2}$, \\
${\rm adm}_{57}=t_1^{3}t_2t_4^{2}t_5^{2}$, & ${\rm adm}_{58}=t_1^{3}t_2t_3^{2}t_6^{2}$, & ${\rm adm}_{59}=t_1^{3}t_2t_3^{2}t_5^{2}$, & ${\rm adm}_{60}=t_1^{3}t_2t_3^{2}t_4^{2}$, \\
${\rm adm}_{61}=t_2t_3t_4^{2}t_5^{2}t_6^{2}$, & ${\rm adm}_{62}=t_2t_3^{2}t_4t_5^{2}t_6^{2}$, & ${\rm adm}_{63}=t_2t_3^{2}t_4^{2}t_5t_6^{2}$, & ${\rm adm}_{64}=t_2t_3^{2}t_4^{2}t_5^{2}t_6$, \\
${\rm adm}_{65}=t_1t_3t_4^{2}t_5^{2}t_6^{2}$, & ${\rm adm}_{66}=t_1t_3^{2}t_4t_5^{2}t_6^{2}$, & ${\rm adm}_{67}=t_1t_3^{2}t_4^{2}t_5t_6^{2}$, & ${\rm adm}_{68}=t_1t_3^{2}t_4^{2}t_5^{2}t_6$, \\
${\rm adm}_{69}=t_1t_2t_4^{2}t_5^{2}t_6^{2}$, & ${\rm adm}_{70}=t_1t_2t_3^{2}t_5^{2}t_6^{2}$, & ${\rm adm}_{71}=t_1t_2t_3^{2}t_4^{2}t_6^{2}$, & ${\rm adm}_{72}=t_1t_2t_3^{2}t_4^{2}t_5^{2}$, \\
${\rm adm}_{73}=t_1t_2^{2}t_4t_5^{2}t_6^{2}$, & ${\rm adm}_{74}=t_1t_2^{2}t_4^{2}t_5t_6^{2}$, & ${\rm adm}_{75}=t_1t_2^{2}t_4^{2}t_5^{2}t_6$, & ${\rm adm}_{76}=t_1t_2^{2}t_3t_5^{2}t_6^{2}$, \\
${\rm adm}_{77}=t_1t_2^{2}t_3t_4^{2}t_6^{2}$, & ${\rm adm}_{78}=t_1t_2^{2}t_3t_4^{2}t_5^{2}$, & ${\rm adm}_{79}=t_1t_2^{2}t_3^{2}t_5t_6^{2}$, & ${\rm adm}_{80}=t_1t_2^{2}t_3^{2}t_5^{2}t_6$, \\
${\rm adm}_{81}=t_1t_2^{2}t_3^{2}t_4t_6^{2}$, & ${\rm adm}_{82}=t_1t_2^{2}t_3^{2}t_4t_5^{2}$, & ${\rm adm}_{83}=t_1t_2^{2}t_3^{2}t_4^{2}t_6$, & ${\rm adm}_{84}=t_1t_2^{2}t_3^{2}t_4^{2}t_5$.
\end{tabular}%
\end{center}
  
Then, we have a direct summand decomposition of $S_6$-modules:
$$ QP_8^{\otimes 6}(\omega)  = S_6({\rm adm}_{1}) \bigoplus S_6({\rm adm}_{61}),$$
where $S_6({\rm adm}_{1})  = \langle \{[{\rm adm}_{j}]_{\omega}:\, 1\leq j\leq 60\} \rangle,$ and $S_6({\rm adm}_{61})  = \langle \{[{\rm adm}_{j}]_{\omega}:\, 61\leq j\leq 84\} \rangle.$ Suppose that $[f]_{\omega}\in [S_6({\rm adm}_{1})]^{S_6}$ and $[g]_{\omega}\in [S_6({\rm adm}_{61})]^{S_6}.$ One has $$f\equiv_{\omega}\sum_{1\leq j\leq 60} \gamma_j{\rm adm}_{j},\ \mbox{ and }\ g\equiv_{\omega}\sum_{61\leq j\leq 84} \beta_j{\rm adm}_{j},\ \ \gamma_j,\, \beta_j\in \mathbb Z/2,\ \mbox{for all $j$}.$$
Using the relations $(\theta_i(f) + f)\equiv_{\omega} 0,$ and $(\theta_i(g) + g)\equiv_{\omega} 0$ for all $i,\, 1\leq i\leq 5,$ one gets $\gamma_1 = \gamma_j$ for any $j,\, 2\leq j\leq 60,$ and $\beta_{j} = 0$ for all $j,\, 61\leq j\leq 84.$ This means that $$[QP_8^{\otimes 6}(\omega)]^{S_6} =\langle [\sum_{1\leq j\leq 60}{\rm adm}_{j}]_{\omega}\rangle.$$
Now, if $[h]_{\omega}\in [QP_8^{\otimes 6}(\omega]^{GL_6},$ then $h\equiv_{\omega} \sum_{1\leq j\leq 60}\zeta{\rm adm}_{j},$ where $\zeta\in\mathbb Z/2.$ Applying the relation $(\theta_6(h) + h)\equiv_{\omega} 0,$ one gets $\zeta{\rm adm}_5 + \mbox{other terms} \equiv_{\omega} 0,$ which implies that $\zeta = 0.$

Thus, $[QP_{8}^{\otimes 6}]^{GL_6}$ is zero, since $\dim [QP_{8}^{\otimes 6}]^{GL_6}\leq \dim [{\rm Ker}((\widetilde {Sq^0_*})_{8})]^{GL_6}.$ 

\medskip

Next, we consider the cases $n = 10.$  Since $\dim QP_{2}^{\otimes 6} = 15,$ and $\dim QP_{10}^{\otimes 6} = 945,$ one has that $\dim {\rm Ker}((\widetilde {Sq^0_*})_{10}) = 945- 15 = 930.$ A basis of ${\rm Ker}((\widetilde {Sq^0_*})_{10}) $ is determined as follows: We have $|\mathscr {C}^{\otimes 5}_{11-2^{0}}| = 280$ (see \cite{Tin}), $|\mathscr {C}^{\otimes 5}_{11-2^{1}}| = 191$ (see \cite{Tin}), $|\mathscr {C}^{\otimes 5}_{11-2^{2}}| = 110$ (see \cite{Phuc3}), and $|\mathscr {C}^{\otimes 5}_{11-2^{3}}| = 25$ (see \cite{Tin}). These results show that the set
$$ B(10):=\big(\bigcup_{0\leq d\leq 3}\mathscr C(d, 10)\big)\setminus \varphi(\mathscr C_{2}^{\otimes 6}),$$
where $|\mathscr C_{2}^{\otimes 6}| = 15,$ has all 930 admissible monomials. Moreover, for each $u\in B(10),$ one sees that $[u]\neq [0],$ and $(\widetilde {Sq^0_*})_{10}([u]) = [0].$ So, a basis of ${\rm Ker}((\widetilde {Sq^0_*})_{10})$ is a set consisting of all the equivalence classes of the admissible monomials in $B(10).$ Moreover, by similar arguments as in the case $n=8,$ we have an isomorphism
$${\rm Ker}((\widetilde {Sq^0_*})_{10})\cong QP_{10}^{\otimes 6}(2, 2, 1) \bigoplus QP_{10}^{\otimes 6}(2, 4)\bigoplus QP_{10}^{\otimes 6}(4, 1,1)\bigoplus QP_{10}^{\otimes 6}(4, 3),$$
where 
$$ \begin{array}{ll}
\medskip
&\dim QP_{10}^{\otimes 6}(2, 2, 1) = \dim (QP_{10}^{\otimes 6})^{0}(2, 2, 1) = 400,\ \dim QP_{10}^{\otimes 6}(2, 4) = 34,\\
& \dim QP_{10}^{\otimes 6}(4, 1, 1) = 280,\ \dim QP_{10}^{\otimes 6}(4, 3) = 216.
\end{array}$$
Using these results, a direct computation shows that the invariants $$[QP_{10}^{\otimes 6}(2, 2, 1)]^{GL_5},\ [QP_{10}^{\otimes 6}(2, 4)]^{GL_5},\ [QP_{10}^{\otimes 6}(4, 1, 1)]^{GL_5},\ [QP_{10}^{\otimes 6}(4, 3)]^{GL_5}$$ are zero. For instance, let us consider the space $[QP_{10}^{\otimes 6}(\omega)]^{GL_6}$ with $\omega:= (2,4).$ We have seen that the cohit space $QP_{10}^{\otimes 6}(\omega)$ has a basis consisting of all the classes represented by the following admissible monomials:

\begin{center}
\begin{tabular}{llrr}
${\rm adm}_{1}=t_2t_3^{2}t_4^{2}t_5^{2}t_6^{3}$, & ${\rm adm}_{2}=t_2t_3^{2}t_4^{2}t_5^{3}t_6^{2}$, & \multicolumn{1}{l}{${\rm adm}_{3}=t_2t_3^{2}t_4^{3}t_5^{2}t_6^{2}$,} & \multicolumn{1}{l}{${\rm adm}_{4}=t_2t_3^{3}t_4^{2}t_5^{2}t_6^{2}$,} \\
${\rm adm}_{5}=t_2^{3}t_3t_4^{2}t_5^{2}t_6^{2}$, & ${\rm adm}_{6}=t_1t_3^{2}t_4^{2}t_5^{2}t_6^{3}$, & \multicolumn{1}{l}{${\rm adm}_{7}=t_1t_3^{2}t_4^{2}t_5^{3}t_6^{2}$,} & \multicolumn{1}{l}{${\rm adm}_{8}=t_1t_3^{2}t_4^{3}t_5^{2}t_6^{2}$,} \\
${\rm adm}_{9}=t_1t_3^{3}t_4^{2}t_5^{2}t_6^{2}$, & ${\rm adm}_{10}=t_1t_2^{2}t_4^{2}t_5^{2}t_6^{3}$, & \multicolumn{1}{l}{${\rm adm}_{11}=t_1t_2^{2}t_4^{2}t_5^{3}t_6^{2}$,} & \multicolumn{1}{l}{${\rm adm}_{12}=t_1t_2^{2}t_4^{3}t_5^{2}t_6^{2}$,} \\
${\rm adm}_{13}=t_1t_2^{2}t_3^{2}t_5^{2}t_6^{3}$, & ${\rm adm}_{14}=t_1t_2^{2}t_3^{2}t_5^{3}t_6^{2}$, & \multicolumn{1}{l}{${\rm adm}_{15}=t_1t_2^{2}t_3^{2}t_4^{2}t_6^{3}$,} & \multicolumn{1}{l}{${\rm adm}_{16}=t_1t_2^{2}t_3^{2}t_4^{2}t_5^{3}$,} \\
${\rm adm}_{17}=t_1t_2^{2}t_3^{2}t_4^{3}t_6^{2}$, & ${\rm adm}_{18}=t_1t_2^{2}t_3^{2}t_4^{3}t_5^{2}$, & \multicolumn{1}{l}{${\rm adm}_{19}=t_1t_2^{2}t_3^{3}t_5^{2}t_6^{2}$,} & \multicolumn{1}{l}{${\rm adm}_{20}=t_1t_2^{2}t_3^{3}t_4^{2}t_6^{2}$,} \\
${\rm adm}_{21}=t_1t_2^{2}t_3^{3}t_4^{2}t_5^{2}$, & ${\rm adm}_{22}=t_1t_2^{3}t_4^{2}t_5^{2}t_6^{2}$, & \multicolumn{1}{l}{${\rm adm}_{23}=t_1t_2^{3}t_3^{2}t_5^{2}t_6^{2}$,} & \multicolumn{1}{l}{${\rm adm}_{24}=t_1t_2^{3}t_3^{2}t_4^{2}t_6^{2}$,} \\
${\rm adm}_{25}=t_1t_2^{3}t_3^{2}t_4^{2}t_5^{2}$, & ${\rm adm}_{26}=t_1^{3}t_3t_4^{2}t_5^{2}t_6^{2}$, & \multicolumn{1}{l}{${\rm adm}_{27}=t_1^{3}t_2t_4^{2}t_5^{2}t_6^{2}$,} & \multicolumn{1}{l}{${\rm adm}_{28}=t_1^{3}t_2t_3^{2}t_5^{2}t_6^{2}$,} \\
${\rm adm}_{29}=t_1^{3}t_2t_3^{2}t_4^{2}t_6^{2}$, & ${\rm adm}_{30}=t_1^{3}t_2t_3^{2}t_4^{2}t_5^{2}$, & \multicolumn{1}{l}{${\rm adm}_{31}=t_1t_2t_3^{2}t_4^{2}t_5^{2}t_6^{2}$,} & \multicolumn{1}{l}{${\rm adm}_{32}=t_1t_2^{2}t_3t_4^{2}t_5^{2}t_6^{2}$,} \\
${\rm adm}_{33}=t_1t_2^{2}t_3^{2}t_4t_5^{2}t_6^{2}$, & ${\rm adm}_{34}=t_1t_2^{2}t_3^{2}t_4^{2}t_5t_6^{2}$. &       &  
\end{tabular}%
\end{center}

Note that $QP_{10}^{\otimes 6}(\omega)\cong (QP_{10}^{\otimes 6})^{0}(\omega)\bigoplus (QP_{10}^{\otimes 6})^{>0}(\omega),$ where $$(QP_{10}^{\otimes 6})^{0}(\omega) = \langle \{[{\rm adm}_j]_{\omega}:\, 1\leq j\leq 30\} \rangle,\ \ (QP_{10}^{\otimes 6})^{> 0}(\omega) = \langle \{[{\rm adm}_j]_{\omega}:\, 31\leq j\leq 34\} \rangle.$$
Suppose that $[f]_{\omega}\in [(QP_{10}^{\otimes 6})^{0}(\omega)]^{S_6},$ and $[g]_{\omega}\in [(QP_{10}^{\otimes 6})^{> 0}(\omega)]^{S_6}.$ One has $f\equiv_{\omega} \sum_{1\leq j\leq 30}\gamma_j{\rm adm}_j,$ and $g\equiv_{\omega} \sum_{31\leq j\leq 34}\beta_j{\rm adm}_j,$ where $\gamma_j,\, \beta_j\in \mathbb Z/2,$ for every $j.$ From the relations $\theta_i(f)\equiv_{\omega} f,$ and $\theta_i(g)\equiv_{\omega} g,$ for all $i,\, 1\leq i\leq 5,$ we deduce that $\gamma_1 = \gamma_2 = \cdots = \gamma_{30},$ and $\beta_j =0,\, j = 31, \ldots, 34.$ So, $$[QP_{10}^{\otimes 6}(\omega)]^{S_6} = \langle [\sum_{1\leq j\leq 30}{\rm adm}_j]_{\omega} \rangle.$$ Now, assume that $[h]_{\omega}\in [QP_{10}^{\otimes 6}(\omega]^{GL_6},$ then $h\equiv_{\omega} \sum_{1\leq j\leq 30}\zeta{\rm adm}_{j}$ with $\zeta\in\mathbb Z/2.$ Using the relation $(\theta_6(h) + h)\equiv_{\omega} 0,$ it can be easily seen that $$\zeta\big(\sum_{1\leq j\leq 5}{\rm adm}_j + \sum_{22\leq j\leq 25}{\rm adm}_j + \sum_{32\leq j\leq 34}{\rm adm}_j\big) \equiv_{\omega} 0,$$ which follows $\zeta = 0.$ This leads to $[QP_{10}^{\otimes 6}(\omega]^{GL_6} = 0.$ 

The above data together with the fact that $[QP_{2}^{\otimes 6}]^{GL_6} = 0$ yield that $[QP_{10}^{\otimes 6}]^{GL_6} = 0.$

\medskip

We now consider the case $n = 11.$ Suppose that $\underline{t}$ is an admissible monomial in $P^{\otimes 6}_{11}.$  Noting that $\deg(\underline{t})$ is odd, and that $t_1^{7}t_2^{3}t_3\in P_{11}^{\otimes 6}$ is the minimal spike with $\omega(t_1^{7}t_2^{3}t_3) = (3,2,1),$ so either $\omega_1(\underline{t}) = 3$ or $\omega_1(\underline{t}) = 5.$ This means that either $\underline{t} = t_it_jt_kz^{2},\, 1\leq i<j<k\leq 6,\, z\in P_{4}^{\otimes 6}$ or $\underline{t} = t_it_jt_kt_lt_mz^{2},\, 1\leq i<j<k<l<m\leq 6,\, z\in P_{3}^{\otimes 6}.$ Since $\underline{t}$ is admissible, $z$ is, too. Furthermore, as computed above, because $$\dim QP^{\otimes 6}_{3} = \dim (QP^{\otimes 6}_{3})^{0} = 41,\ \ \dim QP^{\otimes 6}_{4} = \dim (QP^{\otimes 6}_{4})^{0} = 85,$$ it is easy to check that 
$$ \begin{array}{ll}
\medskip
QP^{\otimes 6}_{3}&\cong (QP^{\otimes 6}_{3})^{0}(1,1)\bigoplus (QP^{\otimes 6}_{3})^{0}(3,0),\\
QP^{\otimes 6}_{4}&\cong (QP^{\otimes 6}_{4})^{0}(2,1)\bigoplus (QP^{\otimes 6}_{4})^{0}(4,0)
\end{array}$$
 where 
$$ \begin{array}{ll}
\medskip
\dim (QP^{\otimes 6}_{3})^{0}(1,1) = 21,\ \ \dim (QP^{\otimes 6}_{3})^{0}(3,0) = 20,\\
\dim (QP^{\otimes 6}_{4})^{0}(2,1) = 70,\ \ \dim (QP^{\otimes 6}_{4})^{0}(4,0) = 15.
\end{array}$$ 
By these, we deduce that the weight vector $\omega(\underline{t})$ belongs to the set $$\{\widetilde{\omega}_{(1)}:=(3,2,1),\, \widetilde{\omega}_{(2)}:=(3,4),\, \widetilde{\omega}_{(3)}:=(5,1,1),\, \widetilde{\omega}_{(4)}:=(5,3)\},$$ and therefore we have an isomorphism
$$ 
 QP_{11}^{\otimes 6}\cong \big(\bigoplus_{1\leq i\leq 4}(QP_{11}^{\otimes 6})^{0}(\widetilde{\omega}_{(i)})\big) \bigoplus \big(\bigoplus_{1\leq i\leq 4}(QP_{11}^{\otimes 6})^{> 0}(\widetilde{\omega}_{(i)})\big).$$
Mothebe, Kaelo, and Ramatebele \cite{MKR} showed that $\dim  QP_{11}^{\otimes 6} = 1205.$ Moreover, we notice that since $|\mathscr C^{\otimes 5}_{12-2^{0}}| = 315,$ $|\mathscr C^{\otimes 5}_{12-2^{1}}| = 280,$ $|\mathscr C^{\otimes 5}_{12-2^{2}}| = 174,$ and $|\mathscr C^{\otimes 5}_{12-2^{3}}| = 45$ (see \cite{Tin}), a direct computation shows that the set $\bigcup_{0\leq d\leq 3}\mathscr C(d, 11)$ has all $1205$ admissible monomials. So, a basis of $QP_{11}^{\otimes 6}$ is a set consisting of all the equivalence classes of the admissible monomials in $\bigcup_{0\leq d\leq 3}\mathscr C(d, 11).$  Therefrom, the following are easily obtained:
$$ \begin{array}{ll}
\medskip
&\dim (QP_{11}^{\otimes 6})^{0}(\widetilde{\omega}_{(1)}) = 880,\ \ \dim (QP_{11}^{\otimes 6})^{0}(\widetilde{\omega}_{(2)}) = 60,\\
\medskip
& \dim (QP_{11}^{\otimes 6})^{0}(\widetilde{\omega}_{(3)}) = 90,\ \ \dim (QP_{11}^{\otimes 6})^{0}(\widetilde{\omega}_{(4)}) = 60,\\
\medskip
&\dim (QP_{11}^{\otimes 6})^{> 0}(\widetilde{\omega}_{(1)}) = 16,\ \ \dim (QP_{11}^{\otimes 6})^{> 0}(\widetilde{\omega}_{(2)}) = 24,\\
& \dim (QP_{11}^{\otimes 6})^{>0}(\widetilde{\omega}_{(3)}) = 30,\ \ \dim (QP_{11}^{\otimes 6})^{>0}(\widetilde{\omega}_{(4)}) = 45.
\end{array}$$ 
Using these results, we state that the invariants $[QP_{11}^{\otimes 6}(\widetilde{\omega}_{(i)})]^{GL_6}$ are zero for all $i,\, 1\leq i\leq 4.$ We shall describe the space $[QP_{11}^{\otimes 6}(\widetilde{\omega}_{(2)})]^{GL_6}.$ The remaining cases can be easily found by similar calculations. One knows that $QP_{11}^{\otimes 6}(\widetilde{\omega}_{(2)})\cong (QP_{11}^{\otimes 6})^{0}(\widetilde{\omega}_{(2)})\bigoplus (QP_{11}^{\otimes 6})^{> 0}(\widetilde{\omega}_{(2)}).$ The $\mathbb Z/2$-subspace $(QP_{11}^{\otimes 6})^{0}(\widetilde{\omega}_{(2)})$ has a basis consisting of all the classes represented by the following 60 admissible monomials:

\begin{center}
\begin{tabular}{llll}
${\rm adm}_{897}=t_2t_3^{2}t_4^{2}t_5^{3}t_6^{3}$, & ${\rm adm}_{898}=t_2t_3^{2}t_4^{3}t_5^{2}t_6^{3}$, & ${\rm adm}_{899}=t_2t_3^{2}t_4^{3}t_5^{3}t_6^{2}$, & ${\rm adm}_{900}=t_2t_3^{3}t_4^{2}t_5^{2}t_6^{3}$, \\
${\rm adm}_{901}=t_2t_3^{3}t_4^{2}t_5^{3}t_6^{2}$, & ${\rm adm}_{902}=t_2t_3^{3}t_4^{3}t_5^{2}t_6^{2}$, & ${\rm adm}_{903}=t_2^{3}t_3t_4^{2}t_5^{2}t_6^{3}$, & ${\rm adm}_{904}=t_2^{3}t_3t_4^{2}t_5^{3}t_6^{2}$, \\
${\rm adm}_{905}=t_2^{3}t_3t_4^{3}t_5^{2}t_6^{2}$, & ${\rm adm}_{906}=t_2^{3}t_3^{3}t_4t_5^{2}t_6^{2}$, & ${\rm adm}_{907}=t_1t_3^{2}t_4^{2}t_5^{3}t_6^{3}$, & ${\rm adm}_{908}=t_1t_3^{2}t_4^{3}t_5^{2}t_6^{3}$, \\
${\rm adm}_{909}=t_1t_3^{2}t_4^{3}t_5^{3}t_6^{2}$, & ${\rm adm}_{910}=t_1t_3^{3}t_4^{2}t_5^{2}t_6^{3}$, & ${\rm adm}_{911}=t_1t_3^{3}t_4^{2}t_5^{3}t_6^{2}$, & ${\rm adm}_{912}=t_1t_3^{3}t_4^{3}t_5^{2}t_6^{2}$, \\
${\rm adm}_{913}=t_1t_2^{2}t_4^{2}t_5^{3}t_6^{3}$, & ${\rm adm}_{914}=t_1t_2^{2}t_4^{3}t_5^{2}t_6^{3}$, & ${\rm adm}_{915}=t_1t_2^{2}t_4^{3}t_5^{3}t_6^{2}$, & ${\rm adm}_{916}=t_1t_2^{2}t_3^{2}t_5^{3}t_6^{3}$, \\
${\rm adm}_{917}=t_1t_2^{2}t_3^{2}t_4^{3}t_6^{3}$, & ${\rm adm}_{918}=t_1t_2^{2}t_3^{2}t_4^{3}t_5^{3}$, & ${\rm adm}_{919}=t_1t_2^{2}t_3^{3}t_5^{2}t_6^{3}$, & ${\rm adm}_{920}=t_1t_2^{2}t_3^{3}t_5^{3}t_6^{2}$, \\
${\rm adm}_{921}=t_1t_2^{2}t_3^{3}t_4^{2}t_6^{3}$, & ${\rm adm}_{922}=t_1t_2^{2}t_3^{3}t_4^{2}t_5^{3}$, & ${\rm adm}_{923}=t_1t_2^{2}t_3^{3}t_4^{3}t_6^{2}$, & ${\rm adm}_{924}=t_1t_2^{2}t_3^{3}t_4^{3}t_5^{2}$, \\
${\rm adm}_{925}=t_1t_2^{3}t_4^{2}t_5^{2}t_6^{3}$, & ${\rm adm}_{926}=t_1t_2^{3}t_4^{2}t_5^{3}t_6^{2}$, & ${\rm adm}_{927}=t_1t_2^{3}t_4^{3}t_5^{2}t_6^{2}$, & ${\rm adm}_{928}=t_1t_2^{3}t_3^{2}t_5^{2}t_6^{3}$, \\
${\rm adm}_{929}=t_1t_2^{3}t_3^{2}t_5^{3}t_6^{2}$, & ${\rm adm}_{930}=t_1t_2^{3}t_3^{2}t_4^{2}t_6^{3}$, & ${\rm adm}_{931}=t_1t_2^{3}t_3^{2}t_4^{2}t_5^{3}$, & ${\rm adm}_{932}=t_1t_2^{3}t_3^{2}t_4^{3}t_6^{2}$, \\
${\rm adm}_{933}=t_1t_2^{3}t_3^{2}t_4^{3}t_5^{2}$, & ${\rm adm}_{934}=t_1t_2^{3}t_3^{3}t_5^{2}t_6^{2}$, & ${\rm adm}_{935}=t_1t_2^{3}t_3^{3}t_4^{2}t_6^{2}$, & ${\rm adm}_{936}=t_1t_2^{3}t_3^{3}t_4^{2}t_5^{2}$, \\
${\rm adm}_{937}=t_1^{3}t_3t_4^{2}t_5^{2}t_6^{3}$, & ${\rm adm}_{938}=t_1^{3}t_3t_4^{2}t_5^{3}t_6^{2}$, & ${\rm adm}_{939}=t_1^{3}t_3t_4^{3}t_5^{2}t_6^{2}$, & ${\rm adm}_{940}=t_1^{3}t_3^{3}t_4t_5^{2}t_6^{2}$, \\
${\rm adm}_{941}=t_1^{3}t_2t_4^{2}t_5^{2}t_6^{3}$, & ${\rm adm}_{942}=t_1^{3}t_2t_4^{2}t_5^{3}t_6^{2}$, & ${\rm adm}_{943}=t_1^{3}t_2t_4^{3}t_5^{2}t_6^{2}$, & ${\rm adm}_{944}=t_1^{3}t_2t_3^{2}t_5^{2}t_6^{3}$, \\
${\rm adm}_{945}=t_1^{3}t_2t_3^{2}t_5^{3}t_6^{2}$, & ${\rm adm}_{946}=t_1^{3}t_2t_3^{2}t_4^{2}t_6^{3}$, & ${\rm adm}_{947}=t_1^{3}t_2t_3^{2}t_4^{2}t_5^{3}$, & ${\rm adm}_{948}=t_1^{3}t_2t_3^{2}t_4^{3}t_6^{2}$, \\
${\rm adm}_{949}=t_1^{3}t_2t_3^{2}t_4^{3}t_5^{2}$, & ${\rm adm}_{950}=t_1^{3}t_2t_3^{3}t_5^{2}t_6^{2}$, & ${\rm adm}_{951}=t_1^{3}t_2t_3^{3}t_4^{2}t_6^{2}$, & ${\rm adm}_{952}=t_1^{3}t_2t_3^{3}t_4^{2}t_5^{2}$, \\
${\rm adm}_{953}=t_1^{3}t_2^{3}t_4t_5^{2}t_6^{2}$, & ${\rm adm}_{954}=t_1^{3}t_2^{3}t_3t_5^{2}t_6^{2}$, & ${\rm adm}_{955}=t_1^{3}t_2^{3}t_3t_4^{2}t_6^{2}$, & ${\rm adm}_{956}=t_1^{3}t_2^{3}t_3t_4^{2}t_5^{2}$.
\end{tabular}%
\end{center}

Suppose that $[f]_{\widetilde{\omega}_{(2)}}\in [(QP_{11}^{\otimes 6})^{0}(\widetilde{\omega}_{(2)})]^{S_6},$ then $f\equiv_{\widetilde{\omega}_{(2)}}\sum_{897\leq j\leq 956}\gamma_j{\rm adm}_j,$ in which $\gamma_j\in \mathbb Z/2$ for all $j.$ By a simple using the relations $\theta_i(f) + f \equiv_{\widetilde{\omega}_{(2)}} 0,$ for $1\leq i\leq 5,$  we get $\gamma_{897} = \gamma_j$ for any $j,\, 898\leq j\leq 956,$ and so $$[(QP_{11}^{\otimes 6})^{0}(\widetilde{\omega}_{(2)})]^{S_6} = \langle [\sum_{897\leq j\leq 956}{\rm adm}_j] \rangle.$$
Next, the $\mathbb Z/2$-subspace $(QP_{11}^{\otimes 6})^{> 0}(\widetilde{\omega}_{(2)})$ has a basis consisting of all the classes represented by the following 24 admissible monomials:

\begin{center}
\begin{tabular}{llll}
${\rm adm}_{957}=t_1t_2t_3^{3}t_4^{2}t_5^{2}t_6^{2}$, & ${\rm adm}_{958}=t_1t_2t_3^{2}t_4^{3}t_5^{2}t_6^{2}$, & ${\rm adm}_{959}=t_1t_2t_3^{2}t_4^{2}t_5^{3}t_6^{2}$, & ${\rm adm}_{960}=t_1t_2t_3^{2}t_4^{2}t_5^{2}t_6^{3}$, \\
${\rm adm}_{961}=t_1t_2^{3}t_3t_4^{2}t_5^{2}t_6^{2}$, & ${\rm adm}_{962}=t_1^{3}t_2t_3t_4^{2}t_5^{2}t_6^{2}$, & ${\rm adm}_{963}=t_1^{3}t_2t_3^{2}t_4t_5^{2}t_6^{2}$, & ${\rm adm}_{964}=t_1^{3}t_2t_3^{2}t_4^{2}t_5t_6^{2}$, \\
${\rm adm}_{965}=t_1^{3}t_2t_3^{2}t_4^{2}t_5^{2}t_6$, & ${\rm adm}_{966}=t_1t_2^{3}t_3^{2}t_4t_5^{2}t_6^{2}$, & ${\rm adm}_{967}=t_1t_2^{3}t_3^{2}t_4^{2}t_5t_6^{2}$, & ${\rm adm}_{968}=t_1t_2^{3}t_3^{2}t_4^{2}t_5^{2}t_6$, \\
${\rm adm}_{969}=t_1t_2^{2}t_3t_4^{3}t_5^{2}t_6^{2}$, & ${\rm adm}_{970}=t_1t_2^{2}t_3t_4^{2}t_5^{3}t_6^{2}$, & ${\rm adm}_{971}=t_1t_2^{2}t_3t_4^{2}t_5^{2}t_6^{3}$, & ${\rm adm}_{972}=t_1t_2^{2}t_3^{3}t_4t_5^{2}t_6^{2}$, \\
${\rm adm}_{973}=t_1t_2^{2}t_3^{3}t_4^{2}t_5t_6^{2}$, & ${\rm adm}_{974}=t_1t_2^{2}t_3^{3}t_4^{2}t_5^{2}t_6$, & ${\rm adm}_{975}=t_1t_2^{2}t_3^{2}t_4t_5^{3}t_6^{2}$, & ${\rm adm}_{976}=t_1t_2^{2}t_3^{2}t_4t_5^{2}t_6^{3}$, \\
${\rm adm}_{977}=t_1t_2^{2}t_3^{2}t_4^{3}t_5t_6^{2}$, & ${\rm adm}_{978}=t_1t_2^{2}t_3^{2}t_4^{3}t_5^{2}t_6$, & ${\rm adm}_{979}=t_1t_2^{2}t_3^{2}t_4^{2}t_5^{3}t_6$, & ${\rm adm}_{980}=t_1t_2^{2}t_3^{2}t_4^{2}t_5t_6^{3}$.
\end{tabular}%
\end{center}

Then, if $[g]_{\widetilde{\omega}_{(2)}}\in [(QP_{11}^{\otimes 6})^{> 0}(\widetilde{\omega}_{(2)})]^{S_6},$ then $g\equiv_{\widetilde{\omega}_{(2)}}\sum_{957\leq j\leq 980}\gamma_j{\rm adm}_j,$ where $\gamma_j\in \mathbb Z/2$ for every $j.$ By the relations $\theta_i(g) + g \equiv_{\widetilde{\omega}_{(2)}} 0,$ for any $i,\, 1\leq i\leq 5,$  one gets $\gamma_j = 0$ for any $j,$ which implies that $$[(QP_{11}^{\otimes 6})^{> 0}(\widetilde{\omega}_{(2)})]^{S_6} = 0.$$
From the above calculations and the fact that $S_6\subset GL_6$, if $[h]_{\widetilde{\omega}_{(2)}}\in [QP_{11}^{\otimes 6}(\widetilde{\omega}_{(2)})]^{GL_6},$ then $$h\equiv_{\widetilde{\omega}_{(2)}}\beta(\sum_{897\leq j\leq 956}{\rm adm}_j)\ \mbox{ with $\beta\in \mathbb Z/2.$}$$ Then, since $\theta_6(h) + h\equiv_{\widetilde{\omega}_{(2)}} 0,$ one gets $\beta{\rm adm}_{897} + \mbox{other terms} \equiv_{\widetilde{\omega}_{(2)}} 0,$ which implies that $\beta = 0.$ Therefore $[QP_{11}^{\otimes 6}(\widetilde{\omega}_{(2)})]^{GL_6} = 0.$

\medskip

Finally, we consider the case $n = 12.$ Because $\dim QP_{3}^{\otimes 6} = 41,$ and $\dim QP_{12}^{\otimes 6} = 1001,$ we see that $\dim {\rm Ker}((\widetilde {Sq^0_*})_{12}) = 960.$ A basis of ${\rm Ker}((\widetilde {Sq^0_*})_{12})$ is determined in Tin \cite{Tin4}. However, his calculations seem very verbose and to have no sense, since ${\rm Ker}((\widetilde {Sq^0_*})_{12})\subset QP_{12}^{\otimes 6},$ and the dimension of $QP_{12}^{\otimes 6}$ has been computed by Mothebe, Kaelo, and Ramatebele \cite{MKR}. Here, similarly to the cases $n = 8$ and $10,$ determining the basis of ${\rm Ker}((\widetilde {Sq^0_*})_{12})$ can be done simply as follows: Since $|\mathscr {C}^{\otimes 5}_{13-2^{0}}| = 190$ (see \cite{MKR}), $|\mathscr {C}^{\otimes 5}_{13-2^{1}}| = 315$ (see \cite{Tin}), $|\mathscr {C}^{\otimes 5}_{13-2^{2}}| = 191$ (see \cite{Tin}), and $|\mathscr {C}^{\otimes 5}_{13-2^{3}}| = 46$ (see \cite{Sum}), the set
$$ B(12):=\big(\bigcup_{0\leq d\leq 3}\mathscr C(d, 12)\big)\setminus \varphi(\mathscr C_{3}^{\otimes 6}),$$
where $|\mathscr C_{3}^{\otimes 6}| = 41,$ has all 960 admissible monomials. Further, we observe that $(\widetilde {Sq^0_*})_{12}([v]) = [0]$ for every $v\in B(12).$ (Note that since $v$ is admissible, $[v]\neq [0]$.) Therefore, the basis of ${\rm Ker}((\widetilde {Sq^0_*})_{12})$ is a set consisting of all the equivalence classes of the admissible monomials in $B(12).$ Therefrom, by similar calculations as in the cases $n = 8,$ and $10$, it can be easily found that $[{\rm Ker}((\widetilde {Sq^0_*})_{12})]^{GL_6}=0$. So, since $[QP_{3}^{\otimes 6}]^{GL_6} =0$, $[QP_{12}^{\otimes 6}]^{GL_6}=0.$ 

\medskip

Thus, we obtain that for each $1\leq n\leq 12,$ the invariants $[QP_{n}^{\otimes 6}]^{GL_6}$ are zero. On the other side,  according to Bruner \cite{Bruner}, for $1\leq n\leq 12,$ we see that the elements $h_1Ph_1$ and $h_0Ph_2$ are non-zero in ${\rm Ext}_{\mathcal A}^{6, 6+10}(\mathbb Z/2, \mathbb Z/2),$ and ${\rm Ext}_{\mathcal A}^{6, 6+11}(\mathbb Z/2, \mathbb Z/2),$ respectively and that ${\rm Ext}_{\mathcal A}^{6, 6+n}(\mathbb Z/2, \mathbb Z/2)  = 0$ for $n\not\in \{10, 11\}.$ These data imply that the sixth algebraic transfer does not detect the non-zero elements $h_1Ph_1$ and $h_0Ph_2.$ This has also been proved by by Ch\ohorn n, and H\`a \cite{C.H1, C.H2} using another method. Therefore, the theorem below is obtained.

\begin{thm}\label{dlc1}
Let $n$ be a positive integer with $n\leq 12.$ Then, the Singer transfer
$$Tr_6^{\mathcal A}: ([QP_{n}^{\otimes 6}]^{GL_6})^{*}\longrightarrow {\rm Ext}_{\mathcal A}^{6, 6+n}(\mathbb Z/2, \mathbb Z/2)$$
is a monomorphism if $n \in \{10, 11\},$ and is a trivial isomorphism otherwise.
\end{thm}

For $n = 13,$ the following remarks are important and useful.

\begin{rema}\label{nxM}
Moetele, and Mothebe \cite{MM} have shown that $QP_{13}^{\otimes 6}$ is $1303$-dimensional. However, this claim is incorrect. We refute this result by proving that $QP_{13}^{\otimes 6}$ has dimension $1294.$  Indeed, firstly, we see that if $t$ is an admissible monomial in $P_{13}^{\otimes 6},$ then by Tin \cite{Tin4}, the weight vector of $t$ is one of the following sequences: $$\overline{\omega}_{(1)}:= (3,3,1),\ \overline{\omega}_{(2)}:= (3,5),\ \overline{\omega}_{(3)}:= (5,2,1),\  \overline{\omega}_{(4)}:= (5,4).$$ So, we have an isomorphism 
$$ 
 QP_{13}^{\otimes 6}\cong \big(\bigoplus_{1\leq i\leq 4}(QP_{13}^{\otimes 6})^{0}(\overline{\omega}_{(i)})\big) \bigoplus \big(\bigoplus_{1\leq i\leq 4}(QP_{13}^{\otimes 6})^{> 0}(\overline{\omega}_{(i)})\big).$$
 Using the formula $\dim (QP_{13}^{\otimes 6})^{0} = \sum_{1\leq  k\leq 5}\binom{6}{k}\dim (QP_{13}^{\otimes\, k})^{>0},$ and the previous results by Peterson \cite{Peterson}, Kameko \cite{Kameko}, the present writer \cite{Phuc4}, Sum \cite{Sum1, Sum2}, one gets $$\dim (QP_{13}^{\otimes 6})^{0} = 3\times \binom{6}{3} + 23\times \binom{6}{4} + 105\times \binom{6}{5} =  1035,$$
and so, it follows that
$$ \dim (QP_{13}^{\otimes 6})^{0}(\overline{\omega}_{(i)})
=\left\{\begin{array}{ll}
765&\mbox{if $i = 1$},\\
0&\mbox{if $i = 2$},\\
240&\mbox{if $i = 3$},\\
30&\mbox{if $i = 4$}.
\end{array}\right.$$
Next, we consider the spaces $(QP_{13}^{\otimes 6})^{> 0}(\overline{\omega}_{(i)}).$ The calculations of Moetele, and Mothebe \cite{MM} indicated that
$$ \dim (QP_{13}^{\otimes 6})^{> 0}(\overline{\omega}_{(i)})
=\left\{\begin{array}{ll}
69 &\mbox{if $i = 1$},\\
15&\mbox{if $i = 2$},\\
144&\mbox{if $i = 3$},\\
40&\mbox{if $i = 4$}.
\end{array}\right.$$
However, the statement for the case $i = 1$ is not true. Indeed, Moetele, and Mothebe \cite{MM} claimed that the following 9 monomials, which have the same weight vector as $\overline{\omega}_{(1)},$ are admissible: 
$$ \begin{array}{ll}
\medskip
 &t_1t_2t_3^{2}t_4^{2}t_5^{6}t_6,\ \ t_1t_2t_3^{2}t_4^{6}t_5^{2}t_6,\ \ t_1t_2t_3^{6}t_4^{2}t_5^{2}t_6,\\
\medskip
& t_1t_2^{2}t_3t_4^{2}t_5^{6}t_6,\ \ t_1t_2^{2}t_3t_4^{6}t_5^{2}t_6,\ \ t_1t_2^{6}t_3t_4^{2}t_5^{2}t_6,\\
&t_1t_2t_3^{2}t_4^{2}t_5^{2}t_6^{5},\ \ t_1t_2^{2}t_3t_4^{2}t_5^{2}t_6^{5},\ \ t_1t_2^{2}t_3^{5}t_4^{2}t_5^{2}t_6.
\end{array}$$
However, by applying the Cartan formula, we get
$$ \begin{array}{ll}
t_1t_2^{2}t_3^{5}t_4^{2}t_5^{2}t_6 =& Sq^{1}(t_1^{2}t_2^{4}t_3^{3}t_4t_5t_6) + Sq^{2}(t_1t_2^{2}t_3^{3}t_4^{2}t_5^{2}t_6 + t_1t_2^{2}t_3^{3}t_4^{2}t_5t_6^{2} + t_1t_2^{2}t_3^{3}t_4t_5^{2}t_6^{2} + t_1t_2^{2}t_3^{5}t_4t_5t_6) \\
&+ Sq^{4}(t_1t_2^{2}t_3^{3}t_4t_5t_6) + X+ \sum Y,
\end{array}$$
where $\omega(Y)<\omega(t_1t_2^{2}t_3^{5}t_4^{2}t_5^{2}t_6) = \overline{\omega}_{(1)},$ and $X$ is the sum of the following 8 monomials: $ t_1t_2^{2}t_3^{3}t_4t_5^{2}t_6^{4},$ $t_1t_2^{2}t_3^{3}t_4t_5^{4}t_6^{2},$ $t_1t_2^{2}t_3^{3}t_4^{2}t_5t_6^{4},$ $t_1t_2^{2}t_3^{3}t_4^{2}t_5^{4}t_6,$ $t_1t_2^{2}t_3^{3}t_4^{4}t_5t_6^{2},$ $t_1t_2^{2}t_3^{3}t_4^{4}t_5^{2}t_6,$ $t_1t_2^{2}t_3^{5}t_4t_5^{2}t_6^{2},$ and $t_1t_2^{2}t_3^{5}t_4^{2}t_5t_6^{2}.$ All of these monomials are less than $t_1t_2^{2}t_3^{5}t_4^{2}t_5^{2}t_6.$ Thus, $ t_1t_2^{2}t_3^{5}t_4^{2}t_5^{2}t_6 \equiv_{\overline{\omega}_{(1)}} X,$ and so $t_1t_2^{2}t_3^{5}t_4^{2}t_5^{2}t_6 $ is an inadmissible monomial.  Similarly, we obtain:
$$ \begin{array}{ll}
\medskip
t_1t_2t_3^{2}t_4^{2}t_5^{6}t_6&\equiv_{\overline{\omega}_{(1)}}  t_1t_2t_3t_4^{2}t_5^{6}t_6^{2}+ t_1t_2t_3^{2}t_4t_5^{6}t_6^{2}+t_1t_2t_3^{2}t_4^{2}t_5^{5}t_6^{2},\\
\medskip
 t_1t_2t_3^{2}t_4^{6}t_5^{2}t_6&\equiv_{\overline{\omega}_{(1)}} t_1t_2t_3t_4^{6}t_5^{2}t_6^{2}+t_1t_2t_3^{2}t_4^{5}t_5^{2}t_6^{2}+t_1t_2t_3^{2}t_4^{6}t_5t_6^{2},\\
\medskip
t_1t_2^{2}t_3t_4^{6}t_5^{2}t_6&\equiv_{\overline{\omega}_{(1)}}t_1t_2t_3t_4^{6}t_5^{2}t_6^{2}+ t_1t_2^{2}t_3t_4^{5}t_5^{2}t_6^{2}+t_1t_2^{2}t_3t_4^{6}t_5t_6^{2},\\
\medskip
 t_1t_2^{2}t_3t_4^{2}t_5^{6}t_6&\equiv_{\overline{\omega}_{(1)}} t_1t_2t_3t_4^{2}t_5^{6}t_6^{2}+ t_1t_2^{2}t_3t_4t_5^{6}t_6^{2}+t_1t_2^{2}t_3t_4^{2}t_5^{5}t_6^{2},\\
\medskip
t_1t_2t_3^{2}t_4^{2}t_5^{2}t_6^{5}&\equiv_{\overline{\omega}_{(1)}}t_1t_2t_3t_4^{2}t_5^{2}t_6^{6}+  t_1t_2t_3^{2}t_4t_5^{2}t_6^{6}+t_1t_2t_3^{2}t_4^{2}t_5t_6^{6},\\
\medskip
 t_1t_2^{2}t_3t_4^{2}t_5^{2}t_6^{5}&\equiv_{\overline{\omega}_{(1)}} t_1t_2t_3t_4^{2}t_5^{2}t_6^{6}+ t_1t_2^{2}t_3t_4t_5^{2}t_6^{6}+t_1t_2^{2}t_3t_4^{2}t_5t_6^{6},\\
t_1t_2t_3^{6}t_4^{2}t_5^{2}t_6&\equiv_{\overline{\omega}_{(1)}} t_1t_2t_3^{3}t_4^{2}t_5^{2}t_6^{4} + t_1t_2t_3^{3}t_4^{2}t_5^{4}t_6^{2}+ t_1t_2t_3^{3}t_4^{4}t_5^{2}t_6^{2}\\
\medskip
&\quad\quad\quad +t_1t_2t_3^{6}t_4t_5^{2}t_6^{2} + t_1t_2t_3^{6}t_4^{2}t_5t_6^{2},\\
t_1t_2^{6}t_3t_4^{2}t_5^{2}t_6&\equiv_{\overline{\omega}_{(1)}} t_1t_2^{3}t_3t_4^{2}t_5^{2}t_6^{4} + t_1t_2^{3}t_3t_4^{2}t_5^{4}t_6^{2}+ t_1t_2^{3}t_3t_4^{4}t_5^{2}t_6^{2} \\
&\quad\quad\quad +t_1t_2^{6}t_3t_4t_5^{2}t_6^{2} + t_1t_2^{6}t_3t_4^{2}t_5t_6^{2},
\end{array}$$
and so, the monomials $t_1t_2t_3^{2}t_4^{2}t_5^{6}t_6,$ $t_1t_2t_3^{2}t_4^{6}t_5^{2}t_6,$ $t_1t_2t_3^{6}t_4^{2}t_5^{2}t_6,$ $t_1t_2^{2}t_3t_4^{2}t_5^{6}t_6,$ $t_1t_2^{2}t_3t_4^{6}t_5^{2}t_6,$ $t_1t_2^{6}t_3t_4^{2}t_5^{2}t_6,$ $t_1t_2t_3^{2}t_4^{2}t_5^{2}t_6^{5},$ and $t_1t_2^{2}t_3t_4^{2}t_5^{2}t_6^{5}$ are inadmissible.  

\medskip

Thus, $(QP_{13}^{\otimes 6})^{> 0}(\overline{\omega}_{(1)}) = \langle \{[{\rm adm}_j]_{\overline{\omega}_{(1)}} = [{\rm adm}_j]:\,1\leq j\leq 60\}\rangle,$ where the admissible monomials ${\rm adm}_j$ are determined as follows:

\begin{center}
\begin{tabular}{llll}
${\rm adm}_{1}=t_1t_2t_3t_4^{2}t_5^{2}t_6^{6}$, & ${\rm adm}_{2}=t_1t_2t_3t_4^{2}t_5^{6}t_6^{2}$, & ${\rm adm}_{3}=t_1t_2t_3t_4^{6}t_5^{2}t_6^{2}$, & ${\rm adm}_{4}=t_1t_2t_3^{2}t_4t_5^{2}t_6^{6}$, \\
${\rm adm}_{5}=t_1t_2t_3^{2}t_4t_5^{6}t_6^{2}$, & ${\rm adm}_{6}=t_1t_2t_3^{2}t_4^{2}t_5t_6^{6}$, & ${\rm adm}_{7}=t_1t_2t_3^{2}t_4^{2}t_5^{3}t_6^{4}$, & ${\rm adm}_{8}=t_1t_2t_3^{2}t_4^{2}t_5^{4}t_6^{3}$, \\
${\rm adm}_{9}=t_1t_2t_3^{2}t_4^{2}t_5^{5}t_6^{2}$, & ${\rm adm}_{10}=t_1t_2t_3^{2}t_4^{3}t_5^{2}t_6^{4}$, & ${\rm adm}_{11}=t_1t_2t_3^{2}t_4^{3}t_5^{4}t_6^{2}$, & ${\rm adm}_{12}=t_1t_2t_3^{2}t_4^{4}t_5^{2}t_6^{3}$, \\
${\rm adm}_{13}=t_1t_2t_3^{2}t_4^{4}t_5^{3}t_6^{2}$, & ${\rm adm}_{14}=t_1t_2t_3^{2}t_4^{5}t_5^{2}t_6^{2}$, & ${\rm adm}_{15}=t_1t_2t_3^{2}t_4^{6}t_5t_6^{2}$, & ${\rm adm}_{16}=t_1t_2t_3^{3}t_4^{2}t_5^{2}t_6^{4}$, \\
${\rm adm}_{17}=t_1t_2t_3^{3}t_4^{2}t_5^{4}t_6^{2}$, & ${\rm adm}_{18}=t_1t_2t_3^{3}t_4^{4}t_5^{2}t_6^{2}$, & ${\rm adm}_{19}=t_1t_2t_3^{6}t_4t_5^{2}t_6^{2}$, & ${\rm adm}_{20}=t_1t_2t_3^{6}t_4^{2}t_5t_6^{2}$, \\
${\rm adm}_{21}=t_1t_2^{2}t_3t_4t_5^{2}t_6^{6}$, & ${\rm adm}_{22}=t_1t_2^{2}t_3t_4t_5^{6}t_6^{2}$, & ${\rm adm}_{23}=t_1t_2^{2}t_3t_4^{2}t_5t_6^{6}$, & ${\rm adm}_{24}=t_1t_2^{2}t_3t_4^{2}t_5^{3}t_6^{4}$, \\
${\rm adm}_{25}=t_1t_2^{2}t_3t_4^{2}t_5^{4}t_6^{3}$, & ${\rm adm}_{26}=t_1t_2^{2}t_3t_4^{2}t_5^{5}t_6^{2}$, & ${\rm adm}_{27}=t_1t_2^{2}t_3t_4^{3}t_5^{2}t_6^{4}$, & ${\rm adm}_{28}=t_1t_2^{2}t_3t_4^{3}t_5^{4}t_6^{2}$, \\
${\rm adm}_{29}=t_1t_2^{2}t_3t_4^{4}t_5^{2}t_6^{3}$, & ${\rm adm}_{30}=t_1t_2^{2}t_3t_4^{4}t_5^{3}t_6^{2}$, & ${\rm adm}_{31}=t_1t_2^{2}t_3t_4^{5}t_5^{2}t_6^{2}$, & ${\rm adm}_{32}=t_1t_2^{2}t_3t_4^{6}t_5t_6^{2}$, \\
${\rm adm}_{33}=t_1t_2^{2}t_3^{3}t_4t_5^{2}t_6^{4}$, & ${\rm adm}_{34}=t_1t_2^{2}t_3^{3}t_4t_5^{4}t_6^{2}$, & ${\rm adm}_{35}=t_1t_2^{2}t_3^{3}t_4^{4}t_5t_6^{2}$, & ${\rm adm}_{36}=t_1t_2^{2}t_3^{4}t_4t_5^{2}t_6^{3}$, \\
${\rm adm}_{37}=t_1t_2^{2}t_3^{4}t_4t_5^{3}t_6^{2}$, & ${\rm adm}_{38}=t_1t_2^{2}t_3^{4}t_4^{3}t_5t_6^{2}$, & ${\rm adm}_{39}=t_1t_2^{2}t_3^{5}t_4t_5^{2}t_6^{2}$, & ${\rm adm}_{40}=t_1t_2^{2}t_3^{5}t_4^{2}t_5t_6^{2}$, \\
${\rm adm}_{41}=t_1t_2^{3}t_3t_4^{2}t_5^{2}t_6^{4}$, & ${\rm adm}_{42}=t_1t_2^{3}t_3t_4^{2}t_5^{4}t_6^{2}$, & ${\rm adm}_{43}=t_1t_2^{3}t_3t_4^{4}t_5^{2}t_6^{2}$, & ${\rm adm}_{44}=t_1t_2^{3}t_3^{2}t_4t_5^{2}t_6^{4}$, \\
${\rm adm}_{45}=t_1t_2^{3}t_3^{2}t_4t_5^{4}t_6^{2}$, & ${\rm adm}_{46}=t_1t_2^{3}t_3^{2}t_4^{4}t_5t_6^{2}$, & ${\rm adm}_{47}=t_1t_2^{3}t_3^{4}t_4t_5^{2}t_6^{2}$, & ${\rm adm}_{48}=t_1t_2^{3}t_3^{4}t_4^{2}t_5t_6^{2}$, \\
${\rm adm}_{49}=t_1t_2^{6}t_3t_4t_5^{2}t_6^{2}$, & ${\rm adm}_{50}=t_1t_2^{6}t_3t_4^{2}t_5t_6^{2}$, & ${\rm adm}_{51}=t_1^{3}t_2t_3t_4^{2}t_5^{2}t_6^{4}$, & ${\rm adm}_{52}=t_1^{3}t_2t_3t_4^{2}t_5^{4}t_6^{2}$, \\
${\rm adm}_{53}=t_1^{3}t_2t_3t_4^{4}t_5^{2}t_6^{2}$, & ${\rm adm}_{54}=t_1^{3}t_2t_3^{2}t_4t_5^{2}t_6^{4}$, & ${\rm adm}_{55}=t_1^{3}t_2t_3^{2}t_4t_5^{4}t_6^{2}$, & ${\rm adm}_{56}=t_1^{3}t_2t_3^{2}t_4^{4}t_5t_6^{2}$, \\
${\rm adm}_{57}=t_1^{3}t_2t_3^{4}t_4t_5^{2}t_6^{2}$, & ${\rm adm}_{58}=t_1^{3}t_2t_3^{4}t_4^{2}t_5t_6^{2}$, & ${\rm adm}_{59}=t_1^{3}t_2^{4}t_3t_4t_5^{2}t_6^{2}$, & ${\rm adm}_{60}=t_1^{3}t_2^{4}t_3t_4^{2}t_5t_6^{2}$.
\end{tabular}%
\end{center}
We prove $\dim (QP_{13}^{\otimes 6})^{>0}(\overline{\omega}_{(1)}) = 60$ by showing that the set $\{[{\rm adm}_j]_{\overline{\omega}_{(1)}} = [{\rm adm}_j]:\,1\leq j\leq 60\}$ is linearly dependent in $(QP_{13}^{\otimes 6})^{> 0}(\overline{\omega}_{(1)}).$ To make this, for any $1\leq u < v\leq 6,$ we use the $\mathcal A$-homomorphism 
$$ \begin{array}{ll}
 \varphi_{(u, v)}: P^{\otimes 6}&\longrightarrow P^{\otimes 5},\\
\ \ \ \ \ \ \ \ \ \ t_j &\longmapsto \left\{ \begin{array}{ll}
{t_j}&\text{if }\;1\leq j \leq u-1, \\
t_{v-1}& \text{if}\; j = u,\\
t_{j-1}&\text{if}\; u+1 \leq j \leq 6.
\end{array} \right.
\end{array}$$
Applying a result in \cite{Phuc4}, we deduce that $\varphi_{(u, v)}({\rm adm}_j)\in P_{13}^{\otimes 5}(\overline{\omega}_{(1)}),$ for each $1\leq j\leq 60.$ In \cite{Phuc4}, we have shown that $(QP_{13}^{\otimes 5})^{>0}(\overline{\omega}_{(1)})$ is an $\mathbb Z/2$-vector space with a basis consisting of all the classes represented by the following monomials:

\begin{center}
\begin{tabular}{llll}
${\rm Adm}_{1}=t_1t_2t_3^{2}t_4^{2}t_5^{7}$, & ${\rm Adm}_{2}=t_1t_2t_3^{2}t_4^{3}t_5^{6}$, & ${\rm Adm}_{3}=t_1t_2t_3^{2}t_4^{6}t_5^{3}$, & ${\rm Adm}_{4}=t_1t_2t_3^{2}t_4^{7}t_5^{2}$, \\
${\rm Adm}_{5}=t_1t_2t_3^{3}t_4^{2}t_5^{6}$, & ${\rm Adm}_{6}=t_1t_2t_3^{3}t_4^{6}t_5^{2}$, & ${\rm Adm}_{7}=t_1t_2t_3^{6}t_4^{2}t_5^{3}$, & ${\rm Adm}_{8}=t_1t_2t_3^{6}t_4^{3}t_5^{2}$, \\
${\rm Adm}_{9}=t_1t_2t_3^{7}t_4^{2}t_5^{2}$, & ${\rm Adm}_{10}=t_1t_2^{2}t_3t_4^{2}t_5^{7}$, & ${\rm Adm}_{11}=t_1t_2^{2}t_3t_4^{3}t_5^{6}$, & ${\rm Adm}_{12}=t_1t_2^{2}t_3t_4^{6}t_5^{3}$, \\
${\rm Adm}_{13}=t_1t_2^{2}t_3t_4^{7}t_5^{2}$, & ${\rm Adm}_{14}=t_1t_2^{2}t_3^{3}t_4t_5^{6}$, & ${\rm Adm}_{15}=t_1t_2^{2}t_3^{3}t_4^{3}t_5^{4}$, & ${\rm Adm}_{16}=t_1t_2^{2}t_3^{3}t_4^{4}t_5^{3}$, \\
${\rm Adm}_{17}=t_1t_2^{2}t_3^{3}t_4^{5}t_5^{2}$, & ${\rm Adm}_{18}=t_1t_2^{2}t_3^{4}t_4^{3}t_5^{3}$, & ${\rm Adm}_{19}=t_1t_2^{2}t_3^{5}t_4^{2}t_5^{3}$, & ${\rm Adm}_{20}=t_1t_2^{2}t_3^{5}t_4^{3}t_5^{2}$, \\
${\rm Adm}_{21}=t_1t_2^{2}t_3^{7}t_4t_5^{2}$, & ${\rm Adm}_{22}=t_1t_2^{3}t_3t_4^{2}t_5^{6}$, & ${\rm Adm}_{23}=t_1t_2^{3}t_3t_4^{6}t_5^{2}$, & ${\rm Adm}_{24}=t_1t_2^{3}t_3^{2}t_4t_5^{6}$, \\
${\rm Adm}_{25}=t_1t_2^{3}t_3^{2}t_4^{3}t_5^{4}$, & ${\rm Adm}_{26}=t_1t_2^{3}t_3^{2}t_4^{4}t_5^{3}$, & ${\rm Adm}_{27}=t_1t_2^{3}t_3^{2}t_4^{5}t_5^{2}$, & ${\rm Adm}_{28}=t_1t_2^{3}t_3^{3}t_4^{2}t_5^{4}$, \\
${\rm Adm}_{29}=t_1t_2^{3}t_3^{3}t_4^{4}t_5^{2}$, & ${\rm Adm}_{30}=t_1t_2^{3}t_3^{4}t_4^{2}t_5^{3}$, & ${\rm Adm}_{31}=t_1t_2^{3}t_3^{4}t_4^{3}t_5^{2}$, & ${\rm Adm}_{32}=t_1t_2^{3}t_3^{5}t_4^{2}t_5^{2}$, \\
${\rm Adm}_{33}=t_1t_2^{3}t_3^{6}t_4t_5^{2}$, & ${\rm Adm}_{34}=t_1t_2^{6}t_3t_4^{2}t_5^{3}$, & ${\rm Adm}_{35}=t_1t_2^{6}t_3t_4^{3}t_5^{2}$, & ${\rm Adm}_{36}=t_1t_2^{6}t_3^{3}t_4t_5^{2}$, \\
${\rm Adm}_{37}=t_1t_2^{7}t_3t_4^{2}t_5^{2}$, & ${\rm Adm}_{38}=t_1t_2^{7}t_3^{2}t_4t_5^{2}$, & ${\rm Adm}_{39}=t_1^{3}t_2t_3t_4^{2}t_5^{6}$, & ${\rm Adm}_{40}=t_1^{3}t_2t_3t_4^{6}t_5^{2}$, \\
${\rm Adm}_{41}=t_1^{3}t_2t_3^{2}t_4t_5^{6}$, & ${\rm Adm}_{42}=t_1^{3}t_2t_3^{2}t_4^{3}t_5^{4}$, & ${\rm Adm}_{43}=t_1^{3}t_2t_3^{2}t_4^{4}t_5^{3}$, & ${\rm Adm}_{44}=t_1^{3}t_2t_3^{2}t_4^{5}t_5^{2}$, \\
${\rm Adm}_{45}=t_1^{3}t_2t_3^{3}t_4^{2}t_5^{4}$, & ${\rm Adm}_{46}=t_1^{3}t_2t_3^{3}t_4^{4}t_5^{2}$, & ${\rm Adm}_{47}=t_1^{3}t_2t_3^{4}t_4^{2}t_5^{3}$, & ${\rm Adm}_{48}=t_1^{3}t_2t_3^{4}t_4^{3}t_5^{2}$, \\
${\rm Adm}_{49}=t_1^{3}t_2t_3^{5}t_4^{2}t_5^{2}$, & ${\rm Adm}_{50}=t_1^{3}t_2t_3^{6}t_4t_5^{2}$, & ${\rm Adm}_{51}=t_1^{3}t_2^{3}t_3t_4^{2}t_5^{4}$, & ${\rm Adm}_{52}=t_1^{3}t_2^{3}t_3t_4^{4}t_5^{2}$, \\
${\rm Adm}_{53}=t_1^{3}t_2^{3}t_3^{4}t_4t_5^{2}$, & ${\rm Adm}_{54}=t_1^{3}t_2^{4}t_3t_4^{2}t_5^{3}$, & ${\rm Adm}_{55}=t_1^{3}t_2^{4}t_3t_4^{3}t_5^{2}$, & ${\rm Adm}_{56}=t_1^{3}t_2^{4}t_3^{3}t_4t_5^{2}$, \\
${\rm Adm}_{57}=t_1^{3}t_2^{5}t_3t_4^{2}t_5^{2}$, & ${\rm Adm}_{58}=t_1^{3}t_2^{5}t_3^{2}t_4t_5^{2}$, & ${\rm Adm}_{59}=t_1^{7}t_2t_3t_4^{2}t_5^{2}$, & ${\rm Adm}_{60}=t_1^{7}t_2t_3^{2}t_4t_5^{2}$.
\end{tabular}%
 \end{center}

Suppose that $\sum_{1\leq j\leq 60}\gamma_j{\rm adm}_j\equiv 0.$ We compute explicitly $\varphi_{(u, v)}(\sum_{1\leq j\leq 60}\gamma_j{\rm adm}_j)$ in terms of admissible monomials ${\rm Adm}_j$ modulo ($\overline{\mathcal A}P^{\otimes 5}_{13}$). Then, a simple computation using the relations $\varphi_{(u, v)}(\sum_{1\leq j\leq 60}\gamma_j{\rm adm}_j)\equiv 0$ gives $\gamma_j = 0,\ j = 1, 2, \ldots, 60.$ This finishes the above statement.

\medskip

Finally, since $QP_{13}^{\otimes 6}\cong \big(\bigoplus_{1\leq i\leq 4}(QP_{13}^{\otimes 6})^{0}(\overline{\omega}_{(i)})\big) \bigoplus \big(\bigoplus_{1\leq i\leq 4}(QP_{13}^{\otimes 6})^{> 0}(\overline{\omega}_{(i)})\big)$, one gets 
$$ \begin{array}{ll}
\dim QP_{13}^{\otimes 6} &= \sum_{1\leq i\leq 4}\big(\dim (QP_{13}^{\otimes 6})^{0}(\overline{\omega}_{(i)}) + \dim (QP_{13}^{\otimes 6})^{>0}(\overline{\omega}_{(i)})\big)\\
&  = 765 + 60 + 15 +  240 + 144 + 30 + 40 = 1294.
\end{array}$$
\end{rema}

To determine the behavior of the sixth transfer in degree $13,$ we need to compute the domain of this transfer in respective degree. A simple computation shows that the invariants $[QP_{13}^{\otimes 6}(\overline{\omega}_{(i)})]^{GL_6} = 0$ for all $i,\,1\leq i\leq 4.$ For instance, we will describe the space $[QP_{13}^{\otimes 6}(\overline{\omega}_{(2)})]^{GL_6}$ in detail. The other spaces can be obtained by similar calculations. Note that $QP_{13}^{\otimes 6}(\overline{\omega}_{(2)}) = (QP_{13}^{\otimes 6})^{>0}(\overline{\omega}_{(2)})$ and that by Moetele, and Mothebe \cite{MM}, $QP_{13}^{\otimes 6}(\overline{\omega}_{(2)})$ has a basis consisting of all the classes represented by the following 15 admissible monomials:

\begin{center}
\begin{tabular}{lllr}
${\rm adm}_{61}=t_1t_2^{2}t_3^{2}t_4^{2}t_5^{3}t_6^{3}$, & ${\rm adm}_{62}=t_1t_2^{2}t_3^{2}t_4^{3}t_5^{2}t_6^{3}$, & ${\rm adm}_{63}=t_1t_2^{2}t_3^{2}t_4^{3}t_5^{3}t_6^{2}$, & \multicolumn{1}{l}{${\rm adm}_{64}=t_1t_2^{2}t_3^{3}t_4^{2}t_5^{2}t_6^{3}$,} \\
${\rm adm}_{65}=t_1t_2^{2}t_3^{3}t_4^{2}t_5^{3}t_6^{2}$, & ${\rm adm}_{66}=t_1t_2^{2}t_3^{3}t_4^{3}t_5^{2}t_6^{2}$, & ${\rm adm}_{67}=t_1t_2^{3}t_3^{2}t_4^{2}t_5^{2}t_6^{3}$, & \multicolumn{1}{l}{${\rm adm}_{68}=t_1t_2^{3}t_3^{2}t_4^{2}t_5^{3}t_6^{2}$,} \\
${\rm adm}_{69}=t_1t_2^{3}t_3^{2}t_4^{3}t_5^{2}t_6^{2}$, & ${\rm adm}_{70}=t_1t_2^{3}t_3^{3}t_4^{2}t_5^{2}t_6^{2}$, & ${\rm adm}_{71}=t_1^{3}t_2t_3^{2}t_4^{2}t_5^{2}t_6^{3}$, & \multicolumn{1}{l}{${\rm adm}_{72}=t_1^{3}t_2t_3^{2}t_4^{2}t_5^{3}t_6^{2}$,} \\
${\rm adm}_{73}=t_1^{3}t_2t_3^{2}t_4^{3}t_5^{2}t_6^{2}$, & ${\rm adm}_{74}=t_1^{3}t_2t_3^{3}t_4^{2}t_5^{2}t_6^{2}$, & ${\rm adm}_{75}=t_1^{3}t_2^{3}t_3t_4^{2}t_5^{2}t_6^{2}$. &  
\end{tabular}
\end{center}

Suppose that $[f]_{\overline{\omega}_{(2)}}\in [QP_{13}^{\otimes 6}(\overline{\omega}_{(2)})]^{S_6},$ then $f\equiv_{\overline{\omega}_{(2)}}\sum_{61\leq j\leq 75}\gamma_j{\rm adm}_j,$ in which $\gamma_j$ belongs to $\mathbb Z/2$ for every $j.$ Using the homomorphisms $\theta_i: P^{\otimes 6}\to P^{\otimes 6}$ and the relations $\theta_i(f)+ f\equiv_{\overline{\omega}_{(2)}} 0,$ for all $i,\, 1\leq i\leq 5,$ we get $\gamma_{61} = \gamma_{62} = \cdots = \gamma_{75}.$ This means that $[QP_{13}^{\otimes 6}(\overline{\omega}_{(2)})]^{S_6} = \langle [\sum_{61\leq i\leq 75}{\rm adm}_j]_{\overline{\omega}_{(2)}} \rangle.$ Now,  since $S_6\subset GL_6,$ if $[g]_{\overline{\omega}_{(2)}}\in [QP_{13}^{\otimes 6}(\overline{\omega}_{(2)})]^{GL_6},$ then $g\equiv_{\overline{\omega}_{(2)}}\sum_{61\leq j\leq 75} \beta{\rm adm}_j,$ in which $\beta\in \mathbb Z/2.$ By the relation $\theta_6(g) + g\equiv_{\overline{\omega}_{(2)}} 0,$ we get $\beta(\sum_{68\leq j\leq 70}{\rm adm}_j)  \equiv_{\overline{\omega}_{(2)}} 0,$ which implies that $\beta = 0.$ This indicates that $[QP_{13}^{\otimes 6}(\overline{\omega}_{(2)})]^{GL_6}$ is trivial. 

Thus, as well known, $$\dim \mathbb Z/2 \otimes_{GL_6} {\rm Ann}_{\overline{\mathcal A}}[P_{13}^{\otimes 6}]^{*} = \dim [QP_{13}^{\otimes 6}]^{GL_6}\leq \sum_{1\leq i\leq 4}\dim [QP_{13}^{\otimes 6}(\overline{\omega}_{(i)})]^{GL_6},$$
and therefore $\mathbb Z/2 \otimes_{GL_6} {\rm Ann}_{\overline{\mathcal A}}[P_{13}^{\otimes 6}]^{*}$ is zero.

Now, using similar idea and the previous results of Peterson \cite{Peterson}, Kameko \cite{Kameko}, Moetele, and Mothebe \cite{MM}, Mothebe, Kaelo, and Ramatebele \cite{MKR}, Sum \cite{Sum1, Sum2, Sum, Sum4}, Tin \cite{Tin, Tin3}, the present writer \cite{Phuc5, Phuc6}, we have the technical theorem below.

\begin{thm}\label{dlsc6}
Let $n$ be a positive integer with $14\leq n\leq 25.$ Then, the dimension of $QP_{n}^{\otimes 6}$ is determined as follows:\\[2mm]
\centerline{\begin{tabular}{c|ccccccccccccccc}
$n$ &$14$ & $15$ & $16$ & $17$ & $18$ & $19$ & $20$ &$21$ & $22$ & $23$ & $24$ & $25$ \cr
\hline
\ $\dim QP_{n}^{\otimes 6}$ & $1660$ & $2184$ & $2451$ &$3135$ &$3941$ & $4998$ & $4622$&$5774$ & $6760$ & $8444$ & $8127$  & $9920$  \cr
\end{tabular}}
\end{thm}

The application of Theorem \ref{dlsc6} yields the next technical result.
\begin{thm}
The domains of the sixth transfer in degrees $n,\, 14\leq n\leq 25$ are given by\\[2mm]
\centerline{
\scalebox{0.9}{
\begin{tabular}{c|ccccccccccccc}
$n$  &$14$ & $15$ & $16$ & $17$ & $18$ & $19$ & $20$ &$21$ & $22$ & $23$ & $24$ & $25$ \cr
\hline
\ $\mathbb Z/2 \otimes_{GL_6} {\rm Ann}_{\overline{\mathcal A}}[P_{n}^{\otimes 6}]^{*}=$  & $\langle [\zeta_{14}] \rangle$ &  $\langle [\zeta_{15}] \rangle$ &  $\langle [\zeta_{16}] \rangle$ &  $\langle [\zeta_{17}] \rangle$ &$0$ & $0$ & $ \langle [\zeta_{20}] \rangle$&$0$ & $0$ & $ \langle [\zeta_{23}] \rangle$ & $0$  & $0$  \cr
\end{tabular}}}\\[4mm]
where the generators $\zeta_{14},\, \zeta_{15},\, \zeta_{16},\, \zeta_{17},\, \zeta_{20},\, \zeta_{23}$ are determined as follows:
$$ \begin{array}{ll}
\zeta_{14}  &=x_1^{(0)}x_2^{(0)}x_3^{(6)}x_4^{(5)}x_5^{(2)}x_6^{(1)}+
x_1^{(0)}x_2^{(0)}x_3^{(6)}x_4^{(4)}x_5^{(3)}x_6^{(1)}+
x_1^{(0)}x_2^{(0)}x_3^{(6)}x_4^{(3)}x_5^{(4)}x_6^{(1)}+
\medskip
x_1^{(0)}x_2^{(0)}x_3^{(6)}x_4^{(2)}x_5^{(5)}x_6^{(1)}\\
&+
x_1^{(0)}x_2^{(0)}x_3^{(6)}x_4^{(1)}x_5^{(6)}x_6^{(1)}+
x_1^{(0)}x_2^{(0)}x_3^{(5)}x_4^{(6)}x_5^{(1)}x_6^{(2)}+
x_1^{(0)}x_2^{(0)}x_3^{(5)}x_4^{(5)}x_5^{(2)}x_6^{(2)}+
\medskip
x_1^{(0)}x_2^{(0)}x_3^{(5)}x_4^{(4)}x_5^{(3)}x_6^{(2)}\\
&+
x_1^{(0)}x_2^{(0)}x_3^{(5)}x_4^{(3)}x_5^{(4)}x_6^{(2)}+
x_1^{(0)}x_2^{(0)}x_3^{(5)}x_4^{(2)}x_5^{(5)}x_6^{(2)}+
x_1^{(0)}x_2^{(0)}x_3^{(5)}x_4^{(1)}x_5^{(6)}x_6^{(2)}+
\medskip
x_1^{(0)}x_2^{(0)}x_3^{(5)}x_4^{(5)}x_5^{(1)}x_6^{(3)}\\
&+
x_1^{(0)}x_2^{(0)}x_3^{(3)}x_4^{(6)}x_5^{(2)}x_6^{(3)}+
x_1^{(0)}x_2^{(0)}x_3^{(6)}x_4^{(2)}x_5^{(3)}x_6^{(3)}+
x_1^{(0)}x_2^{(0)}x_3^{(6)}x_4^{(1)}x_5^{(4)}x_6^{(3)}+
\medskip
x_1^{(0)}x_2^{(0)}x_3^{(5)}x_4^{(2)}x_5^{(4)}x_6^{(3)}\\
&+
x_1^{(0)}x_2^{(0)}x_3^{(3)}x_4^{(4)}x_5^{(4)}x_6^{(3)}+
x_1^{(0)}x_2^{(0)}x_3^{(3)}x_4^{(2)}x_5^{(6)}x_6^{(3)}+
x_1^{(0)}x_2^{(0)}x_3^{(3)}x_4^{(6)}x_5^{(1)}x_6^{(4)}+
\medskip
x_1^{(0)}x_2^{(0)}x_3^{(3)}x_4^{(5)}x_5^{(2)}x_6^{(4)}\\
\end{array}$$

\newpage
$$ \begin{array}{ll}
&+
x_1^{(0)}x_2^{(0)}x_3^{(3)}x_4^{(4)}x_5^{(3)}x_6^{(4)}+
x_1^{(0)}x_2^{(0)}x_3^{(3)}x_4^{(3)}x_5^{(4)}x_6^{(4)}+
x_1^{(0)}x_2^{(0)}x_3^{(3)}x_4^{(2)}x_5^{(5)}x_6^{(4)}+
\medskip
x_1^{(0)}x_2^{(0)}x_3^{(3)}x_4^{(1)}x_5^{(6)}x_6^{(4)}\\
&+
x_1^{(0)}x_2^{(0)}x_3^{(5)}x_4^{(3)}x_5^{(1)}x_6^{(5)}+
x_1^{(0)}x_2^{(0)}x_3^{(6)}x_4^{(1)}x_5^{(2)}x_6^{(5)}+
x_1^{(0)}x_2^{(0)}x_3^{(5)}x_4^{(2)}x_5^{(2)}x_6^{(5)}+
\medskip
x_1^{(0)}x_2^{(0)}x_3^{(3)}x_4^{(4)}x_5^{(2)}x_6^{(5)}\\
&+
x_1^{(0)}x_2^{(0)}x_3^{(5)}x_4^{(1)}x_5^{(3)}x_6^{(5)}+
x_1^{(0)}x_2^{(0)}x_3^{(3)}x_4^{(3)}x_5^{(3)}x_6^{(5)}+
x_1^{(0)}x_2^{(0)}x_3^{(3)}x_4^{(1)}x_5^{(5)}x_6^{(5)}+
\medskip
x_1^{(0)}x_2^{(0)}x_3^{(6)}x_4^{(1)}x_5^{(1)}x_6^{(6)}\\
&+
x_1^{(0)}x_2^{(0)}x_3^{(5)}x_4^{(2)}x_5^{(1)}x_6^{(6)}+
x_1^{(0)}x_2^{(0)}x_3^{(3)}x_4^{(4)}x_5^{(1)}x_6^{(6)}+
x_1^{(0)}x_2^{(0)}x_3^{(3)}x_4^{(3)}x_5^{(2)}x_6^{(6)}+
\medskip
x_1^{(0)}x_2^{(0)}x_3^{(6)}x_4^{(6)}x_5^{(1)}x_6^{(1)},\\
\medskip
\zeta_{15} &= x_1^{(0)}x_2^{(0)}x_3^{(0)}x_4^{(0)}x_5^{(0)}x_6^{(15)},\\
\zeta_{16} &=x_1^{(1)}x_2^{(1)}x_3^{(6)}x_4^{(5)}x_5^{(2)}x_6^{(1)}+
x_1^{(1)}x_2^{(1)}x_3^{(6)}x_4^{(4)}x_5^{(3)}x_6^{(1)}+
x_1^{(1)}x_2^{(1)}x_3^{(6)}x_4^{(3)}x_5^{(4)}x_6^{(1)}+
\medskip
x_1^{(1)}x_2^{(1)}x_3^{(6)}x_4^{(2)}x_5^{(5)}x_6^{(1)}\\
&+
x_1^{(1)}x_2^{(1)}x_3^{(6)}x_4^{(1)}x_5^{(6)}x_6^{(1)}+
x_1^{(1)}x_2^{(1)}x_3^{(5)}x_4^{(6)}x_5^{(1)}x_6^{(2)}+
x_1^{(1)}x_2^{(1)}x_3^{(5)}x_4^{(5)}x_5^{(2)}x_6^{(2)}+
\medskip
x_1^{(1)}x_2^{(1)}x_3^{(5)}x_4^{(4)}x_5^{(3)}x_6^{(2)}\\
&+
x_1^{(1)}x_2^{(1)}x_3^{(5)}x_4^{(3)}x_5^{(4)}x_6^{(2)}+
x_1^{(1)}x_2^{(1)}x_3^{(5)}x_4^{(2)}x_5^{(5)}x_6^{(2)}+
x_1^{(1)}x_2^{(1)}x_3^{(5)}x_4^{(1)}x_5^{(6)}x_6^{(2)}+
\medskip
x_1^{(1)}x_2^{(1)}x_3^{(5)}x_4^{(5)}x_5^{(1)}x_6^{(3)}\\
&+
x_1^{(1)}x_2^{(1)}x_3^{(3)}x_4^{(6)}x_5^{(2)}x_6^{(3)}+
x_1^{(1)}x_2^{(1)}x_3^{(6)}x_4^{(2)}x_5^{(3)}x_6^{(3)}+
x_1^{(1)}x_2^{(1)}x_3^{(6)}x_4^{(1)}x_5^{(4)}x_6^{(3)}+
\medskip
x_1^{(1)}x_2^{(1)}x_3^{(5)}x_4^{(2)}x_5^{(4)}x_6^{(3)}\\
&+
x_1^{(1)}x_2^{(1)}x_3^{(3)}x_4^{(4)}x_5^{(4)}x_6^{(3)}+
x_1^{(1)}x_2^{(1)}x_3^{(3)}x_4^{(2)}x_5^{(6)}x_6^{(3)}+
x_1^{(1)}x_2^{(1)}x_3^{(3)}x_4^{(6)}x_5^{(1)}x_6^{(4)}+
\medskip
x_1^{(1)}x_2^{(1)}x_3^{(3)}x_4^{(5)}x_5^{(2)}x_6^{(4)}\\
&+
x_1^{(1)}x_2^{(1)}x_3^{(3)}x_4^{(4)}x_5^{(3)}x_6^{(4)}+
x_1^{(1)}x_2^{(1)}x_3^{(3)}x_4^{(3)}x_5^{(4)}x_6^{(4)}+
x_1^{(1)}x_2^{(1)}x_3^{(3)}x_4^{(2)}x_5^{(5)}x_6^{(4)}+
\medskip
x_1^{(1)}x_2^{(1)}x_3^{(3)}x_4^{(1)}x_5^{(6)}x_6^{(4)}\\
&+
x_1^{(1)}x_2^{(1)}x_3^{(5)}x_4^{(3)}x_5^{(1)}x_6^{(5)}+
x_1^{(1)}x_2^{(1)}x_3^{(6)}x_4^{(1)}x_5^{(2)}x_6^{(5)}+
x_1^{(1)}x_2^{(1)}x_3^{(5)}x_4^{(2)}x_5^{(2)}x_6^{(5)}+
\medskip
x_1^{(1)}x_2^{(1)}x_3^{(3)}x_4^{(4)}x_5^{(2)}x_6^{(5)}\\
&+
x_1^{(1)}x_2^{(1)}x_3^{(5)}x_4^{(1)}x_5^{(3)}x_6^{(5)}+
x_1^{(1)}x_2^{(1)}x_3^{(3)}x_4^{(3)}x_5^{(3)}x_6^{(5)}+
x_1^{(1)}x_2^{(1)}x_3^{(3)}x_4^{(1)}x_5^{(5)}x_6^{(5)}+
\medskip
x_1^{(1)}x_2^{(1)}x_3^{(6)}x_4^{(1)}x_5^{(1)}x_6^{(6)}\\
&+
x_1^{(1)}x_2^{(1)}x_3^{(5)}x_4^{(2)}x_5^{(1)}x_6^{(6)}+
x_1^{(1)}x_2^{(1)}x_3^{(3)}x_4^{(4)}x_5^{(1)}x_6^{(6)}+
x_1^{(1)}x_2^{(1)}x_3^{(3)}x_4^{(3)}x_5^{(2)}x_6^{(6)}+
\medskip
x_1^{(1)}x_2^{(1)}x_3^{(6)}x_4^{(6)}x_5^{(1)}x_6^{(1)},\\
\zeta_{17} &= x_1^{(0)}x_2^{(0)}x_3^{(5)}x_4^{(5)}x_5^{(5)}x_6^{(2)}+
 x_1^{(0)}x_2^{(0)}x_3^{(5)}x_4^{(5)}x_5^{(6)}x_6^{(1)}+
 x_1^{(0)}x_2^{(0)}x_3^{(3)}x_4^{(5)}x_5^{(8)}x_6^{(1)}+
\medskip
 x_1^{(0)}x_2^{(0)}x_3^{(5)}x_4^{(3)}x_5^{(8)}x_6^{(1)}\\
&+
 x_1^{(0)}x_2^{(0)}x_3^{(3)}x_4^{(6)}x_5^{(7)}x_6^{(1)}+
 x_1^{(0)}x_2^{(0)}x_3^{(5)}x_4^{(7)}x_5^{(4)}x_6^{(1)}+
x_1^{(0)}x_2^{(0)}x_3^{(7)}x_4^{(5)}x_5^{(4)}x_6^{(1)} + 
\medskip
 x_1^{(0)}x_2^{(0)}x_3^{(3)}x_4^{(9)}x_5^{(4)}x_6^{(1)}\\
&+
 x_1^{(0)}x_2^{(0)}x_3^{(9)}x_4^{(3)}x_5^{(4)}x_6^{(1)}+
 x_1^{(0)}x_2^{(0)}x_3^{(3)}x_4^{(9)}x_5^{(3)}x_6^{(2)}+
 x_1^{(0)}x_2^{(0)}x_3^{(9)}x_4^{(3)}x_5^{(3)}x_6^{(2)}+
\medskip
x_1^{(0)}x_2^{(0)}x_3^{(5)}x_4^{(9)}x_5^{(2)}x_6^{(1)}\\
&+
 x_1^{(0)}x_2^{(0)}x_3^{(9)}x_4^{(5)}x_5^{(2)}x_6^{(1)} +
x_1^{(0)}x_2^{(0)}x_3^{(5)}x_4^{(10)}x_5^{(1)}x_6^{(1)}+
 x_1^{(0)}x_2^{(0)}x_3^{(9)}x_4^{(6)}x_5^{(1)}x_6^{(1)}+
\medskip
 x_1^{(0)}x_2^{(0)}x_3^{(3)}x_4^{(11)}x_5^{(2)}x_6^{(1)}\\
&+
 x_1^{(0)}x_2^{(0)}x_3^{(11)}x_4^{(3)}x_5^{(2)}x_6^{(1)} +
 x_1^{(0)}x_2^{(0)}x_3^{(5)}x_4^{(5)}x_5^{(3)}x_6^{(4)}+
 x_1^{(0)}x_2^{(0)}x_3^{(5)}x_4^{(3)}x_5^{(5)}x_6^{(4)}+
\medskip
 x_1^{(0)}x_2^{(0)}x_3^{(3)}x_4^{(5)}x_5^{(5)}x_6^{(4)}\\
&+
 x_1^{(0)}x_2^{(0)}x_3^{(3)}x_4^{(12)}x_5^{(1)}x_6^{(1)}+
 x_1^{(0)}x_2^{(0)}x_3^{(11)}x_4^{(4)}x_5^{(1)}x_6^{(1)}+
 x_1^{(0)}x_2^{(0)}x_3^{(7)}x_4^{(8)}x_5^{(1)}x_6^{(1)}+
\medskip
 x_1^{(0)}x_2^{(0)}x_3^{(7)}x_4^{(7)}x_5^{(1)}x_6^{(2)}\\
&+  
x_1^{(0)}x_2^{(0)}x_3^{(13)}x_4^{(2)}x_5^{(1)}x_6^{(1)} +
 x_1^{(0)}x_2^{(0)}x_3^{(14)}x_4^{(1)}x_5^{(1)}x_6^{(1)}+
 x_1^{(0)}x_2^{(0)}x_3^{(6)}x_4^{(5)}x_5^{(3)}x_6^{(3)}+
\medskip
 x_1^{(0)}x_2^{(0)}x_3^{(5)}x_4^{(3)}x_5^{(6)}x_6^{(3)}\\
&+
 x_1^{(0)}x_2^{(0)}x_3^{(3)}x_4^{(6)}x_5^{(5)}x_6^{(3)} + 
x_1^{(0)}x_2^{(0)}x_3^{(6)}x_4^{(3)}x_5^{(3)}x_6^{(5)}+
 x_1^{(0)}x_2^{(0)}x_3^{(3)}x_4^{(3)}x_5^{(6)}x_6^{(5)}+
\medskip
 x_1^{(0)}x_2^{(0)}x_3^{(3)}x_4^{(6)}x_5^{(3)}x_6^{(5)}\\
&+
x_1^{(0)}x_2^{(0)}x_3^{(5)}x_4^{(3)}x_5^{(3)}x_6^{(6)}+
 x_1^{(0)}x_2^{(0)}x_3^{(3)}x_4^{(5)}x_5^{(3)}x_6^{(6)}+
 x_1^{(0)}x_2^{(0)}x_3^{(3)}x_4^{(3)}x_5^{(5)}x_6^{(6)}+
\medskip
 x_1^{(0)}x_2^{(0)}x_3^{(3)}x_4^{(3)}x_5^{(3)}x_6^{(8)}\\
&+ 
x_1^{(0)}x_2^{(0)}x_3^{(3)}x_4^{(3)}x_5^{(4)}x_6^{(7)} +
x_1^{(0)}x_2^{(0)}x_3^{(3)}x_4^{(5)}x_5^{(2)}x_6^{(7)}+
 x_1^{(0)}x_2^{(0)}x_3^{(3)}x_4^{(6)}x_5^{(1)}x_6^{(7)}+
\medskip
x_1^{(0)}x_2^{(0)}x_3^{(3)}x_4^{(3)}x_5^{(9)}x_6^{(2)}\\
&+
 x_1^{(0)}x_2^{(0)}x_3^{(3)}x_4^{(3)}x_5^{(10)}x_6^{(1)} + 
x_1^{(0)}x_2^{(0)}x_3^{(5)}x_4^{(3)}x_5^{(7)}x_6^{(2)}+
 x_1^{(0)}x_2^{(0)}x_3^{(5)}x_4^{(7)}x_5^{(3)}x_6^{(2)}+
\medskip
 x_1^{(0)}x_2^{(0)}x_3^{(7)}x_4^{(5)}x_5^{(3)}x_6^{(2)},\\
\zeta_{20}  &=x_1^{(3)}x_2^{(3)}x_3^{(6)}x_4^{(5)}x_5^{(2)}x_6^{(1)}+
x_1^{(3)}x_2^{(3)}x_3^{(6)}x_4^{(4)}x_5^{(3)}x_6^{(1)}+
x_1^{(3)}x_2^{(3)}x_3^{(6)}x_4^{(3)}x_5^{(4)}x_6^{(1)}+
\medskip
x_1^{(3)}x_2^{(3)}x_3^{(6)}x_4^{(2)}x_5^{(5)}x_6^{(1)}\\
&+
x_1^{(3)}x_2^{(3)}x_3^{(6)}x_4^{(1)}x_5^{(6)}x_6^{(1)}+
x_1^{(3)}x_2^{(3)}x_3^{(5)}x_4^{(6)}x_5^{(1)}x_6^{(2)}+
x_1^{(3)}x_2^{(3)}x_3^{(5)}x_4^{(5)}x_5^{(2)}x_6^{(2)}+
\medskip
x_1^{(3)}x_2^{(3)}x_3^{(5)}x_4^{(4)}x_5^{(3)}x_6^{(2)}\\
&+
x_1^{(3)}x_2^{(3)}x_3^{(5)}x_4^{(3)}x_5^{(4)}x_6^{(2)}+
x_1^{(3)}x_2^{(3)}x_3^{(5)}x_4^{(2)}x_5^{(5)}x_6^{(2)}+
x_1^{(3)}x_2^{(3)}x_3^{(5)}x_4^{(1)}x_5^{(6)}x_6^{(2)}+
\medskip
x_1^{(3)}x_2^{(3)}x_3^{(5)}x_4^{(5)}x_5^{(1)}x_6^{(3)}\\
&+
x_1^{(3)}x_2^{(3)}x_3^{(3)}x_4^{(6)}x_5^{(2)}x_6^{(3)}+
x_1^{(3)}x_2^{(3)}x_3^{(6)}x_4^{(2)}x_5^{(3)}x_6^{(3)}+
x_1^{(3)}x_2^{(3)}x_3^{(6)}x_4^{(1)}x_5^{(4)}x_6^{(3)}+
\medskip
x_1^{(3)}x_2^{(3)}x_3^{(5)}x_4^{(2)}x_5^{(4)}x_6^{(3)}\\
&+
x_1^{(3)}x_2^{(3)}x_3^{(3)}x_4^{(4)}x_5^{(4)}x_6^{(3)}+
x_1^{(3)}x_2^{(3)}x_3^{(3)}x_4^{(2)}x_5^{(6)}x_6^{(3)}+
x_1^{(3)}x_2^{(3)}x_3^{(3)}x_4^{(6)}x_5^{(1)}x_6^{(4)}+
\medskip
x_1^{(3)}x_2^{(3)}x_3^{(3)}x_4^{(5)}x_5^{(2)}x_6^{(4)}\\
&+
x_1^{(3)}x_2^{(3)}x_3^{(3)}x_4^{(4)}x_5^{(3)}x_6^{(4)}+
x_1^{(3)}x_2^{(3)}x_3^{(3)}x_4^{(3)}x_5^{(4)}x_6^{(4)}+
x_1^{(3)}x_2^{(3)}x_3^{(3)}x_4^{(2)}x_5^{(5)}x_6^{(4)}+
\medskip
x_1^{(3)}x_2^{(3)}x_3^{(3)}x_4^{(1)}x_5^{(6)}x_6^{(4)}\\
\end{array}$$

\newpage
$$ \begin{array}{ll}
&+
x_1^{(3)}x_2^{(3)}x_3^{(5)}x_4^{(3)}x_5^{(1)}x_6^{(5)}+
x_1^{(3)}x_2^{(3)}x_3^{(6)}x_4^{(1)}x_5^{(2)}x_6^{(5)}+
x_1^{(3)}x_2^{(3)}x_3^{(5)}x_4^{(2)}x_5^{(2)}x_6^{(5)}+
\medskip
x_1^{(3)}x_2^{(3)}x_3^{(3)}x_4^{(4)}x_5^{(2)}x_6^{(5)}\\
&+
x_1^{(3)}x_2^{(3)}x_3^{(5)}x_4^{(1)}x_5^{(3)}x_6^{(5)}+
x_1^{(3)}x_2^{(3)}x_3^{(3)}x_4^{(3)}x_5^{(3)}x_6^{(5)}+
x_1^{(3)}x_2^{(3)}x_3^{(3)}x_4^{(1)}x_5^{(5)}x_6^{(5)}+
\medskip
x_1^{(3)}x_2^{(3)}x_3^{(6)}x_4^{(1)}x_5^{(1)}x_6^{(6)}\\
&+
x_1^{(3)}x_2^{(3)}x_3^{(5)}x_4^{(2)}x_5^{(1)}x_6^{(6)}+
x_1^{(3)}x_2^{(3)}x_3^{(3)}x_4^{(4)}x_5^{(1)}x_6^{(6)}+
x_1^{(3)}x_2^{(3)}x_3^{(3)}x_4^{(3)}x_5^{(2)}x_6^{(6)}+
\medskip
x_1^{(3)}x_2^{(3)}x_3^{(6)}x_4^{(6)}x_5^{(1)}x_6^{(1)},\\
\zeta_{23} &= x_1^{(3)}x_2^{(3)}x_3^{(5)}x_4^{(5)}x_5^{(5)}x_6^{(2)}+
 x_1^{(3)}x_2^{(3)}x_3^{(5)}x_4^{(5)}x_5^{(6)}x_6^{(1)}+
 x_1^{(3)}x_2^{(3)}x_3^{(3)}x_4^{(5)}x_5^{(8)}x_6^{(1)}+
\medskip
 x_1^{(3)}x_2^{(3)}x_3^{(5)}x_4^{(3)}x_5^{(8)}x_6^{(1)}\\
&+
 x_1^{(3)}x_2^{(3)}x_3^{(3)}x_4^{(6)}x_5^{(7)}x_6^{(1)}+
 x_1^{(3)}x_2^{(3)}x_3^{(5)}x_4^{(7)}x_5^{(4)}x_6^{(1)}+
x_1^{(3)}x_2^{(3)}x_3^{(7)}x_4^{(5)}x_5^{(4)}x_6^{(1)} + 
\medskip
 x_1^{(3)}x_2^{(3)}x_3^{(3)}x_4^{(9)}x_5^{(4)}x_6^{(1)}\\
&+
 x_1^{(3)}x_2^{(3)}x_3^{(9)}x_4^{(3)}x_5^{(4)}x_6^{(1)}+
 x_1^{(3)}x_2^{(3)}x_3^{(3)}x_4^{(9)}x_5^{(3)}x_6^{(2)}+
 x_1^{(3)}x_2^{(3)}x_3^{(9)}x_4^{(3)}x_5^{(3)}x_6^{(2)}+
\medskip
x_1^{(3)}x_2^{(3)}x_3^{(5)}x_4^{(9)}x_5^{(2)}x_6^{(1)}\\
&+
 x_1^{(3)}x_2^{(3)}x_3^{(9)}x_4^{(5)}x_5^{(2)}x_6^{(1)} +
x_1^{(3)}x_2^{(3)}x_3^{(5)}x_4^{(10)}x_5^{(1)}x_6^{(1)}+
 x_1^{(3)}x_2^{(3)}x_3^{(9)}x_4^{(6)}x_5^{(1)}x_6^{(1)}+
\medskip
 x_1^{(3)}x_2^{(3)}x_3^{(3)}x_4^{(11)}x_5^{(2)}x_6^{(1)}\\
&+
 x_1^{(3)}x_2^{(3)}x_3^{(11)}x_4^{(3)}x_5^{(2)}x_6^{(1)} +
 x_1^{(3)}x_2^{(3)}x_3^{(5)}x_4^{(5)}x_5^{(3)}x_6^{(4)}+
 x_1^{(3)}x_2^{(3)}x_3^{(5)}x_4^{(3)}x_5^{(5)}x_6^{(4)}+
\medskip
 x_1^{(3)}x_2^{(3)}x_3^{(3)}x_4^{(5)}x_5^{(5)}x_6^{(4)}\\
&+
 x_1^{(3)}x_2^{(3)}x_3^{(3)}x_4^{(12)}x_5^{(1)}x_6^{(1)}+
 x_1^{(3)}x_2^{(3)}x_3^{(11)}x_4^{(4)}x_5^{(1)}x_6^{(1)}+
 x_1^{(3)}x_2^{(3)}x_3^{(7)}x_4^{(8)}x_5^{(1)}x_6^{(1)}+
\medskip
 x_1^{(3)}x_2^{(3)}x_3^{(7)}x_4^{(7)}x_5^{(1)}x_6^{(2)}\\
&+  
x_1^{(3)}x_2^{(3)}x_3^{(13)}x_4^{(2)}x_5^{(1)}x_6^{(1)} +
 x_1^{(3)}x_2^{(3)}x_3^{(14)}x_4^{(1)}x_5^{(1)}x_6^{(1)}+
 x_1^{(3)}x_2^{(3)}x_3^{(6)}x_4^{(5)}x_5^{(3)}x_6^{(3)}+
\medskip
 x_1^{(3)}x_2^{(3)}x_3^{(5)}x_4^{(3)}x_5^{(6)}x_6^{(3)}\\
&+
 x_1^{(3)}x_2^{(3)}x_3^{(3)}x_4^{(6)}x_5^{(5)}x_6^{(3)} + 
x_1^{(3)}x_2^{(3)}x_3^{(6)}x_4^{(3)}x_5^{(3)}x_6^{(5)}+
 x_1^{(3)}x_2^{(3)}x_3^{(3)}x_4^{(3)}x_5^{(6)}x_6^{(5)}+
\medskip
 x_1^{(3)}x_2^{(3)}x_3^{(3)}x_4^{(6)}x_5^{(3)}x_6^{(5)}\\
&+
x_1^{(3)}x_2^{(3)}x_3^{(5)}x_4^{(3)}x_5^{(3)}x_6^{(6)}+
 x_1^{(3)}x_2^{(3)}x_3^{(3)}x_4^{(5)}x_5^{(3)}x_6^{(6)}+
 x_1^{(3)}x_2^{(3)}x_3^{(3)}x_4^{(3)}x_5^{(5)}x_6^{(6)}+
\medskip
 x_1^{(3)}x_2^{(3)}x_3^{(3)}x_4^{(3)}x_5^{(3)}x_6^{(8)}\\
&+ 
x_1^{(3)}x_2^{(3)}x_3^{(3)}x_4^{(3)}x_5^{(4)}x_6^{(7)} +
x_1^{(3)}x_2^{(3)}x_3^{(3)}x_4^{(5)}x_5^{(2)}x_6^{(7)}+
 x_1^{(3)}x_2^{(3)}x_3^{(3)}x_4^{(6)}x_5^{(1)}x_6^{(7)}+
\medskip
x_1^{(3)}x_2^{(3)}x_3^{(3)}x_4^{(3)}x_5^{(9)}x_6^{(2)}\\
&+
 x_1^{(3)}x_2^{(3)}x_3^{(3)}x_4^{(3)}x_5^{(10)}x_6^{(1)} + 
x_1^{(3)}x_2^{(3)}x_3^{(5)}x_4^{(3)}x_5^{(7)}x_6^{(2)}+
 x_1^{(3)}x_2^{(3)}x_3^{(5)}x_4^{(7)}x_5^{(3)}x_6^{(2)}+
\medskip
 x_1^{(3)}x_2^{(3)}x_3^{(7)}x_4^{(5)}x_5^{(3)}x_6^{(2)}.
\end{array}$$
\end{thm}

Applying the representation in the lambda algebra of $Tr_6^{\mathcal A},$ one gets
$$ \begin{array}{ll}
Tr_6^{\mathcal A}([\zeta_{14}]) &= [\psi_6(\zeta_{14})] = [\lambda^{2}_0\lambda_3^2\lambda_2\lambda_6 + \lambda^{2}_0\lambda_3^2\lambda_4^2 + \lambda^{2}_0\lambda_3\lambda_5\lambda_4\lambda_2 + \lambda^{2}_0\lambda_7\lambda_1\lambda_5\lambda_1 + \delta(\lambda^{2}_0\lambda_3^2\lambda_9 + \lambda^{2}_0\lambda_3\lambda_9\lambda_3)]\\
\medskip
&= h_0^{2}d_0=h_2Ph_2\in {\rm Ext}_{\mathcal A}^{6, 6+14}(\mathbb Z/2, \mathbb Z/2),\\
\medskip
Tr_6^{\mathcal A}([\zeta_{15}]) &= [\psi_6(\zeta_{15})] = [\lambda_0^{5}\lambda_{15}] = h_0^{5}h_4\in {\rm Ext}_{\mathcal A}^{6, 6+15}(\mathbb Z/2, \mathbb Z/2),\\
Tr_6^{\mathcal A}([\zeta_{16}]) &= [\psi_6(\zeta_{16})] = [\lambda_1^{2}\lambda_3^2\lambda_2\lambda_6 + \lambda_1^{2}\lambda_3^2\lambda_4^2 + \lambda_1^{2}\lambda_3\lambda_5\lambda_4\lambda_2 + \lambda_1^{2}\lambda_7\lambda_1\lambda_5\lambda_1 + \delta(\lambda_1^{2}\lambda_3^2\lambda_9 + \lambda_1^{2}\lambda_3\lambda_9\lambda_3)]\\
\medskip
&=h_1^{2}d_0=h_3Ph_1 = c_0^{2}\in {\rm Ext}_{\mathcal A}^{6, 6+16}(\mathbb Z/2, \mathbb Z/2),\\
Tr_6^{\mathcal A}([\zeta_{17}]) &= [\psi_6(\zeta_{17})] = [\lambda_0^{2}\lambda_3^{3}\lambda_8 + \lambda_0^{2}(\lambda_3\lambda_5^{2} + \lambda_0^{2}\lambda_3^{2}\lambda_7)\lambda_4 + \lambda_0^{2}\lambda_7\lambda_5\lambda_3\lambda_2 + \lambda_0^{2}\lambda_3^{2}\lambda_5\lambda_6\\
& + \delta(\lambda_0^{2}\lambda_3\lambda_5\lambda_{10} + \lambda_0^{2}\lambda_3\lambda_{12}\lambda_3 + \lambda_0^{2}\lambda_4\lambda_7^{2} + \lambda^{3}_0\lambda_{11}\lambda_7)]\\
\medskip
&=h_0^{2}e_0 = h_0h_2d_0\in {\rm Ext}_{\mathcal A}^{6, 6+17}(\mathbb Z/2, \mathbb Z/2),\\
Tr_6^{\mathcal A}([\zeta_{20}]) &= [\psi_6(\zeta_{20})] = [\lambda_3^4\lambda_2\lambda_6 + \lambda_3^4\lambda_4^2 + \lambda_3^{3}\lambda_5\lambda_4\lambda_2 + \lambda_3^{2}\lambda_7\lambda_1\lambda_5\lambda_1 + \delta(\lambda_3^4\lambda_9 + \lambda_3^{3}\lambda_9\lambda_3)]\\
\medskip
&=h_2^{2}d_0=h_0h_2e_0 = h_0^{2}g_1\in {\rm Ext}_{\mathcal A}^{6, 6+20}(\mathbb Z/2, \mathbb Z/2),\\
Tr_6^{\mathcal A}([\zeta_{23}]) &= [\psi_6(\zeta_{23})] = [\lambda_3^{5}\lambda_8 + \lambda_3^{2}(\lambda_3\lambda_5^{2} + \lambda_0^{2}\lambda_3^{2}\lambda_7)\lambda_4 + \lambda_3^{2}\lambda_7\lambda_5\lambda_3\lambda_2 + \lambda_3^{4}\lambda_5\lambda_6\\
& + \delta(\lambda_3^{3}\lambda_5\lambda_{10} + \lambda_3^{3}\lambda_{12}\lambda_3 + \lambda_3^{2}\lambda_4\lambda_7^{2} + \lambda_3^{2}\lambda_0\lambda_{11}\lambda_7)]\\
\medskip
&=h_2^{2}e_0 = h_0h_2g_1\in {\rm Ext}_{\mathcal A}^{6, 6+23}(\mathbb Z/2, \mathbb Z/2).
\end{array}$$
On the other side, for each $13\leq n\leq 25,$ due to Bruner \cite{Bruner}, the elements $h_0^{2}d_0,$ $h_0^{5}h_4,$ $h_1^{2}d_0,$ $h_0^{2}e_0,$ $h_2^{2}d_0,$ and $h_2^{2}e_0$ are non-zero in ${\rm Ext}_{\mathcal A}^{6, 6+14}(\mathbb Z/2, \mathbb Z/2),$ ${\rm Ext}_{\mathcal A}^{6, 6+15}(\mathbb Z/2, \mathbb Z/2),$ ${\rm Ext}_{\mathcal A}^{6, 6+16}(\mathbb Z/2, \mathbb Z/2),$ ${\rm Ext}_{\mathcal A}^{6, 6+17}(\mathbb Z/2, \mathbb Z/2),$ ${\rm Ext}_{\mathcal A}^{6, 6+20}(\mathbb Z/2, \mathbb Z/2),$ and ${\rm Ext}_{\mathcal A}^{6, 6+23}(\mathbb Z/2, \mathbb Z/2),$ respectively. At the same time, ${\rm Ext}_{\mathcal A}^{6, 6+n}(\mathbb Z/2, \mathbb Z/2) = 0$ otherwise. Combining these data, we obtain the following corollary.
\begin{corl}
Let $n$ be a positive integer with $13\leq n\leq 25.$ Then, the sixth transfer
$$Tr_6^{\mathcal A}: \mathbb Z/2 \otimes_{GL_6} {\rm Ann}_{\overline{\mathcal A}}[P_{n}^{\otimes 6}]^{*}\to {\rm Ext}_{\mathcal A}^{6, 6+n}(\mathbb Z/2, \mathbb Z/2)$$
is an isomorphism.
\end{corl}

Next, we investigate the behavior of the sixth transfer in some internal degrees $\geq 26.$ More precisely, we consider generic degree $n_s = 6(2^{s}-1) + 10.2^{s},$ for all $s\geq 0.$ Let us recall that due to Mothebe, Kaelo, and Ramatebele \cite{MKR}, for $s = 0,$ the cohit $\mathbb Z/2$-module   $QP_{n_0}^{\otimes 6}= QP_{10}^{\otimes 6}$ has dimension $945.$ For $s > 0,$ we remark that $n_s = 2^{s+3} + 2^{s+2} + 2^{s+1} + 2^{s} + 2^{s-1} + 2^{s-1} - 6,$ so $\mu(n_s) = 6$ for arbitrary $s > 1,$ which implies that by Kameko's theorem \cite{Kameko}, the iterated homomorphism $((\widetilde {Sq^0_*})_{n_{s}})^{s-1}: QP^{\otimes 6}_{n_s}  \longrightarrow QP^{\otimes 6}_{n_1}$ is an isomorphism of $\mathbb Z/2[GL_6]$-modules, for every $s > 0.$ Thus, we need to compute the dimension of $QP_{n_1}^{\otimes 6} = QP_{26}^{\otimes 6}.$ We observe that the Kameko map $(\widetilde {Sq^0_*})_{26}: QP^{\otimes 6}_{26} \longrightarrow QP^{\otimes 6}_{10}$ is an epimorphism, and $\dim QP^{\otimes 6}_{10} = 945$; so to compute the dimension of $QP_{26}^{\otimes 6},$ we need only to determine the kernel of $(\widetilde {Sq^0_*})_{26}.$ 

Suppose that $t$ is an admissible monomial in $P^{\otimes 6}_{26}$ such that $[t]\in {\rm Ker}((\widetilde {Sq^0_*})_{26}).$ Since $z = t_1^{15}t_2^{7}t_3^{3}t_4\in P^{\otimes 6}_{26}$ is the minimal spike monomial, and $\omega(z) = (4,3,2,1),$ by Singer's criterion on hit polynomials, $\omega_1(t)\geq 4.$ Since $\deg(t)$ is even, either $\omega_1(t) = 4$ or $\omega_1(t) = 6.$ If $\omega_1(t) = 6,$ then $t = t_it_jt_kt_lt_mt_ny^2,$ with $1\leq i<j<k<l<m<n\leq 6,$ and $y\in P^{\otimes 6}_{10}.$ Since $t$ is admissible, by Kameko's criterion on inadmissible monomials, we must have that $y$ is admissible, which implies that $(\widetilde {Sq^0_*})_{26}([t]) = [y]\neq [0],$ which contradicts with the fact that $[t]\in {\rm Ker}((\widetilde {Sq^0_*})_{26}).$ So, $\omega_1(t) = 4,$ and $t$ is of the form $t_it_jt_kt_l\underline{t}^2,$ in which $1\leq i<j<k<l\leq 6,$ and $\underline{t}$ is an admissible monomial of degree $11$ in $\mathcal A$-module $P^{\otimes 6}.$ 
As computed above, $\omega(\underline{t})$ belongs to $\{(3,2,1),\, (3,4),\, (5,1,1),\, (5,3)\}.$ 
Therefore, the weight vector of $t$ is one of the following sequences: $$(4, 3,2,1),\ (4, 3,4),\ (4, 5,1,1),\ (4, 5,3).$$
Then, one has an isomorphism 
$$ \begin{array}{ll}
 {\rm Ker}((\widetilde {Sq^0_*})_{26})&\cong (QP_{26}^{\otimes 6})^{0}\bigoplus (QP_{26}^{\otimes 6})^{>0}(4,3,2,1)\bigoplus (QP_{26}^{\otimes 6})^{>0}(4,3,4)\\
&\quad\quad\quad\quad\bigoplus (QP_{26}^{\otimes 6})^{>0}(4,5,1,1)\bigoplus (QP_{26}^{\otimes 6})^{>0}(4,5,3).
\end{array}$$
According to Walker, and Wood \cite{Walker-Wood}, $QP_{26}^{\otimes 5}$ has dimension $1024.$ We remark that sine $\mu(26) = 4,$ by Sum \cite{Sum1,Sum2}, one gets $\dim (QP_{26}^{\otimes 5})^{0}  = 64\times \binom{5}{4} = 320.$ Since $QP_{26}^{\otimes 5}\cong (QP_{26}^{\otimes 5})^{0}\bigoplus (QP_{26}^{\otimes 5})^{>0},$ $\dim (QP_{26}^{\otimes 5})^{>0} = 1024 - 320 = 704.$ Then, the dimension of $(QP_{26}^{\otimes 6})^{0}$ reads as follows: $$\dim (QP_{26}^{\otimes 6})^{0} = 64\times \binom{6}{4} + 704\times \binom{6}{5} = 960 + 4224 = 5184.$$ 
By direct calculations, we find that 
$$ \begin{array}{ll}
\medskip
&\dim (QP_{26}^{\otimes 6})^{>0}(4,3,2,1)  = 2843, \ \ \dim (QP_{26}^{\otimes 6})^{>0}(4,3,4) = 247,\\
& \dim (QP_{26}^{\otimes 6})^{>0}(4,5,1,1) = 336,\ \ \dim (QP_{26}^{\otimes 6})^{>0}(4,5,3) = 210,
\end{array}$$
Thus, since $QP_{26}^{\otimes 6}\cong  {\rm Ker}((\widetilde {Sq^0_*})_{26})\bigoplus QP_{10}^{\otimes 6},$ one gets 
$$  \begin{array}{ll}
 \dim QP_{26}^{\otimes 6} &= \dim {\rm Ker}((\widetilde {Sq^0_*})_{26})+\dim QP_{10}^{\otimes 6}\\
& = \dim (QP_{26}^{\otimes 6})^{0}+\dim (QP_{26}^{\otimes 6})^{>0}(4,3,2,1)   + \dim (QP_{26}^{\otimes 6})^{>0}(4,3,4) \\
&\quad+  \dim (QP_{26}^{\otimes 6})^{>0}(4,5,1,1) + \dim (QP_{26}^{\otimes 6})^{>0}(4,5,3) + \dim QP_{10}^{\otimes 6}\\
&= 5184 + 2843+247+336+210 + 945 = 9765.
\end{array}$$
Applying this result, we obtain that the invariants $[(QP_{26}^{\otimes 6})(4,3,2,1)]^{GL_6},$ $[(QP_{26}^{\otimes 6})(4,3,4)]^{GL_6},$ $[(QP_{26}^{\otimes 6})(4,5,1,1)]^{GL_6},$ and $[(QP_{26}^{\otimes 6})(4,5,3)]^{GL_6}$ are zero, and so $[{\rm Ker}((\widetilde {Sq^0_*})_{26})]^{GL_6}$ is zero. Now, as shown above, $QP_{n_s}^{\otimes 6}\cong QP_{26}^{\otimes 6},$ for all $s\geq 1.$ Combining these with the fact that $[QP_{10}^{\otimes 6}]^{GL_6} = 0$, and $\dim [QP_{26}^{\otimes 6}]^{GL_6}\leq \dim [{\rm Ker}((\widetilde {Sq^0_*})_{26})]^{GL_6} + \dim [QP_{10}^{\otimes 6}]^{GL_6},$ we may state that 

\begin{thm}\label{dl26}
For each non-negative integer $s,$ the coinvariant $\mathbb Z/2 \otimes_{GL_6} {\rm Ann}_{\overline{\mathcal A}}[P_{n_s}^{\otimes 6}]^{*}$ is trivial.
\end{thm}
Now, it is well-known that (see, for instance, Bruner \cite{Bruner}), the elements $h_1Ph_1$, $h_2^{2}g_1 = h_4Ph_2$, and $D_2$ are non-zero in ${\rm Ext}_{\mathcal A}^{6, 6+n_0}(\mathbb Z/2, \mathbb Z/2)$, ${\rm Ext}_{\mathcal A}^{6, 6+n_1}(\mathbb Z/2, \mathbb Z/2)$, and ${\rm Ext}_{\mathcal A}^{6, 6+n_2}(\mathbb Z/2, \mathbb Z/2),$ respectively. Combining these data, one gets the corollary below.
\begin{corl}\label{hqpl1}
$Tr_6^{\mathcal A}$ does not detect the non-zero elements $h_1Ph_1$, $h_4Ph_2,$ and $D_2.$ Consequently, the transfer homomorphism  $$Tr_6^{\mathcal A}: \mathbb Z/2 \otimes_{GL_6} {\rm Ann}_{\overline{\mathcal A}}[P_{n_s}^{\otimes 6}]^{*}\to {\rm Ext}_{\mathcal A}^{6, 6+n_s}(\mathbb Z/2, \mathbb Z/2)$$ is not an epimorphism for $0\leq s\leq 2.$
\end{corl}
It should be noted that this statement for the element $h_1Ph_1$ has also been proved by Ch\ohorn n, and H\`a \cite{C.H1, C.H2} using another method. In addition, following Bruner \cite{Bruner}, ${\rm Ext}_{\mathcal A}^{6, 6+n_3}(\mathbb Z/2, \mathbb Z/2) =0.$ So, the following is an immediate consequence of the these results and Theorem \ref{dl26}.

\begin{corl}\label{hqpl2}
The sixth transfer $Tr_6^{\mathcal A}: \mathbb Z/2 \otimes_{GL_6} {\rm Ann}_{\overline{\mathcal A}}[P_{n_3}^{\otimes 6}]^{*}\to {\rm Ext}_{\mathcal A}^{6, 6+n_3}(\mathbb Z/2, \mathbb Z/2)$ is an isomorphism.
\end{corl} 

\subsection{The behavior of $Tr_7^{\mathcal A},$ and $Tr_8^{\mathcal A}$  in internal degrees $\leq 15$}\label{pl3}

This section is devoted to investigate the behavior of the Singer transfer of ranks 7 and 8 in internal degrees $\leq 15.$ First of all, we have the following remark.

\begin{rema}
The calculations by Moetele, and Mothebe \cite{MM} showed that $\dim QP_{13}^{\otimes 7} = 5406,$ and $\dim QP_{13}^{\otimes 8} = 18920.$ However, these results are not true. Of course, as shown above (see Remark \ref{nxM}), their calculations are incorrect for the dimension of $QP_{13}^{\otimes 6},$ and therefore, the results for the dimensions of $QP_{13}^{\otimes 7},$ and $QP_{13}^{\otimes 8}$ are, too. We shall prove that $QP_{13}^{\otimes 7},$ and $QP_{13}^{\otimes 8}$ have dimensions $5334,$ and $18592,$ respectively. Indeed, we remark that since the Kameko homomorphism $(\widetilde {Sq^0_*})_{13}: QP_{13}^{\otimes 7}\to QP_{3}^{\otimes 7}$ is an epimorphism of $\mathbb Z/2$-vector spaces, $QP_{13}^{\otimes 7}\cong {\rm Ker}((\widetilde {Sq^0_*})_{13})\bigoplus QP_{3}^{\otimes 7}.$ Since $QP_{3}^{\otimes 7} = (QP_{3}^{\otimes 7})^{0},$ $\dim QP_{3}^{\otimes 7} = \binom{7}{1} + \binom{7}{2}+ \binom{7}{3}  = 63.$ We now determine the kernel of $(\widetilde {Sq^0_*})_{13}.$ We see that if $t$ is an admissible monomial such that $[t]$ belongs to ${\rm Ker}((\widetilde {Sq^0_*})_{13}),$ then the weight vector $\omega(t)\in \{(3,3,1),\, (3, 5),\, (5, 2, 1),\, (5, 4)\}.$ This can be explained as follows: It is straightforward to see that $t_1^{7}t_2^{3}t_3^{3}\in P^{\otimes 7}_{13}$ is the minimal spike monomial, and $\omega(t_1^{7}t_2^{3}t_3^{3}) = (3,3,1).$ So, since $[t]\in {\rm Ker}((\widetilde {Sq^0_*})_{13}),\, [t]\neq [0],$ and $\deg(t)$ is odd, by Singer's criterion on hit polynomials, either $\omega_1(t) = 3$ or $\omega_1(t) = 5.$ Then, either $t = t_it_jt_ky^{2},\, 1\leq i<j<k\leq 7,\, y\in P_{5}^{\otimes 7}$ or $t = t_it_jt_kt_lt_my^{2},\, 1\leq i<j<k<l<m\leq 7,\, y\in P_{4}^{\otimes 7}.$ Since $t$ is admissible, due to Kameko's criterion on inadmissible monomials, $y$ is admissible. Since $QP_{4}^{\otimes 7} = (QP_{4}^{\otimes 7})^{0},$ and $QP_{5}^{\otimes 7} = (QP_{5}^{\otimes 7})^{0},$ one gets 
$$\begin{array}{ll}
\medskip
\dim QP_{4}^{\otimes 7} &= 2\times\binom{7}{2} + 2\times\binom{7}{3} + \binom{7}{4} = 147,\\
\dim QP_{5}^{\otimes 7} &= 3\times\binom{7}{3} + 3\times\binom{7}{4}+ \binom{7}{5} = 231.
\end{array}$$ Furthermore, it is straightforward to check that 
$$ \begin{array}{ll}
\medskip
QP_{4}^{\otimes 7}&\cong (QP_{4}^{\otimes 7})^{0}(2, 1)\bigoplus (QP_{4}^{\otimes 7})^{0}(4, 0),\\
QP_{5}^{\otimes 7}&\cong (QP_{5}^{\otimes 7})^{0}(3, 1)\bigoplus (QP_{5}^{\otimes 7})^{0}(5, 0),
\end{array}$$
where 
$$ \begin{array}{ll}
\medskip
 &\dim (QP_{4}^{\otimes 7})^{0}(2, 1) = 112,\ \dim (QP_{4}^{\otimes 7})^{0}(4, 0) = 35,\\
&\dim (QP_{5}^{\otimes 7})^{0}(3, 1) = 210,\ \dim (QP_{5}^{\otimes 7})^{0}(5, 0) = 21.
\end{array}$$
 Thus, we deduce that the weight vector of $t$ is one of the following sequences:  $$ (3,3,1),\ (3, 5),\ (5, 2, 1),\ (5, 4),$$ which yields that ${\rm Ker}((\widetilde {Sq^0_*})_{13}) \cong (QP_{13}^{\otimes 7})^{0}\bigoplus \mathbb V,$ where $$\mathbb V = (QP_{13}^{\otimes 7})^{ > 0}(3,3,1)\bigoplus (QP_{13}^{\otimes 7})^{ > 0}(3,5)\bigoplus (QP_{13}^{\otimes 7})^{ > 0}(5,2,1)\bigoplus (QP_{13}^{\otimes 7})^{ > 0}(5,4).$$
Now, basing the calculations of $(QP_{13}^{\otimes 6})^{>0}$ above, we have $$\dim (QP_{13}^{\otimes 6})^{>0} = \sum_{1\leq i\leq 4}\dim (QP_{13}^{\otimes 6})^{> 0}(\overline{\omega}_{(i)})= 60+15 + 144+40 = 259.$$
On the other hand, by Peterson \cite{Peterson}, Kameko \cite{Kameko}, the present author \cite{Phuc4}, Sum \cite{Sum1, Sum2}, we see that 
$$ \dim  (QP_{13}^{\otimes 3})^{>0} = 3,\ \dim (QP_{13}^{\otimes 4})^{>0} = 23,\ \ \dim (QP_{13}^{\otimes 5})^{>0} = 105.$$
Combining these data, one gets
$$  \dim (QP_{13}^{\otimes 7})^{0} = \sum_{1\leq  k\leq 6}\binom{7}{k}\dim (QP_{13}^{\otimes\, k})^{>0} =3\times\binom{7}{3} + 23\times \binom{7}{4} + 105\times \binom{7}{5} + 259\times\binom{7}{6} = 4928,$$
and $\dim \mathbb V = 343,$ where
$$ \begin{array}{ll}
\medskip
&\dim (QP_{13}^{\otimes 7})^{ > 0}(3,3,1) = 21,\ \ \dim (QP_{13}^{\otimes 7})^{ > 0}(3,5) = 28,\\
&\dim (QP_{13}^{\otimes 7})^{ > 0}(5,2,1) = 168,\ \ \dim (QP_{13}^{\otimes 7})^{ > 0}(5,4) = 126.
\end{array}$$ 
Of course, the dimension of $\mathbb V$ is obtained by direct calculations using Kameko's criterion on inadmissible monomials. Therefore, we may assert that $$\dim QP_{13}^{\otimes 7} = \dim {\rm Ker}((\widetilde {Sq^0_*})_{13}) +\dim QP_{3}^{\otimes 7} = 4928+343+63= 5334.$$
Similarly, we get
$$ \begin{array}{ll}
\medskip
 \dim QP_{13}^{\otimes 8} &= \sum_{1\leq  k\leq 7}\binom{8}{k}\dim (QP_{13}^{\otimes\, k})^{>0} + \dim (QP_{13}^{\otimes 8})^{>0}\\
\medskip
&= 3\times\binom{8}{3} + 23\binom{8}{4} + 105\times \binom{8}{5} + 259\times\binom{8}{6} + 406\times \binom{8}{7} + 434 = 18592.
\end{array}$$
\end{rema}

Now, the following theorem is obtained by using the above results and direct calculations.
\begin{thm}\label{dl78}
Let $n$ be a positive integer with $n\leq 15.$ Then, the dimensions of $QP^{\otimes 7}_{n}$ and $QP^{\otimes 8}_{n}$ are given by the following table:

\centerline{
\scalebox{0.9}{
\begin{tabular}{c|ccccccccccccccccc}
$n$  &$1$ & $2$ & $3$ & $4$ & $5$ & $6$ & $7$ &$8$ & $9$ & $10$ & $11$ & $12$ & $13$ & $14$ & $15$ \cr
\hline
\ $\dim QP_{n}^{\otimes 7}$ & $7$ & $21$ & $63$  & $147$ & $231$ & $427$ & $729$ &$1238$ & $1785$ & $2792$ & $3900$  & $3983$ & $5334$ & $7091$& $9472$  \cr
\hline
\ $\dim QP_{n}^{\otimes 8}$ & $8$ & $28$ & $92$ &$238$ &$434$ & $868$ & $1598$ & $2863$ & $4515$ & $7412$ & $11151$  & $13209$ & $18592$ & $25872$ & $35723$  \cr
\end{tabular}}
}
\end{thm}

It should be noted that $QP_{n}^{\otimes 7} = (QP_{n}^{\otimes 7})^{0}$ for $n\leq 6,$ and $QP_{n}^{\otimes 8} = (QP_{n}^{\otimes 8})^{0}$ for $n\leq 7.$ The theorem has been proved by Mothebe, Kaelo, and Ramatebele \cite{MKR} for the dimension of $QP_{7}^{\otimes 7}$ and of $QP_{n}^{\otimes h}$ for $h = 7,\, 8$ and $8\leq n\leq 12.$ 

\begin{note}
Let $h$ be an even positive integer with $h\in \{6, 8, 10, 12\}.$ Consider $n= 12$ and the kernel ${\rm Ker}((\widetilde {Sq^0_*})_{12}) $ of the Kameko homomorphism $(\widetilde {Sq^0_*})_{12}: QP^{\otimes h}_{12} \longrightarrow QP^{\otimes h}_{\frac{12-h}{2}}.$ For each $0\leq d\leq 3,$ we put
$$ \mathscr C(d, 12):= \big\{t_l^{2^{d}-1}\mathsf{q}_{l}(t)|\, t\in \mathscr {C}^{\otimes h-1}_{13-2^{d}},\, 1\leq l\leq h\big\}\setminus \varphi\big(\mathscr C^{\otimes h}_{\frac{12-h}{2}}\big),$$
where the map $\mathsf{q}_{l}: P^{\otimes h-1}\to P^{\otimes h}$ is a homomorphism of $\mathcal A$-algebras and determined by 
$$ \mathsf{q}_{l}(t_j) = \left\{ \begin{array}{ll}
{t_j}&\text{if }\;1\leq j \leq l-1, \\
t_{j+1}& \text{if}\; l\leq j \leq h-1.
\end{array} \right.$$
According Mothebe, and Uys \cite{Mothebe}, if $f\in \mathscr C(d, 12),$ then $f$ is admissible. Moreover, $(\widetilde {Sq^0_*})_{12}([f]) = [0],$ which implies that $[f]\in {\rm Ker}((\widetilde {Sq^0_*})_{12}).$ So, based upon Subsections \ref{pl2}, \ref{pl3}, and the results in \cite{MKR}, we notice that
\begin{corl}
For $h\in \{6,8, 10, 12\}$, the dimension of ${\rm Ker}((\widetilde {Sq^0_*})_{12})$ is given as follows:
$$ \begin{array}{ll}
\medskip
\dim {\rm Ker}((\widetilde {Sq^0_*})_{12}) &=  \sum_{0\leq d\leq 3}|\mathscr C(d)| - \big(\big|\bigcap_{0\leq d\leq 3}\mathscr C(d)\big| + \big|\varphi\big(\mathscr C^{\otimes h}_{\frac{12-h}{2}}\big)\big|\big)\\
& =\left\{\begin{array}{ll} 
960 &\mbox{if $h = 6$},\\
13181&\mbox{if $h = 8$},\\
100870&\mbox{if $h = 10$},\\
550846&\mbox{if $h = 12$},\\
\end{array}\right.
\end{array}$$
where $\varphi$ is the up Kameko map $P_{\frac{12-h}{2}}^{\otimes h}\longrightarrow P_{12}^{\otimes h},\, t\longmapsto \prod_{1\leq j\leq h}t_jt^{2}.$\\  Moreover, ${\rm Ker}((\widetilde {Sq^0_*})_{12})$ is an $\mathbb Z/2$-vector space with a basis consisting of all the classes represented by the admissible monomials in $\mathscr C(d, 12).$
\end{corl}

\end{note}

Applying Theorem \ref{dl78}, we claim that

\begin{thm}\label{dl79}
For $h\in \{7, 8\},$ and let $n$ be a positive integer with $n\leq 15.$ Then, the dimension of the coinvariant $\mathbb Z/2 \otimes_{GL_h} {\rm Ann}_{\overline{\mathcal A}}[P_{n}^{\otimes h}]^{*}$ is determined by
$$ \begin{array}{ll}
\dim \mathbb Z/2 \otimes_{GL_h} {\rm Ann}_{\overline{\mathcal A}}[P_{n}^{\otimes h}]^{*}=\dim [QP_{n}^{\otimes h}]^{GL_h}
&=\left\{\begin{array}{ll}
1 &\mbox{if $n = 15$},\\[1mm]
0 &\mbox{otherwise},
\end{array}\right.
\end{array}$$
Moreover
$$ \begin{array}{ll}
\mathbb Z/2 \otimes_{GL_h} {\rm Ann}_{\overline{\mathcal A}}[P_{n}^{\otimes h}]^{*}
&=\left\{\begin{array}{ll}
\langle [x_1^{(0)}x_2^{(0)}x_3^{(0)}x_4^{(0)}x_5^{(0)}x_6^{(0)}x_7^{(15)}] \rangle &\mbox{if $h = 7,$ $n = 15$},\\[1mm]
\langle [x_1^{(0)}x_2^{(0)}x_3^{(0)}x_4^{(0)}x_5^{(0)}x_6^{(0)}x_7^{(0)}x_8^{(15)}] \rangle &\mbox{if $h = 8,$ $n = 15$}.
\end{array}\right.
\end{array}$$
\end{thm}

We shall give the sketch of proof of the case $h = 7,\, n = 11.$ The proofs of other cases use similar idea. First of all, we observe that since the homomorphism $(\widetilde {Sq^0_*})_{11}: QP_{11}^{\otimes 7}\to QP_{2}^{\otimes 7}$ is an epimorphism of $\mathbb Z/2[GL_7]$-modules, one has an isomorphism $QP_{11}^{\otimes 7}\cong {\rm Ker}((\widetilde {Sq^0_*})_{11})\bigoplus QP_{2}^{\otimes 7}.$ By Theorem \ref{dl78}, $QP_{2}^{\otimes 7}$ is $21$-dimensional with the basis $\{[t_it_j]|\, 1\leq i < j\leq 7\}.$ Denote by ${\rm adm}_k$ the admissible monomials of the form $t_it_j,\, 1\leq i<j\leq 7.$ If $g\in [QP_{2}^{\otimes 7}]^{GL_7},$ then $g\equiv \sum_{1\leq k\leq 21}\beta_k{\rm adm}_k,$ in which $\beta_k\in \mathbb Z/2.$ Applying the linear maps $\theta_i$ and the relations $\theta_i(g)\equiv g$ for $1\leq i\leq 7,$ one gets $\beta_k =  0,\, 1\leq k\leq 21,$ which implies that $[QP_{2}^{\otimes 7}]^{GL_7} = 0.$ Thus, $$ \dim [QP_{11}^{\otimes 7}]^{GL_7}\leq \dim [{\rm Ker}((\widetilde {Sq^0_*})_{11})]^{GL_7}.$$
We now show that
$${\rm Ker}((\widetilde {Sq^0_*})_{11})\cong QP_{11}^{\otimes 7}(3, 2,1)\bigoplus QP_{11}^{\otimes 7}(3, 4)\bigoplus QP_{11}^{\otimes 7}(5, 1,1)\bigoplus QP_{11}^{\otimes 7}(5, 3).$$
Indeed, suppose that $t$ is an admissible monomial in $P_{11}^{\otimes 7}$ such that $[t]\in {\rm Ker}((\widetilde {Sq^0_*})_{11}).$ Notice that $t_1^{7}t_2^{3}t_3\in P_{11}^{\otimes 7}$ is the minimal spike monomial with $\omega(t_1^{7}t_2^{3}t_3) = (3,2,1),$ and $\deg(t)$ is odd. So, since $[t]\in {\rm Ker}((\widetilde {Sq^0_*})_{11}),$ and $[t]\neq [0],$ one has that either $\omega_1(t) = 3$ or $\omega_1(t) = 5.$ This means that either $t = t_it_jt_ku^{2},\, 1\leq i<j<k\leq 7,\, u\in P_{4}^{\otimes 7}$ or  $t = t_it_jt_kt_lt_mu^{2},\, 1\leq i<j<k<l<m\leq 7,\, u\in P_{3}^{\otimes 7}.$ By Kameko's criterion on inadmissible monomials, we must have that $u$ is an admissible monomial. So, as shown above, we see that if $u\in P_{3}^{\otimes 7},$ then $\omega(u)\in \{(1,1),\, (5, 0)\},$ and that if $u\in P_{4}^{\otimes 7},$ then $\omega(u)\in \{(2,1),\, (4, 0)\}.$ Therefore, one must have that $\omega(t)\in \{(3,2,1),\, (3,4),\, (5,1,1),\, (5,3)\}.$ This leads to the above assert.  

We have shown that 
$$ \begin{array}{ll}
\medskip
&\dim (QP_{11}^{\otimes 6})^{> 0}(3,2,1) = 16,\ \ \dim (QP_{11}^{\otimes 6})^{> 0}(3,4) = 24,\\
& \dim (QP_{11}^{\otimes 6})^{>0}(5,1,1) = 30,\ \ \dim (QP_{11}^{\otimes 6})^{>0}(5,3) = 45,
\end{array}$$
which implies $\dim (QP_{11}^{\otimes 6})^{> 0} = 16 + 24 + 30 + 45 = 115.$ On the other hand, by Peterson \cite{Peterson}, $\dim (QP_{11}^{\otimes h})^{> 0} = 0$ for $h = 1,\, 2,$ by Kameko \cite{Kameko}, $\dim (QP_{11}^{\otimes 3})^{> 0} = 8,$ by Sum \cite{Sum1, Sum2}, $\dim (QP_{11}^{\otimes 4})^{> 0} = 32,$ and by Tin \cite{Tin}, $\dim (QP_{11}^{\otimes 5})^{> 0} = 75,$ and therefore 
$$ \dim (QP_{11}^{\otimes 7})^{0} = 8\times \binom{7}{3} + 32\times \binom{7}{4} + 75\times \binom{7}{5} +  115\times \binom{7}{6} = 3780.$$
Then, by direct computations, it may be concluded that $QP_{11}^{\otimes 7}(3, 2,1) = (QP_{11}^{\otimes 7})^{0}(3, 2,1),$ and 
$$ \begin{array}{ll}
\medskip
&\dim QP_{11}^{\otimes 7}(3, 2,1) = 2352,\ \dim QP_{11}^{\otimes 7}(3, 4) = 392,\\
& \dim QP_{11}^{\otimes 7}(5, 1,1) = 540,\ \dim QP_{11}^{\otimes 7}(5,3) = 595.
\end{array}$$
We also note that $${\rm Ker}((\widetilde {Sq^0_*})_{11})\cong (QP_{11}^{\otimes 7})^{0}\bigoplus (QP_{11}^{\otimes 7})^{>0}(3, 4)\bigoplus (QP_{11}^{\otimes 7})^{>0}(5, 1,1)\bigoplus (QP_{11}^{\otimes 7})^{>0}(5, 3),$$
where $$ \dim (QP_{11}^{\otimes 7})^{>0}(3, 4) = 14,\ \dim (QP_{11}^{\otimes 7})^{>0}(5, 1,1) = 15,\ \dim (QP_{11}^{\otimes 7})^{>0}(5,3) = 70.$$
Thus, ${\rm Ker}((\widetilde {Sq^0_*})_{11})$ has dimension $3879.$ From these results, it is not too difficult to verify that $[{\rm Ker}((\widetilde {Sq^0_*})_{11})]^{GL_7}$ is zero, and so is $\mathbb Z/2 \otimes_{GL_7} {\rm Ann}_{\overline{\mathcal A}}[P_{11}^{\otimes 7}]^{*}$ since $\dim (\mathbb Z/2 \otimes_{GL_7} {\rm Ann}_{\overline{\mathcal A}}[P_{11}^{\otimes 7}]^{*}) = \dim [QP_{11}^{\otimes 7}]^{GL_7}\leq \dim [{\rm Ker}((\widetilde {Sq^0_*})_{11})]^{GL_7}.$ 

\medskip

Now, by Theorem \ref{dl79}, it can be easily seen that 
$$ \begin{array}{ll}
 Tr_7^{\mathcal A}([x_1^{(0)}x_2^{(0)}x_3^{(0)}x_4^{(0)}x_5^{(0)}x_6^{(0)}x_7^{(15)}]) &= [\psi_7(x_1^{(0)}x_2^{(0)}x_3^{(0)}x_4^{(0)}x_5^{(0)}x_6^{(0)}x_7^{(15)})]\\
\medskip
& = [\lambda_0^{6}\lambda_{15}] = h_0^{6}h_4\in {\rm Ext}_{\mathcal A}^{7, 7+15}(\mathbb Z/2, \mathbb Z/2),\\
 Tr_8^{\mathcal A}([x_1^{(0)}x_2^{(0)}x_3^{(0)}x_4^{(0)}x_5^{(0)}x_6^{(0)}x_7^{(0)}x_8^{(15)}]) &= [\psi_8(x_1^{(0)}x_2^{(0)}x_3^{(0)}x_4^{(0)}x_5^{(0)}x_6^{(0)}x_7^{(0)}x_8^{(15)})] \\
&= [\lambda_0^{7}\lambda_{15}] = h_0^{7}h_4\in {\rm Ext}_{\mathcal A}^{8, 8+15}(\mathbb Z/2, \mathbb Z/2).
\end{array}$$
On the other hand, according to Bruner \cite{Bruner}, for each $1\leq n\leq 15,$ one has that ${\rm Ext}_{\mathcal A}^{7, 7+n}(\mathbb Z/2, \mathbb Z/2) = 0$ for $n\not\in \{11, 15\},$ and ${\rm Ext}_{\mathcal A}^{8, 8+n}(\mathbb Z/2, \mathbb Z/2) = 0$ for $n\neq 15,$ and that the elements $h_0^{2}Ph_2 = h_1^{2}Ph_1,$ $h_0^{6}h_4,$ $h_0^{7}h_4$ are non-zero in ${\rm Ext}_{\mathcal A}^{7, 7+11}(\mathbb Z/2, \mathbb Z/2),$ ${\rm Ext}_{\mathcal A}^{7, 7+15}(\mathbb Z/2, \mathbb Z/2),$ ${\rm Ext}_{\mathcal A}^{8, 8+15}(\mathbb Z/2, \mathbb Z/2),$ respectively. As a consequence, we immediately obtain the following.
\begin{corl}
For $h\in \{7, 8\},$ and $1\leq n\leq 15,$ the cohomological transfer
$$Tr_h^{\mathcal A}: \mathbb Z/2 \otimes_{GL_h} {\rm Ann}_{\overline{\mathcal A}}[P_{n}^{\otimes h}]^{*}\to {\rm Ext}_{\mathcal A}^{h, h+n}(\mathbb Z/2, \mathbb Z/2)$$
is not an epimorphism if $h = 7,\, n = 11,$ and is an isomorphism otherwise.
\end{corl}
Noting that this statement for $h = 7,\, n = 11$ has also been shown by Ch\ohorn n, and H\`a \cite{C.H1, C.H2} using other techniques.

\end{document}